\documentclass[11pt,a4paper,reqno]{amsart}
\usepackage[applemac]{inputenc}
\usepackage[T1]{fontenc}
\usepackage{hyperref}
\usepackage{amsmath}
\usepackage{amsthm}
\usepackage{amsfonts}
\usepackage{amssymb}
\usepackage{graphicx}
\usepackage{color}
\usepackage{bm}
\usepackage{amsbsy}
\usepackage{mathrsfs}
\usepackage{bbm}
\usepackage{dsfont}
\usepackage{graphicx}
\usepackage{subcaption}

\newtheorem{theorem}{Theorem}[section]
\newtheorem{lemma}[theorem]{Lemma}
\newtheorem{ass}[theorem]{Assumption}
\newtheorem{proposition}[theorem]{Proposition}
\newtheorem{corollary}[theorem]{Corollary}
\newtheorem{definition}[theorem]{Definition}

\theoremstyle{remark}
\newtheorem{remark}[theorem]{\it \bf{Remark}\/}

\numberwithin{equation}{section}
\catcode`@=11
\def\section{\@startsection{section}{1}%
	\z@{1.5\linespacing\@plus\linespacing}{.5\linespacing}%
	{\normalfont\bfseries\large\centering}}
\catcode`@=12
\newcommand{\be}{\begin{equation}}
	\newcommand{\ee}{\end{equation}}
\newcommand{\bea}{\begin{eqnarray}}
	\newcommand{\eea}{\end{eqnarray}}
\newcommand{\bee}{\begin{eqnarray*}}
	\newcommand{\eee}{\end{eqnarray*}}

\def\pa{\partial}

\def\CC{\mathbb{C}}
\def\NN{\mathbb{N}}
\def\RR{\mathbb{R}}
\def\ZZ{\mathbb{Z}}

\def\b{\beta}

\catcode`@=11
\def\supess{\mathop{\operator@font Sup\,ess}}
\catcode`@=12
\def\CC{\mathbb{C}}

\def\C{\matchal C}
\def\NN{\mathbb{N}}

\def\RR{\mathbb{R}}
\def\CC{\mathbb{C}}

\def\XX{\mathbb{X}}
\def\ZZ{\mathbb{Z}}

\def\a{\alpha}
\def\e{\varepsilon}

\def\bar#1{{\overline #1}}

\def\R2+{\RR ^2_+}

\def\calE{{\mathcal E}}

\def\calH{{\mathcal H}}

\def\calJ{{\mathcal J}}

\def\calL{{\mathcal L}}
\def\calM{{\mathcal M}}

\def\calO{{\mathcal O}}
\def\calP{{\mathcal P}}

\def\calR{{\mathcal R}}
\def\calS{{\mathcal S}}
\def\calT{{\mathcal T}}

\def\pa{\partial}

\def\lim{\mathop{\rm lim}}

\def\supp{{\rm supp}~}

\def\sup{\mathop{\rm sup}}

\def\l{\lambda}

\def\log{{\rm log}}

\def\pa{\partial}

\def\pa{\partial}
\def\la{\langle}
\def\matchal{\mathcal}
\def\ra{\rangle}

\def\L{\mathcal L}

\begin{document}
	
	\title[]{On stability of self-similar blowup for mass supercritical NLS}
	
	\author{Zexing Li}
	\address{Laboratoire AGM \\ CY Cergy Paris Universit\'e \\ 2 avenue Adolphe Chauvin \\ 95300 Pontoise \\ France}
	\email{zexing.li@u-cergy.fr}

\begin{abstract} We consider the mass supercritical (NLS) $ i\pa_t u+\Delta u+u|u|^{p-1}=0$ in dimension $d\ge 1$ in the mass-supercritical range $0<s_c=\frac{d}{2}-\frac{2}{p-1}<1$. The existence of self-similar blow up dyamics is known \cite{MR2729284}, and suitable self-similar blow up profiles were constructed in \cite{MR4250747}. In this work, we establish the finite codimensional nonlinear asymptotic stability in homogeneous Sobolev space and energy space $H^1$ for a large class of self-similar profiles. The heart of the proof is, following the approach of Beceanu \cite{beceanu2011new},  the derivation of Strichartz dispersive estimates for matrix operators of the form
	$$\mathcal{H} = \left( \begin{array}{cc} \Delta_b -1 &  \\ &  -\Delta_{-b} +1\end{array} \right) -ibs_c + \left( \begin{array}{cc} W_1 & W_2  \\ -\overline{W_2} & -W_1   \end{array} \right),$$
	$b > 0$ and
	\[ \Delta_b = \Delta + ib\left(\frac d2 + x\cdot\nabla\right), \]
	inhomogeneous Sobolev spaces. The deformed Laplacian $\Delta_b$ arises from the renormalization and its operator group $e^{it\Delta_b}$ exhibits self-similar dispersion, which not only recovers the free Strichartz but also enables an extension of resolvent families. Compared with Strichartz estimates based on $\Delta$, this one has a larger admissible region, works for arbitrarily small polynomial decaying potential, and requires no spectral assumption.	\end{abstract}

	\maketitle

\section{Introduction}

\subsection{Setting of the problem}

We consider the nonlinear Schr\"odinger equation
\be \left\{\begin{array}{l} i\pa_t u+\Delta u+u|u|^{p-1}=0\\ u_{|t=0}=u_0\end{array}\right. \tag{NLS} \label{eqNLS}\ee 
for spatial dimension $d \ge 1$ and $p > 1$. The scaling symmetry \[ u_\l (t, x) := \l^{\frac{2}{p-1}}u(\l^2 t, \l x), \quad \l > 0\]
allows us to compute the critical Sobolev norm invariant
\[ \| u_\l (t, \cdot) \|_{\dot{H}^{s_c}} = \| u (\l^2 t, \cdot) \|_{\dot{H}^{s_c}},\quad s_c := \frac{d}{2} - \frac{2}{p-1}.\]
In this paper, we are interested in describing singularity formation mechanisms for nonlinearities above the mass and below the energy conservation laws:
\[ 0 < s_c < 1. \]

\subsection{Type I and Type II blowup.} 
We now recall some key facts related to the description of singularity formation for \eqref{eqNLS} with $0 < s_c < 1$. \\

\noindent\emph{The self-similar law.} From an elementary scaling argument, all blow up solutions satisfy the scaling self-similar lower bound
\be \| u(t)\|_{\dot{H}^{\sigma}} \gtrsim (T-t)^{-\frac{\sigma - s_c}{2}}  \label{eqscaling}\ee where $s_c<\sigma\le 1$ is supercritical. We say a solution blows up self-similarily if it saturates the self-similar law
\be \| u(t)\|_{\dot{H}^{\sigma}} \sim (T-t)^{-\frac{\sigma - s_c}{2}},\quad s_c < \sigma \le 1  \label{eqscalingbis}.
\ee

\noindent \emph{Type II blow}. Type II solutions which blow up strictly faster than \eqref{eqscalingbis} have been constructed both in the critical case $s_c=0$, \cite{MR1852922,MR2150386,MR2169042,MR1995801,MR2061329,MR2116733,MR2122541}), and the supercritical case through the derivation of "ring" solutions which concentrate on a circle, \cite{MR2370365, MR2248831,MR3161317,MR3324912}.\\

\noindent \emph{Type I blow}. However, for $0<s_c<1$, type I blow is conjectured to exist and generate stable blow up dynamics, \cite{MR1696311,MR1487664}. More precisely, the self-similar renormalization 
$$ u(t, x) = \frac{e^{i\gamma(t)}}{(\l(t))^{\frac{2}{p-1}}} v\left(\tau,y\right)$$
with the explicit laws 
\be
\label{eqlaw}
\lambda(t) = \sqrt{2b(T-t)},\quad \gamma(t) = \tau(t), \ \  \tau(t) = -\frac{1}{2b}\ln (T- t), \ \ y=\frac{x}{\lambda(t)}
\ee
maps \eqref{eqNLS} onto the renormalized flow 
\be
\label{enoivnevneiovnovneonvi}
i\pa_\tau v+\Delta v -v+\frac{ib}{2}\left(\frac{2}{p-1}v+y\cdot\nabla v\right)+v|v|^{p-1}=0.
\ee In the range $0 < s_c < 1$, admissible self-similar profiles are expected to exist from numerical grounds, \cite{mclaughlin1986focusing,lemesurier1988focusing,landman1991stability,MR1487664,budd1999new,doi:10.1137/S0036139900382395}.\\

\noindent{\bf Conjecture} (Existence of suitable profiles). {\em  Let $d\ge 1$ and $s_c \in (0, 1)$. Then there exists $b = b(d, s_c)>0$ and a smooth radially symmetric profile $Q_b$ with the following properties.\footnote{The equation \eqref{eqselfsimilar} and self-similar decay \eqref{selfsimilardecay} are essential, in that they can imply the non-vanishing property \eqref{eqnonvanishing} and self-similar decay for higher order derivatives \eqref{selfsimilardecayhigh} via ODE analysis, see \cite[Lemma 2.2]{troy2022zero} and \cite{Limodestability} respectively. We remark that the main result of \cite{troy2022zero} was unfortunately dubious.}
	\begin{enumerate}
\item[\rm (i)]{\rm  Equation:} $Q_b$ is a stationary solution to \eqref{enoivnevneiovnovneonvi}:
\be \Delta Q_b-Q_b+ib\left(\frac{2}{p-1}Q_b+y \cdot \nabla Q_b\right)+Q_b|Q_b|^{p-1}=0 .\label{eqselfsimilar} \ee
\item[\rm (ii)]{\rm Non-vanishing}: \be Q_b(x) \neq 0 \quad \forall x \in \RR^d.  \label{eqnonvanishing}\ee
\item[\rm (iii)]{\rm Self-similar decay}: \bea
\label{selfsimilardecay}
\lim_{r \to \infty}r^{\frac{2}{p-1}} |Q_b(r)| = c_{p} > 0,&& \limsup_{r \to \infty}r^{\frac{p+1}{p-1}} |Q_b'(r)| < \infty;\\
\limsup_{r \to \infty}r^{\frac{2}{p-1} + k} |\pa_r^k Q_b'(r)| < \infty, && \forall \,\, k \ge 2.\label{selfsimilardecayhigh}
\eea 
\end{enumerate}}

The existence of such profiles has been obtained by Bahri-Martel-Rapha\"el \cite{MR4250747} in the range $0<s_c\ll1$ by bifurcating from the critical case $s_c=0$, see also \cite{MR3311593} for the case of (gKdV). Extending the branch to higher values of $s_c$ is a delicate nonlinear ODE problem, and recently Donninger-Sh\"orkhuber \cite{arXiv:2406.16597} constructed the suitable profile for the physical scenario $d=p=3$ with computer assistance. Let us stress that the construction of self-similar profiles is central in the study of singularity formation mechanisms and has been mostly addressed in the energy supercritical case $s_c>1$, see \cite{MR3986939} for the heat equation, \cite{MR4350584} for the wave equation and \cite{MR8522,MR4445442} for compressible fluids.\\

\noindent{\em Existence of self-similar dynamics}. Self-similar profiles satisfying \eqref{selfsimilardecay} are not in the energy space $H^1$ and hence cannot be used as such to produce finite energy type I solutions. The existence of such dynamics was first obtained in \cite{MR2729284} for $0<s_c\ll 1$ using again a bifurcation argument from the critical case $s_c=0$. The analysis provides solutions that satisfy the self-similar law \eqref{eqscaling}, but does not describe the asymptotic structure of the singularity which should be given by the self-similar profile constructed in \cite{MR4250747}.

\subsection{Main results} Our aim in this paper is to provide a robust approach for the proof of the asymptotic stability of self-similar profiles. 

\subsubsection{$\dot H^{s_c+}$ stability theory}

\begin{theorem}[Finite-codimensional asymptotic stability in $\dot H^{s_c+}$] \label{thmcodimstability}
Let $d \ge 1$, $s_c\in (0, 1)$, and $b, Q_b$ satisfying \eqref{eqselfsimilar}, \eqref{eqnonvanishing}, \eqref{selfsimilardecay}. Then for any $0 < \sigma - s_c \ll 1$, there exists a Lipschitz finite-codimensional manifold $\mathcal{M} \subset \dot{H}^\sigma$ of initial data
	\[ u_0 = Q_b + \e_0 \]
	where $\| \e_0 \|_{\dot{H}^\sigma }\ll 1$, such that the corresponding solution to \eqref{eqNLS} blows up at $T = \frac{1}{2b}$ with a decomposition
	\be u(t, x) = \frac{1}{\l(t)^{\frac{2}{p-1}}}(Q_b + \e) \left(t, \frac{x}{\l(t)}\right)e^{i\tau(t)} \label{equmodrep} \ee
	where $(\lambda(t),\gamma(t),\tau(t))$ are given by \eqref{eqlaw} and
	\be \|\e(t)\|_{\dot{H}^\sigma_x}\lesssim (T-t)^{\frac{\sigma - s_c}{2}}.
	\label{eqbdd1}
		\ee
		Moreover, there exists $u_*\in \dot{H}^\sigma$ such that 
		\be
		u(t) - \frac{1}{\l(t)^{\frac{2}{p-1}}}Q_b \left( \frac{x}{\l(t)}\right)e^{i\tau(t)}  \to u_*\quad \mathrm{in}\,\,\dot{H}^\sigma \quad as\quad t\to T.
		\label{eqbdd3}
		\ee		
\end{theorem}


The result reduces the nonlinear asymptotic stability to a linear spectral stability problem. As a corollary, assuming the unstable directions of the linearized operator are generated merely by symmetries of the equation (mode stability), we can obtain the asymptotic stability without losing codimensions. 

\begin{ass}[Mode stability] \label{assmodestab}
    Let $d \ge 1, s_c \in (0, 1)$, and $b, Q_b$ satisfying \eqref{eqselfsimilar}, \eqref{eqnonvanishing}, \eqref{selfsimilardecay}. Let $\calH$ be the linearized operator around $Q_b$ in self-similar coordinates defined by \eqref{Hb}-\eqref{potential}. For any $0 < \sigma - s_c \ll 1$, the discrete spectrum of $\calH$ satisfies 
    \be
      \sigma_{\rm disc}\left( \calH \big|_{(\dot H^\sigma(\RR^d))^2} \right) \cap \left\{ z\in\CC: \Im z < b(\sigma - s_c) \right\} = \{ 0, -bi, -2bi \},
    \ee
     with the corresponding Riesz projections satisfying\footnote{They are instabilities generated by phase rotation, spatial translation and scaling symmetry respectively. See Section \ref{sec63} for detailed discussion.}
    \be
    \dim {\rm Ran} P_0 = 1,\quad \dim {\rm Ran} P_{-bi} = d, \quad \dim {\rm Ran} P_{-2bi} = 1. \label{eqRieszchar}
    \ee
\end{ass}

\begin{remark}[Verification of mode stability]\label{rmkmodestab}
Proving Assumption \ref{assmodestab} for certain self-similar profile is an independent linear problem. It requires an exact characterization of unstable spectrum for a non-self-adjoint operator, which is relatively unbounded with respect to (matrix) Schr\"odinger operators. In our forthcoming work \cite{Limodestability}, we show this property holds for the slightly mass-supercritical profiles constructed in \cite{MR4250747} for $d \le 10$, as expected from the stability result \cite{MR2729284}. 
\end{remark}

\begin{theorem}[Conditional asymptotical stability in $\dot H^{s_c+}$] \label{thmasympstab} 
 Under the same assumption of Theorem \ref{thmcodimstability} plus mode stability Assumption \ref{assmodestab}, for any $0 < \sigma - s_c \ll 1$, there exists $\epsilon_1 \ll 1$ such that for all $\| \e_0 \|_{\dot H^\sigma} < \epsilon_1$, there uniquely exists $(\l_0, x_0, \theta_0) \in \RR \times \RR^d \times \RR$ with 
  \bee
    |\l_0 - 1| + |x_0| + |\theta_0| \lesssim \| \e_0 \|_{\dot{H}^\sigma}
  \eee
  such that the initial data $ u_0 = Q_b + \e_0$
  generates a solution blowing up at $T = \frac{\l_0^2}{2b}$ satisfying the decomposition
  \be
      u(t, x) = \frac{1}{(1 - 2b \l_0^{-2} t)^{\frac{1}{p-1}}}  (Q_b + \e)\left(t, \frac{x-x_0}{\sqrt{\l_0^2-2bt}} \right)  e^{-i\left[\frac{\ln(\frac{1}{2b} - \l_0^{-2}t )}{2b} + \theta_0 \right]} \label{eqss33} \ee
  with the decaying of perturbation \eqref{eqbdd1} and the existence of blowup profile in $\dot H^\sigma$ \eqref{eqbdd3}. 
 \end{theorem}

\noindent{\em Comments on Theorem \ref{thmcodimstability} and Theorem \ref{thmasympstab}.}\\

\noindent{\em 1. Regularity.} \eqref{eqNLS} is local-wellposed in $\dot{H}^\sigma$ with $0 < \sigma - s_c \ll 1$, see Proposition \ref{proplwp}, so the evolution makes sense with merely $u_0 \in \dot{H}^\sigma$. This result is almost sharp in terms of regularity for $\dot{H}^s$ spaces, and the critical $\dot H^{s_c}$ stability is still open. \\

\noindent{\em 2. Manifold structure of Theorem \ref{thmcodimstability}}. Let $B_{\epsilon_0}^{\dot H^\sigma}:= \{ f \in \dot{H}^\sigma, \| f \|_{\dot{H}^\sigma} < \epsilon_0 \}$ with $\epsilon_0 \ll 1$, then we construct a direct sum decomposition \be\dot{H}^\sigma = X_{u} \oplus X_{cs} \label{eqXcus}\ee
with $X_{cs}$ and $X_u$ are $\RR$-linear subspaces and $X_{u}$ finite-dimensional, and a Lipschitz map $\Phi: B_{\epsilon_0}^{\dot H^\sigma} \cap X_{cs} \to X_{u}$ with $\Phi(0) = 0$, such that the manifold in Theorem \ref{thmcodimstability} can be realized as
\be\mathcal{M} = ({\rm Id} + \Phi)(B_{\epsilon_0}^{\dot H^\sigma} \cap X_{cs}) + Q_b. \label{eqmfd}\ee
		Moreover, this manifold is tangential to the center-stable hyperplane $B_{\epsilon_0}^{\dot H^\sigma} \cap X_{cs} + Q_b$. See Theorem \ref{thmfincodimstabHsigmaB} for more accurate formulation. 
		

\subsubsection{$H^1$ stability theory}

The self-similar decay \eqref{selfsimilardecay} makes $Q_b \notin \dot H^s$ of any $s \in [0, s_c]$, which raises the question for the existence, stability and behavior of Type I blowup in the energy space $H^1$.\footnote{Nevertheless, the existence of $H^1$ open set of data leading to Type I blowup with vanishing perturbation \eqref{eqbdd1} is included in the result Theorem \ref{thmasympstab} assuming Assumption \ref{assmodestab}.} Based on the proof of $\dot H^\sigma$ stability, we also prove the counterpart of Theorem \ref{thmcodimstability} and Theorem \ref{thmasympstab} in $H^1$.

\begin{theorem}[Asymptotic stability in $H^1$] \label{thmasympstabH1}
Let $d \ge 1$, $s_c\in (0, 1)$, and $b, Q_b$ satisfying \eqref{eqselfsimilar}, \eqref{eqnonvanishing}, \eqref{selfsimilardecay}, \eqref{selfsimilardecayhigh}. 

\begin{enumerate}
\item[\rm (i)]{\rm Finite-codimensional asymptotic stability:} 
For $R_1 \gg 1$, there exists a Lipschitz finite-codimensional manifold $\mathcal{M} \subset H^1$ of initial data
    such that the corresponding solution to \eqref{eqNLS} blows up at $T = \frac{1}{2b}$ with the decomposition \eqref{equmodrep} for $(\lambda(t),\gamma(t),\tau(t))$ are given by \eqref{eqlaw} and decaying of perturbation 
    \be \|\e(t)\|_{\dot{H}^\sigma \cap \dot H^1}\lesssim (T-t)^{\frac{\sigma - s_c}{2}}
	\label{eqbdd1H1} 
    \ee
    where $0 < \sigma - s_c \ll 1$. 
    Moreover, there exists $u_*\in \dot{H}^\sigma \cap \dot H^1$ and $u^* \in H^{\tilde \sigma}$ for every $0 \le \tilde \sigma < s_c$ such that 
		\bea
		u(t) - \frac{1}{\l(t)^{\frac{2}{p-1}}}Q_b \left( \frac{x}{\l(t)}\right)e^{i\tau(t)}  \to u_* && \mathrm{in}\,\,\dot{H}^\sigma\cap \dot H^1 \quad as\,\, t\to T,
		\label{eqbdd3H1} \\
        u(t) \to u^* && \mathrm{in}\,\,H^{\tilde \sigma} \quad \mathrm{as}\,\, t \to T. \label{eqL2profH1}
		\eea 
        and with $p_c = \frac{2d}{d-2s_c}$, we can quantify the critical norm blowup as $t \to T$
        \be \| u(t) \|_{L^{p_c}} \sim |\log(T-t)|^{\frac{1}{p_c}} (1+o(1)),\quad \| u(t)\|_{\dot{H}^{s_c}} \sim |\log(T-t)|^{\frac 12}(1+o(1)).\label{eqcriticalnorm} \ee
        \item[\rm (ii)]{\rm Conditional asymptotic stability:} If we further assume Assumption \ref{assmodestab}, then there exists an open set of initial data $\calO \subset H^1$ such that its solution to \eqref{eqNLS} blows up at $T = \frac{\l_0^2}{2b}$ and satisfies the decomposition \ref{eqss33} with $|(\l_0, x_0, \theta_0)| \ll 1$. For some $0 < \sigma - s_c \ll 1$, the properties \eqref{eqbdd1H1}-\eqref{eqcriticalnorm} remain true.

        \end{enumerate}
\end{theorem}

\noindent{\em Comments on Theorem \ref{thmasympstabH1}.}\\

\noindent{\em 1. Comparison with stability result \cite{MR2729284}.}
Combined with Remark \ref{rmkmodestab}, this completes \cite{MR2729284} by deriving the asymptotic stability of the self-similar profile and provides a different route map towards the stability. Besides, since our argument is independent of energy structure, it should be natural to generalize to higher regularity $H^k$. \\

\noindent{\em 2. Structure of manifold and open set.} Here we use a different subspace decomposition for $H^1$ as 
\[ H^1 = \tilde X_{cs} \oplus \tilde X_u, \]
with $\tilde X_u \subset C^\infty_c$ and being finite-dimensional. Denote $\chi_{R_1}$ to be a smooth cutoff as \eqref{eqdefchiR} and $B^{H^1}_{\epsilon_1}$ be the $\epsilon_1$-ball in $H^1$. Given an $R_1 \gg 1$, the manifold $\calM$ in (i) and open set $\calO$ in (ii) are constructed as 
\bee
  \calO &=& Q_b \chi_{R_1} + \e^*_{R_1} + B^{H^1}_{\epsilon_1} \\
  \calM &=& Q_b \chi_{R_1} + \e^*_{R_1} + ({\rm Id} + \tilde \Phi_{R_1})(B^{H^1}_{\epsilon_1} \cap \tilde X_{cs})
\eee
where $\e^*_{R_1} \in \tilde X_u \subset C^\infty_c$ with $\| \e^*_{R_1} \|_{H^1} = o_{R_1\to \infty} (1)$, and  $\tilde \Phi_{R_1}: B^{H^1}_{\epsilon_0} \cap \tilde X_{cs} \to \tilde X_u$, $\tilde \Phi_{R_1}(0) = 0$ is a Lipschitz map. See Section \ref{sec73} for detailed description.\\

\noindent{\em 3. Blowup of the critical norm.} 
        The asymptotics of critical norm \eqref{eqcriticalnorm} answers a question in \cite{MR2427005} for the expectation of sharp power low of its logarithmic divergence. It also indicates that subcritical limit profile $u^* \notin \dot H^{s_c}$ as stated in the stability result \cite{MR2729284}. 

\subsection{Strichartz estimates for the self-similar propagator} 

Let the deformed Laplacian be 
\bee \Delta_b := \Delta + ib \left(\frac{d}{2} + y \cdot\nabla \right).\eee

\noindent{\em Structure of the linearized operator}.  The linearization of the time-dependent flow \eqref{enoivnevneiovnovneonvi} close to a self-similar solution $Q_b$ solution to \eqref{eqselfsimilar} leads to a non-self-adjoint linearized problem of the form
\be i \pa_\tau Z + \calH Z = F(Z)\label{eqZ}. \ee
Here $F$ is the nonlinear term, $Z = \left( \begin{array}{c} Z_1 \\ Z_2 \end{array} \right)\in  \CC^2$ and 
\begin{align}
	\mathcal{H} = \left( \begin{array}{cc} \Delta_b -1 &  \\ &  -\Delta_{-b} +1\end{array} \right) - ibs_c + \left( \begin{array}{cc} W_1 & W_2  \\ -\overline{W}_2 & -W_1   \end{array} \right)  
	\label{Hb}
\end{align}
for the explicit potentials
\begin{align} 	W_1 = \frac{p+1}{2}|Q_b|^{p-1}, &\quad W_2 = \frac{p-1}{2}Q_b^2|Q_b|^{p-3}\label{potential}  
\end{align}
which under \eqref{selfsimilardecay} satisfy a polynomial decay
\be |W_i(x)| \lesssim \la x \ra^{-\a},\quad |\nabla W_i(x)| \lesssim \la x \ra^{-\a - 1} \label{eqpotbd0} \ee
for $i = 1, 2$ for some $\a > 0$.\\

\noindent{\em Decay for the linearized operator}. Various strategies have been used in the literature to address the derivation of local-in-space decay estimates for linearized operators close to self-similar profiles: local Virial estimates \cite{MR1896235,MR2150386,MR3090180}, self-adjoint spectral estimates \cite{MR3986939}, distorted Fourier transform \cite{MR2471133}, semigroup estimates \cite{MR3537340,MR4359478,zbMATH07952722}. 
Strichartz estimates were derived in the pioneering work \cite{MR3662440} to control the flow around self-similar blowup of spherically symmetric 3D energy-critical nonlinear wave equations in the renormalized light cone.\\ 

\noindent{\em Strichartz bound}. Our second main result in this paper is the derivation of global in space Strichartz estimates which can be directly used for the proof of the nonlinear stability Theorem \ref{thmcodimstability}. One edge of Strichartz estimate here, compared with semigroup method, is that it requires at most one derivative, hence works nicely for weak and non-analytic nonlinearities. 

Our approach follows the framework developed by Beceanu for the derivation of Strichartz in the $b=0$ case.

\begin{theorem}[Strichartz estimates]\label{thmStrichartz} Let $d \ge 1$, $b > 0$ and $s_c \in (0, 1)$. Let the linear matrix Schr\"odinger system,
	\be i \partial_\tau Z + \mathcal{H}Z = F \label{eqls} \ee
	with $\mathcal{H}$ as in \eqref{Hb}, $W_1$ and $W_2$ satisfying \eqref{eqpotbd0} for some $\a > 0$. Then there exists $\delta_1 >0$ such that for all $\sigma \in (s_c, s_c + \delta_1)$, there exists a finite rank projection operator $P_{\rm disc}:\dot{H}^\sigma\to \dot{H}^\sigma$ such that the following holds. Let $\nu \le b(\sigma-s_c)$ and $P_{\rm ess} := {\rm Id} - P_{\rm disc}$,
	then \begin{align} 
	\label{eqStrichartz} \|e^{\nu t} P_{\rm ess} Z \|_{L^{q_1}_\tau \dot{W}^{\sigma, p_1}_x} \lesssim 
	\| P_{\rm ess}Z(0)\|_{\dot{H^\sigma}} + \| e^{\nu t}P_{\rm ess} F \|_{L^{q_2'}_\tau \dot{W}^{\sigma, p_2'}_x}
\end{align}
	where $(q_i, p_i)$ for $i = 1, 2$ satisfies
	\bea 
	(q, p, d) \neq (2, \infty, d),\quad (q, p) \in \left\{ q \ge 2, 2 < p \le \infty, \frac 2q + \frac dp \ge \frac d2 \right\} \cup \{(\infty, 2)\}. \label{eqadmissible}
	\eea

\end{theorem}

\noindent{\em Comments on Theorem \ref{thmStrichartz}.}\\

\noindent{\em 1. Deformed Laplacian $\Delta_b$.}

Compared with the standard Laplacian $\Delta$, the correction part $ib(\frac d2 + y\cdot \nabla)$ is not relatively bounded and changes the spectrum drastically to be $\sigma(\Delta_b) = \RR$ with $b \neq 0$ (see Lemma \ref{lemctsspec}).
Current results for variable-coefficient Schr\"odinger operators (for example \cite{burq2004strichartz,staffilani2002strichartz,MR2565717}) do not directly apply since $\Delta_b$ is no longer asymptotically flat. 
Fortunately, the semigroup $e^{it\Delta_b}$ enjoys good structure and has explicit representation \eqref{eqrepuhat}-\eqref{eqrepu} (first derived by Carles in \cite{zbMATH02028533,zbMATH06248068}). Here we list two main features:

\begin{enumerate}
\item \emph{$\dot{H}^\sigma$-energy decay}. The commutator of $\Delta_b$ with fractional differentiation $D^\sigma$ gives constant $ib\sigma$, leading to an exponential decay $e^{b\sigma t}$ of $e^{it\Delta_b}$ in $\dot{H}^\sigma$ as indicated in \eqref{eqStrichartz}. 
\item \emph{Self-similar dispersion}. The self-similar rescaling in time makes $e^{it\Delta_b}$ exhibit exponential dispersive decay as $t \rightarrow \infty$. As an important corollary, the resolvent families can be extended onto and even beyond the real axis.
\end{enumerate}

\noindent{\em 2. Comparison with Donninger's work \cite{MR3662440}.} 
In \cite{MR3662440}, Donninger first represents the semigroup $e^{t\mathbf{L}}$ by resolvents $(\l - \mathbf{L})^{-1}$ via Laplace transform, and further solves for Green's functions to represent the resolvents. Then $e^{t\mathbf{L}}$ is a sum of free part and perturbation given by potential, and Strichartz estimates of each part follow intricate spatial estimates. This explicit construction is very clear in treating the radial case.

In this work, we exploit nice structure of the free operator group $e^{it\Delta_b}$ and deal with the potential with a more abstract framework invented by Beceanu \cite{beceanu2011new}, which utilizes the convolution structure of the semigroup and resolvent identities to reduce the Strichartz estimates to a uniform invertibility problem. The semigroup $e^{it\calH}$ is decomposed into a composition of the free part and the potentials (Birman-Schwinger principle) rather than a summation. This strategy avoids explicitly solving resolvent equations and thus can handle non-radial case in the whole space. 
\\

\noindent{\em 3. Comparison with Strichartz estimates for Schr\"odinger equation with potentials.}
There have been many works devoting to Strichartz estimates of $e^{it(-\Delta + W)}$ or its matrix version, arising from Schr\"odinger equation with potentials or linearization around solitons \cite{MR3662440,MR1105875,rodnianski2004time,MR2106340,MR2484934}. Since we replace the Laplace with a deformed one $\Delta_b$, our Strichartz estimates have the following new features:
\begin{enumerate}
	\item  {\em Enlarged admissible region of $(q, p)$.} The self-similar dispersive rate at infinity implies a larger admissible region \eqref{eqadmissible} (Figure 1) for $e^{it\Delta_b}$ \eqref{eqStrichartzfree} and ultimately for $e^{it\calH}$, instead of just $\frac 2q + \frac dp = \frac d2$.
	
	\begin{figure}[h!]
		\centering
		\begin{subfigure}[b]{0.3\linewidth}
			\includegraphics[width=\linewidth]{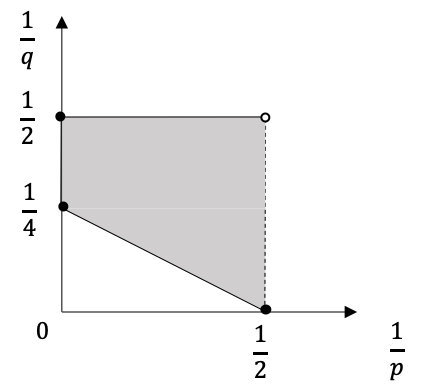}
			\caption{$d=1$}
		\end{subfigure}
		\begin{subfigure}[b]{0.3\linewidth}
			\includegraphics[width=\linewidth]{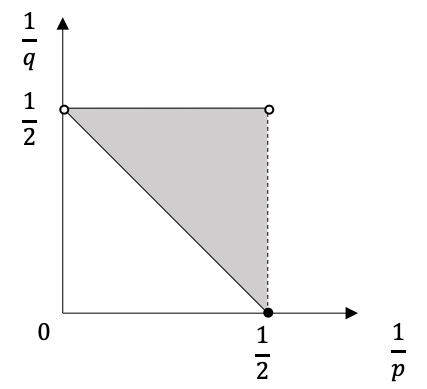}
			\caption{$d=2$}
		\end{subfigure}
		\begin{subfigure}[b]{0.3\linewidth}
			\includegraphics[width=\linewidth]{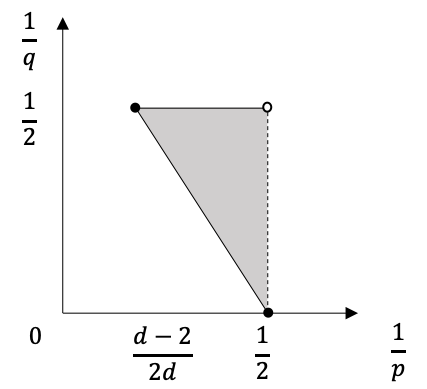}
			\caption{$d\ge 3$}
		\end{subfigure}
		\caption{The admissible region for Strichartz pair $(q, p)$.}
		\label{fig:coffee}
	\end{figure}
	\item  {\em Decay of potential.} To make the potential perturbative in $e^{it(-\Delta + V)}$, we need at least $|x|^{-2}$ decay of $V$ (see \cite{MR2211154} for counterexamples). However, the deformed case $b \neq 0$ can admit arbitrarily small polynomial decay ($\a > 0$ in the sense of \eqref{eqpotbd0}). This amounts to different boundedness of the limiting extension $(-\Delta - (\l + i0))^{-1}$ and $(-\Delta_b - (\l+i0))^{-1}$.
	\item {\em No spectral assumption.} Previous results usually assume the absence of resonance or embedded eigenvalues  \cite{rodnianski2004time,MR2279551}, especially in the scenario of matrix-valued complex-valued potential  \cite{beceanu2011new,MR2484934}. They would essentially affect the dispersive behavior, see \cite{MR2279551} for more discussion. However, we can avoid that trouble by changing spaces $\dot{H}^{\sigma}$ thanks to those two features of $\Delta_b$. When $\sigma$ varies, the $\dot{H}^{\sigma}$-decay implies a shift of essential spectrum, while the discrete spectrum basically remains the same by uniform treatment with extended resolvent families. This decoupling effect allows us to avoid embedded exceptional values.\\
	\end{enumerate}

\noindent{\em 4. Finiteness of discrete spectrum.} The operator $P_{\rm disc}$ is the Riesz projection for discrete spectrum of $\calH$ (see Proposition \ref{prop3}). While the discrete spectra may accumulate toward the essential spectrum (see \cite[Proposition 2.9]{beceanu2011new}) for Schr\"odinger operator, the finite rank of $P_{\rm disc}$ is an implication of the self-similar dispersion. This unconditional finiteness of unstable spectrum for linearized operator around blowup can also be proven through semigroup theory (see \cite{MR4359478} or \cite{zbMATH07588427}). \\

\noindent{\em 5. Restriction on $\sigma$.} The condition $0 < \sigma -s_c \ll 1$ is only to ensure the absence of exceptional value on essential spectrum by the decoupling effect. 
We also obtain $\dot H^\sigma \cap \dot H^1$ Strichartz (Lemma \ref{lemStrichartz1}) as persistence of regularity for linear equation, and generalization to higher regularity is well-expected. However, due to the existence of embedded resonance $\calH (iQ_b, -i\overline{Q_b} )^T = 0$, the Strichartz estimate at sharp regularity is very unclear, even for the radial case\footnote{For wave equation, the linearized operator has no embedding eigenmode in radial case, but has one generated by Lorentz symmetry in the non-radial case.}\\

\subsection{Outline of the proofs and the paper.}\label{s13}

\subsubsection{Proof of Strichartz estimates.} We proof Theorem \ref{thmStrichartz} in Section \ref{s2}-\ref{s5} following the framework of Beceanu \cite{beceanu2011new}. \\

\noindent{\em Step 1. Spectral analysis for scalar operators.}

 In Section \ref{s2}, we will prove self-similar dispersive estimate and free Strichartz estimate for $\Delta_b$ as a self-adjoint operator on $L^2$, extend resolvent families as bounded operators from $L^{p'}$ to $L^{p}$ $(p > 2)$ and discuss their properties as resolvents. Since $\calL(L^{p'} \to L^p)$ is no longer an algebra, we make sense of the resolvent identities and differentiability in a weighted $L^2$ space to exploit oscillation in the Fourier space.

 Next, in Section \ref{s3}, we consider scalar operator  $\bar{V}_2 (-\Delta_b - z)^{-1} \bar{V}_1$ in $\dot{H}^{\sigma}$ with scalar potentials $
	\bar{V}_1, \bar{V}_2$ and prove its compactness and decay in $z$. \\

\noindent{\em Step 2. Spectral analysis for matrix operators.} We come back to the matrix operator $\calH$ in Section \ref{s4}. By decomposing $V = V_1 V_2$, the previous results ensure $I + V_2 (\calH - V - z)^{-1} V_1$ analytic Fredholm as bounded operators in $\dot{H}^{\sigma}$ with $z$ in a region independent of $\sigma$. Hence by picking $\sigma$ close to $s_c$, we can identify and decouple the discrete spectrum with the essential spectrum, and discuss their spectral Riesz projections $P_{\rm disc}$ and $P_{\rm ess}$.

 We conclude by confirming $P_{\rm ess}$ eliminates all singularities of $(\calH - z)^{-1}$.\\

\noindent{\em Step 3. Abstract Wiener's theorem and end of the proof.} In \cite{beceanu2011new}, Beceanu proved the abstract Wiener's theorem which determines the invertibility of an operator in a convolutional Banach algebra through the invertibility of its Fourier transform. Since the Strichartz estimate for $e^{i\tau\calH}$ can be reduced to invertibility of $I - iV_2 e^{i\tau(\calH - V)} V_1$ in $L^2_\tau \dot{H}^{\sigma}_y$, it further boils down to the invertibility of its $\tau$-Fourier transform $I + V_2 (\calH - V -z)^{-1} V_1$ in $\dot{H}^\sigma$ and some compactness. Hence the uniform invertibility of $(\calH - z)^{-1}P_{\rm ess}$ implies a Strichartz estimate for the modified system
\bee i\partial_t Z + \calH P_{\rm ess} + i\mu P_{\rm disc} = F 
\eee
which captures the evolution of the original system in $P_{\rm ess}\dot{H}^\sigma$ and concludes the proof in Section \ref{s5}.

\subsubsection{Proof of $\dot H^\sigma$ stability} In Section \ref{s6}, we prove Theorem \ref{thmcodimstability} and Theorem \ref{thmasympstab}. We start by introducing $J$-invariance to connect \eqref{eqNLS} with the spectrally decoupled system. The Strichartz and spectral gap enables us to construct the solution for any given center-stable initial datum through contraction mapping (namely Lyapunov-Perron method, see for example \cite{MR2480603}). As a corollary, we have a nonlinear unstable-data map from center-stable initial data to the unstable data leading to blowup (Theorem \ref{thmfincodimstabHsigmaB}). One difficulty comes from the nonlinear estimates in Appendix \ref{appC}, where we take $D^\sigma$ derivative for a $C^{\beta}$ nonlinearity with possibly $\beta < \sigma$ by exploiting the non-vanishing of $Q_b$.

For Theorem \ref{thmasympstab}, instead of introducing modulation parameters, track its evolution and redo the construction, we use a Brouwer argument to match initial data, inspired by Donninger-Sh\"orkhuber \cite{MR2909934}. Specifically, for general initial data, we find the suitable symmetry transformation to shift it onto the blowup manifold constructed in Theorem \ref{thmcodimstability}. The mode stability Assumption \ref{assmodestab} ensures the degree of freedom equals the lack of codimensions and non-degeneracy.
This argument not only simplifies the proof by avoiding modulation argument, but also enables more time-independent tools like the Strichartz estimates for stability problem. 

\subsubsection{Proof of $H^1$ stability}
Finally, we prove the $H^1$ result Theorem \ref{thmasympstabH1} in Section \ref{s7}. In order to generalize from $\dot H^\sigma$ stability theory, we introduce the modified spectral projections for localization and $\dot H^\sigma \cap \dot H^1$ Strichartz estimate for improving regularity. 

Critical decay of self-similar profile is a universal difficulty for obtaining localized blowup solution (for instance in fluid \cite{arXiv:1910.14071}, dispersive \cite{MR4359478,zbMATH07952722}, or parabolic problems \cite{arXiv:2404.17228}). Instead of introducing time-dependent damped profiles, we apply the approach from \cite{arXiv:2404.17228} to reparametrize the spectral decomposition and localize the eigenmodes from discrete spectrum (Lemma \ref{lemJinv3}). Interestingly, this argument also helps with the improvement of regularity, since stable eigenmodes have oscillatory nature and may not belong to $\dot H^1$. 

The $\dot H^\sigma \cap \dot H^1$ Strichartz estimate (Lemma \ref{lemStrichartz1}) follows from differentiating the equation and commutator estimates of residuals. Then like Theorem \ref{thmfincodimstabHsigmaB}, we can construct unstable-data map in $\dot H^\sigma \cap \dot H^1$ with modified spectral projections, which can lead to localized blowup datum. Then the proof of Theorem \ref{thmasympstabH1} comes as in Section \ref{s6} for $\dot H^\sigma$ results.

\subsection{Notations.}

The notation in this paper is standard. We use $\hat{f}$ or $\mathscr{F}f$ to denote the Fourier transform of a function $f$, and $D^\a f := \mathscr{F}^{-1} |\xi|^{\a} \mathscr{F} f$ with $\a \in \RR$ for the fractional derivatives. We write balls in Banach space $X$ as $B_R^X:= \{ x \in X: |x| < R \}$. For $X = \RR^d$, we omit the superscript and denote $B_R = B_R^X$, with $B_R^c$ its complement. The Japanese bracket notation is $\la x \ra := \sqrt{1+|x|^2}$. The energy-critical exponent is $2^* := \frac{2d}{d-2}$ for $d\ge 3$ and $2^* := +\infty$ for $d = 1, 2$. 

The weighted $\dot{H}^\sigma$ $(\sigma \ge 0)$ spaces with weight $\la x \ra^{\delta} (\delta \in \RR)$ are written as  

$\| f \|_{\dot{H}^\sigma_{\la x \ra^{\delta} }} := \| \la x \ra^{\delta} D^{\sigma-n} \nabla^n f \|_{L^2}$ for $\sigma \in [n, n+1)$, $n \in \NN_{\ge 0}$.
The Schwartz function space is $\mathscr{S}(\RR^d)$. For the product of Banach spaces $X = X_1 \times X_2 \times ... \times X_n$, we define its norm as $\| f \|_X := \max_j \{ \| f_j \|_{X_j} \}$. We also denote $\calL(X \to Y)$ for Banach space of bounded linear operators from $X$ to $Y$, and $\calL(X) := \calL(X\to X)$. The spectrum of a bounded linear operator $T$ is $\sigma(T)$. 

We denote the smooth radial cut-off function to be 
\be \chi(r) = \left| \begin{array}{ll}
    1 & r \le 1 \\
    0 & r \ge \frac 32
\end{array}\right. \in  C^\infty_{c, rad}(\RR^d),\quad \chi_R = \chi(R^{-1}\cdot). \label{eqdefchiR} \ee

When working with vector-valued functions (say $Z = (Z_1, Z_2)^T$) and corresponding operators (say multiplication by matrix-valued functions $V$), for convenience, we keep the scalar notations of Sobolev norms and operator norms. \\

\textbf{Acknowledgments.} The author is deeply grateful to his supervisor Pierre Rapha\"el for his consistent encouragement and support, and to Yvan Martel, Cl\'ement Mouhot, Charles Collot, and anonymous reviewers for their valuable suggestions. The author also thanks Joackim Bernier, Xiaodong Li, and Tao Zhou for enlightening discussions. This work is supported by funding from ERC project SingWave and ERC project FloWAS.

\section{Spectral analysis of $\Delta_b$}\label{s2}
\label{S2}
This section is devoted to the analysis properties of the deformed Laplacian $\Delta_b= \Delta + ib \left(\frac{d}{2} + x \cdot \nabla \right)$, its semigroup $e^{it\Delta_b}$ and resolvents $(-\Delta_b - z)^{-1}$. 

To begin with, we notice that 
\be
   e^{i\frac{b|x|^2}{4}} \circ (-\Delta_b) \circ e^{-i\frac{b|x|^2}{4}} = -\Delta - \frac{b^2 |x|^2}{4},
\ee
and that $-\Delta + V$ is essentially self-adjoint on $L^2(\RR^d)$ if $V \gtrsim -\la x \ra^2$ (see \cite[Section X.5, P.199]{zbMATH03483022}). Therefore we consider $-\Delta_b$ as the self-adjoint operator on $L^2(\RR^d)$.

\subsection{Free evolution $e^{it\Delta_b}$ and self-similar dispersion} \label{s22}
Since $\Delta_b$ is self-adjoint, we can well-define 
the unitary operator group $e^{it\Delta_b}$ on $L^2(\RR^d)$. In this subsection, we will give exact formulas of $e^{it\Delta_b}$ and derive the self-similar dispersive estimate and Strichartz estimates for such free evolution.\\

\begin{lemma}[Formula for $e^{it\Delta_b}$] For $u_0 \in \mathscr{S}(\RR^d)$, $b \neq 0$,  we have
	\bea 
	(\widehat{e^{it\Delta_b}u_0})(t, \xi) &=&e^{\frac{bdt}{2}} \hat{u}_0 (e^{bt} \xi) e^{-i\frac{e^{2bt} - 1}{2b}|\xi|^2 }, \label{eqrepuhat} \\
   (e^{it\Delta_b}u_0)(t, x) &=& \frac{e^{-\frac{id\pi}{4}\mathrm{sgn}(t)}}{(2 \pi)^{\frac d2}} |\sinh{bt}|^{-\frac d2} \int e^{i\frac{b}{2(1-e^{-2bt})} |xe^{-bt}- y|^2} u_0(y) dy.
   \label{eqrepu} \eea
	\end{lemma}
\begin{proof} The formula \eqref{eqrepu} was derived in \cite[(1.8)]{zbMATH02028533} using Mehler's formula. Here we give another simpler proof using self-similar rescaling of the original free Schr\"odinger equation
	\[ i\pa_t u + \Delta u = 0,\quad u\big|_{t=0} =u_0. \]
    For simplicity, assume $b > 0$. 
	Consider the $L^2$ rescaling
	$$ u(t, x ) = \l(t)^{-\frac d2} v(\tau, y)$$
	with self-similar laws similar to \eqref{eqlaw}
        \[  \l(t) = \sqrt{1-2bt}, \quad\tau(t) = -\frac{\ln\left(1-2bt\right)}{2b},\quad y = \frac{x}{\l(t)}, \]
	so that $\l(t(\tau)) = e^{-b\tau}$, $\l\big|_{t=0} = 1$, and $\tau\big|_{t=0} = 0$. A direct computation implies 
	\[ i\pa_\tau v + \left[\Delta_y + ib\left(\frac d2 + y\cdot\nabla_y\right)\right] v = 0,\quad v\big|_{\tau=0} = u_0. \]
	Hence 
	\be e^{it\Delta}u_0 = \l(t)^{-\frac d2} \left(e^{i\tau(t)\Delta_b} u_0\right)(\l(t)^{-1}\cdot)  \label{eqbto0}\ee
    as $t < \frac{1}{2b}$ namely $\tau < \infty$. This yields \eqref{eqrepuhat} and \eqref{eqrepu}. 
\end{proof}

As in the original Schr\"odinger case, these exact formulas of $e^{it\Delta_b}$ yield dispersive estimates and thereafter Strichartz estimates. 

\begin{proposition}[Dispersive estimates and Strichartz estimates for $e^{it\Delta_b}$] Let $b \neq 0$, $d \ge 1$.
	\begin{enumerate}
\item[\rm (i)]{\rm Dispersive estimates:} For $p \in [2, \infty]$, we have
	\be
	\| e^{it\Delta_b} \|_{L^{p'}\to L^{p}} \lesssim_{d,p}\left\{ \begin{array}{ll} |t|^{-\left(\frac d2 - \frac dp \right)} & |t| \le |b|^{-1} \\ 
		(|b|^{-1}e^{|bt|})^{-\left(\frac d2 - \frac dp \right)} & |t| > |b|^{-1}
	\end{array}  \right.  \label{eqdispersiveest}
	\ee
\item[\rm (ii)]{\rm Strichartz estimates:}
	Let $u(t)$ be solution of $i\partial_t u + \Delta_b u = F $ with $u_0$ and $F$ regular, we have
	\be  \| u(t) \|_{L^{q_1}_t L^{p_1}_x } \lesssim_d \| u_0 \|_{L^2_x} + \| F \|_{L^{q_2'}_t L^{p_2'}_x}. \label{eqStrichartzfree}
	\ee
	where $(q_i, p_i)$ for $i = 1,2$ satisfies the admissible condition \eqref{eqadmissible}.
	\end{enumerate}
\end{proposition}
\begin{proof}
	The dispersive estimates \eqref{eqdispersiveest} follows interpolation between $L^2 \to L^2$ conservation and $L^1 \to L^\infty$ dispersive estimate, which are direct consequences of \eqref{eqrepuhat} and \eqref{eqrepu} respectively. Note that \eqref{eqdispersiveest} also indicates 
	$$\| e^{it\Delta_b}\|_{L^{p'}\to L^p} \lesssim |t|^{-\left(\frac d2 - \frac dp \right)}.$$
	The result of abstract result of Keel-Tao \cite{keel1998endpoint} induces the Strichartz estimates for $q, p \ge 2$ satisfying $\frac 2q + \frac dp = \frac d2$ except $(q, p, d) = (2, \infty, 2)$. For the other indices $(q, p) \in \left\{ q \ge 2, 2<p \le \infty, \frac 2q + \frac dp > \frac d2 \right\}$, \eqref{eqdispersiveest} indicates that $\| e^{it\Delta_b} \|_{L^{p'} \to L^p} \in L^1(\RR) \cap L^{\frac q2}(\RR)$. Hence Young's inequality and H\"older's inequality imply
	\bee
	&&\left|\int_\RR \int_\RR \left( e^{i(t-s)\Delta_b} F(s), G(t)\right)_{L^2(\RR^d)} ds dt  \right| \\
	&\lesssim & \int_\RR \int_\RR \| e^{i(t-s)\Delta_b }\|_{L^{p'}\to L^p} \| F(s)\|_{L^{p'}_x} \| G(t) \|_{L^{p'}_x} ds dt \\
	&\lesssim & \| F\|_{L^{q'}_t L^{p'}_x} \| G\|_{L^{q'}_t L^{p'}_x} \left \|  \| e^{i\tau\Delta_b }\|_{L^{p'}\to L^p}\right \|_{L^{\frac q2}_\tau} \lesssim \| F\|_{L^{q'}_t L^{p'}_x} \| G\|_{L^{q'}_t L^{p'}_x}.
	\eee
	With this estimate, the standard $TT^*$-argument \cite{MR1151250} will conclude the proof. 
	\end{proof}

We emphasize that the dispersion decay rate in \eqref{eqdispersiveest} is exponential due to the self-similar scaling.

\subsection{Extended resolvent families}\label{s23}
Using Laplace transform, we can define the resolvent of $\Delta_b$ through the free evolution $e^{it\Delta_b}$
\bea R_b^+(z) &=& i \int_0^\infty e^{it\Delta_b}e^{itz}dt, \quad \quad \quad \Im z > 0;  \label{eqRb+}\\
 R_b^-(z) &=& -i \int_0^\infty e^{-it\Delta_b} e^{-itz}dt, \quad \Im z < 0. \label{eqRb-}
\eea
They satisfy the following resolvent properties:
\begin{lemma}[Resolvent properties]\label{lemresprop}
	Let $b \neq 0$. For $\Im z > 0$, the resolvent $R^+_b(z)$ satisfies (similarly for $R^-_b(z)$ with $\Im z < 0$):
	\begin{enumerate}
\item[\rm (i)]{\rm Boundedness:} $R_b^+(z)$ is bounded in $L^2$ and bijective from $L^2$ to $D(\Delta_b)$.  \\
\item[\rm (ii)]{\rm Inversion:} 
		\bee R_b^+(z) (-\Delta_b - z) u &=& u \quad u \in D(\Delta_b),\\
		(-\Delta_b - z) R_b^+(z) u &=& u \quad u \in L^2.  \eee
\item[\rm (iii)]{\rm Resolvent identities:} for $\Im z > 0$, $\Im z' > 0$, 
		\bee R_b^+(z) - R_b^+(z') = (z - z')R_b^+(z) R_b^+(z'). \eee
		In particular, $R_b^+(z)$ is analytic on $\{z \in \CC:\Im z > 0\}$ as $\L(L^2)$-valued function with derivative $\partial_z R_b^+(z) = R_b^+(z)^2 \in \L(L^2(\RR^d))$.
	\end{enumerate}
		
		\noindent{Moreover, there holds}
		\begin{enumerate}
\item[\rm (iii')]{\rm Resolvent identities for different families:} For $\Im z > 0$ and $\Im w < 0$,
		\bee 
		R_b^+(z) - R_b^-(w) = (z - w)R_b^+(z) R_b^-(w) = (z-w) R_b^-(w) R_b^+(z). \eee
	\end{enumerate}
	\end{lemma}
\begin{proof}
	This is classical operator semigroup theory \cite[Chap.2 Thm. 1.10]{engel2000one}. 
	\end{proof}

However, the self-similar dispersive estimate \eqref{eqdispersiveest} indicates the integral in \eqref{eqRb+} to converge in $\calL(L^{p'} \to L^p)$ for some $2 < p < 2^*$ as  $\Im z > -|b|\min\{\frac d2, 1\}$. Thus we call them extended resolvent family in a larger range of $z$. The rest of this subsection is devoted to showing these extended resolvent families retain resolvent properties as Lemma \ref{lemresprop} in a modified sense.

\subsubsection{Boundedness}

Firstly, we develop boundedness and a Fourier representation. 

\begin{lemma}\label{lemRbbdd}
	For $|b|> 0$ and the operator $R_b^+(z)$ with $\Im z > -|b|\min\left\{\frac{d}{2}, 1\right\}$ (or $R_b^-(w)$ with $\Im w < |b|\min\left\{\frac{d}{2}, 1\right\}$) defined by \eqref{eqRb+} (or \eqref{eqRb-}), we have the following properties:
	\begin{enumerate}
	    \item Boundedness\footnote{We also conjecture the boundedness at the endpoint case $p = 2^*$ for $d \ge 3$.}: for $p \in \left[2, 2^*\right)$, $R^+_b(z)$ is bounded from $L^{p'}$ to $L^p$ if $\Im z > -|b|\left(\frac d2 - \frac dp\right)$, and it is continuous with respect to $z$ in the operator norm. Moreover, we have
	    \bea \sup_{z: \Im z \ge -|b|\left(\frac d2 - \frac dp\right) + \delta}\| R^+_b (z)\|_{L^{p'} \to L^p} < \infty, &\quad \forall \,\delta > 0; \label{eqresunibdd}\\
	    \sup_{z: \Im z = \lambda }\| R^+_b (z)\|_{L^{p'} \to L^p} \to 0, &\quad \mathrm{as}\,\,\lambda \to +\infty.\label{eqresvan}
	    \eea
	    Similar boundedness holds for $R_b^-(w)$.
	    \item Fourier representation: Let $b > 0$. For $f \in \mathscr{S}'(\RR^d)$ with $\hat{f} \in L^1_{loc}(\RR^d)$ satisfying
	    \be r^{-\frac{d}{2} + \frac{\Im z}{b}}\int_r^\infty \rho^{\frac{d}{2} - 1 - \frac{\Im z}{b}} |\hat{f}|(\rho \omega) d\rho \in L^q(\RR^d), \label{eqRbcond+}
	    \ee 
	    for some $q \in [1, \infty)$, we have $R_b^+(z)f \in \mathscr{S}'(\RR^d)$ with its Fourier transform given by
	    \be \label{eqRb+f} (\widehat{R_b^+(z)f})(r\omega) = \frac{i}{b}r^{-\frac d2 + \frac{\Im z}{b}} e^{\frac{ir^2}{2b}-\frac{i\Re z}{b} \ln r} \int_r^\infty \rho^{\frac d2 - 1 - \frac{\Im z}{b}} e^{-\frac{i\rho^2}{2b}+\frac{i\Re z}{b} \ln \rho} \hat{f}(\rho \omega) d\rho.
	    	\ee
	    	Similarly, if $f \in \mathscr{S}'(\RR^d)$ with $\hat{f} \in L^1_{loc}(\RR^d)$ satisfying
	    	\be  r^{-\frac{d}{2} + \frac{\Im w}{b}}\int_0^r \rho^{\frac{d}{2} - 1 - \frac{\Im w}{b}} |\hat{f}|(\rho \omega) d\rho \in L^q(\RR^d),  \label{eqRbcond-}
	    	\ee
	    	for some $q \in [1, \infty)$, then $R_b^-(w)f \in \mathscr{S}'(\RR^d)$ with Fourier representation
	    \be
	    	\label{eqRb-f} (\widehat{R_b^-(w)f})(r\omega) =  -\frac{i}{b}r^{-\frac d2 + \frac{\Im w}{b}} e^{\frac{ir^2}{2b}-\frac{i\Re w}{b} \ln r}\int_0^r \rho^{\frac d2 - 1 - \frac{\Im w}{b}} e^{-\frac{i\rho^2}{2b}+\frac{i\Re w}{b} \ln \rho} \hat{f}(\rho \omega) d\rho.
	    \ee
	    
	    In particular, all these representations hold for $f \in \mathscr{S}(\RR^d)$. Similar formulas holds for $b < 0$ case. 
	\end{enumerate}
    
\end{lemma}
\begin{proof}We only prove for $R_b^+(z)$ with $b > 0$ and the proofs for the other three cases are almost the same.
	
	(1) The boundedness comes from integrating \eqref{eqdispersiveest}
	\[ \| R_b^+(z) \|_{L^{p'} \to L^p} \lesssim_{d, p} \int_0^{|b|^{-1}} t^{-\left(\frac d2 -\frac dp\right)} e^{-t\Im z} dt + \int_{|b|^{-1}}^\infty |b|^{-\left(\frac d2 -\frac dp\right)} e^{-t\left[\Im z + |b|\left(\frac d2 -\frac dp\right) \right]} dt \]
	 We need $p < 2^*$ for integrability near $t = 0$, and $\Im z > -|b|\left(\frac d2 - \frac dp\right)$ for integrability at infinity. The continuity easily follows
	 \[ \|R_b^+(z) - R_b^+(z') \|_{L^{p'} \to L^p} \le \int_0^\infty \|e^{it\Delta_b} \|_{L^{p'} \to L^p} |e^{itz} - e^{itz'}| dt \to 0 \quad \mathrm{as}\,\,|z-z'|\to 0.   \]
	 The uniformity \eqref{eqresunibdd} and the vanishing \eqref{eqresvan} both follow the monotonicity and vanishing of the above integral concerning $\Im z$. 
	 
	 (2) Formally, the representation formula follows \eqref{eqrepuhat} and \eqref{eqRb+} directly.
	 \bee
	 (\widehat{R_b^+(z)f})(r\omega) &=& i \int_0^\infty e^{itz} \widehat{e^{it\Delta_b}f}(r\omega) dt \\ 
	 & = &i \int_0^\infty e^{itz} e^{\frac{bdt}{2}} \hat{f} (e^{bt}r\omega) e^{-i\frac{e^{2bt} - 1}{2b}r^2 } dt \\
	 &=& i\int_r^\infty e^{iz \frac{1}{b} \ln \frac \rho r} e^{\frac{db}2 \frac{1}{b} \ln \frac \rho r} \hat{f}(\rho \omega) e^{-i\frac{\left( \frac \rho r\right)^2 - 1}{2b} r^2} \frac{d\rho}{b \rho} \\
	 &=&\frac{i}{b}r^{-\frac d2 + \frac{\Im z}{b}} e^{\frac{ir^2}{2b}-\frac{i\Re z}{b} \ln r} \int_r^\infty \rho^{\frac d2 - 1 - \frac{\Im z}{b}} e^{-\frac{i\rho^2}{2b}+\frac{i\Re z}{b} \ln \rho} \hat{f}(\rho \omega) d\rho
	 \eee
	 where in the third equivalence, we change $t$ to the variable $\rho = e^{bt}r \in [r,\infty)$ with $b > 0$. Since $e^{it\Delta_b}$ maps $\mathscr{S}$ to itself and hence maps $\mathscr{S}'$ to itself, the integration on a finite interval is well-defined in $\mathscr{S}'(\RR^d)$. Thus it suffices to verify the integral in the first equivalence to converge in the sense of $\mathscr{S}'(\RR^d)$ to make this computation rigorous. Using \eqref{eqrepuhat}, 
	 \bee  \left| i\int_R^\infty e^{itz} \widehat{e^{it\Delta_b}f}(r\omega) dt\right|
	 &\le& \int_R^\infty e^{-t\Im z} e^{\frac{bdt}{2}} \left| \hat{f} (e^{bt}r\omega) \right| dt \\
	 &\le& b^{-1} r^{-\frac d2 + \frac{\Im z}{b}} \int_{re^{Rb}}^\infty  \rho^{\frac d2 - 1 - \frac{\Im z}{b}} \left| \hat{f} (\rho \omega) \right| d\rho
	 \eee
	 So we can use \eqref{eqRbcond+} as a dominant function in $L^q$ to show dominant convergence. The tail vanishes in $L^q$ and hence the computation makes sense.
	 \end{proof}

    As an application, we prove the full characterization of $\sigma(\Delta_b)$ with $b\neq 0$, which is drastically different from the $b=0$ case.   
    \begin{lemma}[Spectrum of $\Delta_b$]\label{lemctsspec}
    	For $b \neq 0$, $\sigma(-\Delta_b)= \RR$.
    \end{lemma}
    \begin{proof}
    	With the self-adjointness of $-\Delta_b$, we will apply the explicit formula of extended resolvent families to prove $\lambda \in \sigma(-\Delta_b)$ for all $\lambda \in \RR$. We assume $b > 0$ and the case $b < 0$ comes in the same way.
    	
    	If not, then $\lambda \in \rho(-\Delta_b)$ and $-\Delta_b - z$ is invertible near $z$ with an analytic inverse. We denote $R_b(\lambda) = (-\Delta_b - \lambda)^{-1} \in \L(L^2)$, then this analyticity implies 
    	\bee R_b^{+}(\lambda + i\epsilon) \to R_b(\lambda) \quad \mathrm{in}\,\,\L(L^2)&\,\,\mathrm{as}\,\,\epsilon \to 0^+; \\
    	R_b^{-}(\lambda - i\epsilon) \to R_b(\lambda) \quad \mathrm{in}\,\,\L(L^2)&\,\,\mathrm{as}\,\,\epsilon \to 0^+.
    	\eee
    	Also recall the continuity and boundedness in Lemma \ref{lemRbbdd}. Suppose $\epsilon < \frac{|b|}{4}$, we also have 
    	\bee R_b^{+}(\lambda + i\epsilon) \to R_b^+(\lambda) \quad \mathrm{in}\,\,\L(L^\frac{4d}{2d+1} \to L^\frac{4d}{2d-1})&\,\,\mathrm{as}\,\,\epsilon \to 0^+; \\
    	R_b^{-}(\lambda - i\epsilon) \to R_b^-(\lambda) \quad \mathrm{in}\,\,\L(L^\frac{4d}{2d+1} \to L^\frac{4d}{2d-1})&\,\,\mathrm{as}\,\,\epsilon \to 0^+.
    	\eee
    	Thus for $f \in \mathscr{S}(\RR^d)$, we should have $R^+_b(\l)f = R_b(\lambda)f = R_b^-(\l)f$, while 
    	\[ \mathscr{F}(R_b^+(\l)f - R_b^-(\l)f)(r\omega) = \frac{i}{b}r^{-\frac d2 + \frac{\Im z}{b}} e^{\frac{ir^2}{2b}-\frac{i\Re z}{b} \ln r} \int_0^\infty \rho^{\frac d2 - 1 - \frac{\Im z}{b}} e^{-\frac{i\rho^2}{2b}+\frac{i\Re z}{b} \ln \rho} \hat{f}(\rho \omega) d\rho  \]
    	which may not be zero when the integral is non-zero. This contradiction shows that $\lambda \in \sigma(-\Delta_b)$ and finishes the proof. 
    \end{proof}

    \subsubsection{Inversion property}
    
    Next, we come to the inversion property. We first check the strong continuity of the semigroup in the sense of tempered distribution.
    
    \begin{lemma}\label{lemsemicts}
    	For $z \in \CC$ and $f \in \mathscr{S}(\RR^d)$, the orbital map
    	\[ C_{z, f} : t \mapsto e^{it\Delta_b} e^{itz} f \]
    	is continuous and differentiable from $\RR$ to $\mathscr{S}(\RR^d)$ with 
    	\be \frac{d}{dt}C_{z, f}(t) := \lim_{h \to 0} \frac{1}{h}\left(e^{i(h+t)(\Delta_b+z)} f - e^{it(\Delta_b+z)} f\right) = i( \Delta_b + z) e^{it(\Delta_b+z)} f \label{eqdifsemi}\ee
    	where the convergence is in the sense of $\mathscr{S}(\RR^d)$ and is uniform in finite time. 
    	The same statements also hold if we replace $\mathscr{S}(\RR^d)$ by $\mathscr{S}'(\RR^d)$. 
\end{lemma}
    
    \begin{proof}
    	The conclusion (continuity, differentiability and uniform convergence of derivatives) in $\mathscr{S}(\RR^d)$ follow direct computations using \eqref{eqrepuhat}. For example, for the continuity at $t = 0$ with $z = 0$, we compute
    	\bee &&\mathscr{F}(e^{it\Delta_b} f- f )(\xi) \\
    	&=& \left( e^{\frac {bdt}{2} }- 1\right)\hat{f}(e^{bt}\xi) e^{-i\frac{e^{2bt}-1}{2b}|\xi|^2} 
    	+ \left( \hat{f}(e^{bt}\xi) - \hat{f}(\xi) \right) e^{-i\frac{e^{2bt}-1}{2b}|\xi|^2} \\
    	&&+  \hat{f}(\xi) \left( e^{-i\frac{e^{2bt}-1}{2b}|\xi|^2} -1\right).
    	\eee

    	It is easy to check that each term on the right-hand side is vanishing in the Schwartz space. Then since 
    	\bee \la e^{it\Delta_b} f, \varphi \ra_{\mathscr{S}', \mathscr{S}} =  \la f,  e^{-it\Delta_b} \varphi \ra_{\mathscr{S}', \mathscr{S}}, \quad \la \Delta_b f, \varphi \ra_{\mathscr{S}', \mathscr{S}} = \la f,\Delta_b \varphi \ra_{\mathscr{S}', \mathscr{S}},
    	\eee
    	the conclusion in $\mathscr{S}'(\RR^d)$ follows the one in $\mathscr{S}(\RR^d)$. 
    	\end{proof}
    
    Now we prove the inversion property within the range of the previous lemma.
    
    \begin{lemma}[Inverseion properties for extended resolvent families]\label{leminv}
    	Let $p \in \left[2, 2^*\right)$ and $z \in \CC$ satisfying $\Im z > -|b| \left( \frac d2 - \frac dp\right)$.
    	If $ u \in L^{p'}$, we have
    	\be (-\Delta_b - z) R_b^+(z)  u = u. \label{eqinv1} \ee
    	Moreover, if $u \in L^{p'}$ and $\Delta_b u \in L^{p'}$, then
    	\be
    	R_b^+(z) (-\Delta_b - z) u = u.\label{eqinv2}
    	\ee
    	Similar identities hold for $R^-_b(z)$ with $z$ in the corresponding range.
    \end{lemma}
  
    \begin{proof}
    	First observing that for such $p$ and $z$ in the lemma and $f \in L^{p'}$, we have
    	\be e^{ih(\Delta_b+z)} R_b^+(z) f = e^{ih(\Delta_b+z)} i\int_0^\infty e^{it(\Delta_b+z)}f dt = i\int_h^\infty e^{it(\Delta_b+z)} f dt \label{eqinvconv} \ee
    	since the integral converges in $L^{p}$. 
    	Therefore for $u \in L^{p'}$,
    	\[ \frac{i}{h}\left(e^{ih(\Delta_b+z)} - I\right) R_b^+(z) u = \frac{1}{h} \int_0^h e^{it(\Delta_b+z)} u dt. \]
    	As $u, R_b^+(z) u \in \mathscr{S}'(\RR^d)$, the continuity of $C_{z, u}$ and \eqref{eqdifsemi} for $C_{z, R_b^+(z)u}$ in $\mathscr{S}'(\RR^d)$ from Lemma \ref{lemsemicts} imply that 
    	both sides of the above equation will converge to \eqref{eqinv1} as $h \to 0$. 
    	
    	For \eqref{eqinv2}, we start from
    	\[ \frac{1}{h} \left(e^{ih (\Delta_b+z)} - I \right) \int_0^T e^{it(\Delta_b+z)} u dt =\int_0^T e^{it(\Delta_b+z)} \frac{1}{h} \left(e^{ih (\Delta_b+z)} - I \right) u dt \]
    	for $T > 0$. 
    	Letting $h$ goes to $0$, the uniform convergence of \eqref{eqdifsemi} implies that both sides converges in $\mathscr{S}'(\RR^d)$ to 
    	\[ i(\Delta_b + z) \int_0^T e^{it(\Delta_b + z)}udt = i\int_0^T e^{it(\Delta_b + z)} (\Delta_b + z) udt.  \]
    	Since $\Delta_b u \in L^{p'}$ and $u \in L^{p'}$, the right-hand side converge in $L^p$ and the left-hand side converges in $\mathscr{S}'$ as $T\to \infty$. Then \eqref{eqinv2} follows \eqref{eqinv1}. 
    \end{proof}


    \subsubsection{Resolvent identity and differentiability} 

	Although resolvent identities formally hold, it is not trivial to see how extended resolvents compose with $\L(L^{p'}\to L^p)$ no longer an algebra. The following lemma rigorously verifies such resolvent identities. In particular, the composition makes sense as an operator from $L^{p'}$ to $L^{p}$.
	
    \begin{lemma}[Resolvent identities for extended resolvent families]\label{lemresident}
        Let $b \neq 0$. For $z, z'$ such that $\Im z \neq \Im z' > -|b|\min\left\{\frac d2, 1 \right\}$, and for $p$ satisfies  $\frac{1}{p} < \frac{\min\{ \Im z, \Im z'\}}{|b|d} + \frac{1}{2}$ and $p \in [2,2^*)$, we have
        \begin{enumerate}
        	\item[\rm (i)]{\rm  Resolvent identity for the same family:}
        	\be R_b^+(z) - R_b^+(z') = (z - z')R_b^+(z) R_b^+(z')=(z-z')R_b^+(z)  R_b^+(z')  \ee
        	holds as bounded operator from $L^{p'}$ to $L^p$.
        	Similar statement holds when we replace $+$ by $-$ and $z$, $z'$ by $-z$, $-z'$.
        	\item[\rm (ii)]{\rm Resolvent identity for different families:} Let $w := -z'$ and suppose $$\Im z > \Im w,$$ then we have
        	\be R_b^+(z) - R_b^-(w) = (z - w)R_b^+(z) R_b^-(w) = (z-w) R_b^-(w) R_b^+(z) \label{eqresident2}\ee
        \end{enumerate}  
    \end{lemma}
    \begin{proof}
    	The first requirement on $p$ guarantees $R_b^+(z), R_b^+(z'), R_b^-(w)$ are bounded from $L^{p'}$ to $L^p$ according to Lemma \ref{lemRbbdd}. So the boundedness of composition follows the resolvent identities. We only need to verify the resolvent identity on $f \in \mathscr{S}(\RR^d)$ and focus on the case $b > 0$ for example. 
    	
    	(1) First we check the validity of Fourier representation formula \eqref{eqRb+f} acting on $R_b^+(z')f$. Note that $\widehat{R_b^+(z')f}(r\omega)$ has mild singularity near the origin:
    	\be |\widehat{R_b^+(z')f}(r\omega)| \lesssim 
    	\left\{ \begin{array}{ll} 
    		r^{-\frac d2 + \frac{\Im z'}{b}} & \frac{\Im z'}{b} < \frac d2  \\
    		|\ln r| & \frac{\Im z' }{b} = \frac d2 \\
    	    1 & \frac{\Im z'}{b} > \frac d2  \end{array} \right.\quad r \le 1,  \label{eqRb+sing}\ee
        and it decays faster than any polynomial order at infinity. Thus 
    	\[ r^{-\frac d2+\frac{\Im z }{b}}\int_r^\infty \rho^{\frac{d}{2} -1 - \frac{\Im z}{b} } \left| \widehat{R_b^+(z')f} (\rho \omega)\right| d\rho \lesssim \left\{ 
    	\begin{array}{ll} r^{\min \left\{-\frac d2 + \frac{\Im z }{b}, -\frac{d}{2} + \frac{\Im z'}{b}, 0\right\}} |\ln r|  & r \le 1 \\ r^{-k} & r \ge 1  \end{array} \right.  \]
    	for any $k \ge 1$, which implies \eqref{eqRbcond+}. Using \eqref{eqRb+f} and changing the order of integration, we get
    	\bee 
    	&& \mathscr{F}\left(R_b^+(z)R_b^+(z') f\right) (r\omega) \\
        & = & -\frac{1}{b^2} r^{-\frac d2} e^{\frac{ir^2}{2b}-\frac{i z}{b} \ln r} \int_r^\infty \mu^{\frac d2 - 1} e^{-\frac{i\mu^2}{2b}+\frac{i z'}{b} \ln \mu} \hat{f}(\mu \omega) \int_r^\mu \rho^{-1}e^{\frac{i (z-z')}{b} \ln \rho} d\rho d\mu \\
        & = & -\frac{1}{b^2} r^{-\frac d2} e^{\frac{ir^2}{2b}-\frac{i z}{b} \ln r} \int_r^\infty \mu^{\frac d2 - 1} e^{-\frac{i\mu^2}{2b}+\frac{i z'}{b} \ln \mu} \hat{f}(\mu \omega) \left[ \frac{b}{i} \frac{1}{z-z'} \left( e^{i\frac{z-z'}{b} \ln \mu} - e^{i\frac{z-z'}{b} \ln r} \right) \right] d\mu \\
        & = & \frac{1}{z-z'}\left(\widehat{R_b^+(z)f} - \widehat{R_b^+(z')f}\right).
    	\eee
    	
    	(2) Now we have $|\widehat{R^-_b(w)f}(r\omega)| \lesssim \la r \ra^{-\frac d2 + \frac{\Im w}{b}}$ for $f \in \mathscr{S}(\RR^d)$. Since $\Im z > \Im w $, the integral satisfies
    	\[ r^{-\frac d2 + \frac{\Im z }{b}} \int_r^\infty \rho^{\frac{d}{2} -1 - \frac{\Im z}{b} } \left| \widehat{R_b^-(w)f} (\rho \omega)\right| d\rho \lesssim r^{-\frac d2 + \frac{\Im z }{b}}  \la r \ra^{-\frac{\Im z - \Im w}{b}}. \]
    	So \eqref{eqRbcond+} holds for $\frac 1q \in \left(\frac{1}{2} - \frac{\Im z}{bd}, \frac{1}{2} - \frac{\Im w}{bd} \right)$. Similarly applying the Fourier representation formulas, we obtain
    	\bee 
    	&& \mathscr{F}\left(R_b^+(z)R_b^-(w) f\right) (r\omega) \\
    	& = & \frac{1}{b^2} r^{-\frac d2} e^{\frac{ir^2}{2b}-\frac{i z}{b} \ln r} \bigg[ \int_0^r \mu^{\frac d2 - 1} e^{-\frac{i\mu^2}{2b}+\frac{i w}{b} \ln \mu} \hat{f}(\mu \omega) \int_r^\infty \rho^{-1}e^{\frac{i (z-w)}{b} \ln \rho} d\rho d\mu \\
    	&& + \int_r^\infty \mu^{\frac d2 - 1} e^{-\frac{i\mu^2}{2b}+\frac{i w}{b} \ln \mu} \hat{f}(\mu \omega) \int_\mu^\infty \rho^{-1}e^{\frac{i (z-w)}{b} \ln \rho} d\rho d\mu  \bigg] \\
    	& = & -\frac{1}{b^2} r^{-\frac d2} e^{\frac{ir^2}{2b}-\frac{i z}{b} \ln r} \bigg[ \int_0^r \mu^{\frac d2 - 1} e^{-\frac{i\mu^2}{2b}+\frac{i w}{b} \ln \mu} \hat{f}(\mu \omega) d\mu \frac{-b}{i} \frac{1}{z-w}  e^{i\frac{z-w}{b} \ln r} \\
    	&& +\int_r^\infty \mu^{\frac d2 - 1} e^{-\frac{i\mu^2}{2b}+\frac{i w}{b} \ln \mu} \hat{f}(\mu \omega)   \frac{-b}{i} \frac{1}{z-w}  e^{i\frac{z-w}{b} \ln \mu} d\mu \bigg] \\
    	& = & \frac{1}{z-w}\left(\widehat{R_b^+(z)f} - \widehat{R_b^-(w)f}\right).
    	\eee
        In the second equivalence, the integrability of $\rho$ comes from $\Im z > \Im w$. 
    	
    	And for $R_b^-(w) R_b^+(z) f$, we see from \eqref{eqRb+sing} and $\Im z > \Im w$ that for $r \le 1$, 
    	\[ r^{-\frac d2+\frac{\Im w}{b}} \int_0^r \rho^{\frac d2 - 1 -\frac{\Im w}{b}} |\widehat{R_b^+(z)f}(\rho \omega)| d\rho \lesssim 
    	\left\{ \begin{array}{ll} 
    		r^{-\frac d2 + \frac{\Im z}{b}} & \frac{\Im z}{b} < \frac d2  \\
    		C_\e r^{-\e} & \frac{\Im z}{b} = \frac d2 \\
    		1 & \frac{\Im z}{b} > \frac d2  \end{array} \right. \]
    	where $\e$ can be an arbitrary positive number, and from the fast decay at infinity that the left-hand side is controlled by $r^{-\frac d2 + \frac{\Im w}{b}}$ at infinity. These decay and $\Im z > \Im w$ imply integrability \eqref{eqRbcond-}, and the last equality in \eqref{eqresident2} follows similar computation as above. 
    \end{proof}

    However, those resolvent identities does not imply differentiability since $R_b^+(z) R_b^+(z') \in \calL(L^{p'} \to L^p)$ is not convergent when $z' \to z$. We address this issue by invoking a weaker topology $\L(L^2_{\la x \ra^\delta} \to L^p )$ to exploit the oscillation in Fourier space.

    \begin{lemma}[Differentiability of extended resolvent families]\label{lemanalyticity}
    	For $ \delta \in (0, 1)$, $\Im z > -|b|\min\left\{\frac d2, 1\right\}$ and $p \in [ 2, 2^*)$ satisfying 
    	\be \delta > \frac d2- \frac dp > -\frac{\Im z}{|b|}, \label{eqanalyticitycond} \ee
        we have 
    	\bea \left\| \widehat{R_{b}^+(z)^2 f} \right\|_{L^{p'}}&\lesssim_{b,d,p, z, \delta}& \| \hat{f}\|_{H^\delta} \label{eqana1} \\
    	\left\| \mathscr{F}\{(R_b^+(z') -  R_b^+(z))R_{b}^+(z) f\} \right\|_{L^{p'}} &\lesssim_{b,d,p, z, \delta}& |z-z'| \| \hat{f}\|_{H^\delta}  \label{eqana2}
    	\eea
    	where the second inequality holds when $|z'-z| \ll 1$. In particular, 
    	\be R_b^+(z')R_b^+(z) \to R_b^+(z)^2 \quad \text{in} \,\,\L(L^2_{\la x \ra^\delta} \to L^p )\,\,\text{as}\,\,z' \to z. \ee
    	Similarly bounds and convergence hold for $R_b^-(z)$ with the corresponding $z$. 
\end{lemma}

\begin{proof}
	Again, we only consider the case $b > 0$. Note that $L^2_{\la x \ra^\delta}$ is embedded in $L^{p'}$ by H\"older's inequality and $R_b^\pm(z) \in \L(L^{p'} \to L^p)$ with $z$ in the corresponding range. It suffices to prove \eqref{eqana1} and \eqref{eqana2}. \\
	
	\underline{1. Proof of \eqref{eqana1}.}  Following the computation in the previous lemma, we have for $f \in \mathscr{S}(\RR^d)$,
	\bee 
	&& \mathscr{F}\left(R_b^+(z)R_b^+(z) f\right) (r\omega) \\
	& = & -\frac{1}{b^2} r^{-\frac d2} e^{\frac{ir^2}{2b}-\frac{i z}{b} \ln r} \int_r^\infty \mu^{\frac d2 - 1} e^{-\frac{i\mu^2}{2b}+\frac{i z}{b} \ln \mu} \hat{f}(\mu \omega) \int_r^\mu \rho^{-1} d\rho d\mu \\
	&=& -\frac{1}{b^2} r^{-\frac d2} e^{\frac{ir^2}{2b}-\frac{i z}{b} \ln r} \int_r^\infty \mu^{\frac d2 - 1} e^{-\frac{i\mu^2}{2b}+\frac{i z}{b} \ln \mu} \ln \left(\frac \mu r\right) \hat{f}(\mu \omega) d\mu.
	\eee
	
	Let $\psi \in C^\infty_{c, rad}(\RR^d)$, non-negative and $\int_{\RR^d} \psi(x) dx = 1$. Define its rescaling $\psi_\l (r) := \l^{-d}\psi(\l^{-1} r)$. Also recall the smooth cutoff $\chi$ from \eqref{eqdefchiR}. Now decompose $\hat{f}(\xi)$ by
	\be \hat{f} = g_1 + g_2 + g_3 \label{eqanadecomp1}
	\ee
	with 
	\bea g_1(\xi) &=& \hat{f}(\xi) \chi(\xi), \label{eqanadecomp2}\\
	g_2(\xi) &=& (1-\chi(\xi))  (\hat{f}* \psi_{|\xi|^{-1}})(\xi) , \label{eqanadecomp3}\\
	g_3(\xi) &=& (1-\chi(\xi)) \left[ \hat{f}(\xi)  - (\hat{f}* \psi_{|\xi|^{-1}})(\xi) \right]. \label{eqanadecomp4}
	\eea
	Then
	\bee  &&\mathscr{F}\left(R_b^+(z)R_b^+(z) f\right) (r\omega) \\
	 &=& -\frac {1}{b^2} r^{-\frac d2}e^{i\frac{r^2}{2b} - i\frac{z}{b}\ln r}  \int_r^\infty e^{-i \frac{\rho^2}{2b} + i\frac zb \ln \rho}\rho^{\frac d2 - 1}\ln \left(\frac \rho r\right)  \left(g_1(\rho \omega) + g_2(\rho \omega) + g_3(\rho \omega) \right) d\rho \\
	&=:& G_1(r\omega) + G_2(r\omega) + G_3(r\omega).
	\eee
    We will estimate each term respectively. Before that, we specify a small $\e > 0$ such that
    \be \frac d2 - \frac dp >- \frac{\Im z}{b} + \e \label{eqepsilonsmall}\ee
    and will bound the logarithm term by $\ln \left(\frac \rho r\right) \lesssim_{\e} \left(\frac \rho r \right)^\e$. \\
	
	\emph{(1) Term $G_1$}. The support of $\chi$ indicates that $\supp G_1 \subset B_{\frac 32}$. For $r \le \frac 32$, by estimating $\ln t \lesssim_\e t^\e$ for $t = \frac{\rho}{r} \ge 1$ and applying Cauchy-Schwarz inequality, 
	\bee
	|G_1(r\omega)|
	&\lesssim_{\e}&  b^{-2} r^{-\frac d2 + \frac{\Im z}{b}}  \int_r^\frac 32 \rho^{\frac d2 - 1 - \frac{\Im z}{b}} \left(\frac \rho r\right)^{\e} |\hat{f}(\rho \omega)|d\rho \\
	&\lesssim_{\e}& b^{-2} r^{-\frac d2 - \max\left\{-\frac{\Im z}{b} + \e,0 \right\}} |\ln r| \left(\int_0^\infty |\hat{f}(\rho \omega)|^2 \rho^{d-1} d\rho \right)^{\frac 12}.
	\eee
	The choice of $\e$ ensures $r^{-\frac d2 - \max\left\{-\frac{\Im z}{b} + \e,0 \right\}}|\ln r|  \in L^{p'}(B_{\frac 32})$. Also, $p' \le 2$ enables us to bound $L^{p'}(\mathbb{S}^{d-1})$ by $L^2(\mathbb{S}^{d-1})$. Consequently, we obtain
	$\| G_1\|_{L^{p'}} \lesssim_{b, d, \delta, p, z} \| \hat{f}\|_{L^2}$. \\
	
	    \emph{(2) Term $G_2$}. We integrate by parts over the quadratic phase to gain decay of $\rho$
	    \bee &&\int_r^\infty e^{-i \frac{\rho^2}{2b} + i\frac zb \ln \rho}\rho^{\frac d2 - 1} \ln \left(\frac \rho r\right) g_2(\rho \omega)  d\rho\\
	    &=& \int_r^\infty \rho^{\frac d2 - 1} \left[\frac{-b}{i\rho}\partial_\rho e^{-\frac{i\rho^2}{2b}} \right]e^{\frac{i z}{b} \ln \rho}\ln \left(\frac \rho r\right) g_2(\rho \omega) d\rho \\
	    &=&  -ib \int_r^\infty  e^{-\frac{i\rho^2}{2b}} \rho^{\frac d2 -2}e^{\frac{i z}{b} \ln \rho}
	    \Bigg\{ \ln \left(\frac \rho r\right)  \partial_\rho g_2(\rho\omega)  
	       + \rho^{-1}\left[ \left(i\frac zb + \frac d2 - 2\right)   \ln \left(\frac \rho r\right) + 1 \right] g_2(\rho \omega)  \Bigg\} d\rho 
	    \eee
	where in the second equivalence we exploit $\ln (\rho/r)\big|_{\rho = r} = 0$  to eliminate the boundary term.
	Thus with the $\epsilon$ satisfying \eqref{eqepsilonsmall},
	we have
	\bee &&|G_2(r\omega)| \lesssim_{z, b, d, \e} r^{-\frac d2 + \frac{\Im z}{b}} \int_r^\infty \rho^{\frac d2 - 2 - \frac{\Im z}{b}}\left(\frac \rho r\right)^{\e} \left( \rho^{-1} |g_2(\rho\omega)| + |\partial_\rho g_2(\rho \omega)| \right) d\rho \\
	&\lesssim& r^{-\frac d2 + \frac{\Im z}{b}-\e} \int_{\max\{r, 1\}}^\infty \rho^{\frac d2 - 2 - \frac{\Im z}{b}+\e}\left( |(\hat{f}* \psi_{|\xi|^{-1}})(\rho \omega)| + |\partial_\rho (\hat{f}* \psi_{|\xi|^{-1}})(\rho\omega)| \right) d\rho
	\eee
	where we bound $\partial_r \chi$ by a constant. 
	Now we estimate the integral term with Cauchy-Schwarz inequality as
	\bee
	&&\int_{\max\{r, 1\}}^\infty \rho^{\frac d2 - 2 - \frac{\Im z}{b}+\e} \left( |(\hat{f}* \psi_{|\xi|^{-1}})(\rho \omega)| + |\partial_\rho (\hat{f}* \psi_{|\xi|^{-1}})(\rho\omega)| \right) d\rho \\
	& \lesssim_{z, b, \delta, \e}& \la r \ra^{-1-\frac{\Im z}{b}+\e} \left(\int_1^\infty |(\hat{f}* \psi_{|\xi|^{-1}})(\rho \omega)|^2 \rho^{d-1} d\rho \right)^{\frac 12} \\
	&+&
	 \la r \ra^{-\frac{\Im z}{b} - \delta+\e} \left(\int_1^\infty |\pa_\rho (\hat{f}* \psi_{|\xi|^{-1}})(\rho \omega)|^2  \rho^{d-3+2\delta} d\rho \right)^{\frac 12}
	\eee
	Noticing that the radial and angular parts of $G_2$ are decoupled, we have
	\bee
	&&\| G_2 \|_{L^{p'}} \\
	&\lesssim_{z, b, \delta, d, \e} &
	 \left\|  r^{-\frac d2 + \frac{\Im z}{b}-\e} \la r \ra^{-1-\frac{\Im z}{b}+\e} \right\|_{L^{p'}(r^{d-1}, [0,\infty])}
	\left\| \left( \int_1^\infty |(\hat{f}* \psi_{|\xi|^{-1}})(\rho \omega)|^2 \rho^{d-1} d\rho \right)^{\frac 12}  \right\|_{L^{p'}(\mathbb{S}^{d-1})} \\
	& +& \left\|  r^{-\frac d2 + \frac{\Im z}{b}-\e} \la r \ra^{-\delta-\frac{\Im z}{b}+\e} \right\|_{L^{p'}(r^{d-1}, [0,\infty])}
	\left\| \left(\int_1^\infty |\pa_\rho (\hat{f}* \psi_{|\xi|^{-1}})(\rho \omega)|^2  \rho^{d-3+2\delta} d\rho \right)^{\frac 12}  \right\|_{L^{p'}(\mathbb{S}^{d-1})} \\
	&\lesssim& 
	\left( \| r^{-\frac d2 + \frac{\Im z}{b}-\e} \|_{L^{p'}(B_1)} + 
	\| r^{-\frac d2 -\delta} \|_{L^{p'}(B_1^c)}\right) \\
	&&\cdot \left( \| \hat{f}* \psi_{|\xi|^{-1}} \|_{L^{2}(B_1^c)} + \| \partial_\rho (\hat{f}* \psi_{|\xi|^{-1}})|\cdot|^{-1+\delta} \|_{L^{2}(B_1^c)} \right).
	\eee
	From the parameter condition \eqref{eqanalyticitycond}, choice of $\e$ and the estimates \eqref{eqfrac1},\eqref{eqfrac4} in Lemma \ref{lemfrac}, we get $\| G_2\|_{L^{p'}} \lesssim_{z, b, d, \delta} \|\hat{f} \|_{H^{\delta}}$.\\
	
	\emph{(3) Term $G_3$}. We estimate the integral part directly
	\bee
	&&\int_{\max\{r, 1\}}^\infty \rho^{\frac d2 - 1 - \frac{\Im z}{b}} \left(\frac \rho r\right)^\e |f(\rho \omega) - (\hat{f}* \psi_{|\xi|^{-1}})(\rho \omega)| d\rho \\
	&\lesssim_{z, b, \delta, \e}& r^{-\delta-\frac{\Im z}{b}} \left(\int_1^\infty |\hat{f}(\rho\omega)-(\hat{f}* \psi_{|\xi|^{-1}})(\rho \omega)|^2 \rho^{d-1+2\delta} d\rho \right)^{\frac 12} 
	\eee
	Similarly, using parameter conditions and \eqref{eqfrac3} from Lemma \ref{lemfrac}, we see $\| G_3\|_{L^{p'}} \lesssim_{z, b, d, \delta} \|\hat{f} \|_{H^{\delta}}$.\\
	
	These three parts conclude the proof of \eqref{eqana1}.\\

	 \underline{2. Proof of \eqref{eqana2}.} Define 
	 $$\Theta(t, w) := \left\{ \begin{array}{ll} \frac{e^{tw} -1}{w} & w \in \CC - \{ 0\}, \\ t & w =0,  \end{array}\right. $$
	 with $t \ge 0$ and let 
	 $$a := \frac{i (z' - z)}{b},\quad |a| \ll 1.$$ 
	 Recalling the computation in Lemma \ref{lemresident}, we can uniformly write 
	 \bee &&\mathscr{F}\left(R_b^+(z') R_b^+(z) f\right) (r\omega) 
	 = -\frac{1}{b^2} r^{-\frac d2} e^{\frac{ir^2}{2b}-\frac{i z}{b} \ln r} \int_r^\infty \mu^{\frac d2 - 1} e^{-\frac{i\mu^2}{2b}+\frac{i z}{b} \ln \mu} \Theta\left(\ln \left(\frac \rho r\right), a  \right)  \hat{f}(\mu \omega) d\mu. \eee
	 $\Theta$ satisfies simple estimates
	 \bee \left| \Theta(t, w) - \Theta(t, 0) \right| \lesssim \left\{ \begin{array}{ll} |w|t^2 & t \le 2|w|^{-1}  , \\ t e^{|w|t} & t \ge \frac{1}{2} |w|^{-1},  \end{array}\right. \\
	 \left| \partial_t( \Theta(t, w) - \Theta(t, 0)) \right| \lesssim \left\{ \begin{array}{ll} |w|t & t \le 2|w|^{-1}  , \\ e^{|w|t} & t \ge \frac{1}{2} |w|^{-1}.  \end{array}\right.
	 \eee
	 Like the first part, we decompose $\hat{f}$ into three parts as \eqref{eqanadecomp1}-\eqref{eqanadecomp4} and estimate each contribution $G_1$, $G_2$ and $G_3$ as
	 \bee  &&\mathscr{F}\left(R_b^+(z') R_b^+(z) f\right) (r\omega) -  \mathscr{F}\left(R_b^+(z)^2 f\right) (r\omega) \\
	 &=& -\frac{1}{b^2} r^{-\frac d2} e^{\frac{ir^2}{2b}-\frac{i z}{b} \ln r} \int_r^\infty \rho^{\frac d2 - 1} e^{-\frac{i\rho^2}{2b}+\frac{i z}{b} \ln \rho}\left[ \Theta\left(\ln \left(\frac \rho r\right), w  \right) -\Theta\left(\ln \left(\frac \rho r\right), 0  \right)  \right] \sum_{k=1}^3 g_k(\rho \omega) d\rho \\
	 &=:& \sum_{k=1}^3 G_k(r\omega) .
	 \eee
	 We control each part similarly to the counterpart for \eqref{eqana1} but with an $|a|$ factor. We only sketch the proof and focus on the additional difficulty from $\Theta$. Note that $\ln(\rho/r) \le 2|a|^{-1}$ and $\ln (\rho/r)\ge \frac 12 |a|^{-1}$ are equivalent to $\rho \le re^{2|a|^{-1}}$ and  $\rho \ge re^{\frac 12|a|^{-1}}$ respectively. \\

	 \emph{(1) Term $G_1$}. Again, $G_1$ is supported in $B_{\frac 32}$. So we assume $r \le \frac 32$ here. Take $\e_1 > 0$ such that $\frac d2 - \frac dp > - \frac{\Im z}{b} + \e_1$ as before and $0 < \e \ll \e_1$ specified later. We assume $|a| \le \e$.
	 
	 If $re^{|a|^{-1}} > 1$, we know $\rho \le re^{2|a|^{-1}}$ for all $\rho \in \left[r, \frac 32\right]$, so
	 \bee
	 |G_1(r\omega)| &\lesssim_{\e}& b^{-2} r^{-\frac d2 + \frac{\Im z}{b}}  \int_r^\frac 32 \rho^{\frac d2 - 1 - \frac{\Im z}{b}} |a|\left(\frac \rho r\right)^{2\e} |\hat{f}(\rho \omega)|d\rho\\
     &\lesssim_{b, z, \e}& |a| b^{-2} r^{-\frac d2 - \max\left\{ -\frac{\Im z}{b} + 2\e, 0\right\}} |\ln r|\left(\int_r^\frac 32 |\hat{f}(\rho \omega)|^2 \rho^{d-1} d\rho \right)^{\frac 12}.
	 \eee
	 If $re^{|a|^{-1}} \le 1$, using that $e^{-|a|^{-1}} \lesssim |a|$ when $|a| \ll 1$, 
	 \bee &&|G_1(r)| \\
	 &\lesssim_\e& \left[ \left(\int_r^{e^{|a|^{-1} }r}  \rho^{-1-2\frac{\Im z }{b}} |a|^2  \left(\frac \rho r\right)^{4\epsilon}d\rho \right) + \left( \int_{e^{|a|^{-1}} r}^\frac 32 \rho^{-1-2\frac{\Im z }{b}} \left(\frac \rho r\right)^{2\epsilon+2|a|}  d\rho\right)\right]^{\frac 12} 
	 \\
	 &&  \cdot  b^{-2} r^{-\frac d2 +\frac{\Im z}{b}}  \left( \int_0^\infty |\hat{f}(\rho\omega)|^2 \rho^{d-1}d\rho \right)^{\frac 12} \\
	 &\lesssim& \left[(|a|^2 + e^{-2|a|^{-1}})  \int_{r}^\frac 32 \rho^{-1-2\frac{\Im z}{b}} \left(\frac \rho r\right)^{4\epsilon+2|a|} d\rho \right]^{\frac 12}
	 \\
	 && \cdot  b^{-2} r^{-\frac d2 +\frac{\Im z }{b}}   \left( \int_0^\infty |\hat{f}(\rho\omega)|^2 \rho^{d-1}d\rho \right)^{\frac 12} \\
	 &\lesssim_{\e, b, z}& |a|r^{-\frac d2 - \max\left\{ -\frac{\Im z}{b} + 2\e+|a|, 0\right\}} |\ln r| \left( \int_0^\infty |\hat{f}(\rho\omega)|^2 \rho^{d-1} \right)^{\frac 12}.
	 \eee
	 So letting $3\e \le \e_1$, we can finish the estimate as before with an $|a|$ factor. \\

      \emph{(2) Terms $G_2$ and $G_3$}. For $G_2$, we integrate by parts over the quadratic phase to gain one order decay of $\rho$ with the cost of creating $\partial_\rho g_2(\rho \omega)$ and having weights $\left[ \Theta\left(\ln \left(\frac \rho r\right), a  \right) -\Theta\left(\ln \left(\frac \rho r\right), 0  \right)  \right]$ and $\partial_t \left[ \Theta\left(\ln \left(\frac \rho r\right), a  \right) -\Theta\left(\ln \left(\frac \rho r\right), 0  \right)  \right]$. These new weights (compared with the estimate for $G_2$ of \eqref{eqana1}) can be bounded by $C_{\e}|a|(\rho  / r)^{5\e + |a|}$ as we do in $G_1$, so it will not affect the estimate if we take $\e, |a|$ small enough. Therefore, the estimate of $G_2$ follows in the same way with a factor $|a|$ and so is $G_3$.
     
	 \end{proof}

\subsection{Improvement of decay and regularity}

Recall that $(-\Delta - z)^{-1}$ with $\Im z \neq 0$ as a convolution with exponential decaying function can retain polynomial spatial decay. A similar effect holds for our resolvent of $\Delta_b$, which will be used to improve the decay of generalized eigenfunctions. 

\begin{lemma}\label{lemimpreg}
	For $d \ge 1$, $\delta, \delta' \in (0, 1)$ and $z \in \CC$ satisfying 
	\be\frac{\Im z}{|b|} > \delta > 3 \delta' \label{eqimpregcond} \ee then, we have
	\be \left\| \la x \ra^{\delta'} R_{b}^+(z) f \right\|_{L^2}\lesssim_{b,d, z, \delta, \delta'} \|\la x \ra^\delta f\|_{L^2} \ee
	And the same bound holds for $R_b^-(z)$ with the corresponding $z$.
\end{lemma}

\begin{proof}We will prove the equivalent formulation on the Fourier side 
	\bea \left\| \widehat{R_{b}^+(z) f} \right\|_{H^{\delta'}}\lesssim_{b,d, z, \delta, \delta'} \|\hat{f}\|_{H^\delta}. \label{eqimpreg1}
	\eea
	Also since $\Im z > 0$, the $L^2$ bound from Lemma \ref{lemRbbdd} reduces the task to bound the $\dot{H}^{\delta'}$-norm of left-hand side.
	Recall the equivalent expression of $\dot{H}^\delta$-norm (see for example \cite[Proposition 1.37]{bahouri2011fourier})
	\be \| f \|_{\dot{H}^\delta}^2 = C_{\delta, d} \int_{\RR^d}  \int_{\RR^d} \frac{\left|f(x - y) - f(x)\right|^2}{|y|^{d+2\delta}} dy dx.\label{eqHdelta} \ee
	We have 
	\begin{align}
		\left\| \widehat{R_{b}^+(z) f} \right\|_{\dot H^{\delta'}}^2 &\lesssim_{d, \delta'}  \int_{|\eta| \ge \e_0} |\eta|^{-(d+2\delta')} \int_{\RR^d} 2\left( |\widehat{R_{b}^+(z) f}(\xi -\eta)|^2 +  |\widehat{R_{b}^+(z) f}(\xi)|^2\right) d\xi d\eta \label{eqimprove1}\\
		&+\int_{|\eta|\le \e_0} \int_{\RR^d} \frac{|\widehat{R_{b}^+(z) f}(\xi-\eta) - \widehat{R_{b}^+(z) f}(\xi)|^2 }{|\eta|^{d+2\delta'}} d\xi d\eta \label{eqimprove2}
	\end{align}
	The $\e_0 = \e_0(d)$ will be taken small enough. Notice that the $L^2$ bound implies $ |\eqref{eqimprove1}| \lesssim_{\delta', d} \| \hat{f}\|_{L^2}^2$. For \eqref{eqimprove2}, we will show
	\be \int_{\RR^d} \left|\widehat{R_{b}^+(z) f}(\xi-\eta) - \widehat{R_{b}^+(z) f}(\xi)\right|^2 d\xi \lesssim_{d, b, z, \delta} |\eta|^{\frac{2\delta}{3}} |\ln |\eta|| \| \hat{f}\|_{H^\delta}^2,   \label{eqimprove3}
	\ee
	for fixed $\eta \le \e_0$, and it concludes the proof of \eqref{eqimpreg1} because $\delta > 3 \delta'$.
	We will decompose the domain of $\xi$ to show \eqref{eqimprove3}. We first let $100\e_0^{\frac 13} < 1$ such that $50 \e_0^\frac 13 < 1 < (50\e_0^{\frac 13})^{-1}$.  \\
	
	\underline{1. For $|\xi| \le 20|\eta|^{\frac 13}$.} Since $\| \hat{f} \|_{L^{p}} \lesssim \| \hat{f}\|_{H^\delta}$ with $\frac 12 - \frac 1p = \frac \delta d$, we have
	\bee  \int_r^\infty \rho^{\frac d2 - \frac{\Im z}{b} - 1} |\hat{f}(\rho \omega)| d\rho 
	\lesssim_{z, b, d, p,\delta} r^{\delta - \frac{\Im z}{b}} \left( \int_r^\infty |\hat{f}(\rho \omega)|^p \rho^{d-1} \right)^{\frac 1p}
	\eee
	where we used $ \frac{Im z}{b} > \delta$. Then 
	\bee
	&& \int_{B_{20|\eta|^{1/3}}} |\widehat{R_{b}^+(z) f}(\xi-\eta) - \widehat{R_{b}^+(z) f}(\xi)|^2  d\xi
	\le  \int_{B_{30|\eta|^{1/3}}} |\widehat{R_{b}^+(z) f}(\xi)|^2 d\xi \\
	&\lesssim_{z, b, d, p,\delta}& \int_0^{30|\eta|^{1/3}}  r^{2\left(-\frac d2 +\frac{\Im z}{b} \right) + 2\left(\delta - \frac{\Im z}{b}\right)} r^{d-1} dr \left( \int_{\mathbb{S}^{d-1}}  \left( \int_r^\infty |\hat{f}(\rho \omega)|^p \rho^{d-1} \right)^{\frac 2p} d\omega \right) \\
	& \lesssim_{d, p, \delta} & \eta^{\frac{2\delta}{3}} \| \hat{f} \|_{H^\delta}^2
	\eee 
	as required in \eqref{eqimprove3}. \\
	
	\underline{2. For $|\xi| \ge (20|\eta|^{\frac 13})^{-1}$.}
	Here we mollify $\hat{f}$ to use integration by parts. Let $\hat{g}_1(\xi) = \hat{f} * \psi_{|\xi|^{-1}} (\xi)$, and $\hat{g}_2 (\xi) = \hat{f}(\xi) - \hat{g}_1(\xi)$. Then like the estimate in Lemma \ref{lemanalyticity}, we have
	\bee
	&&\int_r^\infty e^{-i \frac{\rho^2}{2b} + i\frac zb \ln \rho}\rho^{\frac d2 - 1} \hat{g}_1(\rho \omega)  d\rho
	=
	\int_r^\infty \left( \frac{-b}{i\rho}\partial_\rho e^{-i \frac{\rho^2}{2b}}\right) e^{i\frac zb \ln \rho}\rho^{\frac d2 - 1} \hat{g}_1(\rho \omega)  d\rho  \\
	&=& -ibr^{\frac d2 -2}e^{-i \frac{r^2}{2b} + i\frac zb \ln r } \hat{g}_1(r \omega) -ib \int_r^\infty e^{-i \frac{\rho^2}{2b} + i\frac zb \ln \rho}\rho^{\frac d2 - 2} \left( \rho^{-1}\left( \frac d2 - 2 + i\frac zb \right) \hat{g}_1(\rho \omega) + \pa_\rho \hat{g}_1(\rho \omega) \right)  d\rho
	\eee
	Thus 
	\bee |\widehat{R^+_b(z)g_1}(r\omega)| &\lesssim_{z, b, d}& r^{-2}|\hat{g}_1(r\omega)| + r^{-\frac d2 + \frac{\Im z}{b}} \int_r^\infty \rho^{\frac d2 - 2 - \frac{\Im z}{b}} \left( \rho^{-1}|\hat{g}_1(\rho\omega)| + |\partial_\rho \hat{g}_1(\rho \omega)| \right) d\rho 
	\eee
	So applying \eqref{eqfrac1} and \eqref{eqfrac4} from Lemma \ref{lemfrac} and using $\frac{\Im z}{b} > \delta$,
	\bee
	\int_{B_{ 10^{-1}|\eta|^{-1/3} }^c} |\widehat{R_{b}^+(z) g_1}(\xi)|^2 d\xi 
	&\lesssim_{z, b, d, \delta}& |\eta|^{\frac 23} \| \hat{g}_1 \|_{L^2(B_1^c)}^2  + |\eta|^{\frac{2\delta}{3}} \| \pa_r \hat{g}_1 |\cdot|^{-1+\delta} \|_{L^2(B_1^c)}^2 \\
	&\lesssim_{b, \delta} & |\eta|^{\frac{2\delta}{3}} \| \hat{f}\|_{H^\delta}^2.
	\eee
	The $\hat{g}_2$ part has the same bound as above due to direct estimate via \eqref{eqfrac3}.\\
	
	\underline{3. For $ |\xi| \in (20|\eta|^{\frac 13}, (20|\eta|^{\frac 13})^{-1})$.}
	We begin with decomposing $\hat{f}$ into $\hat{f}_- := \hat{f} \chi_{(10|\eta|^{\frac 13})^{-1}}$ and $\hat{f}_+ := \hat{f} - \hat{f}_-$. Then since $\hat{f}_+$ supported outside $B_{(20|\eta|^{\frac 13})^{-1}}$, we can estimate the integral
	in $\widehat{R_b^+(z)f_+}$ for $r = |\xi| < (20|\eta|^{\frac 13})^{-1}$ like the previous case,
	\bee 
	&&\left| \int_r^\infty e^{-i \frac{\rho^2}{2b} + i\frac zb \ln \rho}\rho^{\frac d2 - 1} \hat{f}_+(\rho \omega)  d\rho  \right|  \\
	&\lesssim_{z, b, d, \delta} & |\eta|^{\frac 23 +\frac{\Im z}{3b}} \left( \int_1^\infty \left| \hat{f}_+ * \psi_{\rho^{-1}}(\rho \omega)\right|^2 \rho^{d-1}d\rho \right)^{\frac 12} \\
	&+&  |\eta|^{\frac{\delta}{3}+\frac{\Im z}{3b}} \left( \int_1^\infty \left[\left| \rho^{-1} \pa_\rho \left( \hat{f}_+ * \psi_{\rho^{-1}}\right)(\rho \omega) \right|^2 + \left| \left(\hat{f}_+ - \hat{f}_+ * \psi_{\rho^{-1}}\right)(\rho \omega) \right|^2\right]  \rho^{d-1+2\delta}d\rho \right)^{\frac 12} 
	\eee
	So similarly by Lemma \ref{lemfrac},
	\bee &&\int_{B_{ 10^{-1}|\eta|^{-1/3} } - B_{ 10|\eta|^{1/3}}} |\widehat{R_{b}^+(z) f_+}(\xi)|^2 d\xi  \\
	&\lesssim_{z, b,d,\delta}& \int_{10|\eta|^{1/3}}^{10^{-1}|\eta|^{-1/3}} r^{-d+\frac{2\Im z}{b}} r^{d-1} dr \cdot  |\eta|^{\frac{2\delta}{3}+\frac{2\Im z}{3b}} \| \hat{f}_+\|_{H^\delta}^2 \lesssim_{z, b} |\eta|^{\frac{3\delta}{2}} \| \hat{f}\|_{H^\delta}^2
	\eee
	which is the desired bound in \eqref{eqimprove3}. 
	
	Finally, for $\hat{f}_-$, we need to use the interacton. Let $\xi = r\omega$ and $\xi - \eta = \tilde{r}\tilde{\omega}$ written in polar coordinates, and for convenience, define $A(r) := r^{-\frac d2} e^{i\frac{r^2}{2b}-i\frac{z}{b}\ln r}$ to write 
	\[ \widehat{R_b^+(z)f}(r\omega) = \frac{i}{b} A(r) \int_r^\infty \left(A(\rho)\rho\right)^{-1}\hat{f}(\rho\omega)d\rho. \]
	Now we decompose
	\bea &&-ib\left(\widehat{R_{b}^+(z) f_-}(\xi-\eta) - \widehat{R_{b}^+(z) f_-}(\xi)\right) \nonumber\\
	&=& \left( A(\tilde{r}) - A(r) \right) \int_{\tilde{r}}^\infty \left(A(\rho)\rho\right)^{-1} \hat{f}_-(\rho \tilde{\omega}) d\rho \label{eqimp31} \\
	&+& A(r)  \int_{\tilde{r}}^\infty \left[ \left(A(\rho)\rho\right)^{-1} -\left(A\left(\rho +r-\tilde{r}\right) \left(\rho +r-\tilde{r}\right) \right)^{-1} \right] \hat{f}_-(\rho \tilde{\omega}) d\rho  \label{eqimp32}\\
	&+& A(r) \int_r^\infty \left(A(\rho)\rho\right)^{-1} \left[ \hat{f}_- ((\rho + \tilde{r} - r)\tilde{\omega}) - \hat{f}_- (\rho\omega) \right]d\rho \label{eqimp33}
	\eea
	and discuss their contributions to \eqref{eqimprove3} respectively. \\
	
	\emph{(1) Esimate on \eqref{eqimp33}.}
	We plug in a mollified $\hat{f}_-$ to get the $H^\delta$-norm directly.
	\bea 
	\eqref{eqimp33} &=& A(r) \int_r^\infty \left(A(\rho)\rho\right)^{-1} \left[ \hat{f}_- ((\rho + \tilde{r} - r)\tilde{\omega}) - \hat{f}_- * \psi_{|\eta|^{\frac 13}} (\rho\omega) \right]d\rho \label{eqimp35}\\
	&+& A(r) \int_r^\infty \left(A(\rho)\rho\right)^{-1} \left[ \hat{f}_- (\rho\omega) - \hat{f}_- * \psi_{|\eta|^{\frac 13}} (\rho\omega) \right]d\rho \label{eqimp36}
	\eea
	We only consider \eqref{eqimp35} because \eqref{eqimp36} will follow likewise.
	Since $\rho \le (5|\eta|^{\frac 13})^{-1}$, $|\xi| \ge 20|\eta|^{\frac 13}$ and $|\eta|\ll 1$, we have
	\[ |(\rho + \tilde{r} - r)\tilde{\omega} - \rho\omega| = \left| (\rho + |\xi-\eta| - |\xi|) \frac{\xi - \eta}{|\xi-\eta|} - \rho \frac{\xi}{|\xi|} \right| \le \left(\frac{2\rho}{|\xi - \eta|} + 1\right)|\eta| \le \frac{1}{5}|\eta|^{\frac 13}  \]
	where the ${\frac 1 3}$-power is used. 
	Define 
	\be\gamma\left(\rho, \frac{\xi}{|\xi|}; |\xi|, \eta\right) :=  (\rho + |\xi-\eta| - |\xi|) \frac{\xi - \eta}{|\xi-\eta|}. \label{eqimprovevar}\ee
	Similar as the proof of \eqref{eqfrac3}, we have
	\bee &&|\eqref{eqimp35}| \le r^{-\frac d2 + \frac{\Im z}{b}} \int_r^\infty \rho^{\frac d2 - 1 - \frac{\Im z}{b}} \int_{B_{|\eta|^{\frac 13}}} \psi_{|\eta|^{\frac 13}}(\sigma) \left| \hat{f}_-\left( \gamma(\rho, \omega; r, \eta) \right) - \hat{f}_-\left(\rho \frac{\xi}{|\xi|} - \sigma\right)   \right| d\sigma d\rho \\
	&\le &r^{-\frac d2 + \frac{\Im z}{b}} \int_r^\infty \rho^{\frac d2 - 1 - \frac{\Im z}{b}} \left( \int_{B_{\frac{3|\eta|^{1/3}}{2}}} \frac{\left| \hat{f}_-\left( \gamma(\rho, \omega; r, \eta) \right) - \hat{f}_-\left( \gamma(\rho, \omega; r, \eta) -\sigma \right) \right|^2}{|\sigma|^{d+2\delta}} d\sigma \right)^{\frac 12} |\eta|^{\frac{\delta}{3}}d\rho \\
	&\lesssim_{d, b, z, \delta}& r^{-\frac d2} |\eta|^{\frac \delta 3}
	\left( \int_{5|\eta|^{\frac 13}}^{5^{-1}|\eta|^{-\frac 13}} \int_{B_\frac{3|\eta|^{1/3}}{2}} \frac{\left| \hat{f}_-\left( \gamma(\rho, \omega; r, \eta) \right) - \hat{f}_-\left( \gamma(\rho, \omega; r, \eta) -\sigma \right) \right|^2}{|\sigma|^{d+2\delta}} d\sigma \rho^{d-1} d\rho \right)^{\frac 12}
	\eee
	where the integration domain of $\rho$ comes from support of $\hat{f}_-$ and that $|\gamma(\rho, \omega; r, \eta)| = \rho + |\xi - \eta| - |\xi|$. 
	
	Let $\rho \omega = \tilde{\xi}$. For any $r \in  (20|\eta|^{\frac 13}, (20|\eta|^{\frac 13})^{-1})$ and $\eta \ll 1$, it is elementary to check that $\gamma(\cdot; r, \eta)$ is almost an identity map on that domain of $\rho$. More precisely, it maps $B_{5^{-1}|\eta|^{-\frac 13}} - B_{5|\eta|^{\frac 13}}$ into $B_{4^{-1}|\eta|^{-\frac 13}} - B_{4|\eta|^{\frac 13}}$ and $\left| \frac{\partial_i \gamma(\cdot;r, \eta)}{\partial \tilde{\xi}_j} - \delta_{ij} \right| \lesssim_{d} |\eta|^{\frac 23} \ll 1$ uniformly for $\tilde{\xi}$, $r$ in the above domain and $|\eta| \ll 1$, which implies that $\gamma$ is a bijection from $B_{5^{-1}|\eta|^{-\frac 13}} - B_{5|\eta|^{\frac 13}}$ onto its range with uniformly bounded Jacobian. Therefore, after further decreasing $\e_0$ if necessary, we have
	\bee 
	&&\int_{B_{ 20^{-1}|\eta|^{-1/3} } - B_{ 20|\eta|^{1/3}}} |\eqref{eqimp35}|^2 d\xi \\ 
	&\lesssim_{b, d, z, \delta}& |\eta|^\frac{2\delta}{3} \int^{20^{-1}|\eta|^{-1/3}}_{20|\eta|^{1/3} } r^{-1} \int_{B_{\frac{1}{5|\eta|^{\frac 13}}} - B_{5|\eta|^{\frac 13}}} \int_{\RR^d} \frac{\left| \hat{f}_-\left( \gamma(\tilde{\xi}; r, \eta) \right) - \hat{f}_-\left( \gamma(\tilde{\xi}; r, \eta) -\sigma \right) \right|^2}{|\sigma|^{d+2\delta}} d\sigma d\tilde{\xi} dr \\ 
	& \lesssim_{d} & |\eta|^\frac{2\delta}{3} \int^{10^{-1}|\eta|^{-1/2}}_{10|\eta|^{1/2} } r^{-1} \int_{B_{\frac{1}{4|\eta|^{\frac 13}}} - B_{4|\eta|^{\frac 13}}} \int_{\RR^d} \frac{\left| \hat{f}_-( \gamma) - \hat{f}_-( \gamma -\sigma ) \right|^2}{|\sigma|^{d+2\delta}} d\sigma d\gamma dr \\
	& \lesssim_{b, d,z, \delta}& |\eta|^\frac{2\delta}{3} \left| \ln |\eta| \right| \| \hat{f}\|_{H^\delta}^2.
	\eee
	This concludes the estimate for \eqref{eqimp35}. \eqref{eqimp36} follows likewise.
	\\
	
	\emph{(2) Estimates on \eqref{eqimp31} and \eqref{eqimp32}.}
	Notice that $|A'(r)| \lesssim_{d, z, b} \left( r^{-1} + r \right) |A(r)|$ and $|\tilde{r} - r| \le |\eta|$. Using the fundamental theorem of calculus to bound the difference term in \eqref{eqimp31} and \eqref{eqimp32}, we will gain $|\eta|$ and a factor $r+r^{-1}$ or $\rho + \rho^{-1}$ compared with the original estimate for $|\widehat{R^+_b(z)f_-}|$. Since $\supp \hat{f}_- \subset B_{5^{-1}|\eta|^{-1/3}}$ and $r = |\xi| \ge 20|\eta|^{\frac 13}$, the extra term can be bounded by $|\eta|^{-1/3}$. Thus we can estimate trivially
	\bee
	\left|\eqref{eqimp31}\right| + \left|\eqref{eqimp32}\right| 
	&\lesssim_{b, d, z}& |\eta|^{\frac 23} r^{-\frac d2 + \frac{\Im z}{b}} \int_{\tilde{r}}^{(5|\eta|^{1/2})^{-1}} \rho^{\frac d2 - 1 - \frac{\Im z}{b}} |\hat{f}_-(\rho \tilde{\omega})| d\rho  \\
	& \lesssim_{d}& |\eta|^{\frac 23} r^{-\frac d2} \left(\int_{10|\eta|^{\frac 13}}^{5^{-1}|\eta|^{-\frac 13}} |\hat{f}_-(\rho \tilde{\omega})|^2 \rho^{d-1} d\rho \right)^{\frac 12}.
	\eee
	Here the decay in $|\eta|$ is enough, and in the integration over $\xi$, we can perform a similar change of variable like \eqref{eqimprovevar} to make $\rho\omega \mapsto \rho \tilde{\omega}$ with the given $r$ and $\eta$. This implies the last piece in the bound \eqref{eqimprove3} and concludes the proof of this lemma.
\end{proof}

\subsection{Conjugation with $D^\sigma$}
	
	Finally, we mention the simple but important fact about conjugating $\Delta_b$, $e^{it\Delta_b}$ and $R_b^\pm(z)$ with $D^\sigma$. This implies the $\dot{H}^\sigma$-energy decay of $e^{it\Delta_b}$.
	
	\begin{lemma}\label{lemcom} For $b \neq 0$ and $\sigma \ge 0$, we have
		\bea
		  D^\sigma \Delta_b D^{-\sigma}&=& \Delta_b + ib\sigma. \label{eqcomDeltab} \\
		  \label{eqcomgroup} D^\sigma e^{it\Delta_b} D^{-\sigma} &=& e^{it\Delta_b} e^{-b\sigma t}. \\
		\label{eqcomres} D^\sigma R_b^\pm(z) D^{-\sigma} &=& R_b^\pm (z + ib\sigma).
		\eea
		\end{lemma}
	\begin{proof}
		The first identity \eqref{eqcomDeltab} easily follows $[D^\sigma, x\cdot \nabla] = \sigma D^\sigma$. 
	The other two relations also come from the represntation formulas of $e^{it\Delta_b}$ \eqref{eqrepuhat} and of $R_b^\pm(z)$ \eqref{eqRb+} respectively:
	\bee \mathscr{F} D^\sigma e^{it\Delta_b} D^{-\sigma} f &=& |\xi|^{\sigma }e^{\frac{bdt}{2}} \hat{f} (e^{bt} \xi)|e^{bt}\xi|^{-\sigma} e^{-i\frac{e^{2bt} - 1}{2b}|\xi|^2 } = e^{-b\sigma t}\mathscr{F} e^{it\Delta_b} f, \\
	 D^\sigma R_b^\pm(z) D^{-\sigma} &=&\pm i\int_0^\infty e^{\pm itz}D^\sigma e^{\pm it\Delta_b} D^{-\sigma} dt  \\
	&=& \pm i\int_0^\infty e^{\pm itz} e^{\pm it\Delta_b} e^{\mp b\sigma t} = R_b^\pm(z + ib\sigma).
	\eee
	\end{proof}

\section{Resolvent estimates for $\Delta_b$ with potential}\label{s3}

In this section, we consider the compactness and decay in $z$ of operator families $V_2 R_b^\pm(z) V_1$ on $\dot{H}^\sigma$ where $V_1$, $V_2$ are scalar potentials. From the commutator relation \eqref{eqcomres}, it's equivalent to discuss 
$$D^\sigma V_2 R_b^\pm(z) V_1 D^{-\sigma} = D^\sigma V_2 D^{-\sigma}  R_b^\pm(z + ib\sigma)  D^{\sigma}V_1 D^{-\sigma} $$
as an operator on $L^2$. The potential  will satisfy polynomial decay 
\be |V(x)|\lesssim \la x \ra^{-\a},\quad |\nabla V(x)| \lesssim \la x \ra^{-\a-1} \tag{\ref{eqpotbd0}} \ee
for some $\a > 0$ as required in Theorem \ref{thmStrichartz}.

We also remark that for every lemma and proposition involving $R_b^\pm(z)$ in this section, we will only prove for the case $R_b^+(z)$ with $b > 0$, because the other three cases ($b < 0$ or $R^-_b(z)$) can be done in almost the same way.

\subsection{Compactness}

In this subsection, we prove compactness results for the extended resolvent coupled with potential. In contrast to $b=0$ case, here any small polynomial decay of $V$ is enough to guarantee compactness.

\begin{lemma}
	\label{lempotcpt2} 
	For $d \ge 1$, $0 < \sigma<  \min \{\frac d2, 1\}$ and $V \in C^1(\RR^d)$ satisfying \eqref{eqpotbd0} with $ \a > 0$, then
    \begin{enumerate}
    	\item[\rm (i)] The operator $D^\sigma V D^{-\sigma} R^\pm_b(z)$ for $\pm \Im z > 0 $ is bounded and compact on $L^2(\RR^d)$.
    	\item[\rm (ii)] The operator $D^\sigma V_1 D^{-\sigma} R^+_b(z) D^\sigma V_2 D^{-\sigma}$ for $\Im z > - |b| \min\{\a, 1, \frac{d}{2}-\sigma \}$ (and $D^\sigma V_1 D^{-\sigma} R^-_b(z) D^\sigma V_2 D^{-\sigma}$ for $\Im z < |b| \min\{\a, 1, \frac{d}{2}-\sigma \}$) is bounded and compact on $L^2(\RR^d)$.
    	\end{enumerate}
\end{lemma}

\begin{proof}
	(i) From Riesz-Kolmogorov compactness criterion, it suffices to prove 
	\begin{align}
		\| D^\sigma V D^{-\sigma} R^+_b(z) \|_{L^2 \to L^2} <& \infty, \label{eqcpt01} \\
		\| \mathbbm{1}_{B_R^c} D^\sigma V D^{-\sigma} R^+_b(z) \|_{\L(L^2)} \to 0,&\quad \mathrm{as}\,\,R\to \infty,\label{eqcpt02}\\
		\| \mathbbm{1}_{B_R^c}\mathscr{F}  D^\sigma V D^{-\sigma} R^+_b(z)  \|_{\L(L^2)}\to 0,&\quad \mathrm{as}\,\,R\to \infty. \label{eqcpt03}
	\end{align}

    (1) The boundedness \eqref{eqcpt01} follows from Lemma \ref{lemRbbdd}, Lemma \ref{lempotbdd} and $\Im z > 0$ so that $D^\sigma V D^{-\sigma} \in \L(L^2)$ and $R^+_b(z) \in \L(L^2)$.\\
    
    (2) For the spatial localization \eqref{eqcpt02}, we will prove
    \be \| (1 - \chi(R^{-1}|x|))D^\sigma V D^{-\sigma} \|_{L^p \to L^2} \to 0,\quad \mathrm{as} \,\,R\to +\infty \label{eqcpt21}  \ee
    with $\frac 12 \ge \frac 1p >\max\{0,  \frac 12 - \frac{\a}{d}, \frac{\sigma}{d} \}$
    and $\chi(r) \in C^\infty_{c, rad}(\RR^d)$ to be $1$ when $0 \le r \le 1$ and $0$ when $r \ge \frac 32$. Then \eqref{eqcpt02} follows the case $p = 2$.

    From Lemma \ref{lempotbdd}, we can approximate $V$ by some $C^\infty_c(\RR^d)$ function and only prove these asymtotics for such good $V$. 
    Suppose $\supp V \subset B_M$ with $M \gg 1$ and $R \gg M$. We act the operator on $f \in \mathscr{S}(\RR^d)$. Recall the physical representation of $D^\sigma$
    \[ D^\sigma h (x) = C_{d, \sigma} \lim_{\e \to 0+}\int_{B_{\e}^c} \frac{h(x+z) - h(x)}{|z|^{d+\sigma}}dz.\]
    Note that here $VD^{-\sigma}f$ is supported in $B_M$ while $|x| \ge R \gg M$, so
    \bee \left(1-\chi(R^{-1}x)\right) D^\sigma V D^{-\sigma} f (x)
    = \left(1-\chi(R^{-1}x)\right) C \int_{\RR^d} \int_{\RR^d} \frac{V(x+z)f(y)}{|x+z-y|^{d-\sigma}|z|^{d+\sigma}}dydz
    \eee
    which immediately implies the pointwise estimate
    \bee 
    &&\left| \left(1-\chi(R^{-1}x)\right) D^\sigma V D^{-\sigma} f (x) \right| \\
    &\lesssim& \mathbbm{1}_{B_R^c}(x) |x|^{-{d+\sigma}} \| V\|_{L^\infty} \int_{\RR^d} \int_{\RR^d}\frac{|f(y)|}{|z'-y|^{d-\sigma}}\chi(M^{-1}z') dydz'\\
    &\le& \mathbbm{1}_{B_R^c}(x) |x|^{-{d+\sigma}} \| V\|_{L^\infty} \| f\|_{L^p} \| D^{-\sigma}\left(\chi(M^{-1}\cdot)\right) \|_{L^{p'}}
    \eee
    Since $\frac 1q:= \frac{1}{p'} + \frac{\sigma}{d} < 1$, we know $\| D^{-\sigma}\left(\chi(M^{-1}\cdot)\right) \|_{L^{p'}} \lesssim \| \chi(M^{-1}\cdot)\|_{L^{q}} < \infty$ and hence 
    \bee &&\| \left(1-\chi(R^{-1}x)\right) D^\sigma V D^{-\sigma} f \|_{L^2} \\
    &\lesssim& \||x|^{-(d+\sigma)}\|_{L^2(B^c_R)} \| V\|_{L^\infty} \| f\|_{L^p} \| \chi(M^{-1}\cdot)\|_{L^{q}} \lesssim R^{-\frac d2 - \sigma} \| f\|_{L^p} \eee
    which concludes \eqref{eqcpt21} and hence \eqref{eqcpt02}. \\
    
    (3) Finally, we will prove the Fourier localization (or called equicontinuity) \eqref{eqcpt03}. 
    
	We again begin with some reductions. Firstly, we can replace $V$ by a Schwartz approximation (still denoted by $V$) with $\hat{V}$  compactly supported.
	Secondly, we apply the semigroup decomposition of $R_b^+(z)$ \eqref{eqRb+} and remove small intervals near $0$ and $\infty$ because 
	\[ \left\| \int_0^{K^{-1}} e^{it\Delta_b}e^{itz} dt \right\|_{\L(L^2)} + \left\|\int_K^\infty e^{it\Delta_b}e^{itz} dt \right\|_{\L(L^2)} = o_{K\to \infty} (1)\]
	for our $\Im z > 0$. 
	Therefore \eqref{eqcpt03} is reduced to: for any $[a_1, a_2] \subset (0,\infty)$ and $V \in \mathscr{S}(\RR^d)$, $\hat{V} \in C^\infty_c(\RR^d)$, 
	\be \left\| \mathscr{F} D^\sigma V D^{-\sigma} \int_{a_1}^{a_2} e^{it\Delta_b} e^{itz} f dt \right\|_{L^2(B_R^c)} \le o_{R\to +\infty}(1) \|f\|_{L^2}. \ee
    Via direct computation and integration by parts over the quadratic phase, we get
	\bee && \left[\mathscr{F} D^\sigma V D^{-\sigma} \int_{a_1}^{a_2} e^{it\Delta_b} e^{itz} f dt \right] (\xi)  \\
	&=& |\xi|^{\sigma} \int_{\RR^d} \int_{a_1}^{a_2} \hat{V}(\xi - e^{-bt} \zeta) |\zeta|^{-\sigma} e^{t\left(b\sigma - \frac{bd}{2} + iz \right) } e^{-i\frac{1-e^{-2bt}}{2b}|\zeta|^2 } \hat{f}(\zeta) dt  d\zeta \\
	&=& |\xi|^\sigma \int_{\RR^d} |\zeta|^{-\sigma} \hat{f}(\zeta) \int_{a_1}^{a_2}  e^{t\left(b\sigma - \frac{bd}{2} + iz \right) } \frac{ie^{2bt}}{|\zeta|^{2}}  \frac{d}{dt}\left(e^{-i\frac{1-e^{-2bt}}{2b}|\zeta|^2 }\right) \hat{V}(\xi - e^{-bt} \zeta) dt d\zeta \\
	&=& i|\xi|^\sigma \int_{\RR^d} |\zeta|^{-\sigma-2} \hat{f}(\zeta) \Bigg[ \left( e^{t\left(2b + b\sigma - \frac{bd}{2} + iz \right) } e^{-i\frac{1-e^{-2bt}}{2b}|\zeta|^2 } \hat{V}(\xi - e^{-bt} \zeta)  \right)\bigg|_{a_1}^{a_2}  \\
	&& - \int_{a_1}^{a_2}e^{-i\frac{1-e^{-2bt}}{2b}|\zeta|^2 } \frac{d}{dt} \left( e^{t\left(2b + b\sigma - \frac{bd}{2} + iz \right) }  \hat{V}(\xi - e^{-bt}\zeta)\right) dt  \Bigg].
	\eee
	Suppose $\supp \hat{V} \subset B_{M}$. Then if $|\xi| \ge 2M$, the support of $\hat{V}$ indicates $|\xi - e^{-bt}\zeta| \le M$ and hence $|\xi| \sim |e^{-bt}\zeta| \sim_{a_1, a_2} |\zeta|$. Hence taking $R \ge 2M$ and letting $C = 2b+b\sigma - \frac{bd}{2}$, we have
	\bee && \left| \left[\mathscr{F} D^\sigma V D^{-\sigma} \int_{a_1}^{a_2} e^{it\Delta_b} e^{itz} f dt \right] (\xi)  \right| \\
	&\lesssim_{a_1, a_2}& |\xi|^{-2} \int_{\RR^d} |\hat{f}(\zeta)| \Bigg[ \int_{a_1}^{a_2}C e^{Ct} \left(|\hat{V}(\xi - e^{-bt}\zeta)| + |\xi||\nabla \hat{V}(\xi - e^{-bt}\zeta)| \right)dt \\
	&& + \sum_{i = 1, 2} e^{Ca_i} |\hat{V}(\xi - e^{-ba_i}\zeta )| \Bigg]d\zeta \\
	& \lesssim_{a_1, a_2, C}& |\xi|^{-1} \sup_{\tau \in [e^{-ba_2}, e^{-ba_1}]} \int_{\RR^d} |\hat{f}(\zeta)| \left(|\hat{V}(\xi - \tau \zeta)| +  | \nabla \hat{V} (\xi - \tau \zeta)| \right) d\zeta.
	\eee 
	Using Young's inequality and $\tau \sim_{b, a_1, a_2} 1$, we get 
	\bee \left\| \mathscr{F} D^\sigma V D^{-\sigma} \int_{a_1}^{a_2} e^{it\Delta_b} e^{itz} f dt \right\|_{L^2(B_R^c)} 
	\lesssim_{a_1, a_2, b, C} R^{-1} \| f\|_{L^2} \|\hat{V}\|_{W^{1, 1}} \eee
	when $R \ge 2M$. That confirms \eqref{eqcpt03} and hence part (i) of this lemma.\\
	
	(ii) Let $L :=  D^\sigma V_1 D^{-\sigma} R^+_b(z) D^\sigma V_2 D^{-\sigma} $. We again exploit Riesz-Kolmogorov compactness criterion. Notice that the condition on $\Im z$ enables us to pick 
	\be \frac{1}{p} \in \left( \frac 12 - \frac \a d, \frac 12\right] \cap \left( \frac \sigma d, \frac 12\right] \cap \left( \frac{1}{2} - \frac{1}{2^*},\frac 12 + \frac{\Im z}{|b|d} \right)\label{eqrangep}\ee
	Then the boundedness of $L$ comes from
	Lemma \ref{lemRbbdd} and Lemma \ref{lempotbdd}
	\[ \| L \|_{L^2 \to L^2} \le \| D^\sigma V_2 D^{-\sigma}\|_{L^2\to L^{p'}} \| R_b^+(z) \|_{L^{p'} \to L^p} \| D^\sigma V_1 D^{-\sigma}\|_{L^p\to L^{2}} < \infty. \]
	And the spatial localization follows \eqref{eqcpt21}. 
	
	Finally, we need to show the following equicontinuity condition
	\be \left\| Lf(\cdot + h) - Lf \right\|_{L^2} \le o_{|h|\to 0}(1) \| f \|_{L^2}. \label{eqcpt3}\ee
		Similar to the proof of Lemma \ref{lempotcpt2}, we begin with several reductions. Again, we replace $V_1$, $V_2$ by their $C^\infty_c$ approximations and peel off small neighborhoods of $0$ and $\infty$ in the semigroup decomposition of $R_b^+(z)$ \eqref{eqRb+}. In particular, from the compactness of interval and uniform boundedness of $e^{itz}$, it reduces to consider a uniform pointwise estimate: for any $[a_1, a_2] \subset (0, \infty)$,
	\[ \sup_{t \in [a_1, a_2]} \| (D^\sigma V_1 D^{-\sigma} e^{it\Delta_b} D^\sigma V_2 D^{-\sigma} f)(\cdot - h) - D^\sigma V_1 D^{-\sigma} e^{it\Delta_b} D^\sigma V_2 D^{-\sigma} f\|_{L^2} = o_{|h|\to 0}(1) \| f \|_{L^2}, \]
	By commutator formula \eqref{eqcomgroup} and uniform boundedness of $e^{b\sigma t}$, it further comes down to
	\[ \sup_{t \in [a_1, a_2]} \| (D^\sigma V_1 e^{it\Delta_b} V_2 D^{-\sigma} f)(\cdot - h) - D^\sigma V_1 e^{it\Delta_b}  V_2 D^{-\sigma} f\|_{L^2} = o_{|h|\to 0}(1) \| f \|_{L^2}. \]
	Next, notice that we may further decompose $V_i = V_{i1}V_{i2}$ with $V_{i1}, V_{i2} \in C^\infty_c$ for $i = 1, 2$. Using the boundedness $\|D^\sigma V_{11} D^{-\sigma}\|_{L^\frac{4d}{2d-1}\to L^{2}}$ and $\|D^\sigma V_{22}D^{-\sigma}\|_{L^2\to L^{\frac{4d}{2d+1}}}$ from Lemma \ref{lempotbdd}, it suffices to prove 
	\bee \sup_{t \in [a_1, a_2]} \| (D^\sigma V_{12} e^{it\Delta_b}  V_{21} D^{-\sigma} f)(\cdot - h) - D^\sigma V_{12}  e^{it\Delta_b}V_{21} D^{-\sigma} f\|_{L^\frac{4d}{2d-1}} = o_{|h|\to 0}(1) \| f \|_{L^\frac{4d}{2d+1}}.\label{eqcpt31} \eee
	Finally, with the $L^2\to L^2$ uniform boundedness
	\[ \sup_{t \in [a_1, a_2]} \| D^\sigma V_{12} D^{-\sigma} e^{it\Delta_b} D^\sigma V_{21} D^{-\sigma}\|_{L^2 \to L^2} <\infty \]
	and interpolation, \eqref{eqcpt3} boils down to
	\be \sup_{t \in [a_1, a_2]} \| (D^\sigma V_{12}  e^{it\Delta_b} V_{21} D^{-\sigma} f)(\cdot - h) - D^\sigma V_{12}  e^{it\Delta_b}  V_{21} D^{-\sigma} f\|_{L^\infty} = o_{|h|\to 0}(1) \| f \|_{L^1}. \label{eqcpt32}\ee
	For convenience, we denote $V_{12} =: W_1$ and $V_{21} =: W_2$. 
	
	Suppose $f \in \mathscr{S}(\RR^d)$, using representation formula of $D^\sigma$, $D^{-\sigma}$ and $e^{it\Delta_b}$ \eqref{eqrepu}, we can compute 
	\bee 
	&& W_1 e^{it\Delta_b} W_2 D^{-\sigma} f(\cdot) \\
	&=& C W_1(\cdot) \int_{\RR^d} |\sinh{bt}|^{-\frac d2}  e^{i\frac{b|(\cdot)e^{-bt}- y|^2}{2(1-e^{-2bt})} }W_2(y)\int_{\RR^d} f(z)|y-z|^{-(d-\sigma)} dzdy  
	\eee
	Let 
	\be \Phi_t(x, y):=e^{i\frac{b|x e^{-bt}- y|^2}{2(1-e^{-2bt})} },\quad A(t):= |\sinh{bt}|^{-\frac d2}. \label{eqaugfunc1}\ee
	We can write
	\bee 
	&& (D^\sigma W_1 e^{it\Delta_b} W_2 D^{-\sigma} f)(x)\\
	&=& C' A(t) \int_{\RR^d} \int_{\RR^d} \int_{\RR^d} \frac{W_1(x+w) \Phi_t(x+w, y)-W_1(x)\Phi_t(x,y) }{|w|^{d+\sigma}} W_2(y)\frac{f(z)}{|y-z|^{d-\sigma}}dz dy dw
	\eee
	Define 
	\be F_t(x, y):= W_1(x)\Phi_t(x, y) W_2(y).\label{eqaugfunc2}\ee
	Then 
	\bee
	&& (D^\sigma W_1 e^{it\Delta_b} W_2 D^{-\sigma} f)(x-h) -(D^\sigma W_1 e^{it\Delta_b} W_2 D^{-\sigma} f)(x) \\
	&=& C' A(t) \int_{\RR^d} \int_{\RR^d} \int_{\RR^d} \frac{F_t(x+w-h, y) - F_t(x-h, y) - F_t(x+w, y) + F_t(x, y) }{|w|^{d+\sigma}} \\
	&&\qquad \qquad \qquad \qquad \qquad \frac{f(z)}{|y-z|^{d-\sigma}}dz dy dw
	\eee
	Note that $\Phi_t(x, y) \in C^\infty_{loc}(\RR^d \times \RR^d)$ and thus $F_t(x, y) \in C^\infty_c(\RR^d \times \RR^d)$ as $W_1, W_2 \in C^\infty_c(\RR^d)$. Suppose $\supp F_t \subset B_R \times B_R$ and  $|h|\le 1$. Using
	\[ |F_t(q-h, y) - F_t(q, y)| \le \| F_t\|_{C^1(\RR^{2d})} |h| \mathbbm{1}_{B_{R+1}}(y) \]
	for $|w| \ge 1$
	and 
	\bee &&|F_t(x+w-h, y) - F_t(x+w-h, y) -F_t(x+w,h) + F_t(x, y)| \\
	&\le &\| F_t\|_{C^2(\RR^{2d})} |h||w| \mathbbm{1}_{B_{R+1}}(y) \eee
	for $|w| \le 1$, we obtain the pointwise estimate
	\bee 
	|(D^\sigma W_1 e^{it\Delta_b} W_2 D^{-\sigma} f)(x-h) -(D^\sigma W_1 e^{it\Delta_b} W_2 D^{-\sigma} f)(x)|
	\lesssim_{R, d, \sigma}  \| F_t\|_{C^2}A(t) |h| \| f \|_{L^1}
	\eee
	Since $\| F_t\|_{C^2}$ and $A(t)$ depend on $t$ continuously, this confirms \eqref{eqcpt32} and thus \eqref{eqcpt3}. The proof of this lemma is concluded.
\end{proof}

We remark that this lemma used two mechanisms of Fourier localization: for (i) we used the oscillation in time of $e^{it\Delta_b}$, while for (ii) it boils down to a smoothing effect of Schr\"odinger propagation with limited speed.

\subsection{Decay as $z \to \infty$}

In this subsection, we discuss the qualitative decay for $R_b^+(z)$ as $|z|$ large. 

\begin{lemma}[Qualitative decay of extended resolvent families]\label{lemopdecay}
	Let $d \ge 1$,  $0 < \sigma < \min \{\frac d2, 1\}$, and $V_1$, $V_2 \in C^1(\RR^d)$ satisfying \eqref{eqpotbd0} with $ \a > 0$. For any $\delta > 0$, we have for  $\Im z \ge \delta - |b| \min\{\a, 1\}$ 
	\be \| D^\sigma V_1 D^{-\sigma} R^+_b(z) D^\sigma V_2 D^{-\sigma}\|_{L^2 \to L^2} \to 0, \quad \mathrm{as} \,\, |z| \to +\infty. \ee 
	The similar result hold for $D^\sigma V_1 D^{-\sigma} R^-_b(z) D^\sigma V_2 D^{-\sigma}$ with $\Im z \le -\delta + |b| \min\{\a, 1\}$.
\end{lemma}
\begin{proof}
    From Lemma \ref{lemRbbdd}, Lemma \ref{lempotbdd} and the $\delta$ gap, we know $D^\sigma V_1 D^{-\sigma} R^+_b(z) D^\sigma V_2 D^{-\sigma}$ is uniformly bounded in $L^2\to L^2$ for $\Im z \ge \delta - |b| \min\{\a, 1\}$. Then via the reduction as in proving \eqref{eqcpt3} in Lemma \ref{lempotcpt2} (ii), it suffices to show for $W_1, W_2 \in C^\infty_c(\RR^d)$, any $[a_1, a_2] \subset (0, \infty)$,
	\be \left\| \int_{a_1}^{a_2} e^{b\sigma t} e^{itz} D^\sigma W_{1}  e^{it\Delta_b} W_{2} D^{-\sigma} f dt \right\|_{L^\infty} = o_{|z|\to +\infty}(1) \| f \|_{L^1}. \label{eqdecay1}\ee

	Following the computation in proving \eqref{eqcpt3}, we define 
	    \be \Phi_t(x, y):=e^{i\frac{b|x e^{-bt}- y|^2}{2(1-e^{-2bt})} },\,\, A(t):= \left|\frac{e^{bt} - e^{-bt}}{2b} \right|^{-\frac{d}{2}}, \,\, F_t(x, y):= W_1(x)\Phi_t(x, y) W_2(y).\label{eqdecay21}\ee
	and then we can write 
	\bee
	 (D^\sigma W_1 e^{it\Delta_b} W_2 D^{-\sigma} f)(x)
	= \int_{\RR^d} K_t(x, z)  f(z)dz
	\eee
	where
	\[ K_t(x, z) := A(t) e^{b\sigma t} \int_{\RR^d} \int_{\RR^d} \frac{F_t(x+w, y) - F_t(x, y) }{|w|^{d+\sigma}}\frac{1}{|y-z|^{d-\sigma}} dy dw. \] 
	We claim
		\be \sup_{t \in [a_1, a_2]} \left (\|K_t (x, z) \|_{L^\infty(\RR^d \times \RR^d)} + \|\pa_t K_t (x, z) \|_{L^\infty(\RR^d \times \RR^d)} \right )  < \infty. \label{eqdecay23} \ee
	Indeed, note that $F_t, \pa_t F_t \in C^\infty_c(\RR^d \times \RR^d)$ whose support is independent of $t$ and $A(t) \in C^\infty([a_1, a_2])$. We can estimate them similarly as in Lemma \ref{lempotcpt2} (ii) to conclude \eqref{eqdecay23}. 
	
	Now we are in place to use a contradictory argument to show \eqref{eqdecay1}. If not, then there exists $\epsilon_0 > 0$ and sequences of $t_n \in [a_1, a_2]$, $z_n \in \{\Im z \ge \delta_0 - |b|\min\{\a, 1\}\}$ such that $|z_n| \to +\infty$ and for any $n$,
	\be  \left\| \int_{a_1}^{a_2} e^{b\sigma t} e^{itz_n} D^\sigma W_{1}  e^{it\Delta_b} W_{2} D^{-\sigma} f dt \right\|_{L^\infty} \ge \epsilon_0 \| f \|_{L^1}.  \label{eqdecay2}\ee 
	Firstly, \eqref{eqdecay2} and \eqref{eqdecay23} indicate a uniform bound of $\Im z_n$ since
	 \bee 
	  &&\left\| \int_{a_1}^{a_2} e^{b\sigma t} e^{itz_n} D^\sigma W_{1}  e^{it\Delta_b} W_{2} D^{-\sigma} f dt \right\|_{L^\infty} \\
	  &\le& \sup_{x\in \RR^d} \int_{a_1}^{a_2} e^{-t\Im z_n} \int_{\RR^d} |K_t(x, z)| |f(z)| dz dt 
	  = o_{\Im z_n \to +\infty}(1) \|f\|_{L^1}.
	 \eee
	 So we can take a subsequence such that $\Im z_n \to c_1 \in [\delta - |b| \min\{\a, 1\}, \infty)$ and $\Re z_n \to +\infty$. Now decompose
	\bea
	 && \int_{a_1}^{a_2} e^{itz_n} D^\sigma W_{1}  e^{it\Delta_b} W_{2} D^{-\sigma} f dt  \nonumber\\ 
	 &=& \int_{a_1}^{a_2} (e^{itz_n} - e^{it(\Re z_n + ic_1)}) D^\sigma W_{1}  e^{it\Delta_b} W_{2} D^{-\sigma} f dt \label{eqdecay3} \\
	 && + 
	 \int_{a_1}^{a_2} e^{it(\Re z_n + ic_1)} D^\sigma W_{1}  e^{it\Delta_b} W_{2} D^{-\sigma} f dt. \label{eqdecay4}
	\eea
	 Evidently, \eqref{eqdecay3} is $o_n(1) \| f\|_{L^1}$ since $\sup_{t \in [a_1, a_2]} |e^{itz_n} - e^{it(\Re z_n + ic_1)}| = o_n(1)$. For \eqref{eqdecay4}, we first estimate the kernel 
	 \bee 
	  && \left| \int_{a_1}^{a_2} e^{it\Re z_n} e^{-tc_1} K_t(x, z) dt  \right| \\
	  &=& (i\Re z_n)^{-1}\left|  \left(e^{it\Re z_n} e^{-tc_1} K_t(x, z)\right)\bigg|^{a_2}_{t = a_1} - \int_{a_1}^{a_2} e^{it\Re z_n} \partial_t (e^{-tc_1} K_t(x, z)) dt  \right| \\
	  & \lesssim_{a_1, a_2, c_1} & (\Re z_n)^{-1} \sup_{t\in[a_1, a_2]} \left( \|K_t(x, z)\|_{L^\infty}  + \|\pa_t K_t(x, z)\|_{L^\infty} \right)  = o_{n}(1)
	 \eee
	 where the boundedness comes from \eqref{eqdecay23}.
	 These two vanishing integrals contradict with \eqref{eqdecay2} and hence close this lemma. 
\end{proof}

Using the inversion property of $R_b^\pm(z)$ Lemma \ref{leminv}, we have the following stronger qualitative decay for $f$ with better regularity. 

\begin{lemma}[Quantitative decay of extended resolvent families]\label{lemopdecay2}
    For $d \ge 1$, $p \in \left[2, 2^*\right)$ and $y > -|b|\left(\frac d2 - \frac dp\right)$, we have for $f \in \mathscr{S}(\RR^d)$
    \begin{align}
     \label{eqCauchy1} \left\| R_b^\pm(\l \pm iy) f \right\|_{L^{p}} &\lesssim  \la \l \ra^{-1} \left( \| f \|_{L^p \cap L^{p'}} +  \|\Delta_b f \|_{L^{p'}} \right) \\
    \| (R_b^+(\l + iy) - R_b^-(\l - iy) ) f \|_{L^p} &\lesssim \la \l \ra^{-2} \left( \| f \|_{L^p \cap L^{p'}} +  \|\Delta_b f \|_{L^{p'}} + \|\Delta_b^2 f \|_{L^{p'}} \right) \label{eqCauchy2}
    \end{align}
\end{lemma}

\begin{proof}
	Firstly, from the boundedness Lemma \ref{lemRbbdd}, we have the left-hand side of both cases bounded by $\| f \|_{L^{p'}}$ with a constant independent of $\l$. We will then focus on the decay of $\l$ with $\l \neq 0$.
	
	Since $f \in \mathscr{S}(\RR^d)$ and thus $\Delta_b f\in \mathscr{S}(\RR^d)$, the equation \eqref{eqinv2} indicates
	\be R_b^\pm(\l \pm iy) f = -\l^{-1} f + \l^{-1} R_b^\pm(\l \pm iy) \left(-\Delta_b \mp iy \right)f. \label{eqCauchy3}\ee
	It indicates
	\[ \| R_b^\pm (\l + iy) f\|_{L^{p}} \lesssim_y |\l|^{-1}\left( \|  f \|_{L^p \cap L^{p'}} + \| \Delta_b f \|_{L^{p'}} \right) \]
	which confirms \eqref{eqCauchy1}. This equation and $\Delta_b^2 f \in \mathscr{S}(\RR^d)$ also implies
	\bee
	&&\left( R_b^+ (\l + iy) - R_b^-(\l-iy)  \right)f \\
	&=& \l^{-2} \left[ R_b^+(\l + iy) \left( \Delta_b^2 +2iy\Delta_b -y^2 \right)f + R_b^-(\l - iy) \left( -\Delta_b^2 +2iy\Delta_b +y^2 \right)f +2iy f  \right]
	 \eee
	 Hence we get the desired decay in \eqref{eqCauchy2}
	 \[ \| \left( R_b^+ (\l + iy) - R_b^-(\l-iy)  \right)f \|_{L^{p}} \lesssim_y |\l|^{-2} \left( \| f \|_{L^p \cap L^{p'}} + \| \Delta_b f \|_{L^{p'}}+ \| \Delta_b^2 f \|_{L^{p'}} \right). \]
\end{proof}

\section{Spectral analysis of $\calH$}\label{s4}

In this section, we work on the matrix operator $\calH$ \eqref{Hb} with the following assumption as in Theorem \ref{thmStrichartz}:
\be \left \{ \begin{array}{l}
	d \ge 1,\quad  b > 0, \quad 0 < s_c < \min \left \{\frac d2, 1 \right \},\\
	W_1, W_2 {\rm \,\, satisfy \,\,\eqref{eqpotbd0} \,\,with}\,\, \a > 0.
 \end{array} \right .
\label{eqassump}
\ee
 The goal of this section is characterization of the spectrum and spectral projections of $\calH$ (Proposition \ref{prop3} and Proposition \ref{prop4}). As a corollary, we show the singularity in the extended resolvent of $\calH$ can be removed by applying $P_{\rm ess}$ (Corollary \ref{coroSbVbdd}).  \\

As preliminary, we define the matrix counterparts of $\Delta_b$, $e^{it\Delta_b}$ and $R_b^\pm(z)$ as
\bea
\calH_b &:=& \left( \begin{array}{cc} \Delta_b -1 &  \\ &  -\Delta_{-b} +1\end{array} \right)\\
e^{it\calH_b} &=& \left( \begin{array}{ll} e^{it(\Delta_b - 1)} &  \\ & e^{-it(\Delta_{-b} -1)}  \end{array} \right) \\
\calS_b^\pm (z) &:=& \pm i\int_0^{\infty}e^{\mp it\calH_b} e^{\pm itz}dt = \left( \begin{array}{ll} -R_b^\mp(-1-z) &  \\ & R_{-b}^\pm(-1+z)  \end{array} \right).\label{eqSb}
\eea
Then all the results in Section \ref{s2}-\ref{s3} can be easily transferred to the matrix case. In particular, we have the commutator relations from Lemma \ref{lemcom}
	\be
D^\sigma \calH_b D^{-\sigma} = \calH_b + ib\sigma,\quad  D^\sigma e^{it\calH_b} D^{-\sigma} = e^{it\calH_b} e^{-b\sigma t},\quad D^\sigma \calS_b^\pm(z) D^{-\sigma} = \calS_b^\pm (z - ib\sigma). \label{eqcomSb}
\ee
We also decompose the matrix potential $V$ as
\bea V = V_1 V_2,\quad V_1 = \left( \begin{array}{ll} 1 &  \\ &  -1\end{array} \right) \left( \begin{array}{ll} W_1 & W_2  \\ \overline{W}_2 & W_1   \end{array} \right)^{\frac 12},\quad V_2 = \left( \begin{array}{ll} W_1 & W_2  \\ \overline{W}_2 & W_1   \end{array} \right)^{\frac 12}.\label{eqV1V2}
\eea
Then the $V_1$ and $V_2$ will satisfy \eqref{eqpotbd0} for order $\frac{\a}{2} > 0$.

\subsection{Characterization of exceptional sets}

In this subsection, we first use analytic Freholm theory to discuss exceptional values of the operator $I + V_2 \calS_b^\pm(z+ibs_c) V_1$ and identify a spectral gap.\\

To begin with, we take
\be \left| \begin{array}{l}
     \delta_0 := \frac{1}{16} \min \left\{ \frac d2 -s_c , 1-s_c,  \a \right\} > 0,  \\
     p_0 := \left( \max \left\{ \frac{d-\a}{2d}, \frac{d-2}{2d}, \frac{\sigma}{d} \right\} + \frac{\delta_0 s_c}{d^2 (p+1)}\right)^{-1} \in \left ( 2, \frac{2d}{d-2}\right ), \\
     \tilde p_0 := \frac{2d}{d-6\delta_0} \in (2, p_0),
\end{array}\right.\label{eqdelta0} \ee
and analyze $I + V_2 \calS_b^\pm(z+ibs_c) V_1$ in $\dot{H}^\sigma$ for $\sigma \in [s_c, s_c + \delta_0]$ for $\pm \Im z \ge -\delta_0 b$. We stress that $\delta_0$ and therefore the range of $z$ is independent of $\sigma$, while $p_0$ depends on $\sigma$. 

\begin{lemma}
	\label{prop1} Assume \eqref{eqassump}. Let $\delta_0$ given as in \eqref{eqdelta0}. For any $\sigma \in [s_c, s_c+\delta_0]$,
	the two families of operators $I+V_2 \calS_b^{\pm}(z +ibs_c) V_1$ are uniformly bounded analytic Fredholm operators in $\L(\dot{H}^\sigma)$ for $\pm\Im z \ge -\delta_0 b$ respectively. Moreover, they are invertible when $|z|$ is large enough, and their inversions are uniformly bounded. 
\end{lemma}

\begin{proof}

	Using the commutator relation \eqref{eqcomSb}, the claim can be formulated for 
	\be T^\pm(z) := I + D^{\sigma} V_2 D^{-\sigma} \calS_b^{\pm} (z - ib(\sigma - s_c)) D^{\sigma} V_1 D^{-\sigma}\ee
	to be uniformly bounded, analytic Fredholm operator families on $L^2(\RR^d)$ with the same domain of $z$. For simplicity, we only prove the claim for $T^+(z)$. We denote $\delta := \sigma - s_c \in [0, \delta_0]$ and use the $p_0$ from \eqref{eqdelta0}.
	
	\emph{(1) Uniform boundedness.} From Lemma \ref{lemRbbdd} and the choice of $p_0$ \eqref{eqdelta0}, 
	$$ \sup_{\delta \in [0,\delta_0]} \sup_{\Im z \ge -\delta_0 b}\| \calS_b^{+}(z-ib\delta) \|_{L^{p_0'}\to L^{p_0}} \le \sup_{\Im z \ge -2\delta_0 b}\| \calS_b^{+}(z) \|_{L^{p_0'}\to L^{p_0}} < \infty. $$
	From Lemma \ref{lempotbdd}, $D^\sigma V_{i} D^{-\sigma}$ is bounded from $L^2$ to $L^{p_0'}$ and $L^{p_0}$ to $L^2$ both for $i = 1, 2$. These two bounds imply the uniform boundedness of $\| T^+(z) \|_{\L(L^2(\RR^d))}$ for $\Im z \ge -b\delta_0$. \\
	
	\emph{(2) Analyticity.} 
	Since $4\delta_0 < \min\{\a/2, 1\}$, \eqref{eqpotbdd2} implies that $D^\sigma V_1 D^{-\sigma}$ is bounded from $L^2$ to $L^2_{\la x \ra^{4\delta_0}}$. Noticing that $4 \delta_0 > 3\delta_0 = \frac d2 - \frac d{\tilde p_0} > 2\delta_0 > -\frac{\Im z}{b}$, Lemma \ref{lemanalyticity} indicates that $\pa_z \calS_b^+(z-ib\delta)$ is bounded in $L^2_{\la x \ra^{4\delta_0}}$ to $L^{\tilde p_0}$. Combined with the boundedness $D^{\sigma }V_2 D^{-\sigma}\in \L(L^{\tilde p_0} \to L^2)$, we get the analyticity of $T^+$.  \\

	\emph{(3) Fredholm.} With $2\delta_0 < \min \{ \a, 1, \frac d2 - \sigma \}$, this directly follows the compactness Lemma \ref{lempotcpt2} (ii).\\
	
	\emph{(4) Invertibility and boundedness for $|z| \gg 1$.} Since $2\delta_0 < \min \{ \a/2, 1\}$, Lemma \ref{lemopdecay} implies $\| T^+(z) - I \|_{\L(\dot{H}^\sigma)} \to 0$ as $|z| \to +\infty$. Therefore it is invertible when $|z|$ is large enough and the inverse is uniformly bounded. 
\end{proof}

Now we can define the exceptional sets and values.

\begin{definition}[Exceptional set]
	Assume \eqref{eqassump} and let $\delta_0$ as in \eqref{eqdelta0}. the exceptional sets of $\calH$ in $\dot{H}^\sigma$  are defined as 
	\be \calE^\pm_\sigma := \left\{ z \in \CC, \pm \Im z \ge -b\delta_0: I+V_2 \calS_b^\pm (z+ibs_c)V_1 \in \L(\dot{H}^\sigma) \,\,\text{is  not invertible}  \right\} \label{eqdefexceptional}\ee
	and their elements are called exceptional values. 
	\end{definition}

Lemma \ref{prop1} then implies a uniform characterization of these exceptional sets.

\begin{lemma}[Gap of exceptional set]\label{lemgap}
	Assume \eqref{eqassump} and let $\delta_0$ as in \eqref{eqdelta0}. Then $\bigcup_{\pm} \bigcup_{\sigma \in [s_c, s_c+\delta_0]} \calE^\pm_{\sigma}$ is a finite discrete set. Moreover, there exists $0 < \delta_1 \le \delta_0$ such that 
	\be \left( \bigcup_{\pm} \bigcup_{\sigma \in [s_c, s_c+\delta_0]} \calE^\pm_{\sigma} \right) \cap \{z\in \CC : 0 < \Im z < b\delta_1 \} = \emptyset. \label{eqdelta1} \ee
	\end{lemma}
\begin{proof}
	For every $\sigma \in [s_c, s_c+\delta_0]$ and each sign of $\pm$, Lemma \ref{prop1} and analytic Fredholm theory (see for example \cite[Theorem VI.14]{reed1972methods}) indicate that $\calE^{\pm}_\sigma$ is a discrete, bounded (hence finite) set in $\{z \in \CC: \pm \Im z \ge -b\delta_0 \}$. By interpolation, $I+V_2 \calS_b^\pm (z+ibs_c)V_1$ is invertible on $\dot{H}^{s_c}$ and $\dot{H}^{s_c+\delta_0}$ implies its invertibility on $\dot{H}^\sigma$ for all $\sigma \in [s_c, s_c+\delta_0]$, so
	\[ \bigcup_{\pm} \bigcup_{\sigma \in [s_c, s_c+\delta_0]} \calE^\pm_{\sigma} \subset \bigcup_{\pm} \left( \calE^\pm_{s_c} \cup \calE^\pm_{s_c + \delta_0}  \right) \]
	is still a finite discrete set, which guarantees the existence of $\delta_1$ and \eqref{eqdelta1}.
	\end{proof}

From now on, we will only consider $\sigma \in (s_c, s_c+\delta_1)$ so that the exceptional set is away from $\sigma\left((\calH_b+ibs_c)\big|_{\dot{H}^\sigma}\right) = \RR + ib(\sigma-s_c)$, which will be identified as $\sigma_{\rm ess}(\calH\big|_{\dot{H}^\sigma} )$.
 
\subsection{Characterization of the spectrum and spectral projections of $\calH$}
In this subsection, we will connect the operator $I + V_2 \calS_b^\pm(z+ibs_c)V_1$ to extended resolvents of $\calH = \calH_{b} -ibs_c + V$ on $\dot{H}^\sigma$ via (formal) resolvent identities \eqref{eqSbV}, and then prove the goal of this section Proposition \ref{prop3} and Proposition \ref{prop4}. This connection is made clear in the following lemma.

\begin{lemma}[Extended resolvents for $\calH$] \label{prop2} Assume \eqref{eqassump} and take $\delta_0, p_0, \tilde p_0$ from \eqref{eqdelta0} and $\delta_1$ from Lemma \ref{lemgap}.
	For $\sigma \in (s_c, s_c+\delta_1)$, we define 
	\be \calO_\sigma^{\pm} := \{ z\in \CC: \pm \Im z \ge -b\delta_0\} - \calE_\sigma^\pm \label{eqzrange} \ee
	and define the extended resolvent families of $\calH$ as
	\be \calS_{b,V}^\pm(z+ibs_c) := \calS_b^\pm (z+ibs_c) -  \calS_b^\pm (z+ibs_c) V_1 (I + V_2  \calS_b^\pm (z+ibs_c) V_1 )^{-1} V_2 \calS_b^\pm(z+ibs_c)   \label{eqSbV} \ee
	for $z \in \calO_{\sigma}^\pm$. Then they satisfy the following properties:
	\begin{enumerate}
		\item {\rm Well-definedness and boundedness:} $\calS^\pm_{b, V}(z+ibs_c) \in \L(\dot{W}^{\sigma, p_0'} \to \dot{W}^{\sigma, p_0})$. Moreover, we have the following uniform boundedness for every $\nu > 0$
		\bea \sup_{\substack{ z: \pm \Im z \ge -b\delta_0 \\ \mathrm{dist}(z, \calE_\sigma^\pm) > \nu} } 
		\| \calS_{b, V}^\pm (z+ibs_c)\|_{\dot{W}^{\sigma, p_0'} \to \dot{W}^{\sigma, p_0} } < \infty;&& \\
		\sup_{\substack{ z: \pm (\Im z -  b(\sigma-s_c))\ge 0 \\ \mathrm{dist}(z, \calE_\sigma^\pm) > \nu} } 
		\| \calS_{b, V}^\pm (z+ibs_c)\|_{\dot{W}^{\sigma, p'} \to \dot{W}^{\sigma, p} } < \infty, && \forall \,2 < p \le p_0. \label{eqSbVbdd}
		\eea
		\item {\rm Resolvent: }when $ \pm \left(\Im z - b(\sigma - s_c)\right) > 0$, $\calS_{b, V}^\pm(z+ibs_c) = \left(\calH - z\right)^{-1} \in \L(\dot{H}^\sigma)$ is the resolvent of $\calH$ on $\dot{H}^\sigma$.
		\item {\rm Relation with $\calS_b^\pm(z)$: }
		\begin{align}
			\calS^\pm_{b, V}(z+ibs_c) &= \calS^\pm_{b}(z+ibs_c)  - \calS^\pm_{b}(z+ibs_c)  V \calS^\pm_{b, V}(z+ibs_c)  \label{eqSb1} \\
			&= \calS^\pm_{b}(z+ibs_c)  - \calS^\pm_{b, V}(z+ibs_c)  V \calS^\pm_{b}(z+ibs_c) \label{eqSb2}
			\\
			I - V_2 \calS_{b, V}^{\pm}(z+ibs_c) V_1 &=(I + V_2 \calS_b^\pm (z+ibs_c) V_1)^{-1} \in \L(\dot{H}^\sigma) \label{eqSb3}
		\end{align}
		\item {\rm Resolvent identity:} for $z, z' \in \calO_\sigma^{\pm}$, we have 
		\be
		\begin{aligned}
			\calS^\pm_{b, V}(z+ibs_c) - \calS^\pm_{b, V}(z'+ibs_c) &= (z-z')\calS^\pm_{b, V}(z+ibs_c) \calS^\pm_{b, V}(z'+ibs_c)\\
			&=(z-z')\calS^\pm_{b, V}(z'+ibs_c) \calS^\pm_{b, V}(z+ibs_c)
		\end{aligned} \label{eqSbVres1}
		\ee
		in $\L(\dot{W}^{\sigma, p_0'} \to \dot{W}^{\sigma, p_0})$. Moreover, if $z \in \calO_\sigma^-$ and $z' \in \calO_\sigma^+$ also satisfy
		$$\Im z' -b(\sigma-s_c) > \Im z - b(\sigma - s_c),$$ 
		then we also have
		\be 
		\begin{aligned}  \calS^-_{b, V}(z+ibs_c) - \calS^+_{b, V}(z'+ibs_c) &= (z-z')\calS^-_{b, V}(z+ibs_c) \calS^+_{b, V}(z'+ibs_c) \\
			&= (z-z')\calS^+_{b, V}(z'+ibs_c) \calS^-_{b, V}(z+ibs_c)
		\end{aligned}  \label{eqSbVres2}
		\ee
		in $\L(\dot{W}^{\sigma, p_0'} \to \dot{W}^{\sigma, p_0})$.
		\item {\rm Differentiability:} $z\mapsto \calS_{b, V}^\pm(z+ibs_c)$ is differentiable on $\{ z\in \CC: \pm \Im z \ge -b\delta_0\} - \calE_\sigma^\pm$ as bounded operators from $\dot{H}^\sigma_{\la x \ra^{4\delta_0}}$ to $\dot{W}^{\sigma, \tilde p_0}$. Moreover, for any $\delta \in (0, \delta_0)$, we define $p_\delta > 2$ as
		\be \frac d2 - \frac{d}{p_\delta} = \frac \delta 2.\ee
		Then $z\mapsto \calS^\pm_{b,V}(z+ibs_c)$ is also differentiable on $\{ z \in \CC : \pm \Im z \ge b(\sigma -s_c) \} - \calE_\sigma^\pm$ as bounded operators from $\dot{H}_{\la x \ra^\delta}^\sigma$ to $\dot{W}^{\sigma, p_\delta}$.
	\end{enumerate}
\end{lemma}

We remark that the formula \eqref{eqSbV} can be derived formally through resolvent identities. Now we use it as a definition and then check the properties of resolvents via direct algebraic computation.

\begin{proof}\emph{(1) Well-definedness and boundedness.} The definition of exceptional set \eqref{eqdefexceptional} and Lemma \ref{prop1} imply that 
	$ (I + V_2  S_b^\pm (z+ibs_c) V_1 )^{-1} \in \L(\dot{H}^\sigma)$ 
	when $z\in \calO_\sigma^\pm$ and it is uniform bounded away from $\calE_\sigma^\pm$. Thus (1) easily follows the bounds of $\calS_{b}^\pm(z-ib(\sigma -s_c))$, $D^\sigma V_i D^{-\sigma}$ with $i=1,2$ as in the proof of (1) of Lemma \ref{prop1}. \\
	
	\emph{(2) Resolvent.} Similarly, when $\pm \left(\Im z - b(\sigma - s_c)\right) > 0$, Lemma \ref{lempotbdd} and Lemma \ref{lemRbbdd} imply boundedness on $\calS_b^\pm(z+ibs_c)$ and $\calS_{b, V}^\pm (z+ibs_c)$ in $\L(\dot{H}^\sigma)$.
	The inversion property follows $\calS_b^\pm(z+ibs_c) = (\calH_{b}- ibs_c - z)^{-1}$ and direct algebraic computation. For example, we check $(\calH-z) \calS_{b,V}^+(z+ibs_c) = I$ (where we denote $\tilde{z} := z + ibs_c$ for simplicity): 
	\bee 
	&&(\calH-z) \calS_{b,V}^+(\tilde{z})  \\
	&=& \left( \calH_b -\tilde{z} + V \right) \left( \calS_b^\pm (\tilde{z}) -  \calS_b^\pm (\tilde{z}) V_1 \left( I + V_2 \calS_b^\pm (\tilde{z}) V_1\right)^{-1} V_2 \calS_b^\pm(\tilde{z})  \right) \\
	&=& I + V\calS_b^\pm (\tilde{z})  - V_1 \left( I + V_2 \calS_b^\pm (\tilde{z}) V_1\right) \left( I + V_2 \calS_b^\pm (\tilde{z}) V_1\right)^{-1} V_2 \calS_b^\pm(\tilde{z}) = I
	\eee
	where every operator except $\calH_b$ makes sense in $\L(\dot{H}^\sigma)$. \\
	
	\emph{(3) Relation with $\calS_b^\pm(z)$.}  Those three identities \eqref{eqSb1}-\eqref{eqSb3} all follow direct computation, with boundedness in \eqref{eqSb3} indicated in (1). \\
		
	\emph{(4) Resolvent identity.} The resolvent identities for $\calS_{b, V}^\pm$ follow those of $\calS_b^\pm$ (Lemma \ref{lemresident}) and direct computation. \\
	
		\emph{(5) Analyticity.} Lemma \ref{prop1} and analytic Fredholm theory indicate that $(I +V_2 \calS_{b}^\pm(z+ibs_c)V_1 )^{-1}$ is a meromorphic $\L(\dot{H}^\sigma)$-valued function on $\{ z \in \CC: \pm \Im z \ge - b\delta_0 \}$ with finite rank poles $\calE^\pm_\sigma$. 
	As in the above proof of Lemma \ref{prop1} (2), Lemma \ref{lemanalyticity} and \eqref{eqpotbdd2} imply respectively the boundedness of $\pa_z \calS_{b}^\pm(z+ibs_c)$ in $\L(\dot{H}^\sigma_{\la x \ra^{4\delta_0}} \to \dot{W}^{\sigma, p_0})$ and boundedness $V_i \in \L(\dot{H}^\sigma \to \dot{H}^\sigma_{\la x \ra^{4\delta_0}})$, which implies the differentiablity and boundedness for $\pm \Im z \ge -b\delta_0$.
	
	When $\pm \Im z \ge b(\sigma -s_c)$, any triplet $(z, p_\delta, \delta)$ with $0 < \delta < \delta_0$ will satisfy the condition \eqref{eqanalyticitycond} of Lemma \ref{lemanalyticity}. Then the analyticity follows a similar argument.
\end{proof}

We are ready to prove the main proposition of this subsection which characterizes the spectrum and corresponding Riesz projection for $\calH$.  

\begin{proposition}[Characterization of $\sigma(\calH)$] \label{prop3}
	Assume \eqref{eqassump}. Take $p_0$ from \eqref{eqdelta0} and $\delta_1$ from Lemma \ref{lemgap}. Let $\sigma \in (s_c, s_c+\delta_1)$ and view $\calH$ as a closed operator on $\dot{H}^\sigma$. The following statements hold:
	\begin{enumerate}
		\item {\rm Characterization of spectrum:}
		\bea \sigma_{\rm ess}(\calH) = \{ \RR + ib(\sigma - s_c) \},\quad \sigma_{disc}(\calH)  =  \tilde{\calE}_\sigma. \label{eqHspec}
		\eea
		where $\tilde{\calE}_\sigma$ is the truncated exceptional set
		\be \tilde{\calE}^\pm_\sigma := \calE^\pm_\sigma \cap \{ z\in \CC: \pm( \Im z - b(\sigma - s_c) ) > 0 \},\qquad \tilde{\calE}_\sigma := \cup_{\pm}\tilde{\calE}^\pm_\sigma.  \label{eqtruncexceptional} \ee
		\item {\rm Regularity of Riesz projections} There exists $2 < p_1 \le p_0$ and $\gamma_0 > 0$ such that for any $\xi \in \tilde{\calE}_\sigma$,
		the range of\footnote{For convenience, we denote 
        \[ \calS_{b, V}(z+ibs_c) = \calS_{b, V}^\pm (z+ibs_c),\quad {\rm when}\,\, \pm(\Im z - b(\sigma - s_c)) > 0, \] namely making it the resolvent of $\calH$ on $\dot H^\sigma$ according to Lemma \ref{prop2} (2). So is $\calS_b(z+ibs_c)$.
        }
		\be  P_\xi := -\frac{1}{2\pi i} \oint_{|z-\xi|=\e} S_{b, V}(z+ibs_c)  dz \label{eqdefproj1} \ee
		and its $\dot{H}^\sigma$ adjoint $P_\xi^*$ are spanned by functions belonging to $\dot{H}^\sigma_{\la x \ra^{\gamma_0}} \cap \dot{W}^{\sigma, p_1'} \cap \dot{W}^{\sigma, p_0}$.
		Thus, each such projection is bounded from $\dot{W}^{\sigma, p_0'} + \dot{W}^{\sigma, p_1}$ to $\dot{H}^\sigma_{\la x \ra^{\gamma_0}} \cap \dot{W}^{\sigma, p_1'} \cap \dot{W}^{\sigma, p_0}$. 
		\item {\rm Spectral projections:} Define spectral projections of $\calH$ onto $\sigma_{\rm disc}$ and $\sigma_{\rm ess}$ as
		\be P_{\rm disc} := \sum_{\xi \in \tilde{\calE}_\sigma} P_\xi,\qquad P_{\rm ess} = I - P_{\rm disc}. \label{eqdefproj2}\ee
		Then $P_{\rm disc} \in \L (\dot{W}^{\sigma, p_0'} + \dot{W}^{\sigma, p_1} \to \dot{H}^\sigma_{\la x \ra^{\gamma_0}} \cap \dot{W}^{\sigma, p_1'} \cap \dot{W}^{\sigma, p_0})$ and $P_{\rm ess} \in \cap_{p \in [p_1', p_1]}\L(\dot{W}^{\sigma, p})$. Both operators are commutable with $\calH$. 
	\end{enumerate}
\end{proposition}

\begin{proof}
	(1) From Lemma \ref{prop2} (2), we know $\sigma(\calH) \subset \{\RR + ib(\sigma -s_c)\} \cup \tilde{\calE}_\sigma$. Now we show the inverse inclusion and the classification.
	
	Suppose $\xi \notin \sigma(\calH)$ and $\Im \xi \neq b(\sigma - s_c)$. Then by resolvent identities, we can recover \eqref{eqSbV} and \eqref{eqSb3}, implying that $\xi \notin \tilde{\calE}_\sigma$. This indicates $\tilde{\calE}_\sigma \subset \sigma(\calH)$. 
	
	Notice that the operator $\calH_b - ib\sigma $ is self-adjoint on $\dot{H}^\sigma$ with $\sigma_{\rm ess}(\calH_b - ib\sigma ) = \RR$ from Lemma \ref{lemctsspec} and the commutator relation \eqref{eqcomDeltab}. Since Lemma \ref{lempotcpt2} (i) indicates compactness of $V_2 \calS_b^\pm(z-ib\sigma)$ on $\dot{H}^\sigma$ when $\Im z \neq 0$, the formula \eqref{eqSbV} and Weyl's criterion (see for example \cite[Theorem XIII.14]{reed1972methods4}) indicate that $\sigma_{\rm ess}(\calH_b - ib\sigma) = \sigma_{\rm ess}(\calH - ib(\sigma-s_c)) = \RR$. That is  $\sigma_{\rm ess}(\calH) = \RR+ib(\sigma - s_c)$. Thus \eqref{eqHspec} is confirmed.
	\\
	
	(2) Thanks to the finite-rank nature of $P_\xi$ for $\xi \in \tilde{\calE}_\sigma$, its boundedness is equivalent to check the regularity of $\mathrm{Ran}(P_\xi)$ and $\mathrm{Ran}(P_\xi^*)$. 
	
	Since $\mathrm{dist}( \tilde{\calE}_\sigma, \RR+ib(\sigma-s_c))>0$, the boundedness of $\calS_{b, V}(z + ibs_c)$ from \eqref{eqSbVbdd} indicates $P_\xi \in \L(\dot{W}^{\sigma, p'} \to \dot{W}^{\sigma, p})$ and $\mathrm{Ran}(P_\xi)\cup \mathrm{Ran}(P_\xi^*) \subset \dot{W}^{\sigma, p}$ for any $2 \le p \le p_0$. 
    Noticing the embedding $L^2_{\la x \ra^{\gamma}} \hookrightarrow L^{2-\e}$ for some $\e=\e(\gamma) >0$, it suffices to prove $\mathrm{Ran}(P_\xi) \cup \mathrm{Ran}(P_\xi^*) \subset \dot{H}^\sigma_{\la x \ra^\gamma}$ for some $\gamma > 0$.
	
	We first consider $\mathrm{Ran} (P_\xi)$. The range of $P_\xi$ for $\xi \in \sigma_{dist}(\calH)$ are composed of generalized eigenfunctions of finite algebraic multiplicity. Take 
	\[ \beta = \frac 12 \min\left\{ \a, 1, \frac{d}{2} - \sigma,   \mathrm{dist}( \tilde{\calE}_\sigma, \RR+ib(\sigma-s_c)) \right\}, \]
	and 
	\[ \beta_n := 4^{-n}\beta. \]
	We claim that if $f \in \dot{H}^\sigma$ and $(\calH - \xi)^{n} f = 0$, then $D^\sigma f \in L^2_{\la x \ra^{\beta_n}}$. 
	Indeed, if $(\calH - \xi) f = 0$, then acting on $\calS_{b}(\xi+ibs_c)$, we get
	\[ f = -\calS_b(\xi+ibs_c) (Vf).  \]
	Due to our choice of $\beta$, \eqref{eqpotbdd2} indicates that $D^\sigma (Vf) \in L^2_{\la x \ra^{\beta}}$ and Lemma \ref{lemimpreg} implies that $D^\sigma f = -\calS_b(\xi-ib(\sigma - s_c)) D^\sigma (Vf) \in L^2_{\la x \ra^{\beta_1}} $. Inductively, if $(\calH - \xi)f = g$ with $D^\sigma g \in L^2_{\la x \ra^{\beta_{n-1}}}$, then $D^\sigma f =-D^\sigma \calS_b(\xi+ibs_c) (Vf - g) \in L^2_{\la x \ra^{\beta_n}}$. Since there are only finite exceptional values, each of which has finite multiplicity, we conclude that the existence of $\gamma_0 > 0$ such that all generalized eigenfunctions belong to $\dot{H}^\sigma_{\la x \ra^{\gamma_0}}$.
	
		For $P_\xi^*$, 
	it is the Riesz projection onto $\bar{\xi} \in \sigma_{disc}(\calH^*)$
	with $\calH^* = \calH_b - ib(2\sigma - s_c) + D^{-2\sigma} \bar{V}^T D^{2\sigma}$. So for  $f \in \dot{H}^\sigma$ satisfying $(\calH^* - \bar{\xi}) f = g$, we have
	\[ D^\sigma f = \calS_b(\bar{\xi}+ib(\sigma - s_c) ) \left[ D^\sigma g - (D^{-\sigma} \bar{V}^T D^\sigma) D^\sigma f \right]. \]
	\eqref{eqpotbdd3}, Lemma \ref{lemimpreg} and a similar inductive argument indicate the desired regularity for $\mathrm{Ran} \,P_{\rm disc}^*$ and end the proof. 
	
	(3) The boundedness follows (2) and the commutability follows the resolvent property from Lemma \ref{prop2}(2). 
\end{proof}

Next, we will derive a weak representation formula for $P_{\rm ess}$ via extended resolvents $\calS_{b, V}^\pm(z+ibs_c)$ on the essential spectrum.

\begin{proposition}[Representation of $P_{\rm ess}$]\label{prop4} Assume \eqref{eqassump} and take $p_0$ from \eqref{eqdelta0} and $\delta_1$ from Lemma \ref{lemgap}.
	Let $\sigma \in (s_c, s_c+\delta_1)$ and $f, g \in \mathscr{S}'(\RR^d)$ such that 
	\be f, g \in \dot{W}^{\sigma, p} \cap \dot{W}^{\sigma, p'},\quad \calH_b D^\sigma f, \calH_b D^\sigma g, \calH_b^2 D^\sigma f \in L^{p'} \label{eqPccond}
	\ee for some $p \in (2, p_0]$. We have
	\be \la P_{\rm ess} f, g\ra_{\dot{H}^\sigma} = \frac{1}{2\pi i} \int_{\RR} \left\la \left(\calS_{b, V}^+(\l +ib\sigma) -  \calS_{b, V}^-(\l +ib\sigma)\right) f, g  \right\ra_{\dot{H}^\sigma} d\l \label{eqrepPc}\ee
	where the integral converges absolutely.
\end{proposition}

\begin{proof}
	Firstly, we define and estimate the semigroup $e^{it(\calH_b -ib\sigma + V)}$ on $\dot{H}^\sigma$. It is equivalent to considering the evolution
	\[ i\partial_t u + \calH_b u - ib\sigma u +Vu = 0 \]
	in $\dot{H}^\sigma$ and bound $u(t)$ with $u(0)$. 
	Let $v = D^{\sigma} u$ and act $D^\sigma$ on the equation. With the commutator relation \eqref{eqcomSb}, it becomes
	\be i\pa_t v + \calH_b v = -D^\sigma V D^{-\sigma} v \label{eqlinearH}\ee
	on $L^2$. 
	Applying Duhamel's formula, a straightforward estimate indicates 
	\[ \| v(t)\|_{L^2} \le \| v(0)\|_{L^2} + \| D^{\sigma} V D^{-\sigma} \|_{\L(L^2)}  \int_0^t \| v(s)\|_{L^2} ds. \]
	This estimate easily enables a contraction argument for local well-posedness of \eqref{eqlinearH} with the lifespan only depending on $\| D^{\sigma} V D^{-\sigma} \|_{\L(L^2)}$. Thus the global well-posedness follows, ensuring $e^{it(\calH_b -ib\sigma + V)}$ a well-defined semigroup on $\dot{H}^\sigma$. By Gronwall's inequality, the bound implies 
	$\| v(t) \|_{L^2} \le e^{t \| D^{\sigma} V D^{-\sigma} \|_{\L(L^2)}} \| v_0\|_{L^2}$, namely 
	\be  \|e^{it(\calH_b -ib\sigma + V)}\|_{\L(\dot{H}^\sigma)} \le  e^{t \| D^{\sigma} V D^{-\sigma} \|_{\L(L^2)}}, \quad t \in \RR \ee 
	where $\| D^{\sigma} V D^{-\sigma} \|_{\L(L^2)} < \infty$ by Lemma \ref{lempotbdd}. 
	
	Next, we prove that for $y \gg 1$ and $f, g \in \mathscr{S}'(\RR^d)$ satisfying \eqref{eqPccond}, 
	\be \la f, g  \ra_{\dot{H}^\sigma} =\frac{-i}{2\pi} \int_{\RR} \left\la \left(\calS_{b, V}^+(\l +ib\sigma + iy) -  \calS_{b, V}^-(\l +ib\sigma-iy) \right) f, g  \right\ra_{\dot{H}^\sigma} d\l. \label{eqrep1}\ee
	We first choose $y$ large enough such that $\| \calS_b^\pm(z+ib\sigma \pm iy) \|_{\L(\dot{H}^\sigma)} \ll 1$ by Lemma \ref{lemopdecay}, leading to $(I + V_2 \calS_b^\pm(\l +ib\sigma \pm iy)V_1)^{-1}$ and hence $\calS_{b, V}^\pm(\l +ib\sigma \pm iy)$ uniformly bounded in $\L(\dot{H}^\sigma)$ for $\l \in \RR$.
	The formula \eqref{eqrep1} therefore makes sense. Increase $y$ if necessary to make $y > 1 + \|D^\sigma V D^{-\sigma} \|_{\L(L^2)}$. So for any $f, g \in \mathscr{S}'(\RR^d)$ satisfying \eqref{eqPccond}, we consider the function of $t$
	$$\la e^{it(\calH_b-ib\sigma+V)}e^{-y|t|} f, g \ra_{\dot{H}^\sigma} =  \la e^{it(\calH-ib(\sigma-s_c)}e^{-y|t|} f, g \ra_{\dot{H}^\sigma}$$
	which belongs to $L^1(\RR)$ due to its exponential decay at $\pm \infty$. Using the integral representation of resolvents (like \eqref{eqSb}), the Fourier transform of this function is
	\be
	\int_\RR e^{-it\l} \la e^{it(\calH-ib(\sigma-s_c)}e^{-y|t|} f, g  \ra_{\dot{H}^\sigma} dt = -i \left\la \left(\calS_{b, V}^+(\l +ib\sigma + iy) -  \calS_{b, V}^-(\l +ib\sigma-iy) \right) f, g  \right\ra_{\dot{H}^\sigma} \label{eqpropr31}
	\ee
	We claim that this is again in $L^1(\RR)$. Indeed, we first expand $\calS_{b, V}^\pm(z)$ with \eqref{eqSbV} to obtain
	\bee
	&&\mathrm{RHS\,\, of \,\,\eqref{eqpropr31}} =
	-i \left\la \left(\calS_{b}^+(\l + iy) -  \calS_{b}^-(\l -iy) \right) D^{\sigma}f, D^\sigma g  \right\ra_{L^2} \\
	&+& i \sum_{\pm} \pm \left\la \calS_{b}^\pm(\l \pm iy) D^{\sigma} V_1 (I + V_2 \calS_{b}^\pm(\l+ib\sigma \pm iy ) V_1)^{-1} V_2 D^{-\sigma} \calS_{b}^\pm(\l \pm iy)  D^{\sigma}f, D^\sigma g  \right\ra_{L^2}.
	\eee
	Now we apply Lemma \ref{lemopdecay2} with $p = p_0$ to gain decay in $\l$. Apply \eqref{eqCauchy2} to the first term and \eqref{eqCauchy1} to the last two terms after transposing the first $\calS_b^\pm(\l \pm iy)$ onto $D^\sigma g$. Together with the uniform boundedness of $(I + V_2 \calS_b^\pm(\l +ib\sigma \pm iy)V_1)^{-1}$ in $\L(\dot{H}^\sigma)$ and $D^\sigma V_i D^{-\sigma} \in \L(L^{p_0} + L^2 \to L^2) \cap \L(L^2\to L^{p_0'} \cap L^2)$, for any $ p \in (2, p_0]$, we have
	\be
	\begin{aligned} 
		\left|\mathrm{RHS\,\, of \,\,\eqref{eqpropr31}}\right| \lesssim \la \l \ra^{-2}& \left( \| D^\sigma f \|_{L^{p} \cap L^{p'}} + \| \calH_b D^\sigma f \|_{L^{p'}} +  \| \calH_b^2 D^\sigma f \|_{L^{p'}} \right) \\
		& \cdot \left( \| D^\sigma g \|_{L^{p} \cap L^{p'}} + \| \calH_b D^\sigma g \|_{L^{p'}} \right) \in L^1(\RR)
	\end{aligned}\label{eqdecayl} 
	\ee
	So \eqref{eqpropr31} implies \eqref{eqrep1} through the Fourier inversion formula. Note that this estimate also holds for $y \ge 0$ whenever $\{\RR + i(b\sigma -bs_c \pm y) \}$ does not intersect with $\tilde{\calE}^\pm_\sigma$.
	
	Finally, we will shift the integration contour in \eqref{eqrep1} toward the essential spectrum $\{\RR + ib(\sigma -s_c)\}$. Lemma \ref{prop2} (5) guaranteed the analyticity of $\left\la\calS_{b, V}^\pm(z +ibs_c) f, g  \right\ra_{\dot{H}^\sigma}$ on the corresponding half-plane except $\tilde{\calE}^\pm_\sigma$, so we take contours consisting of boundaries of $[-R, R] \times [0, y]$, $[-R, R]\times [-y, 0]$ and small circular contours around all exceptional values. Lemma \ref{lemopdecay} and \eqref{eqdecayl} with $y = 0$ guarantees the residual to vanish when $R$ goes to $+\infty$, so we finally arrived at
	\bee \la f, g  \ra_{\dot{H}^\sigma} =\frac{1}{2\pi i} \int_{\RR} \left\la \left(\calS_{b, V}^+(\l +ib\sigma) -  \calS_{b, V}^-(\l +ib\sigma) \right) f, g  \right\ra_{\dot{H}^\sigma} d\l + \sum_{\xi \in \tilde{\calE}} \la P_\xi f, g\ra_{\dot{H}^\sigma} \label{eqrep2} \eee
	which infers \eqref{eqrepPc}. 
\end{proof}

At the end of this step, we use this representation formula to show $P_{\rm ess}$ can remove the singularity of the resolvent $\calS_{b, V}^-(z+ibs_c)$.

\begin{corollary} \label{coroSbVbdd} Assume \ref{eqassump} and take $\delta_1$ from Lemma \ref{lemgap} and $p_1$ from Proposition \ref{prop3}. For $\sigma \in (s_c, s_c+\delta_1)$, we have for any $p \in (2, p_1]$, $ \calS_{b, V}^-(z+ibs_c) P_{\rm ess}$ is a well-defined bounded operator from $\dot{W}^{\sigma, p'}$ to $\dot{W}^{\sigma, p}$ for $\Im z \le b(\sigma - s_c)$, and has the uniform bound 
	\be \sup_{z: \Im z \le b(\sigma-s_c)} \left\|  \calS_{b, V}^{-}(z+ibs_c) P_{\rm ess} \right\|_{\dot{W}^{\sigma, p'} \to \dot{W}^{\sigma, p}} < \infty. \ee
\end{corollary}

\begin{proof}
	Fix a $p \in (2, p_1]$. We will use the representation formula \eqref{eqrepPc}, boundedness of $P_{\rm ess} \in \L(\dot{W}^{\sigma, p'})$ and resolvent identities \eqref{eqSbVres1} \eqref{eqSbVres2}. Recalling the boundedness \eqref{eqSbVbdd}, we only need to estimate for $z$ close to $\tilde{\calE}^-_\sigma$. In particular, we can also assume $\Im z - b(\sigma -s_c) \le -\delta_2$ for some $\delta_2 > 0$. 
	
	Firstly, we check the validity of the representation formula
	\be
	\begin{aligned} 
		\left\la \calS_{b, V}^{-}(z + ibs_c) P_{\rm ess} f, g \right\ra_{\dot{H}^\sigma} 
		=
		\frac{1}{2\pi i} \int_{\RR} \left\la  \calS_{b, V}^{-}(z + ibs_c)  \left(\calS_{b, V}^+(\l +ib\sigma) -  \calS_{b, V}^-(\l +ib\sigma) \right)f, g  \right\ra_{\dot{H}^\sigma} 
	\end{aligned} \label{eqSbVrep}\ee
	for $\Im z \le b(\sigma-s_c) - \delta_2$, $z \notin \tilde{\calE}^-_\sigma$ and $f, g \in \mathscr{S}$. Since\footnote{Here the adjoint operator are in the sense of $L^2$ rather than $\dot{H}^\sigma$.}
	\bee 
	&&\left\la \calS_{b, V}^{-}(z + ibs_c) P_{\rm ess} f, g \right\ra_{\dot{H}^\sigma} = \left\la D^\sigma \calS_{b, V}^{-}(z + ibs_c) D^{-\sigma} D^\sigma P_{\rm ess} f, D^\sigma g  \right\ra_{L^2} \\
	&=&  \left\la  D^\sigma P_{\rm ess} f, (D^\sigma \calS_{b, V}^{-}(z + ibs_c) D^{-\sigma})^* D^\sigma g  \right\ra_{L^2} \\
	&=& \left\la  P_{\rm ess} f, D^{-\sigma} (D^\sigma \calS_{b, V}^{-}(z + ibs_c) D^{-\sigma})^* D^\sigma g  \right\ra_{\dot{H}^\sigma}
	\eee
	It suffices to prove $f$ and $G :=  D^{-\sigma} (D^\sigma \calS_{b, V}^{-}(z + ibs_c) D^{-\sigma})^* D^\sigma g$ satisfying \eqref{eqPccond}. 
	From the resolvent property Lemma \ref{prop2} (2), we have 
	\bee &&   (D^\sigma \calS_{b, V}^{-}(z + ibs_c) D^{-\sigma})^* 
	=  \left[(\calH_b + ib(\sigma - s_c) -z + D^\sigma V D^{-\sigma})^{-1}\right]^* \\
	&=& (\calH_b - ib(\sigma - s_c) -\bar{z} + D^{-\sigma} \bar{V}^T D^{\sigma})^{-1} =: (\calH_b - \zeta + A)^{-1} \in \L (L^2)
	\eee
	where $ A := D^{-\sigma} \bar{V}^T D^\sigma$ and $\zeta := ib(\sigma - s_c) + \bar{z}$ which satisfies $\Im \zeta \le - \delta_2$.
	From \eqref{eqpotbdd3}, we know for $0 \le \gamma \ll 1$, $A\in  \L(L^2 \to L^2_{\la x \ra^\gamma})$. Now from resolvent identities
	\[ (\calH_b - \zeta + A )^{-1} - \calS_{b}(\zeta) = -(\calH_b - \zeta + A )^{-1} A \calS_b(\zeta) = -\calS_b (\zeta) A (\calH_b - \zeta + A )^{-1}, \]
	we have 
	\[ (\calH_b - \zeta + A )^{-1} = \calS_{b}(\zeta) - \calS_{b}(\zeta) A \calS_{b}(\zeta) + \calS_{b}(\zeta) A (\calH_b - \zeta + A )^{-1} A \calS_{b}(\zeta). \]
	Using Lemma \ref{lemimpreg}, the boundedness of $A$ above and 
	$g \in \mathscr{S}(\RR^d)$, we can find $0<\tilde{\gamma} \ll 1$, $0 < \tilde{p} -2 \ll 1$ such that 
	$D^\sigma G =  (\calH_b - \zeta + A )^{-1} D^\sigma g \in L^2_{\la x \ra^{\tilde{\gamma}}} \cap L^{\tilde{p}}$. 
	Also noticing that $\calH_b (\calH_b - \zeta + A )^{-1} =  I + (\zeta - A) (\calH_b - \zeta + A )^{-1}$, we know $G$ satisfies $G \in \dot{W}^{\sigma, \tilde{p}} \cap \dot{W}^{\sigma,\tilde{p}'}$ and $\calH_b D^\sigma G \in L^{\tilde{p}'}$. With $f \in \mathscr{S}$, this confirms $f$ and $G$ satisfying \eqref{eqPccond} and therefore the validity of \eqref{eqSbVrep}. 
	
	Now for $z_0$ with $\Im z_0 \le b(\sigma-s_c) - \delta_2$, $z_0 \notin \tilde{\calE}^-_\sigma$, we apply \eqref{eqSbVrep} and the resolvent property of $\calS_{b, V}^-(z_0 + ibs_c)$ to get
	\bee 
	&&\left\la \calS_{b, V}^{-}(z_0 + ibs_c) P_{\rm ess} f, g \right\ra_{\dot{H}^\sigma} \\
	&=&
	\frac{1}{2\pi i} \int_{\RR} \left\la \calS_{b, V}^{-}(z_0 + ibs_c) \left(\calS_{b, V}^+(\l +ib\sigma) -  \calS_{b, V}^-(\l +ib\sigma) \right) f, g  \right\ra_{\dot{H}^\sigma} \\
	&=& \frac{1}{2\pi i} \int_{\RR} \left\la \frac{1}{z_0 - (\l + ib(\sigma - s_c))} \left(\calS_{b, V}^+(\l +ib\sigma) -  \calS_{b, V}^-(\l +ib\sigma) \right) f, g  \right\ra_{\dot{H}^\sigma}
	\eee
	where the integral converges absolutely due to Proposition \ref{prop4}. Now 
	fix an arbitrary $z_1 \in \{ \Im z \le b(\sigma - s_c)-\delta_2 \} - \tilde{\calE}^-_\sigma$ so that $\calS_{b, V}^{-}(z_0 + ibs_c) P_{\rm ess}$ is bounded from $\dot{W}^{\sigma, p'}$ to $\dot{W}^{\sigma, p}$. For any other $z_2 \in \{ \Im z \le b(\sigma - s_c) - \delta_2 \}$, we have
	\bee 
	&&\left\la \calS_{b, V}^{-}(z_1 + ibs_c) P_{\rm ess} f, g \right\ra_{\dot{H}^\sigma} - \left\la \calS_{b, V}^{-}(z_2 + ibs_c) P_{\rm ess} f, g \right\ra_{\dot{H}^\sigma} \\
	&=& \frac{1}{2\pi i} \int_{\RR} \bigg\la \frac{z_2 - z_1}{[z_1 - (\l + ib(\sigma - s_c))][z_2 - (\l + ib(\sigma - s_c))]} \\
	&& \quad\quad\quad \left(\calS_{b, V}^+(\l +ib\sigma) -  \calS_{b, V}^-(\l +ib\sigma) \right) f, g  \bigg\ra_{\dot{H}^\sigma}.
	\eee
	The $\la \l \ra^{-2}$ decay and uniform boundedness of $\calS_{b, V}^{\pm}(\l + ib\sigma) \in \L(\dot{W}^{\sigma,p'} \to \dot{W}^{\sigma,p})$ indicate the boundedness of $\calS_{b, V}^-(z_2 + ibs_c)P_{\rm ess}$, particularly when $z_2$ is close to $\tilde{\calE}^-_\sigma$. This enables us to extend the definition $\calS_{b, V}^-(z+ibs_c)P_{\rm ess}$ to $z \in \tilde{\calE_{\sigma}}$ and implies the uniform boundedness.
\end{proof}

\section{Abstract Wiener's Theorem and Proof of Strichartz estimates}
\label{s5}

In this section, we will prove Theorem \ref{thmStrichartz}. We first recall the abstract Wiener's theorem of Beceanu from \cite{beceanu2011new} (also \cite{beceanu2009critical}), and then present the proof of our Strichartz estimates by verifying conditions for a modified Schr\"odinger system. 

\subsection{Abstract Wiener's Theorem}

In this subsection, we introduce the setting and abstract tool in Beceanu's work \cite{beceanu2011new,beceanu2009critical}, and then present a version for our later use. \\

Let $H$ be a complex Hilbert space, $\L(H)$ be the space of bounded linear operators from $H$ to itself, and $L^1_t \L (H)$ be the Banach space of operator-valued $L^1$ functions on $\RR$. We add a convolutional unit 
$$I_K := \delta_{t=0}I_{H}$$
to define the normed vector space $K$ 
\bea K &:=& L^1_t \L (H) \oplus \CC \{I_K\}; \label{eqdefK} \\
\| \mu + aI_K \|_{K} &:=& \| \mu \|_{L^1_t \L(H)} + |a|, \qquad\forall a \in \CC, \,\,\mu \in L^1_t \L(H).
\eea
Then it satisfies the following properties:
\begin{lemma}[Properties of $K$]
	Let $K$ defined as \eqref{eqdefK}. Then the following statements are true.
	\begin{enumerate}
		\item  $K$ is an unital Banach algebra.
		\item Multiplication by bounded continuous functions is bounded on $K$:
		\be \| fk \|_{K} \le \| f \|_{C^0(\RR)} \|k\|_{K}.  \ee
		\item If $k = \mu(t) + a I_K \in K$, then $k^* = \mu(t)^* + \bar{a} I_K \in K$ as well.
		\item The  Fourier transform is well-defined by 
		\be \hat{k}(\l) := \int_\RR e^{-it\l} k(t) dt, \ee
		which satisfies $ \| \hat{k} \|_{C^0_\l \L(H)} \le \|k  \|_K$.
	\end{enumerate}
\end{lemma}
\begin{proof}
	The algebra structure is given by $L^1_t$-convolution
	\[ (\mu + a I_K) * (\nu + bI_K) = \mu * \nu + a\nu + b \mu + ab I_K. \]
	Then all these properties are trivial. 
\end{proof}

We stress that the invertibility of $k \in K$ means there exists $k^{-1} \in K$ such that $k * k^{-1} = k^{-1} * k = I_K$.
 \\

Now we are in place to state the abstract Wiener's theorem, a sufficient condition for $k \in K$ to be invertible characterized by invertibility of $\hat{k}(\l) \in \L(H)$ for all $\l$. We also collect the criterion for the upper triangular property (meaning that $A \in K$ is supported on $[0, \infty)$) to be preserved under inversion. 

\begin{theorem}[{\cite[Theorem 3.3]{beceanu2009critical}}, {\cite[Theorem 2.3]{beceanu2011new}}, ] \label{thmWiener}
	Let $K$ be the operator algebra of \eqref{eqdefK}. If $A \in K$ is invertible, then $\hat{A}(\l) \in \L(H)$ is invertible for every $\l \in \RR$. Conversely, assume $\hat{A}(\l)$ is invertible for each $\l \in \RR$, $A - I_K = L \in L^1_t(\L(H))$, and
	\be  \lim_{\e\to 0}\| L(\cdot + \e) - L\|_K = 0,\quad \lim_{R\to \infty}\|1 - \psi(t/R)L(t) \|_K = 0 \label{eqWienercond} \ee
	where $\psi \in C^\infty_c([-2, 2])$ satisfies $\psi(t) = 1$ for $t \in [-1, 1]$. Then $A$ is invertible and $A^{-1} - I_K \in L^1_t(\L (H))$.
\end{theorem}

\begin{lemma}[{\cite[Lemma 3.4]{beceanu2009critical}}, {\cite[Lemma 2.5]{beceanu2011new}}] \label{lemWiener1}
	Given $A \in K$ upper triangular with $A^{-1} \in K$, $A^{-1}$ is upper triangular if and only if $\hat{A}$ can be extended to a weakly analytic family of invertible operators in the lower half-plane (namely it is continuous up to the boundary, uniformly bounded, and has uniformly bounded inverse). 
\end{lemma}

The following lemma was also implicitly used in Beceanu's work \cite{beceanu2011new}. It bridges such invertibility with our Strichartz estimates. We present a proof for completeness. 

\begin{lemma}\label{lemWiener2}
	If $A, A^{-1} \in K$ and both are upper triangular, then $\int_0^t A(t-s) ds \in \L(L^2_t([0, \infty), H))$ has an inversion $\int_0^t A^{-1}(t-s)ds \in \L(L^2_t([0, \infty), H))$.
\end{lemma}

\begin{proof}
	We first show the boundedness. Since $A = I_K + L$ with $L \in L^1_t \L(H)$, for $u \in L^2_t([0, \infty), H)$, 
	\bee 
	&&\left\| \int_0^t A(t-s) u(s) ds \right\|_{L^2_t([0, \infty), H)}  \\
	&\le& \left\| \int_0^t \|L(t-s)\|_{\L(H)} \|u(s)\|_H ds \right\|_{L^2_t(\RR)}  + \| u\|_{L^2_t([0, \infty), H )} 
	 = \| A \|_K \| u \|_{L^2_t ([0, \infty), H ) }
	\eee
	from Young's inequality. And similar for $\int_0^t A^{-1}(t-s)ds$. That $\int_0^t A^{-1}(t-s)ds$ is the inverse of $\int_0^t A(t-s)ds$ can be checked using $A^{-1} *A= A*A^{-1} = I_K$.
\end{proof}

To conclude this subsection, we present the following version of abstract Wiener's theorem for operators with exponential decay in $t$. 

\begin{corollary}[Abstract Wiener's theorem with exponential decay]\label{coroWiener}
	Let $K$ be the Banach algebra of \eqref{eqdefK} and $\epsilon \in \RR$. If $A = I_K + L \in K$ is upper triangular with $L \in L^1_{t,loc}\L(H)$ and satisfies
	\begin{enumerate}
		\item $\| L \|_{L^1_t e^{\epsilon t}\L(H)} = \int_0^\infty \| L(t) \|_{\L(H)}e^{\epsilon t} dt < \infty$.
		\item $ \lim_{\epsilon\to 0}\| L(\cdot + \epsilon) - L\|_{L^1_t e^{\epsilon t} \L(H)} = 0.$
		\item The Fourier transform $\hat{A}(\l)$ can be extended to a weakly analytic family of invertible operators on $H$ for $\{z \in \CC: \Im z \le \epsilon\}$.
	\end{enumerate}
    Then $\int_0^t A(t-s)ds \in \L (L^2_t([0, \infty), e^{\epsilon t} H))$ has inversion $\int_0^t A^{-1}(t-s)ds \in \L (L^2_t([0, \infty), e^{\epsilon t} H))$.
\end{corollary}
\begin{proof}
	Let $B(t) := A(t) e^{\epsilon t}$. Then $B = I_K + L_B$ is upper triangular with $L_B \in L^1_t \L(H)$. Since
	\[\hat{B}(\l) = \int_0^\infty e^{-it\l} e^{\epsilon t} A(t) dt = \hat{A}(\l +i\epsilon)\]
	is a weak analytic family in the lower-half plane,
	$B$ meets the conditions of Theorem \ref{thmWiener} (the second limit in \eqref{eqWienercond} follows (1) here), Lemma \ref{lemWiener1} and Lemma \ref{lemWiener2}. Hence there exists $B^{-1} \in K$. For $\mu, \nu \in K$, a computation as in Lemma \ref{lemWiener2} indicates
	\bee \left((e^{-\epsilon (\cdot)} \mu) * (e^{-\epsilon (\cdot)} \nu)\right) (t) = \int_0^t e^{-\epsilon(t-\tau)} (\mu* \nu)(t-\tau) d\tau.
	\eee
	Thus $B^{-1} * (e^{\epsilon (\cdot)} A) = (e^{\epsilon (\cdot)} A) * B^{-1} = I_K$ implies $(e^{-\epsilon (\cdot)}B^{-1}) * A = A * (e^{-\epsilon (\cdot)}B^{-1}) = I_K$. By the uniqueness of inversion in $K$, we know $B^{-1} = e^{\epsilon (\cdot)} A^{-1} \in L^1_t \L(H)$. Finally, for the boundedness, by acting $\int_0^t e^{\epsilon(t-s)} A(t-s) ds$ on $e^{\epsilon (\cdot)} u$, Lemma \ref{lemWiener2} indicates 
	\[ \left\|  \int_0^t A(t-s) ds \right\|_{\L (L^2_t([0, \infty), e^{\epsilon t}H ))} \le \| e^{\epsilon(\cdot)}A \|_{L^1_t\L(H)} < \infty.  \]
	Replacing $A$ by $A^{-1}$, we obtain the same bound which concludes the proof.
\end{proof}

\subsection{Modified Schr\"odinger system and Proof of Theorem \ref{thmStrichartz}} 
We consider the modified system 
\be i\pa_t \tilde{Z} + \calH P_{\rm ess} \tilde{Z}+ i\mu P_{\rm disc} \tilde{Z} = \tilde{F} \label{eqmod} \ee
where $\mu  > b(\sigma-s_c)$ and  $P_{\rm disc}$ and $P_{\rm ess}$ are spectral projections of $\calH$ defined in \eqref{eqdefproj2}. Our aim of this section is to apply Corollary \ref{coroSbVbdd} to verify the condition in Corollary \ref{coroWiener} and prove the Strichartz estimate.

As a preparation, since 
\be \calH P_{\rm ess} + i\mu P_{\rm disc} = \calH_b - ibs_c + V - P_{\rm disc} (\calH - i\mu) =: \calH_b - ibs_c + \tilde{V}, \label{eqtildeH} \ee
we first derive a decomposition of the modified potential $\tilde{V} := V - P_{\rm disc} (\calH - i\mu)$.
\begin{lemma}[Decomposition of modified potential]\label{lemdecompV} Assume \eqref{eqassump} and take $\delta_1$ from Lemma \ref{lemgap} and $p_1$ from Proposition \ref{prop3}.
	For $\sigma \in (s_c, s_c+\delta_1)$ and $\mu  > b(\sigma-s_c)$, there exist $p_2 \in (2, 2^*)$, $\gamma_2 > 0$ and a decomposition 
	\be \tilde{V} = \tilde{V}_1\tilde{V}_2 \ee
	such that $\tilde{V_1} \in \L(\dot{H}^\sigma, \dot{W}^{\sigma, p_2'} \cap \dot{H}^\sigma_{\la x \ra^{\gamma_2}})$ and $\tilde{V}_2 \in \L(\dot{W}^{\sigma, p_2} + \dot{H}^\sigma, \dot{H}^\sigma)$.
\end{lemma}
\begin{proof}
	We will specify $0 < \gamma_1 < \min\{\a, \gamma_0\}$ with $\gamma_0$ from Proposition \ref{prop3} and decompose trivially 
	\be \tilde{V_1} = \la x \ra^{-\gamma_1},\quad \tilde{V_2} = \la x \ra^{\gamma_1} \left(  V - P_{\rm disc}(\calH - i\mu)  \right). \label{eqdecompV} \ee
	Lemma \ref{lempotbdd} and \eqref{eqpotbdd2} implies $\tilde{V}_1 \in \L(\dot{H}^\sigma, \dot{W}^{\sigma, p_2'} \cap \dot{H}^\sigma_{\la x \ra^{\gamma_2}})$ and $\la x \ra^{\gamma_1} V \in \L(\dot{W}^{\sigma, p_2} + \dot{H}^\sigma, \dot{H}^\sigma)$ for some $p_2 > 2$ and $\gamma_2 > 0$. Notice that for $\xi \in \tilde{\calE}_\sigma = \sigma_{disc}(\calH)$, 
	\[P_\xi^1 := P_\xi (\calH - \xi) = -\frac{1}{2\pi i} \int_{|z-\xi|=\e} (z - \xi) \calS_{b, V}(z+ibs_c) dz\]
	satisfies $P_\xi^1 P_\xi = P_\xi P_\xi^1 = P_\xi^1$ by functional calculus (see for example \cite[Theorem V.8.1]{MR564653}) and similar for $(P_\xi^1)^*$ and $P_\xi^*$. In particular, we have $\mathrm{Ran} P_\xi^1 \subset \mathrm{Ran} P_\xi$ and $\mathrm{Ran} (P_\xi^1)^* \subset \mathrm{Ran} P_\xi^*$. Thus Proposition \ref{prop3} indicates that $P_{\rm disc} (\calH - i\mu)$ is bounded from $\dot{W}^{\sigma, p_1} +\dot{H}^\sigma$ to $\dot{H}^\sigma_{\la x \ra^{\gamma_0}}$ and the bound of $\tilde{V_2}$ follows \eqref{eqpotbdd4}. 
\end{proof}

Then we can define the resolvent of \eqref{eqtildeH} and derive its relation with $\calS_{b,V}(\l+ibs_c)$ and $\calS_{b}(\l+ibs_c)$. A formal computation suggests that
\be (P_{\rm ess} \calH + i\mu P_{\rm disc} - z) \left( \calS_{b, V}^{-}(z+ibs_c) P_{\rm ess} - (z - i\mu)^{-1} P_{\rm disc} \right) = I. \ee
This motivates the following definition and lemma.

\begin{lemma}[Extended resolvents for $\tilde{\calH}$] \label{lemdecompV2} Assume \eqref{eqassump} and take $\delta_1$ from Lemma \ref{lemgap}.
	For  $\sigma \in (s_c, s_c+\delta_1)$, $\mu > b(\sigma - s_c)$ and $\tilde{V}_1$, $\tilde{V}_2$ obtained in Lemma \ref{lemdecompV}, we define 
	\be \tilde{\calS}^-_{b, V}(z + ibs_c) := \calS_{b, V}^-(z+ibs_c) P_{\rm ess} - (z - i\mu)^{-1} P_{\rm disc} 
	\label{eqdefSbVt} \ee
	for $\Im z  \le b(\sigma - s_c)$. It satisfies the following properties.
	\begin{enumerate}
		\item {\rm Uniform boundedness:} for some $p_3 \in (2, 2^*)$, $\tilde{S}_{b, V}^-(z+ibs_c)$ is well-defined and uniformly bounded in $\L(\dot{W}^{\sigma, p_3'}\to \dot{W}^{\sigma, p_3})$ for $\Im z  \le b(\sigma - s_c)$.
		\item {\rm Relation with $\calS_b^-$:} the following identities hold for $\Im z  \le b(\sigma - s_c)$
			\begin{align} 
			\tilde{\calS}_{b,V}^-(z+ibs_c) = \calS_b^- (z+ibs_c) -&  \calS_b^- (z+ibs_c) \tilde{V}_1 (I + \tilde{V}_2  \calS_b^- (z+ibs_c) \tilde{V}_1 )^{-1} \tilde{V}_2 \calS_b^-(z+ibs_c)   \label{eqSbVt} \\
			I - \tilde{V}_2 \tilde{\calS}_{b, V}^{-}(z+ibs_c) \tilde{V}_1 &=(I + \tilde{V}_2 \calS_b^- (z+ibs_c) \tilde{V}_1)^{-1}. \label{eqSbVt2}
		\end{align}
	Moreover, $I + \tilde{V}_2 \calS_b^- (z+ibs_c) \tilde{V}_1$ and $I - \tilde{V}_2 \tilde{\calS}_{b, V}^{-}(z+ibs_c) \tilde{V}_1$ are both analytic and uniformly bounded in $\L(\dot{H}^\sigma)$ for $\Im z \le b(\sigma - s_c)$.
	\end{enumerate}
\end{lemma}

\begin{proof}
   Firstly we discuss the commutability issue. Let $z \in \CC - \tilde{\calE}_\sigma$ with $\Im z \le b(\sigma - s_c)$. The resolvent identities \eqref{eqSbVres1} and \eqref{eqSbVres2} indicate that $\calS_{b, V}^-(z +ibs_c)$ is commutable with all $\calS_{b, V}(z' + ibs_c)$ for $\Im z' \neq b(\sigma - s_c)$ and $z' \notin \tilde{\calE}_{\sigma}$. This implies
   \[ P_{\rm disc} \calS_{b, V}^- (z+ibs_c) = \calS_{b, V}^- (z+ibs_c) P_{\rm disc} \]
   and hence 
   \be P_{\rm ess} \calS_{b, V}^- (z+ibs_c) = \calS_{b, V}^- (z+ibs_c) P_{\rm ess}. \label{eqPccom} \ee
   From the uniform boundedness Corollary \ref{coroSbVbdd}, the identity \eqref{eqPccom} makes sense for all $\Im z \le b(\sigma - s_c)$. 

   From now on, we suppose $z\in \CC$ only satisfying $\Im z \le b(\sigma -s_c)$ and take $p_3 = \min\{ p_1, p_2\}$ with $p_1$ from Proposition \ref{prop3} and $p_2$ from Lemma \ref{lemdecompV}. Then Lemma \ref{prop2} and Corollary \ref{coroSbVbdd} imply (1).
   
   For (2), we claim the following resolvent identities
   \bea 
     \tilde{\calS}_{b, V}^- (z +ibs_c) - \calS_b^-(z+ibs_c) &=& -\calS_b^-(z+ibs_c) \tilde{V} \tilde{\calS}_{b, V}^- (z +ibs_c)  \label{eqSbVt3} \\ 
     &=& -\tilde{\calS}_{b, V}^- (z +ibs_c) \tilde{V} \calS_b^-(z+ibs_c),  \label{eqSbVt4}
   \eea	
   which imply \eqref{eqSbVt} and \eqref{eqSbVt2} easily. 
   These two resolvent identities will follow some computation involving the above commutability and inversion property of $\calS_{b}^-$. Take \eqref{eqSbVt3} as an example. From definitions of $\tilde{\calS}_{b, V}^-$, $\calS_{b, V}^-$, $\tilde{V}$ and identities \eqref{eqSb1}, \eqref{eqSb2} and \eqref{eqPccom}, we have (the variable $z +ibs_c$ of $\calS_b^-$, $\calS_{b, V}^-$ or $\tilde{\calS}_{b, V}^-$ are omitted in the following computation for convenience)
   \bee
     \mathrm{LHS \,\,of \,\,\eqref{eqSbVt3}} &=& \left( \calS_{b ,V}^- - \calS_b^-  \right)P_{\rm ess} - \left((z-i\mu)^{-1} + \calS_b^- \right) P_{\rm disc}, \\
     \mathrm{RHS \,\,of \,\,\eqref{eqSbVt3}} &=& -\calS_b^- \left(V - P_{\rm disc}(\calH - i\mu) \right) \left( \calS_{b, V}^- P_{\rm ess} - (z-i\mu)^{-1} P_{\rm disc}  \right) \\
     &=& -\calS_b^- V \calS_{b, V}^- P_{\rm ess} - \calS_b^- \left[ -V (z-i\mu)^{-1} + (\calH - i\mu)(z-i\mu)^{-1} \right] P_{\rm disc} \\
     &=& \left( \calS_{b ,V}^- - \calS_b^-  \right)P_{\rm ess} - \calS_b^- P_{\rm disc} - (z-i\mu)^{-1} \calS_b^- (\calH_{b} - z - ibs_c) P_{\rm disc}
   \eee
   Noticing that for $u \in \dot{W}^{\sigma, p_3'}$, we have $P_{\rm disc} u \in \dot{W}^{\sigma, p_3'}$ and $\calH_b P_{\rm disc} u \in \dot{W}^{\sigma, p_3'}$ since $P_{\rm disc} u$ is a generalized eigenfunction of $\calH = \calH_b - ibs_c + V$. Then the last term is just $(z-i\mu)^{-1} P_{\rm disc}$ by the inversion property \eqref{eqinv2}, and \eqref{eqSbVt3} follows. The identity \eqref{eqSbVt4} comes similarly with \eqref{eqinv1}.

   Finally, we prove the analyticity and uniform boundedness of $I + \tilde{V}_2 \calS_b^- (z+ibs_c) \tilde{V}_1$ and $I - \tilde{V}_2 \tilde{\calS}_{b, V}^{-}(z+ibs_c) \tilde{V}_1$. We first see the analyticity of $\tilde{V}_2 \calS_{b}^- (z+ibs_c) \tilde{V}_1$ by specifying a $0 < \gamma_3 \le \gamma_2$ small enough such that  $p_{\gamma_3} := \frac{2d}{d-\gamma_3} \le p_3$. Then differentiability of $\calS_{b}^-(z+ibs_c)$ (Proposition \ref{prop3} (5)) and boundedness of $\tilde{V}_1$ and $\tilde{V}_2$ (Lemma \ref{lemdecompV2}) imply the desired analyticity. 
   Next, Lemma \ref{lempotcpt2} (ii) and finite-rankness of $P_{\rm disc}$ easily indicate that $\tilde{V}_2 \tilde{\calS}_{b, V}^- (z+ibs_c) \tilde{V}_1$ is a compact operator on $\dot{H}^\sigma$. Hence we can use a similar argument as in Lemma \ref{prop1} plus the uniform boundedness of $\tilde{\calS}_{b, V}^-(z+ibs_c)$ to finish the proof.
\end{proof}

Now we are ready to prove the Strichartz estimate.\\

\begin{proof}[Proof of Theorem \ref{thmStrichartz}]

\underline{Step 1. Application of abstract Wiener's theorem.}
Let $H = \dot{H}^\sigma$ and $K$ be the Banach algebra as in \eqref{eqdefK}.
Define 
\be L(t) := -i\tilde{V}_2 e^{it(\calH_{b}-ibs_c)} \tilde{V}_1 \ee
for $t \ge 0$ and 
\be A = I_K + L.\ee
We will check the three conditions in Corollary \ref{coroWiener} with $\epsilon = b(\sigma - s_c)$ to apply it.  \\

(1) Firstly, the dispersive estimate \eqref{eqdispersiveest} and commutator relation \eqref{eqcomgroup} indicate that $e^{it(\calH_b - ibs_c)} \in L^1_t e^{b(\sigma - s_c)t}\L (\dot{W}^{\sigma, p_2'} \to \dot{W}^{\sigma, p_2})$. Combined with Lemma \ref{lemdecompV}, we have $L \in L^1_t e^{b(\sigma - s_c)t} \L(\dot{H}^\sigma)$ and hence $A \in K$ is upper triangular.  \\

(2) Next, we show that for any $[a_1, a_2] \in (0,\infty)$, 
\be \sup_{t \in [a_1, a_2]} \| L(\cdot + \epsilon) - L(\cdot) \|_{\L(\dot{H}^\sigma)} = o_{\epsilon}(1). \label{eqinteract00} \ee
It implies the second condition thanks to the integrability $L \in  L^1_t e^{b(\sigma - s_c)t} \L(\dot{H}^\sigma)$. Recall \eqref{eqdecompV} that 
$$\tilde{V}_2 =  \la x \ra^{\gamma_1}  V + \left(-\la x \ra^{\gamma_1} P_{\rm disc}(\calH - i\mu)\right) =: \tilde{V}_{2, 1} + \tilde{V}_{2,2}. $$
We will treat these two parts separately. 

\emph{(2.a) The multiplication part $\tilde{V}_{2, 1} := \la x \ra^{\gamma_1} V$.} We first approximate $\tilde{V}_1$ and $\tilde{V}_{2, 1}$ by $W_1, W_2 \in \mathscr{S}$ with $\hat{W}_1, \hat{W}_2$ compactly supported. Supposing $\hat{W}_1, \hat{W}_2 \in C^\infty_c (B_M)$ with $M > 0$, we consider $\epsilon < \min\{(10b)^{-1}, (2e^{b(a_2 + 1)}M)^{-4}\}$ and let $R = \epsilon^{-\frac 14}$. We decompose for any $t \in [a_1, a_2]$ as
\bea &&\|  W_2 (e^{i(t+\epsilon)(\calH_b - ibs_c)} - e^{it(\calH_b -ibs_c)}) W_1 f \|_{\dot{H}^\sigma} \label{eqinteract0}\\
&\le& \left\|  W_2   \chi(R^{-1}D) (e^{i(t+\epsilon)(\calH_b - ibs_c)} - e^{it(\calH_b -ibs_c)})  W_1 f  \right\|_{\dot{H}^\sigma} \label{eqinteract1} \\
&+& \sum_{\pm} \left\|  W_2   (1- \chi(R^{-1}D)) e^{i(t\pm \frac{\epsilon}{2})(\calH_b - ibs_c)}  W_1 f  \right\|_{\dot{H}^\sigma} \label{eqinteract2}
\eea 
where  $\chi$ is the smooth cutoff from \eqref{eqdefchiR}. 
For the low-frequency part \eqref{eqinteract1}, notice that 
\bee &&\mathscr{F} \left(\chi(R^{-1}D) (e^{i(t+\epsilon)\Delta_b} - e^{it\Delta_b})  W_1 f\right) (\eta) \\
&=& \chi(R^{-1}\eta) \int_{\RR^d} \bigg[ \left(e^{b(t+\epsilon)\frac d2} e^{-i\frac{e^{2b(t+\epsilon)} -1}{2b} |\eta|^2}   \hat{W}_1 (e^{b(t+\epsilon)} \eta - \zeta)  \right)  \\
&& \quad \quad - \left(e^{bt\frac d2} e^{-i\frac{e^{2bt} -1}{2b} |\eta|^2}   \hat{W}_1 (e^{bt} \eta - \zeta) \right) \bigg] |\zeta|^{-\sigma} \widehat{D^{\sigma}f}(\zeta) d\zeta
\eee
where the absolute value of the bracket can be bounded by
\[ 2(b+b^{-1}) \epsilon R^2 e^{b(a_2 + 1)\left(\frac d2 + 2\right)} \left[\int_0^1 |\nabla \hat{W}_1(e^{b(t+\epsilon s)}\eta - \zeta )| ds  + \sum_{\pm} |\hat{W}_1(e^{b(t\pm \frac{\epsilon}{2})}\eta - \zeta )| \right]  \]
where we get smallness $\epsilon R^2 = \epsilon^{\frac 12}$.
The compact support of $\hat{W}_1$ and $\hat{W}_2$ also indicates $|\xi| \sim |\eta| \sim |\zeta|$ when $|\xi| \gg M$ so that $|\zeta|^{-\sigma}$ cancels $|\xi|^\sigma$ from the $\dot{H}^\sigma$ norm. Therefore, Young's inequality implies
\[ \eqref{eqinteract1} \lesssim_{d, b, a_2, W_1, W_2} \epsilon^{\frac 12} \| f \|_{\dot{H}^\sigma}. \]

For the high-frequency part, we again estimate pointwisely in the Fourier space. Let $\tilde{f} = D^\sigma f$ and $\tilde{t}$ to be either $t \pm \frac{\epsilon}{2}$. Compute
\bea &&\mathscr{F} D^{\sigma} W_2 D^{-\sigma}\left( 1- \chi(R^{-1}D)\right) e^{i\tilde t\Delta_b} D^{\sigma} W_1 D^{-\sigma} \tilde f (\xi) \label{eqfinitespeed} \\
&=& C |\xi|^{\sigma} \int_{\RR^d} \hat{W}_2(\xi - \eta) (1- \chi(R^{-1} \eta)) e^{b\tilde t\left(\sigma + \frac d2\right)}e^{-i \frac{e^{2b\tilde t} - 1}{2b} |\eta|^{2}} \left(\widehat{ W_1 D^{-\sigma} \tilde f}\right)(e^{b\tilde t}\eta) d\eta \nonumber \\
&=& -C \frac{be^{b\tilde t\left(\sigma + \frac d2\right)}}{e^{2b\tilde t}-1} |\xi|^\sigma \int_{\RR^d} \int_{\RR^d} e^{-i \frac{e^{2b\tilde t} - 1}{2b} |\eta|^{2}} \nonumber\\
&& \qquad \nabla_{\eta} \cdot \left[ \hat{W}_2(\xi - \eta) (1- \chi(R^{-1} \eta))\frac{\eta}{|\eta|^2} \hat{W}_1 (e^{b\tilde t}\eta - \zeta)  \right] |\zeta|^{-\sigma} \hat{\tilde f}(\zeta) d\zeta d\eta \nonumber
\eea
where we integrate by parts over $e^{-i \frac{e^{2b\tilde t} - 1}{2b} |\eta|^{2}}$. So from the support of $1-\chi(R^{-1} \eta)$ and $\hat{V}_1$, $\hat{V}_2$, we know $|\xi| \sim |\zeta| \sim |\eta| \ge R$. Hence by Young's inequality,
\bee |\eqref{eqfinitespeed}| \lesssim_{d, \sigma} R^{-1} \left|\frac{be^{b\tilde t\left(\sigma+1 + \frac d2\right)}}{e^{2b\tilde t}-1}\right|\left\{ (|\hat{W}_2| + |\nabla \hat{W}_2|) * \left[ (|\hat{W}_1| + |\nabla \hat{W}_1|) * |\hat{\tilde f}|\right](e^{b\tilde t}\cdot)  \right\} (\xi).
\eee
Young's inequality again implies
\[ \eqref{eqinteract2} \lesssim_{d, b, a_1, a_2, W_1, W_2} \epsilon^{-\frac 14} \| f \|_{\dot{H}^\sigma}. \]
So far, we have verified 
\[ \sup_{t \in [a_1, a_2]} \eqref{eqinteract0} = o_{\epsilon}(1)\|f \|_{\dot{H}^\sigma}. \]

\emph{(2.b) The finite-rank part $\tilde{V}_{2, 2} := -\la x \ra^{\gamma_1} P_{\rm disc}(\calH - i\mu)$.} We can assume $\tilde{V}_{2, 2} g= \sum_{i = 1}^N (g, u_i)_{\dot{H}^\sigma}  \la x \ra^{\gamma_1} v_i$ with $v_i \in \mathrm{Ran}P_{\rm disc}$ and $u_i \in \mathrm{Ran} P_{\rm disc}^*$. Since $u_i \in \dot{W}^{\sigma, p_2'}$ and $\la x \ra^{\gamma_1} v_i \in \dot{H}^\sigma$ (from Proposition \ref{prop3} and \eqref{eqpotbdd4}), we can approximate $u_i$ with $\tilde{u}_i \in \mathscr{S}(\RR^d)$ and obtain uniform smallness for $t \in [a_1, a_2]$ from the strongly continuity of the semigroup
\[ \left(e^{i(t+\epsilon)(\calH_b - ibs_c)} - e^{it(\calH_b -ibs_c)}\right)^{*} \tilde{u}_i \to 0\quad \mathrm{in}\,\,\dot{H}^\sigma \,\, \mathrm{as}\,\,\epsilon \to 0. \]

So \eqref{eqinteract00} and hence the second condition are confirmed by combining these two estimates for $\tilde{V}_{2, 1}$ and $\tilde{V}_{2,2}$. \\

(3) Recalling the definition of $\calS_b^-$ in \eqref{eqSb}, we compute
\[\hat{A}(z) = I - i \int_0^\infty \tilde{V}_2 e^{it(\calH_b - ibs_c)}\tilde{V}_1 e^{-itz} dt = I + \tilde{V}_2 \calS_b^-(z + ibs_c) \tilde{V}_1.   \]
The third follows Lemma \ref{lemdecompV2} (2) directly. \\

\underline{Step 2. Proof of Strichartz estimate \eqref{eqStrichartz}.}
Now we define 
\[ U_b := \int_0^t e^{i(t-s)(\calH_b-ibs_c)}ds. \]
The conclusion of Corollary \ref{coroWiener} implies
\be \left(\int_0^t A(t-s)ds\right)^{-1} = \left(I - i\tilde{V}_2 U_b \tilde{V}_1\right)^{-1} \in \L(L^2_t ([0,\infty), e^{b(\sigma - s_c) t}\dot{H}^\sigma). \label{eqfin1} \ee
Also from the Strichartz estimate for $e^{it\Delta_b}$ from \eqref{eqStrichartzfree} and the commutator relation \eqref{eqcomgroup}, 
we have
\be U_b \in \L\left(L_t^{q_2'} e^{tb(\sigma - s_c)} \dot{W}^{\sigma,p_2'} \to L_t^{q_1} e^{tb(\sigma - s_c)} \dot{W}^{\sigma, p_1} \right) \label{eqfin2} \ee
	where $(q_i, p_i)$ for $i = 1,2$ satisfies the admissible condition \eqref{eqadmissible}.


Finally, we will come back to the modified system \eqref{eqmod}-\eqref{eqtildeH} and prove \eqref{eqStrichartz}. 
Viewing $\tilde{V}\tilde{Z}$ as the source term and applying Duhamel's formula, we get
\bee 
\tilde{Z} = U_b \left[ \delta_{t=0}\tilde{Z}(0) - i(-\tilde{V}\tilde{Z} + \tilde{F})\right].
\eee
Then we decompose $\tilde{V}= \tilde{V}_1 \tilde{V}_2$ according to Lemma \ref{lemdecompV}, act $\tilde{V}_2$ on both sides and invert $I - i\tilde{V}_2 U_b \tilde{V}_1$ to get the formula for $\tilde{V}_2 \tilde{Z}$
\[  \tilde{V}_2 \tilde{Z} = (I - i\tilde{V}_2 U_b \tilde{V}_1)^{-1} \tilde{V}_2 U_b (\delta_{t=0} \tilde{Z}(0) -i \tilde{F} ).\]
Plugging this representation into the previous formula for $Z$, we get
\be \tilde{Z} = \left[U_b + iU_b \tilde{V}_1 (I - i\tilde{V}_2 U_b \tilde{V}_1)^{-1} \tilde{V}_2 U_b  \right]  (\delta_{t=0} \tilde{Z}(0) -i \tilde{F} )\ee
The bounds \eqref{eqfin1}, \eqref{eqfin2} and boundedness of $\tilde{V}_1$ and $\tilde{V}_2$ from Lemma \ref{lemdecompV} indicate that 
\be \left\| e^{bt(\sigma - s_c)}\tilde{Z} \right\|_{L^{q_1}_t \dot{W}^{\sigma, p_1}_x} 
\lesssim \| \tilde{Z}(0)\|_{\dot{H}^\sigma}
 +  \left\|e^{bt(\sigma - s_c)}  \tilde{F} \right\|_{L_t^{q_2'} \dot{W}^{\sigma, p_2'}_x}.  \ee
 We emphasize that the enlarged admissible region of \eqref{eqfin1} and \eqref{eqfin2} is necessary because $\tilde{V}_1$ and $\tilde{V}_2$ can only slightly change the integrability.
 
Now consider $Z$ with $Z(0) = Z_0 \in \dot{H}^\sigma$ satisfying the original system \eqref{eqls} and let $\tilde{Z}(0) = P_{\rm ess} Z(0) \in \dot{H}^\sigma$ evolve in the modified system \eqref{eqmod}. Then we will have $P_{\rm ess} Z(t) = \tilde{Z}(t)$ and the Strichartz estimate above implies \eqref{eqStrichartz} for $\nu = b(\sigma -s_c)$. \\

For general $\nu \le b(\sigma - s_c)$, it suffices to recover \eqref{eqfin1} and \eqref{eqfin2} with $b(\sigma -s_c)$ substituted by $\nu$. Noticing that the condition of Corollary \ref{coroWiener} for $\epsilon := \nu < b(\sigma - s_c)$ also holds, the above arguments also imply \eqref{eqfin1} for $\nu$. Since 
$$ \left \|e^{it\Delta_b} e^{-[b(\sigma-s_c) - \nu]t} \right \|_{L^{q'} \to L^{q}} \le \left \|e^{it\Delta_b} \right \|_{L^{q'} \to L^{q}} \lesssim |t|^{-\left ( \frac d2 - \frac dq \right )}, $$
this dispersive estimate similarly yields the free Strichartz \eqref{eqfin2} with exponential weight $e^{\nu t}$. These two elements conclude the proof of for all  $\nu \le b(\sigma - s_c)$.

\end{proof}

\section{Nonlinear stability in $\dot H^\sigma$}\label{s6}

In this section, we turn to the nonlinear stability theory self-similar blowup of \eqref{eqNLS} in $\dot H^\sigma$. More precisely, we will construct a finite-codimensional Lipschitz manifold of initial data leading to the self-similar blowup plus a decaying perturbation in $\dot{H}^\sigma$, and show that this implies the asymptotic stability assuming mode stability Assumption \ref{assmodestab}. 

We begin by clarifying our setting. Recall \eqref{eqNLS} under the self-similar renormalization \eqref{eqlaw} with $u = Q_b + \e$, $Z = \left(\begin{array}{c} \e \\ \bar \e \end{array}\right)$ as 
\be i \partial_\tau Z + \mathcal{H}Z = F(Z). \tag{\ref{eqZ}} \ee
where $\calH$ is the linear operator \eqref{Hb} with potential \eqref{potential}, $Z = \left( \begin{array}{c} Z_1 \\ Z_2 \end{array} \right)\in \Bbb C^2$ and $F = \left( \begin{array}{c} -R(Z_1, Q_b) \\ \overline{R(\bar{Z_2}, Q_b)} \end{array} \right)$ with 
\begin{align}
	R(\e, v) = (\e+v)|\e+v|^{p-1} - v|v|^{p-1} - \frac{p+1}{2}|v|^{p-1}\e - \frac{p-1}{2} v^2 |v|^{p-3} \bar{\e}.\label{eqNL}
\end{align}
From the decay of self-similar profile \eqref{selfsimilardecay}, the potential satisfies \eqref{eqpotbd0} with $\a = 2$. Therefore, the exponent from \eqref{eqdelta0} now reads as 
\be \left| \begin{array}{l}
     \delta_0 := \frac{1}{16} \min \left\{ \frac d2 -s_c , 1-s_c \right\} > 0,  \\
     p_0 := \left( \max \left\{ \frac{d-2}{2d}, \frac{\sigma}{d} \right\} + \frac{\delta_0 s_c}{d^2 (p+1)}\right)^{-1} \in \left ( 2, \frac{2d}{d-2}\right ).
\end{array}\right.\label{eqp0sec6} \ee

\subsection{$\mathcal{J}$-invariance and spectrally decoupled system}\label{S51}
In this subsection, we will transform the nonlinear Schr\"odinger equation into a renormalized system with stable modes and unstable modes decoupled. The extra degree of freedom in the system will be removed by an $\RR$-linear operator $\calJ$ (borrowed from \cite{MR2480603}).\\

\noindent{\em $\calJ$-invariance}.  We define the $\RR$-linear subspace
\be X_+ := \left\{ \left( \begin{array}{c} \e \\ \bar{\e} \end{array} \right): \e \in \dot{H}^{\sigma}(\RR^d) \right\} \subset \dot{H}^\sigma(\RR^d) \times  \dot{H}^\sigma(\RR^d)  \label{eqX+}  \ee
and the $\RR$-linear transform on $(\dot H^\sigma)^2$
\be \calJ:  \left( \begin{array}{c} Z_1 \\ Z_2 \end{array} \right) \mapsto \overline{ \left( \begin{array}{c} Z_2 \\ Z_1 \end{array} \right)} \label{eqJ} \ee
Obviously, for $Z \in (\dot{H}^\sigma)^2$, $Z \in X_+$ if and only if $\calJ Z = Z$. With the help of $\calJ$, we can show $X_+$ is invariant under the evolution \eqref{eqZ}.
\begin{lemma}[$\calJ$-invariance for \eqref{eqZ}] \label{lemJinv1}
	 We have 
	\bee  \calJ \circ i &=& -i \circ \calJ, \\
	\calJ \circ \calH &=&-\calH \circ \calJ, \\
	\calJ F(Z) &=& -F(\calJ Z).
	 \eee
    Thus if $Z(0) \in X_+$, the solution $Z(\tau)$ of \eqref{eqZ} will satisfies $Z(\tau) \in X_+$ for all $\tau$ in the lifespan. 
	\end{lemma}
\begin{proof}
	Those identities follows the definition of $\calJ$ and $i$, $\calH$, $F(Z)$. We only remark that $\overline{\Delta_b} = \Delta_{-b}$.
	
	For the invariance, suppose $Z(0) \in X_+$ generates a solution $Z(\tau)$. Since $\calJ Z$ satisfies the same equation with initial data $\calJ(Z)(0) = Z(0) \in X_+$, the uniqueness of solution from Proposition \ref{proplwp2} indicates $Z(\tau) = \calJ Z(\tau)$ for all $\tau$, which implies $Z(\tau) \in X_+$. 
	\end{proof}

Define the exponent 
\be \a_c :=d\left( \frac{1}{2} - \frac{1}{p+1}\right) = \frac{d}{d + 2 - 2s_c} \in (s_c, 1) \cap \left(\frac{d}{d+2}, \frac d2 \right) \label{eqac}\ee
such that  $ \dot{H}^{\a_c}(\RR^d) \hookrightarrow L^{p+1}(\RR^d)$. 
Recall the local-wellposedness of \eqref{eqNLS} and \eqref{eqZ} in $\dot{H}^\sigma$, $\sigma \in (s_c, \a_c)$ from Proposition \ref{proplwp} and Proposition \ref{proplwp2} respectively. The identification of these two systems via $\calJ$-invariance follows.

\begin{lemma}\label{lemidentsolu}
	Let $d \ge 1$, $0 < s_c < \min\left\{ \frac d2 , 1 \right\}$, $\sigma \in (s_c, \a_c)$. Let $u$ solves \eqref{eqNLS} with initial data $u_0 \in \dot{H}^\sigma$, and $Z$ solves \eqref{eqZ} with initial data $Z_0 = \left( \begin{array}{c} \e_0 \\ \bar{\e_0} \end{array} \right)\in (\dot{H}^\sigma)^2$. 
	Take any $T > 0$ and $\l(t)$, $\tau(t)$ from \eqref{eqlaw}.
	If the initial data are related by
	$$u_0 = \lambda(0)^{-\frac{2}{p-1}}(Q_b + \e_0)(x/\lambda(0) )e^{i\tau(0)},$$ then these two solutions identify by $Z(\tau, y)  = \left( \begin{array}{c} \e(\tau, y) \\ \bar{\e}(\tau, y) \end{array} \right)$ and
	\be u(t, x) = \l(t)^{-\frac{2}{p-1}}(Q_b + \e)\left (t, \frac{x}{\l(t)} \right )e^{i\left (\tau(t)\right )}. \label{equtoe}\ee
\end{lemma}
\begin{proof}
The form $Z = \left( \begin{array}{c} \e \\ \bar{\e} \end{array} \right) \in X_+$ follows the previous lemma. Using the equation of $Q_b$ \eqref{eqselfsimilar}, we can check the $u$ defined in \eqref{equtoe} satisfies \eqref{eqNLS} with initial data $u_0$. Thus the local well-posedness Proposition \ref{proplwp} implies the correspondence.
	\end{proof}

\mbox{}

Next, we show the spectral projection of $\calH$ from Section \ref{s4} also preserves $X_+$, so that we can decouple the evolution of \eqref{eqZ} in $X_+$ into stable, unstable and center parts. 

\begin{lemma}[Spectrum and $\calJ$-invariance of $\calH$]\label{lemJinv2} Let $d \ge 1$, $ 0 < s_c < \min \{\frac d2, 1\}$ and assume the existence of $Q_b$ satisfying \eqref{eqselfsimilar}, \eqref{eqnonvanishing} and \eqref{selfsimilardecay}. For the operator $\calH$ in \eqref{Hb} with potential \eqref{potential}, there exists $0 < \delta_1 < \delta_0$ with $\delta_0$ from \eqref{eqp0sec6}, such that for $\sigma \in (s_c, s_c + \delta_1)$, $\calH\big|_{\dot{H}^\sigma}$ has the spectrum decomposition
	\be \sigma_{\rm ess}(\calH\big|_{\dot{H}^\sigma}) = \{ \RR + ib(\sigma - s_c) \}, \quad \sigma_{\rm disc}(\calH\big|_{\dot{H}^\sigma}) = \calE_s \cup \calE_u \ee
	with the finite discrete sets $\calE_s$, $\calE_u$ satisfying
	\be \calE_s \subset \{z \in \CC : \Im z \ge b\delta_1\},\quad \calE_u \subset \{z \in \CC: \Im z \le 0\}. \label{eqdecompE} \ee
	The corresponding spectral projections as $P_c:= P_{\rm ess}$\footnote{We replace subscription $\rm{ess}$ by $c$ for simplicity.}, $P_s$ and $P_u$ will satisfy
	\begin{enumerate}
		\item $I = P_{\rm c} + P_u + P_s$.
		\item $P_s, P_u \in \L((\dot{W}^{\sigma, p_1} + \dot{W}^{\sigma, p_0'})^2 \to (\dot{W}^{\sigma, p_0} \cap \dot{W}^{\sigma, p_1'})^2)$ for $p_0$ as \eqref{eqp0sec6} and some $p_1 > 2$.
		\item {\rm $\calJ$-invariance: }
		\be  \calJ \circ P_\a = P_\a \circ \calJ, \quad \a \in \{s, u, c\}.\label{eqcomPJ}
		\ee
	\end{enumerate}
\end{lemma}
\begin{proof}The spectral decomposition is a direct conclusion of Proposition \ref{prop3} with the $\delta_1$ and decoupling \eqref{eqdecompE} coming from Lemma \eqref{lemgap}. The only statement left is the $\calJ$-invariance property.
	
	It suffices to prove \eqref{eqcomPJ} for $\a \in P_s$.  From the boundedness and gap of $\calE_s$ and $\calE_u$, we can pick contours $\gamma_s, \gamma_u \subset \CC$ to be symmetric about $i\RR$ such that
	\be P_\a =- \frac{1}{2\pi i} \oint_{\gamma_\a} (\calH - z)^{-1} dz,\quad \a \in \{u, s\}. \label{eqcontourPa} \ee
	Thus $\gamma_\a = -\bar{\gamma}_\a$ as sets but with reversed orientation as oriented curves. Denote $J := \left( \begin{array}{cc} 0 & 1 \\ 1& 0 \end{array} \right)$, then $\calJ Z = \overline{JZ}$. Direct computation implies $J\calH J = -\overline{\calH}$. Thus
	\bee
	\calJ P_\a Z &=& \frac{1}{2\pi i} \oint_{\gamma_\a} \overline{J  (\calH - z)^{-1} J } d\bar{z}  \overline{JZ} =  \frac{1}{2\pi i} \oint_{\gamma_\a}  (-\calH - \bar{z})^{-1} d\bar{z}  \overline{JZ} \\
	&= &\frac{1}{2\pi i} \oint_{-\bar{\gamma}_\a}  (\calH - \omega)^{-1} d\omega  \overline{JZ} = -\frac{1}{2\pi i} \oint_{\gamma_\a}  (\calH - \omega)^{-1} d\omega  \overline{JZ} = P_\a \calJ Z
	\eee
	for $\a \in \{u, s\}$. 
\end{proof}

\mbox{}

\noindent{\em Spectrally decoupled system \eqref{eqZdecouple} in $X_+$.}
 Applying these projections onto \eqref{eqZ}, we get the decoupled system
\be
\left| \begin{array}{l}
i\pa_\tau Z_c + \calH Z_c = F_c \\
i\pa_\tau Z_s + \calH Z_s =F_s \\
i\pa_\tau Z_u + \calH Z_u = F_u 
\end{array}\right. \label{eqZcsu}
\ee
with $Z_\a = P_\a Z$, $F_\a = P_\a F$ for $\a \in \{u, s, c\}$. 
We further parametrize the evolution in $X_s := P_s X_+$ and $X_u:= P_u X_+$ as $\RR^{d_s}$ and $\RR^{d_u}$ respectively\footnote{We cannot view $X_u$ or $X_s$ as complex-valued space since $X_+$ is only $\RR$-linear.}.
Then the system \eqref{eqZcsu} finally turns into 
\be
 \left|\begin{array}{l}
   i\pa_\tau Z_c + \calH Z_c = F_c \\
    	\frac{d}{d\tau}z_s + A_s z_s= f_s, \\  
   \frac{d}{d\tau}z_u + A_u z_u = f_u. 
 \end{array}\right.\label{eqZdecouple}
\ee
where
\begin{align*}
 z_\a(\tau) = \phi_\a(Z_\a(\tau)) \in \RR^{d_\a},& \qquad f_\a(\tau) = \phi_\a (-iF_\a(\tau)) \in \RR^{d_\a},\\
A_\a = \phi_\a^{-1} \circ& (-i\calH P_\a) \circ \phi_\a \in \RR^{d_\a \times d_\a}, 
\end{align*}
with $\phi_\a: P_\a X_+ \to \RR^{d_\a}$ a $\RR$-linear homeomorphism for $\a \in \{ u, s\}$. Moreover, they satisfy
 \bea \Re\left( \mathrm{eig}(A_s) \right)\ge b\delta_1, &\quad& \Re \left (\mathrm{eig}(A_u)\right ) \le 0, \label{eqspecA}\\ |f_\a| \sim \|F_\a \|_{\dot{H}^{\sigma}} \lesssim \| F \|_{\dot{H}^{\sigma} + \dot{W}^{\sigma, p_0'}}, &\quad &
|z_\a| \sim \|Z_\a\|_{\dot{W}^{\sigma, q}}\,\,  \forall q \in [2, p_0] \label{eqzProj} \eea
from Lemma \ref{lemJinv2}. 

\subsection{Finite-codimensional stability in $\dot H^\sigma$}

In this subsection, we prove a more accurate version of Theorem \ref{thmcodimstability}. We first recall $X_+$ space from \eqref{eqX+} and define the natural isometry 
\be \Pi: X_+\to \dot H^\sigma(\RR^d), \qquad \Pi\left( f, \bar f\right)^T = f. \label{eqdefPi}\ee

\begin{theorem}[$\dot H^\sigma$ finite-codimensional asymptotic stability] \label{thmfincodimstabHsigmaB}
    Let $d \ge 1$, $s_c \in (0, 1)$, and $b, Q_b$ satisfying  \eqref{eqselfsimilar}, \eqref{eqnonvanishing}, \eqref{selfsimilardecay}. 
  Then for every $0 < \sigma - s_c \ll 1$, the following statements hold true.
  \begin{enumerate} 
      \item Finite codimensional stability: There exists $\epsilon_0 \ll 1$ and a Lipschitz map
      \[ \Phi: B_{\epsilon_0}^{X_+} \cap P_{cs} X_+ \to P_u X_+  \]
      such that for all $\e_{cs} \in B_{\epsilon_0}^{\dot H^\sigma} \cap \Pi P_{cs} X_+$, the initial data
      \[ u_0 = Q_b + \e_{cs} + \Pi \Phi(\Pi^{-1}\e_{cs}) \]
      generates a solution of \eqref{eqNLS} blowing up at $T = \frac{1}{2b}$ satisfying  
      \bea 
      u(t, x) &=& \frac{1}{\l(t)^{\frac{2}{p-1}}}(Q_b + \e) \left(t, \frac{x}{\l(t)}\right)e^{i\tau(t)}, 
      \label{eqss1}\\
        \| \e(t) \|_{\dot H^\sigma_x} &\lesssim& (1-2bt)^{\frac{\sigma - s_c}{2}},\label{eqss2}
           \eea
	where $(\lambda(t),\gamma(t),\tau(t))$ are given by \eqref{eqlaw}, and there exists $u_* \in \dot H^\sigma$ such that 
        \be u(t) - \frac{1}{\l(t)^{\frac{2}{p-1}}}Q_b \left( \frac{x}{\l(t)}\right)e^{i\tau(t)}  \to u_*\quad \mathrm{in}\,\,\dot{H}^\sigma \quad as\quad t\to T.
		\label{eqss3}\ee
   
      Here $P_u$ and $P_{cs} := P_c + P_s$ are from Lemma \ref{lemJinv2}. 
    \item Contraction of the unstable-data map:  $\Phi$ satisfies $\Phi(0) = 0$ and
      \be \| \Phi(\e_1) - \Phi(\e_2) \|_{X_+} \lesssim \left( \| \e_1\|_{X_+} + \| \e_2 \|_{X_+} \right)^{\min \{ p-1, 1\}} \| \e_1 - \e_2 \|_{X_+}.\label{eqLipestPhi}  \ee
  \end{enumerate}
\end{theorem}
    
    The statement (1) clearly implies Theorem \ref{thmcodimstability}. 

\begin{proof}[Proof of Theorem \ref{thmfincodimstabHsigmaB}] We choose $\sigma \in (s_c, s_c + \frac 12 \delta_1)$ with $\delta_1$ from Lemma \ref{lemJinv2}, and we denote
\be \mu = b(\sigma - s_c) \le \frac 12 \delta_1 \le \frac 12 \delta_0, 
\label{eqdeltachoice} \ee 
where $\delta_0$ is from \eqref{eqp0sec6}.
We introduce the following exponents:
\be  \frac{1}{r_1} = \frac{\sigma}{d} + \frac{1}{p+1} \frac{d-2\sigma}{d},\quad \frac{1}{r_2} = \frac{p}{p+1}\frac{d-2\sigma}{d}, \quad \frac{1}{q_1} = \frac{(p-1)(d-2\sigma)}{4(p+1)}.
\label{eqr1r2q1}\ee
Using $\a_c$ from \eqref{eqac}, we have $\frac{1}{r_1} = \frac{1}{2} - \frac{\a_c}{d} \frac{d-2\sigma}{d} \in \left ( \frac{1}{2^*},\frac 12 \right )$, $\frac{1}{q_1} = \frac{\a_c(d-2\sigma)}{2d}$.
One can then check $\frac{2}{q_1} + \frac{d}{r_1} = \frac d2 $, $\frac{1}{r_1'} = \frac{1}{r_1} + \frac{p-1}{pr_2}$, $\dot{W}^{\sigma, r_1'} \hookrightarrow L^{r_2}$ and $\dot{W}^{\sigma, r_1} \hookrightarrow L^{pr_2}$. Moreover, recalling $\delta_0, p_0$ from \eqref{eqp0sec6}, we have $r_1 < p_0$ since
\[ \frac 1{r_1} - \frac{1}{p_0} = \min \left\{ \frac{d(1-\a_c)+ 2\sigma \a_c}{d^2}, \frac{d-2\sigma}{d(p+1)}\right \} - \frac{\delta_0 s_c}{d^2 (p+1)} > 0. \]
We also define
\be \beta := \frac{1}{q_1'} - \frac{p}{q_1} = \frac{\sigma-s_c}{\frac d2 - s_c} > 0,\quad \tilde p = \min \left\{ \frac{p+1}{2}, \frac 32 \right\}.\label{eqbeta}\ee 


\mbox{}

\underline{Step 1. Reduction to contraction mapping estimates.} Let $\epsilon_0 \ll 1$. For any 
\[ (Z_{c0}, z_{s0}) \in P_c X_+ \times \RR^{d_s}, \quad \max \{ \| Z_{c0}\|_{(\dot H^\sigma)^2}, |z_{s0}| \} =: \epsilon \le \epsilon_0, \]
 We denote $\bm{Z} = (Z_c, z_s, z_u)$, and define the linear map $\bm{\Psi}(\bm{Z}) = (\bm{\Psi}_c(\bm{Z}), \bm{\Psi}_s(\bm{Z}), \bm{\Psi}_u(\bm{Z}))$ as
\be
  \left( \begin{array}{c}\bm{\Psi}_c(\bm{Z}) \\ \bm{\Psi}_s(\bm{Z}) \\ \bm{\Psi}_u(\bm{Z}) \end{array} \right) = \left( \begin{array}{c}
 e^{it\calH} Z_{c0} - i \int_0^t e^{i(t-\tau)\calH} P_c F d\tau \\ 
 e^{-tA_s} z_{s0} + \int_0^t e^{-(t-\tau)A_s} f_s d\tau \\
- \int_t^\infty e^{(\tau-t)A_u} f_u d\tau
  \end{array} \right)
\ee
 Define the Banach space as
\bee
 \XX = \XX_c \times \XX_s \times \XX_u := S \times C^0_{e^{\frac{\tilde p + 1}{2}\mu t}}\left([0, \infty), \RR^{d_s} \right) \times C^0_{e^{\frac{\tilde p + 1}{2}\mu t}}\left([0, \infty), \RR^{d_u} \right)  
\eee
where the Strichartz norm is 
\be
  S = C^0_{e^{\mu t}} ([0, \infty), (\dot{H}_y^\sigma)^2(\RR^d) )  \cap \left(L^2_{e^{\mu t}}\cap L^{q_1}_{e^{\mu t}}\right) ([0, \infty), (\dot{W}_y^{\sigma, r_1})^2(\RR^d) ). \label{eqdefStrinorm}
\ee
We claim that for $\bm{Z}, \tilde{\bm{Z}} \in B^{\XX}_{C_0 \epsilon}$ with $\epsilon_0 \ll 1$, $C_0 \gg 1$, we have 
\bea
\| e^{\tilde p \mu \tau} F(Z)\|_{L^{q_1'}_\tau\dot W^{\sigma, r_1'}_y} &\lesssim & \| \bm{Z}\|_{\XX}^{\min \{ p, 2 \}} \label{eqPsicontract00}  \\
  \| \bm{\Psi}_c(\bm{Z}) \|_{\XX_c} &\le& \frac 12 C_0 \left( \epsilon + \| \bm{Z} \|_{\XX}^{\min \{ p, 2 \}}\right), \label{eqPsicontract01} \\
  \| \bm{\Psi}_s(\bm{Z}) \|_{\XX_s} &\le& \frac 12 C_0 \left( \epsilon + \| \bm{Z} \|_{\XX}^{\min \{ p, 2 \}}\right), \label{eqPsicontract02}\\
  \| \bm{\Psi}_u(\bm{Z}) \|_{\XX_u} &\lesssim&  \| \bm{Z} \|_{\XX}^{\min \{ p, 2 \}}, \label{eqPsicontract03}\\
  \| \bm{\Psi}(\bm{Z}) - \bm{\Psi}(\tilde{\bm{Z}}) \|_{\XX} &\lesssim& \epsilon^{\min \{p-1, 1 \}} \| \bm{Z} - \tilde{\bm{Z}} \|_{\XX}. \label{eqPsicontract04}
\eea

Postponing the proof for later steps, we now show that these estimates \eqref{eqPsicontract00}-\eqref{eqPsicontract04} imply Theorem \ref{thmfincodimstabHsigmaB}. Clearly, they imply that $\bm{\Psi}$ is a contraction mapping in $B^{\XX}_{C_0 \epsilon}$, leading to a global solution $Z \in X_+$ solving \eqref{eqZ} and hence $u$ of the form \eqref{eqss1} solving \eqref{eqNLS}. The unstable-data map $\Phi$ is determined by $(Z_{c0}, z_{s0}) \mapsto z_{u0} := -\int_0^\infty e^{\tau A_u} f_u d\tau$, whose contraction estimate \eqref{eqLipestPhi} follows from \eqref{eqPsicontract04} noticing that the $\bm{\Psi}_u$ component is defined independent of $(Z_{c0}, z_{s0})$. 

For the estimate of $\epsilon$ \eqref{eqss2}, we combine the boundedness of $\bm{Z}$ with regularity of eigenspaces \eqref{eqzProj} to see
\be \| e^{\mu \tau} Z \|_S \lesssim \| e^{\mu \tau} Z_c \|_{S} + \| e^{\mu\tau} z_s \|_{L^2_\tau} +  \| e^{\mu\tau} z_u \|_{L^2_\tau} \lesssim \epsilon. \label{eqestZfinal}
\ee
Noticing that $\tau(t) = -\frac{\ln (1-2bt)}{2b}$, $\mu = b(\sigma - s_c)$ and hence $e^{\mu\tau} = (1-2bt)^{-\frac{\sigma -s_c}{2}}$, this is the boundedness of $\e$ \eqref{eqss2}. 

Finally, we check the limit profile in the original coordinates \eqref{eqss3}. From
\[ i\pa_\tau Z + (\calH_b - ibs_c) Z = F - VZ, \]
we have
\bee  e^{-i\tau(\calH_b - ibs_c)}Z(\tau) &=& \left (Z(0) - i \int_0^\infty e^{-i\tau'(\calH_b - ibs_c)} (F - VZ)(\tau') d\tau'\right ) \\
&+&i\int_\tau^\infty e^{-i\tau'(\calH_b - ibs_c)} (F - VZ)(\tau') d\tau' \\
&=:& Z_+ + \mathfrak{r}(\tau).\eee
With \eqref{eqestZfinal}, \eqref{eqPsicontract00} and that $V: \dot W^{\sigma, r_1'} \to \dot W^{\sigma, r_1}$ from Lemma \ref{lempotbdd}, the Strichartz estimate for $e^{it\calH_b}$ \eqref{eqStrichartzfree} implies 
\bee \left\| \mathfrak{r}(\tau) \right \|_{\dot{H}^\sigma} 
\lesssim \| e^{ \mu \tau } VZ \|_{L^2_\tau ([\tau,\infty), \dot{W}^{\sigma, r_1'}_y)} + \| e^{\mu \tau} F \|_{L^{q_1'}_\tau([\tau, \infty), \dot{W}^{\sigma, r_1'}_y) }\to 0  \eee
as $\tau\to +\infty$. Thus
\[ e^{b(\sigma - s_c) \tau}Z(\tau) - e^{i\tau(\calH_b - ib\sigma)} Z_+ = e^{i\tau(\calH_b - ib\sigma)} \mathfrak{r}(\tau) \to 0,\quad \mathrm{in}\,\, (\dot{H}^\sigma)^2 \]
as $\tau \to +\infty$ since $\| e^{i\tau(\calH_b - ib\sigma)} \mathfrak{r}(\tau)  \|_{(\dot{H}^\sigma)^2}  = \| \mathfrak{r}(\tau)  \|_{(\dot{H}^\sigma)^2}$. Denote  $\e_+$ and $\e_r$ as the first component of $Z_+$ and $\mathfrak{r}$ respectively. We invoke \eqref{eqbto0} to go back to the original coordinates
\bee  \frac{1}{\l(t)^{\frac{2}{p-1}}} \e\left (t, \frac{x}{\l(t)}\right ) e^{i\tau(t)}& = &\frac{1}{\l(t)^{\frac d2}} \left(e^{i\tau(t)(\Delta_b - 1)}(\e_+ + \e_r(\tau(t)) \right)\left (\frac{x}{\l(t)}\right ) e^{i\tau(t)}\\
&=&e^{it\Delta}\e_+(x) + (e^{it\Delta}\e_r(\tau(t)))(x).
\eee
This implies \eqref{eqbdd3} with $u_* := e^{i\frac{1}{2b} \Delta}\e_+$.

\mbox{}

\underline{Step 2. Onto estimate \eqref{eqPsicontract00}-\eqref{eqPsicontract03}}.

\textit{Case 1. $1 < p \le 2$.} We first obtain a spatial nonlinear estimate from the fractional difference estimate Lemma \ref{lemfracdif}. Note that $Q_b$ is non-vanishing \eqref{eqnonvanishing} and satisfies the condition \eqref{eqslowvar} with $a = 1$. Then we have
\be 
\| F \|_{\dot{W}^{\sigma, r_1'}} \lesssim \| Z \|_{L^{pr_2^-}}^p + \| Z \|_{L^{pr_2}}^{p-1} \| Z \|_{\dot{W}^{\sigma, r_1}} \lesssim \| Z \|_{\dot{W}^{\sigma, r_1} \cap \dot{W}^{\sigma, r_1^-}}^p \label{eqFnonlinear1}
\ee
for $1 < p \le 2$.
Here $r^-_i$ with $i = 1,2$ are fixed such that $0 < r_i - r_i^- \ll 1$ and $\dot{W}^{\sigma, r_1^-} \hookrightarrow L^{pr_2^-}$. We also pick $q_1^+$ and $\beta^+$ such that $\frac{2}{q_1^+} + \frac{d}{r_1^-} = \frac d2$, $\beta^+ = \frac{1}{q_1'} - \frac{p}{q_1^+}.$

Next, we verify the nonlinear estimates in the Strichartz norms \eqref{eqPsicontract00}. Compute 
\bea
&& \| e^{ \tilde p \mu \tau} F \|_{L^{q_1'}_\tau \dot{W}^{\sigma, r_1'}_y} \lesssim \left\| e^{ \tilde p \mu \tau}\left(  \| Z \|_{\dot{W}^{\sigma, r_1}_y}^p + \|  Z \|_{\dot{W}^{\sigma, r_1^-}_y}^{p}  \right) \right\|_{L^{q_1'}_\tau}  \nonumber\\
& \lesssim & \left\| e^{ \tilde p \mu \tau}\left(  \| Z_c \|_{\dot{W}^{\sigma, r_1}_y}^p + \|  Z_c \|_{\dot{W}^{\sigma, r_1^-}_y}^{p}  + |z_u|^p + |z_s|^p \right) \right\|_{L^{q_1'}_\tau} \nonumber\\
&\lesssim &  \| e^{-(p  \mu -  \tilde p \mu)\tau} \|_{L^{q_1'}_\tau} \left( \|e^{  \mu \tau} z_s \|_{L^\infty_\tau}^p + \|e^{  \mu \tau} z_u \|_{L^\infty_\tau}^p \right) + \| e^{-(p\mu- \tilde p \mu) \tau} \|_{L^{\frac{1}{\beta}}_\tau} \| e^{\mu \tau}Z_c \|_{L^{q_1}_{\tau}\dot{W}^{\sigma, r_1}_y}^p  \nonumber\\
&+&  \| e^{-(p\mu- \tilde p \mu) \tau} \|_{L^{\frac{1}{\beta^+}}_\tau} \| e^{\mu \tau}Z_c \|_{L^{q_1^+}_{\tau}\dot{W}^{\sigma, r_1^-}_y}^{p} \nonumber\\
& \lesssim &  \| e^{\mu \tau}Z_c \|_{S}^p +  \|e^{  \mu \tau} z_s \|_{L^\infty_\tau}^p + \|e^{  \mu \tau} z_u \|_{L^\infty_\tau}^p \lesssim \| \bm{Z}\|_{\XX}^p  \nonumber \label{eqnonlinearF}
\eea 
where the second inequality follows \eqref{eqzProj} and that $r_1 \in (2, p_0)$. 

Finally, we check \eqref{eqPsicontract01}-\eqref{eqPsicontract03}. For \eqref{eqPsicontract01}, we apply the Strichartz estimate \eqref{eqStrichartz} from Theorem \ref{thmStrichartz} to see
\bee
 \| \bm{\Psi}_c(\bm{Z}) \|_{\XX_c} \lesssim \| Z_{c0} \|_{\dot{H}^\sigma_y} + \|e^{\mu\tau} F \|_{L^{q_1'}_\tau \dot{W}^{\sigma, r_1'}_y} \lesssim 
 \| Z_{c0} \|_{\dot{H}^\sigma_y} + \| \bm{Z}\|_{\XX}^p  
\eee
Then \eqref{eqPsicontract01} holds if we choose $C_0$ larger than twice the constant here. 
For \eqref{eqPsicontract02}, Note that \eqref{eqspecA} and $\delta_1 \ge 2(\sigma - s_c)$ indicate $\| e^{-A_s \tau}\|_{\RR^{d_s} \to \RR^{d_s}} \lesssim e^{-\frac{\tilde p + 1}{2}\mu \tau}$. By Duhamel's formula for $z_s$-evolution in \eqref{eqZdecouple} and \eqref{eqzProj}, we have
\bee 
|z_s(\tau)| &\le &\left |e^{-A_s \tau}z_{s0} \right | + \left |\int_0^\tau e^{-A_s(\tau-\tau')} f_s(\tau') d\tau'\right| \\
& \lesssim &e^{-\frac{\tilde p + 1}{2}\mu\tau} \left( |z_{s0}| +  \| e^{\frac{\tilde p + 1}{2}\mu\tau'} \| F(\tau') \|_{\dot{W}^{\sigma, r_1'}_y} \|_{L^1_{\tau'}([0, \tau])}  \right)  \\
&\lesssim  &e^{-\frac{\tilde p + 1}{2}\mu\tau} \left( |z_{s0}| +  \| e^{-\frac{\tilde p-1}{2}\mu\tau'}  \|_{L^{q_1}_{\tau'}}  \| e^{\tilde p \mu \tau} F \|_{L^{q_1'}_\tau \dot{W}^{\sigma, r_1'}_y} \right),
\eee
which implies \eqref{eqPsicontract02} by further applying \eqref{eqPsicontract00} and increasing $C_0$ if necessary. Similarly, the estimate of $z_u$ \eqref{eqPsicontract03} follows $\| e^{A_u \tau} \|_{\RR^{d_u} \to \RR^{d_u}} \lesssim e^{\frac{\tilde p - 1}{2} \mu \tau}$ from \eqref{eqspecA} and the nonlinear estimate \eqref{eqPsicontract00} that 
\bee
|z_u(\tau)| &\lesssim& e^{-\frac{\tilde p - 1}{2} \mu \tau} \left\|  e^{\frac{\tilde p - 1}{2} \mu \tau'} \| F(\tau') \|_{\dot{W}^{\sigma, r_1'}_y}  \right\|_{L^1_{\tau'}([\tau, \infty))} \\
&\lesssim& e^{-\frac{\tilde p - 1}{2} \mu \tau}  \| e^{-\frac{\tilde p + 1}{2} \mu \tau} \|_{L^{q_1}_{\tau'} [\tau, \infty)} \| e^{\tilde p \mu \tau'} F \|_{L^{q_1'}_\tau \dot{W}^{\sigma, r_1'}_y} \lesssim e^{-\tilde p \mu \tau} \| \bm{Z}\|_{\XX}^p.
\eee
 
 \mbox{}
 
\underline{Case 2. $p > 2$. } We only sketch the proof since the argument is similar. Like \eqref{eqFnonlinear1}, we have a similar nonlinear estimate
\be 
\| F \|_{\dot{W}^{\sigma, r_1'}} \lesssim \| Z \|_{\dot{W}^{\sigma, r_1} \cap \dot{W}^{\sigma, r_1^-}}^p + \| Z \|_{\dot{W}^{\sigma, r_1} \cap \dot{W}^{\sigma, r_1^-}}^2 \label{eqFnonlinaer2}
\ee
for $p > 2$ from Lemma \ref{lemfracdif}. To deal with the additional quadratic terms, we introduce $\beta_1 := \frac{1}{q_1'} - \frac{2}{q_1} > \frac{1}{q_1} - \frac{p}{q_1} = \beta$, and correspondingly $\beta_1^+ = \frac{1}{q_1'} - \frac{2}{q_1^+}$. 
Then noticing that $\tilde p = \frac 32$ for $p > 2$, the weighted nonlinear estimate \eqref{eqPsicontract00} becomes
\bee 
&& \| e^{ \tilde p \mu \tau} F \|_{L^{q_1'}_\tau \dot{W}^{\sigma, r_1'}_y} \\
&\lesssim &  \| e^{-(p\mu -  \tilde p \mu)\tau} \|_{L^{q_1'}_\tau} \left( \|e^{\mu \tau} z_s \|_{L^\infty_\tau}^p + \|e^{\mu \tau} z_u \|_{L^\infty_\tau}^p \right) + \| e^{-(p\mu -  \tilde p \mu) \tau} \|_{L^{\frac{1}{\beta}}_\tau} \| e^{\mu \tau}Z_c \|_{L^{q_1}_{\tau}\dot{W}^{\sigma, r_1}_y}^p  \\
&+&  \| e^{-(p\mu -  \tilde p \mu) \tau} \|_{L^{\frac{1}{\beta^+} }_\tau} \| e^{\mu \tau}Z_c \|_{L^{q_1^+}_{\tau}\dot{W}^{\sigma, r_1^-}_y}^{p} \\
&+ &  \| e^{-(2\mu -  \tilde p \mu)\tau} \|_{L^{q_1'}_\tau} \left( \|e^{\mu \tau} z_s \|_{L^\infty_\tau}^2 + \|e^{\mu \tau} z_u \|_{L^\infty_\tau}^2 \right) + \| e^{-(2\mu -  \tilde p \mu) \tau} \|_{L^{\frac{1}{\beta_1}}_\tau} \| e^{\mu \tau}Z_c \|_{L^{q_1}_{\tau}\dot{W}^{\sigma, r_1}_y}^2  \\
&+&  \| e^{-(2\mu -  \tilde p \mu) \tau} \|_{L^{\frac{1}{\beta_1^+}}_\tau} \| e^{\mu \tau}Z_c \|_{L^{q_1^+}_{\tau}\dot{W}^{\sigma, r_1^-}_y}^2 \\
& \lesssim & \| \bm{Z}\|_{\XX}^p +  \| \bm{Z}\|_{\XX}^2 \lesssim  \| \bm{Z}\|_{\XX}^2.
\eee
The other estimates \eqref{eqPsicontract01}-\eqref{eqPsicontract03} follow verbatim as above plus this estimate. 

\mbox{}

\underline{Step 3. Contraction estimate \eqref{eqPsicontract04}}.

The proof is mainly a nonlinear estimate for the difference flow of \eqref{eqZdecouple}.  
	Without loss of generality, we assume $\e \ge \tilde{\e}$. We also employ notations $\tilde F = F(\tilde Z)$ and $\triangle F = \tilde{F} - F$. 
	Since the unit ball in $L^{q_1'}_\tau([0,\infty),  (\dot{W}^{\sigma, r_1'}_y)^2)$ is closed in $L^{q_1}_\tau ([0, \infty), (L^{r_2}_y)^2)$ and $e^{\tilde{p} \delta \tau} \triangle F \in B_1^{L^{q_1'}_\tau \dot{W}^{\sigma, r_1'}_y}$ from \eqref{eqPsicontract00} and $\epsilon \le \epsilon_0 \ll 1$,  the inverse mapping theorem then implies
	\[ \left \| e^{\tilde{p} \delta \tau} \triangle F  \right \|_{L^{q_1'}_\tau  \dot{W}^{\sigma, r_1'}_y} \sim  \left \| e^{\tilde{p} \delta \tau} \triangle F  \right \|_{L^{q_1'}_\tau L^{r_2}_y}.  \]
	
	Note that Lemma \ref{lemnondif} indicates 
	\bee 
	 \| \triangle F\|_{ L^{r_2} } \lesssim \| \triangle Z \|_{ L^{pr_2} } \left( \| Z \|_{ L^{pr_2} }^{p-1} + \| \tilde{Z}\|_{ L^{pr_2} }^{p-1} \right ) 
	 \lesssim \| \triangle Z \|_{ \dot{W}^{\sigma, r_1}} \left (\| Z \|_{ \dot{W}^{\sigma, r_1}}^{p-1} + \|  \tilde{Z} \|_{ \dot{W}^{\sigma, r_1}}^{p-1} \right )
	\eee
	for $1 < p \le 2$ and
		\bee 
	\| \triangle F\|_{ L^{r_2} }
	\lesssim \| \triangle Z \|_{ \dot{W}^{\sigma, r_1}} \left (\| Z \|_{ \dot{W}^{\sigma, r_1}}^{p-1} + \|  \tilde{Z} \|_{ \dot{W}^{\sigma, r_1}}^{p-1} + \| Z \|_{ \dot{W}^{\sigma, r_1}} + \|  \tilde{Z} \|_{ \dot{W}^{\sigma, r_1}} \right )
	\eee
	for $p > 2$. Similar to Step 2, we can derive 
	\bee
	  \left \| e^{\tilde{p} \delta \tau} \triangle F  \right \|_{L^{q_1'}_\tau  \dot{W}^{\sigma, r_1'}_y} \lesssim \| \triangle \bm{Z} \|_{\XX} \left( \| \bm{Z}\|_{\XX}^{\bar p - 1} +  \|\tilde{\bm{Z}} \|_{\XX}^{\bar p - 1}\right) \lesssim \epsilon^{\bar p - 1} \| \triangle \bm{Z} \|_{\XX}
	\eee
	where $\bar{p} := \min \{ p, 2\}$, and thereafter, the linear estimates of Strichartz, $e^{-A_s \tau}$ and $e^{A_u \tau}$ imply the contraction estimate \eqref{eqPsicontract04} as in Step 1. 
\end{proof}

\subsection{Asymptotic stability in $\dot H^\sigma$}\label{sec63}

In this subsection, we will prove the nonlinear asymptotic stability in $\dot H^\sigma$ by assuming mode stability Assumption \ref{assmodestab} by a Brouwer argument for matching initial data.

We begin by computing the eigenpairs generated by symmetry. 
Define
\be \label{eqdefxi01b}
\xi_{0} = i \left(\begin{array}{c} Q_b \\ -\bar{Q}_b \end{array}\right), \,\,
\xi_{1} = \frac 12 \left(\begin{array}{c} \Lambda Q_b \\ \overline{\Lambda Q_b} \end{array}\right); \,\,
\zeta_{j} = \left(\begin{array}{c} \pa_j Q_b \\ \overline{\pa_j Q_b} \end{array}\right),\,\, 1 \le j \le d. 
\ee
where $\Lambda = \frac{2}{p-1} + x\cdot\nabla$. 
Then $\xi_0, \xi_1, \zeta_j \in X_+$ from their symmetry and the decay of $Q_b$ \eqref{selfsimilardecay}. 
By differentiating the self-similar profile equation \eqref{eqselfsimilar} with respect to phase rotation, scaling and spatial translation, we obtain the algebraic relations
\bee
    \calH_b \xi_0 = 0,\quad \calH_b \xi_1 = -2bi\xi_1 - i\xi_0;\quad
\calH_b \zeta_j = -bi \zeta_j, \quad 1 \le j \le d.
\eee
Thus
\be
\calH_b \xi_0 = 0,\quad (\calH_b + 2bi) \left( \xi_0 + 2b\xi_1 \right) = 0,\quad 
(\calH_b + bi)\zeta_j = 0. \label{eqexpliciteigenfunc}
\ee

\mbox{}

\begin{proof}[Proof of theorem \ref{thmasympstab}]
Let $\epsilon_1 \le \epsilon_0$ to be determined later, with $\epsilon_0$ from Theorem \ref{thmfincodimstabHsigmaB}. The explicit eigen relations \eqref{eqexpliciteigenfunc} and mode stability Assumption \ref{assmodestab} imply that  
\be P_u X_+ = {\rm span}_\RR(\xi_0, \xi_1, \zeta_1,..., \zeta_d).  \label{eqRieszcharX+} \ee

\mbox{}

\textit{Step 1. Reduction to matching initial data.}
  Fix $\e_0 \in B^{\dot H^\sigma}_{\epsilon_1}$. For $\vec \a = (\l_0, x_0, \theta_0)$, we define the modulated profile $w_0^{\vec \a}$ to satisfy
  \[ (Q_b + \e_0)(x) = \left(Q_b + w^{(\vec \a)}_0 \right)\left( \frac{x-x_0}{\l_0} \right) \l_0^{-\frac d2 + s_c} e^{-i\theta_0}, \]
  namely 
  \be
    w^{(\vec \a)}_0(y) := \left( Q_b + \e_0 \right) \left( \l_0 y + x_0 \right) \l_0^{\frac d2 - s_c}e^{i\theta_0} - Q_b(y).
  \ee
  We claim that for $\epsilon_1$ small enough, there exists $C_* > 0$, such that there exists $\vec \a \in B^{\RR^{d+2}}_{C_* \| \e_0\|_{\dot H^\sigma}}$ satisfying 
  \be
     \Phi \left( (1-P_u) \Pi^{-1} w^{(\vec \a)}_0 \right) = P_u \Pi^{-1} w^{(\vec \a)}_0 \label{eqwalpha0}
  \ee
  where $\Phi$ is the unstable-data map defined in Theorem \ref{thmfincodimstabHsigmaB}.
  
  Now we prove that \eqref{eqwalpha0} concludes the proof of Theorem \ref{thmasympstab}.
  Let 
  $$\tilde u_0 := Q_b + w^{(\vec \a)}_0 = Q_b + \Pi (1-P_u) \Pi^{-1} w^{(\vec \a)}_0 + \Pi \Phi \left( (1-P_u) \Pi^{-1} w^{(\vec \a)}_0 \right).$$
  Since
  $$\| \Pi (1-P_u) \Pi^{-1} w^{(\vec \a)}_0 \|_{\dot H^\sigma} \lesssim \| \e_0 \|_{\dot H^\sigma} + |\vec \a| \lesssim \epsilon_1, $$
  when $\epsilon_1$ is small enough, we have $\| \Pi (1-P_u) \Pi^{-1} w^{(\vec \a)}_0 \|_{\dot H^\sigma} \le \epsilon_0$ and hence Theorem \ref{thmfincodimstabHsigmaB} implies that $\tilde u_0$
  generates a self-similar blowup
  \[ \tilde u(t, x) = \frac{1}{(1-2bt)^{\frac{1}{p-1}}}  (Q_b + w)\left( t, \frac{x}{\sqrt{1-2bt}} \right)  e^{-i\frac{\ln(\frac{1}{2b} - t)}{2b}}, \quad w\big|_{t=0} = w^{(\vec \a)}_0, \]
  with $w$ satisfies $\| w(t) \|_{\dot H^\sigma_x} \lesssim (1-2bt)^{\frac{\sigma - s_c}{2}}$ and there exists a blowup profile $\tilde u_*$ in the sense of \eqref{eqss3}. 
  Then from the symmetry of \eqref{eqNLS}, the solution generated by $u_0 = Q_b + \e_0$ becomes 
  \[ u(t, x) = \tilde u\left(\frac{t}{\l_0^2}, \frac{x-x_0}{\l_0} \right) \l_0^{-\frac d2 + s_c} e^{-i\theta_0}, \]
  so \eqref{eqss33} holds with $\e(t) = w(\l_0^{-2} t)$ has desired decaying when $t \to T := \frac{\l_0^2}{2b}$. The vanishing of $\e(t)$ implies the uniqueness of the asymptotics decomposition \eqref{eqss33}, and hence the uniqueness of the modulation parameter $\vec \a$. The blowup profile for $u$ is $u_* = \tilde u_*\left(\frac{\cdot - x_0}{\l_0}\right) \l_0^{-\frac d2 + s_c} e^{-i\theta_0}$. 

\mbox{}

\textit{Step 2. Parameterization and Brouwer argument.}
  Now we prove the claim by solving \eqref{eqwalpha0}. 
  Let 
  \[ r_{\vec \a} := Q_b(\l_0 \cdot + x_0)\l_0^{\frac d2 - s_c} e^{i\theta_0} - Q_b,\quad v_{\vec \a} = w^{(\vec \a)}_0 - r_{\vec \a} = \e_0(\l_0 \cdot + x_0)\l_0^{\frac d2 - s_c} e^{i\theta_0}, \]
  then \eqref{eqwalpha0} becomes
  \be P_u \Pi^{-1} r_{\vec \a} = \Phi \left( (1-P_u) \Pi^{-1} w^{(\vec \a)}_0 \right) - P_u \Pi^{-1} v_{\vec \a}. \label{eqwalpha1} \ee
  Since $P_u\big|_{P_uX_+} = I\big|_{P_u X_+}$, we compute
  \[  P_u \Pi^{-1} \left(\left( \pa_{\theta_0}, \pa_{\l_0}, \pa_{(x_0)_1},...,\pa_{(x_0)_d} \right)  r_{\vec \a}\right)\Big|_{\vec \a = \vec 0} = \left( \xi_0, 2\xi_2, \zeta_1,...,\zeta_d \right).  \]
  So with \eqref{eqRieszcharX+}, we can parameterize $P_u X_+$ by
  \[ \calP: P_u X_+ \to \RR^{d+2},\quad  c_1 \xi_0+ 2c_2 \xi_1 + \sum_{j=1}^d c_{j+2} \zeta_j \mapsto \vec c,  \] 
  which is clearly well-defined by the linear independence of those modes \eqref{eqdefxi01b}. Then $\calP P_u \Pi^{-1} (\nabla_{\vec \a} r_{\vec \a}) \big|_{\vec \a = \vec 0} = {\rm Id}_{\RR^{d+2}}$ and \eqref{eqwalpha1} turns into
  \be
  \begin{split}
    \vec \a = \calT_{\e_0} (\vec \a) :&= - \calP P_u \Pi^{-1}\left(r_{\vec \a} - \vec \a \cdot (\nabla_{\vec \a} r_{\vec \a}) \big|_{\vec \a = \vec 0}\right) \\
    &+ \calP\left[ \Phi \left( (1-P_u) \Pi^{-1} w^{(\vec \a)}_0 \right) - P_u \Pi^{-1} v_{\vec \a} \right]. 
    \end{split} 
    \label{eqwalpha2}
  \ee
  From the boundedness of $P_u$, $\Pi^{-1}$, $\calP$ and \eqref{eqLipestPhi}, together with that for $|\vec \a| \le 1$,
  \[ \left\| v_{\vec \a} \right\|_{\dot{H}^\sigma} \lesssim \| \e_0 \|_{\dot H^\sigma},\quad \| r_{\vec \a} \|_{\dot{H}^\sigma} \lesssim |\vec \a|,\quad \|r_{\vec \a} - \vec \a \cdot (\nabla_{\vec \a} r_{\vec \a}) \big|_{\vec \a = \vec 0} \|_{\dot{H}^\sigma} \lesssim |\vec \a|^2, \]
  we have for $|\vec \a| \le 1$, there exists $C_* > 0$ independent of $\e_0, \vec \a$ such that
  \bee
 \left|\calT_{\e_0}(\vec \a)\right| &\le& \frac 12 C_* \left( \| \e_0 \|_{\dot H^\sigma} + |\vec \a|^{\min \{p, 2\}} \right).
  \eee
  Suppose 
  $$\epsilon_1 < (2C_*^{\min\{p, 2\}})^{\frac{1}{\min \{ p-1, 1\}}},$$
  then with $\| \e_0 \|_{\dot H^\sigma} \le \epsilon_1$, we have
  \[ \calT_{\e_0} \left(B^{\RR^{d+2}}_{C_* \| \e_0\|_{\dot H^\sigma}}\right) \subset B^{\RR^{d+2}}_{C_* \| \e_0\|_{\dot H^\sigma}}, \]
  and the Brouwer fixed point theorem implies the existence of solution $\vec \a$ to \eqref{eqwalpha2} and hence \eqref{eqwalpha0}. That concludes the proof.
\end{proof}

\section{Nonlinear stability in $H^1$}\label{s7}

\subsection{Modified spectral projections}

In this subsection, we apply the strategy from \cite[Section 2.5]{arXiv:2404.17228}. Our goal is to construct the modified spectral projections by truncating ${\rm Ran} P_{\rm disc}$. 

\begin{lemma}[Modified spectral projections] \label{lemJinv3} Under the assumption of Lemma \ref{lemJinv2} and assuming additionally \eqref{selfsimilardecayhigh} for $Q_b$, there exist linear operators $\tilde P_c, \tilde P_s, \tilde P_u \in \calL((\dot H^\sigma)^2)$ satisfying 
\begin{enumerate}
    \item $I = \tilde P_c + \tilde P_s + \tilde P_u$; for $\a, \a' \in \{ s, u, c\}$,  $\tilde P_\a \tilde P_{\a'}= \tilde P_\a \delta_{\a, \a'}$. 
    \item For $ \a \in \{ s, u \}$, $\dim ({\rm Ran} \tilde P_\a) = \dim ({\rm Ran} P_\a ) < \infty$,  and ${\rm Ran} \tilde P_\a \subset C^\infty_c(\RR^d)$. 
    \item $\tilde P_s, \tilde P_u \in \calL\left( (\dot{W}^{\sigma, p_1} + \dot{W}^{\sigma, p_0'})^2 \to H^N \right)$ for any $N \ge 0$ with $p_0$ $p_1$ as in Lemma \ref{lemJinv2}. 
    \item $\calJ$-invariance: $\calJ \circ \tilde P_\a = \tilde P_\a \circ \calJ$ for $\a \in \{ s, u, c\}$. 
    \item Composition with original spectral projections: 
    \bea
    \tilde P_\a P_\a = \tilde P_\a, && P_\a \tilde P_\a = P_\a,\quad {\rm for }\,\, \a \in \{ s, u\}; \label{eqspecprojcom1} \\
    \tilde P_c P_c = P_c, && P_c \tilde P_c = \tilde P_c.\label{eqspecprojcom2}
    \eea
    In particular, ${\rm Ran} P_c  = {\rm Ran} \tilde P_c$. 
    \item Action of $\calH$ on stable/unstable subspaces: View $\tilde P_\a \calH P_\a \in \calL({\rm Ran} \tilde P_\a)$ as matrix operators for $\a \in \{s, u \}$, then
    \be \Im\left( \mathrm{eig}(\tilde P_s \calH  P_s) \right)\ge b\delta_1, \quad \Im \left (\mathrm{eig}(\tilde P_u \calH  P_u)\right ) \le 0. \ee
\end{enumerate}
\end{lemma}
\begin{proof} Within this proof, we will write the inner product $(\cdot, \cdot) := (\cdot, \cdot)_{(\dot H^\sigma)^2}$ for simplicity. 

    \underline{Step 1. Symmetry of discrete spectrum.} In this step, we will show that for $\a \in \{ u, s\}$, there exists a set of functions $\{ \psi_{\a;j}, \varphi_{\a;j} \}_{j = 1}^{d_\a} \subset X_+ \cap (\dot W^{\sigma, p_0} \cap \dot W^{\sigma, p_1'})^2$ such that 
    \be P_\a = \sum_{j = 1}^{d_\a} (\cdot, \psi_{\a;j})  \varphi_{\a; j}, \quad (\varphi_{\a;i}, \psi_{\a;j}) = \delta_{ij}.\label{eqformproj} \ee
For simplicity, we only prove the case for $\a = u$. 

    To begin with, we recall the contour integral definition of $P_u$ \eqref{eqcontourPa} and decompose 
    \[ P_u = P_{\Im} + \sum_\pm P_{\Re \pm}  \]
    by choosing contour $\gamma_\Im \subset \CC$, $\gamma_{\Re\pm} \subset  \{z \in \CC: \pm \Re z > 0 \}$ respectively which encompass the unstable modes $\calE_u \cap i\RR$, $\calE_u \cap  \{z \in \CC: \pm \Re z > 0 \}$ respectively. Then $P_\Im, P_{\Re\pm}$ are orthogonal Riesz projections, namely 
    \be P_\beta P_{\beta'} =  P_\beta \delta_{\beta, \beta'} \quad {\rm for}\,\, \beta, \beta' \in \{\Im, \Re+, \Re- \}. \label{eqPImReproj}\ee
    Moreover, we can require the symmetry with respect to $i\RR$ that $\gamma_\Im = -\bar \gamma_\Im$ and $\gamma_{\Re \pm} = - \bar \gamma_{\Re \mp}$ as sets. Thereafter, a similar computation as in Lemma \ref{lemJinv2} leads to 
    \be \calJ P_\Im = P_\Im \calJ,\quad \calJ P_{\Re\pm} = P_{\Re \mp} \calJ. \label{eqPImResym} \ee

    Next, we claim that we can pick a basis $\{v_j\}_{j = 1}^{M_1} \subset X_+$ for ${\rm Ran} P_\Im \big|_{(\dot H^\sigma)^2}$. Indeed, from Lemma \ref{lemJinv1}, we have $(\calH - \l)^n \varphi_\l = 0$ if and only if $(\calH + \bar \l)^n \calJ \varphi_\l = 0$. Because $P_{\Im} = \sum_{\l \in \calE_u \cap i\RR} P_\l$, we see  $\varphi_\l \in {\rm Ran} P_\Im \big|_{(\dot H^\sigma)^2}$ implies $\calJ\varphi_\l \in {\rm Ran} P_\Im \big|_{(\dot H^\sigma)^2}$ and thus $\varphi_\l + \calJ \varphi_\l, i(\varphi_\l - \calJ \varphi_\l) \in {\rm Ran} P_\Im \big|_{(\dot H^\sigma)^2}$. Since the last vectors belongs to $X_+$, the choice of $\{ v_j \}$ follows.
    
    Now we also take $\{w_j \}_{j=1}^{M_2}$ to be the basis of ${\rm Ran} P_{\Re+} \big|_{(\dot H^\sigma)^2}$. Then there exists $\{ h_j\}_{j = 1}^{M_1} \subset {\rm Ran} (P_{\Im} \big|_{(\dot H^\sigma)^2})^*$, $\{ g_j\}_{j=1}^{M_2} \subset {\rm Ran} (P_{\Re+} \big|_{(\dot H^\sigma)^2})^*$ such that 
    \[ P_\Im = \sum_{j=1}^{M_1}(\cdot, h_j) v_j, \quad P_{\Re+} = \sum_{j=1}^{M_2}(\cdot, g_j) w_j.\]
    The symmetry \eqref{eqPImResym} indicates that 
    \[ P_{\Re-} = \calJ P_{\Re+}\calJ = \sum_{j = 1}^{M_2} \overline{(\calJ \cdot, g_j)} \calJ w_j = \sum_{j = 1}^{M_2} (\cdot, \calJ  g_j) \calJ w_j, \]
    and that $h_j = \calJ h_j$, namely $h_j \in X_+$. The projection property \eqref{eqPImReproj} then implies 
    \be (v_i, h_j) = (w_i, g_j) = \delta_{ij},\quad (v_i, g_j) = (w_i, h_j) = (\calJ w_i, g_j) = 0,\quad \forall i,j. \label{eqvwghortho} \ee
    Therefore we can compute
    \bee
    P_\a  &=& P_{\Im} + P_{\Re+} + P_{\Re-} = \sum_{j=1}^{M_1} (\cdot, h_j) v_j + \sum_{j = 1}^{M_2} \left[ (\cdot, g_j) w_j + (\cdot, \calJ g_j) \calJ w_j \right] \\
    &=& \sum_{j=1}^{M_1} (\cdot, h_j) v_j +  \sum_{j = 1}^{M_2} \left[ (\cdot, g_j + \calJ g_j) \frac{w_j + \calJ w_j}{2} + (\cdot, i(g_j - \calJ g_j)) i\frac{w_j -\calJ w_j}{2} \right].
    \eee
    This clearly implies the choice of $\{\psi_{u;j}, \varphi_{u;j}\}_{j=1}^{d_u} \subset X_+$ with $d_u = M_1 + 2M_2$ to ensure the representation formula in \eqref{eqformproj}, and the orthogonal conditions follow from \eqref{eqvwghortho}. Finally, the additional regularity $\psi_{u;j}, \varphi_{u;j} \in (\dot W^{\sigma, p_0} \cap \dot W^{\sigma, p_1'})^2$ was verified in Proposition \ref{prop3}. 

    \mbox{}

    \underline{Step 2. Definition and properties of modified spectral projections.} For $R > 0$ we define the matrix by 
    \be M_{\a;R} = \{ (\chi_R \varphi_{\a;i}, \psi_{\a;j}) \}_{1 \le i, j \le d_\a} \in \RR^{d_\a \times d_\a},\quad {\rm for}\,\, \a \in \{ u, s\}.   \label{eqdefMaR} \ee
    Since $\varphi_{\a; i}, \psi_{\a; j} \in (\dot H^\sigma)^2$, we have $M_{\a; R} = I_{d_\a} + o_R(1)$. So we can fix a $R_0 \gg 1$, such that $M_{u;R_0}$, $M_{s; R_0}$ are invertible, and define for $\a \in \{ u, s \}$ that
    \be
      \left(\tilde \varphi_{\a; 1}, ..., \tilde \varphi_{\a; d_\a}\right)^\top = M_{\a;R_0}^{-1} \left(\chi_{R_0}\varphi_{\a; 1}, ..., \chi_{R_0}\varphi_{\a; d_\a}\right)^\top,  \label{eqdefmodspecproj}
    \ee
    and 
    \be
      \tilde P_\a = \sum_{j = 1}^{d_\a} (\cdot, \psi_{\a;j})_{(\dot H^\sigma)^2} \tilde \varphi_{\a;j},\quad {\rm for\,\,} \a \in \{u, s\};\quad \tilde P_c = I - \tilde P_s - \tilde P_u. \label{eqdeftildePa} \ee

    From the orthogonality \eqref{eqformproj} and definition of $\tilde \varphi_{\a; j}$ \eqref{eqdefmodspecproj}, we have 
    \be (\tilde \varphi_{\a;i}, \psi_{\a;j}) = \delta_{ij}  \label{eqmodspecortho} \ee
    which implies the projection property (1). With $Q_b \in C^\infty_{loc}$ from \eqref{selfsimilardecayhigh}, standard elliptic interior regularity theory implies that ${\rm Ran} P_s \cup {\rm Ran} P_u \subset C^\infty_{loc}$. Thus ${\rm Ran} \tilde P_\a = \chi_{R_0} {\rm Ran} P_\a \subset C^\infty_c(\RR^d)$ for $\a \in \{ u, s\}$ and it concludes (2) and (3). (4) follows from $\chi_{R_0} \varphi_{\a; j} \in X_+$.

    For (5), \eqref{eqspecprojcom1} follows a direct computation using the definition \eqref{eqdefmodspecproj} and the orthogonality conditions \eqref{eqformproj}, \eqref{eqmodspecortho}, while \eqref{eqspecprojcom2} comes from \eqref{eqspecprojcom1} and $\tilde P_\a P_{\a'} = P_{\a} \tilde P_{\a'} = 0$ for $(\a, \a') = (u, s)$ or $(s, u)$. For (6), it suffices to observe that $\tilde P_\a \calH P_\a$ under the basis $\{\tilde \varphi_{\a;j}\}$ has the same matrix representation as $\calH \big|_{{\rm Ran} P_\a}$  under the basis $\{ \varphi_{\a;j}\}$ by definition. 
\end{proof}

Similar to \eqref{eqZcsu} and \eqref{eqZdecouple}, we let 
\bee
\tilde   Z_\a = \tilde P_\a Z,\quad \tilde F_\a = \tilde P_\a F,\quad \tilde X_\a = \tilde P_\a X_+,\quad {\rm for\,\,} \a \in \{ u, s, c\}.
\eee
And for $\a \in \{u, s \}$, we define the $\RR$-linear homeomorphism $\phi_\a: \tilde X_\a \to \RR^{d_\a}$ as parametrization, 
\[\tilde z_\a =\tilde \phi_\a (\tilde Z_\a),\quad \tilde z_\a  =\tilde \phi_\a (\tilde Z_\a), \quad \tilde A_\a = \phi_\a^{-1} \circ (-i \tilde P_\a \calH P_\a) \circ \phi_\a \in \RR^{d_\a \times d_\a}. \]
From Lemma \ref{lemJinv3}, they satisfy for $\a \in \{s, u \}$ and all $N \ge 0$ and $q \in [2, p_0]$ that
 \bea \Re\left( \mathrm{eig}(\tilde A_s) \right)\ge b\delta_1, && \Re \left (\mathrm{eig}(\tilde A_u)\right ) \le 0, \label{eqspecA1}\\ |\tilde f_\a| \sim_N \|\tilde F_\a \|_{H^N} \lesssim_N \| F \|_{\dot{H}^{\sigma} + \dot{W}^{\sigma, p_0'}}, &&
|\tilde z_\a| \sim_{q, N} \| \tilde Z_\a \|_{\dot W^{\sigma, q} \cap W^{1, q} \cap H^N_{\la x \ra^{N}}}. \label{eqzProj1} \eea

Now we compute the modified spectrally decoupled flow. We first compute using \eqref{eqspecprojcom1}-\eqref{eqspecprojcom2} that
\bee
  &&\tilde P_\a \calH Z = \tilde P_\a P_\a \calH Z = \tilde P_\a \calH P_\a Z =  \tilde P_\a \calH P_\a \tilde P_\a Z = \tilde P_\a \calH P_\a \tilde Z_\a,\quad \a \in \{ u, s\}, \\
  &&\tilde P_c \calH Z = \tilde P_c \calH (\tilde Z_c + \tilde Z_s + \tilde Z_u) = \calH \tilde Z_c + \tilde P_c \calH (\tilde Z_s + \tilde Z_u) 
\eee
So applying the modified projection $\tilde P_s, \tilde P_u, \tilde P_c$ to the self-similar flow \eqref{eqZ} yields 
\be
\left| \begin{array}{l}
i\pa_\tau \tilde Z_c + \calH \tilde Z_c = \tilde F_c - \tilde P_c \calH (\tilde Z_s + \tilde Z_u) \\
\frac{d}{d\tau}\tilde z_s + \tilde A_s \tilde z_s= \tilde f_s, \\  
   \frac{d}{d\tau}\tilde z_u + \tilde A_u \tilde z_u = \tilde f_u. 
\end{array}\right.  \label{eqZdecouple2}
\ee

\subsection{Strichartz estimate in $\dot H^\sigma \cap \dot H^1$}

\begin{lemma}[$\dot H^\sigma \cap \dot H^1$ Strichartz estimate]\label{lemStrichartz1}
    Under the assumption of Lemma \ref{lemJinv2} and assuming additionally \eqref{selfsimilardecayhigh} for $Q_b$, 
    the following linear estimate holds for the linear system $i\pa_t Z + \calH Z = F$ 
    \be
    \| P_{\rm ess} Z \|_{\tilde S_\sigma \cap \tilde S_1} \lesssim \| P_{\rm ess} Z (0) \|_{\dot H^\sigma \cap \dot H^1} + \| P_{\rm ess} F \|_{\tilde S_\sigma' \cap \tilde S_1'} \label{eqStrichartz2}
\ee
where 
\begin{align*}
\tilde S_\a &= \left(C^0_{e^{\mu \tau}} \dot H^\a \cap L^{q_1}_{e^{\mu \tau}} \dot W^{\a, r_1} \cap L^2_{e^{\mu \tau}} (\dot W^{\a, r_0} \cap \dot W^{\a, r_1}) \right)([0,\infty) \times \RR^d)\\
\tilde S_\a' &= (L^{q_2'}_{e^{\mu \tau}} + L^2_{e^{\mu \tau}}) \dot W^{\a, r_2'}([0,\infty) \times \RR^d)
\end{align*}
with $\mu = b(\sigma - s_c)$, $0 < r_0 - 2 \ll 1$; for $i = 1, 2$, $r_i \in (2, p_0]$ for $p_0$ from \eqref{eqp0sec6},  and $q_i \ge 2$ is determined by $\frac{2}{q_i} + \frac{d}{r_i} = \frac d2$. 
\end{lemma}

\begin{proof}
The bound for $\tilde S_\sigma$ follows from the $\dot H^\sigma$ Strichartz estiamte \eqref{eqStrichartz}, so it suffices to control $\| P_{\rm ess}Z \|_{\tilde S_1}$, which will basically follow from \eqref{eqStrichartz} plus some commutator estimates and boundedness of spectral projections. 

We also note that $\| \cdot \|_{D^{1-\sigma} \dot W^{\sigma, q}} \sim \|\cdot\|_{\dot W^{1, q}}$ for $q \in (1, \infty)$, thanks to the $L^q$ boundedness of Riesz transform $\calR_i := \pa_i D^{-1}$ and that $\sum_i^2 \calR_i = -1$. In particular $\tilde S_1 = D^{1-\sigma} \tilde S_\sigma$ and $\tilde S_1' = D^{1-\sigma} \tilde S_\sigma'$. 

\mbox{}

Firstly, we compute with \eqref{eqcomDeltab} that
\[ D^{1-\sigma} \calH = \left( \calH_b + ib(1-\sigma) + D^{1-\sigma} V D^{-(1-\sigma)} \right) D^{1-\sigma} = \left( \calH_b + ib(1-\sigma)\right) D^{1-\sigma} - [V, D^{1-\sigma}]. \]
Therefore applying $P_{\rm ess} D^{1-\sigma} P_{\rm ess}$ the original linear equation \eqref{eqZ} yields
\bee
  \left(i\pa_t + \calH + ib(1-\sigma)\right) P_{\rm ess} D^{1-\sigma}  Z_{\rm ess} = P_{\rm ess} D^{1-\sigma} F_{\rm ess} + P_{\rm ess} [V, D^{1-\sigma}] Z_{\rm ess}
\eee
where we denote $Z_{\rm ess} = P_{\rm ess} Z$, $F_{\rm ess} = P_{\rm ess} F$. With $\nu = \mu - b(1 - \sigma)$, the $\dot H^\sigma$-Strichartz estimate \eqref{eqStrichartz} implies 
\bee
  \| P_{\rm ess} D^{1-\sigma}  Z_{\rm ess} \|_{\tilde S_\sigma}  
  &\lesssim& \| P_{\rm ess} D^{1-\sigma}  Z_{\rm ess}(0) \|_{\dot H^\sigma} + \| P_{\rm ess} [V, D^{1-\sigma}] Z_{\rm ess} \|_{L^2_\tau \dot{W}^{\sigma, 2_-}_x}  \\
   &+& \| e^{\mu t}P_{\rm ess} D^{1-\sigma} F_{\rm ess}  \|_{(L^{q_1'}_\tau + L^2_\tau) (\dot{W}^{\sigma, r_1'}_x + \dot W^{\sigma, 2_-}_x)} 
\eee
Recall the boundedness of $P_{\rm dist}$ from Proposition \ref{prop3} (3), so that $P_{\rm ess} = {\rm Id} - P_{\rm dist} \in \calL( \dot W^{\sigma, 2_-}) \cap \calL(\dot H^\sigma) \cap \calL ( \dot W^{\sigma, r_1'} \to  \dot W^{\sigma, r_1'}  + \dot W^{\sigma, 2_-})$ with $0 < 2 - 2_- \ll 1$. 
Plus the commutator estimate in Lemma \ref{lempotbdd2} (iv) (where we used $|\nabla^2 V|\lesssim \la x \ra^{-4}$ from \eqref{selfsimilardecayhigh}), with $0 < 2_+ - 2 \ll 1$ and control of $\| Z_{\rm ess} \|_{L^2_\tau \dot W^{\sigma, 2_+}_x}$ by Strichartz estimate \eqref{eqStrichartz}, we obtain
\be 
  \| P_{\rm ess} D^{1-\sigma}  Z_{\rm ess} \|_{\tilde S_\sigma}  \lesssim \| Z_{\rm ess}(0) \|_{\dot H^1} + \| F_{\rm ess} \|_{\tilde S_1'} + \|Z_{\rm ess} \|_{\tilde S_\sigma}. \label{eqest1Stri}
\ee

\mbox{}

Next, we claim that 
\be P_{\rm dist} \circ \chi_R \circ D^{1-\sigma} \in \calL(\dot H^\sigma + \dot W^{\sigma, 2_+} \to \dot H^\sigma \cap \dot W^{\sigma, p_0}) \label{eqbddPdistchiR} \ee
for any $R > 0$ and $r \in [2, p_0]$. Indeed, by rewriting $P_{\rm dist} = \sum_j (\cdot, \psi_j)_{\dot H^\sigma} \varphi_j$ and using that $\dot H^\sigma + \dot W^{\sigma, 2_+} \subset L^{2_\sigma} + L^{2_{\sigma+}}$ with $2_\sigma = \frac{2d}{d-2\sigma} \in (2, \infty)$, the claim follows that $D^{1-\sigma} \chi_R D^{2\sigma}\psi_j \in L^1 \cap L^\infty$ and $\varphi_j \in \dot H^\sigma \cap \dot W^{\sigma, p_0}$. From Proposition \ref{prop3}, the regularity of $\varphi_j$ is verified and $\psi_j \in {\rm Ran} P_{\rm dist}^* \subset \dot W^{\sigma, p_1} \cap \dot W^{\sigma, p_1'}$ for some $p_1 > 2$. We will show that this implies $D^\sigma \psi_j \in C^\infty_{loc} \cap L^2$. Then $D^{2\sigma} \psi_j \in C^\infty_{loc}$, and the desired regularity of $D^{1-\sigma} \chi_R D^{2\sigma} \psi_j$ follows, which concludes \eqref{eqbddPdistchiR}. 

Suppose $\psi_j$ is an eigenfunction, namely $(\calH^* - \bar \xi)\psi_j = 0$. Since $V^\top \in C^\infty_{loc}$ and $D^{-\sigma} V^{\top} D^\sigma \in \calL(L^2)$ from \eqref{eqpotbdd3}, the elliptic regularity theory implies $D^\sigma \psi_j \in H^2_{loc}$ and then $D^{2\sigma} \psi_j \in C^\infty_{loc}$ exploiting that locality for the equation for $D^{2\sigma} \psi_j$. When $\psi_j$ is an generalized eigenfunction $(\calH^* - \bar \xi)^N\psi_j = 0$, the smoothness can be similarly shown by induction on $N$. 

\mbox{}

Now we are in place to bound the unstable part of $D^{1-\sigma} Z_{\rm ess}$. Lemma \eqref{lempotbdd} indicates that $\| (1-\chi_R) \|_{\dot W^{\sigma, (\frac 12 + \frac \sigma d - \frac 1{q})^{-1}}}\lesssim \| (1-\chi_R) \|_{\dot W^{1, (\frac 1d + \frac 12 - \frac 1{q})^{-1}}} \lesssim R^{-d(\frac 12 - \frac 1q)}$ for $2 < q < \min \{ \frac{2d}{d-2\a}, \frac d\sigma \}$, so with  the regularity of $\varphi_j, \psi_j$ above, we have  
\bee 
&&\| P_{\rm dist} \circ (1- \chi_R) \|_{\dot H^\sigma + \dot W^{\sigma, 2_+} \to \dot H^\sigma \cap \dot W^{\sigma, p_0} } \\
&\lesssim& \| D^\sigma (1-\chi_R) D^{-\sigma} \|_{L^2 \to L^{(2_+)'}}+ \| D^\sigma (1-\chi_R) D^{-\sigma} \|_{L^{2_+} \to L^2 }  = o_R(1).
\eee
We compute using this and \eqref{eqbddPdistchiR} 
\bea
   && \| P_{\rm dist} D^{1-\sigma}  Z_{\rm ess} \|_{\tilde S_\sigma} \lesssim \| P_{\rm dist} \chi_R D^{1-\sigma}  Z_{\rm ess} \|_{\tilde S_\sigma} + \| P_{\rm dist} (1-\chi_R) D^{1-\sigma}  Z_{\rm ess} \|_{\tilde S_\sigma} \nonumber \\
   &\le& C_R \| e^{\mu \tau}  Z_{\rm ess} \|_{C^0 \dot H^\sigma \cap (L^{q_1} \cap L^2) \dot W^{\sigma, 2_+}} + o_R(1) \| e^{\mu \tau} D^{1-\sigma} Z_{\rm ess} \|_{C^0 \dot H^\sigma \cap (L^{q_1} \cap L^2) \dot W^{\sigma, 2_+}}\nonumber \\
   &\lesssim& C_R \| Z_{\rm ess}\|_{\tilde S_\sigma} + o_R(1) \| Z_{\rm ess} \|_{\tilde S_1} \label{eqest2Stri}
\eea
Here we requires $r_0 \le 2_+$ so that $C^0_{e^\mu \tau} \dot H^\sigma \cap (L^{q_1}_{e^\mu \tau} \cap L^2_{e^\mu \tau}) \dot W^{\sigma, 2_+} \hookrightarrow \tilde S_\sigma$.

\mbox{}

Finally, since $\| Z_{\rm ess} \|_{\tilde S_1} \sim \| D^{1-\sigma} Z_{\rm ess}\|_{\tilde S_\sigma} \lesssim  \| P_{\rm ess} D^{1-\sigma}  Z_{\rm ess} \|_{\tilde S_\sigma} +  \| P_{\rm dist} D^{1-\sigma}  Z_{\rm ess} \|_{\tilde S_\sigma}$, we can conclude \eqref{eqStrichartz2} by summing up \eqref{eqest1Stri} and \eqref{eqest2Stri}, taking $R \gg 1$ to absorb $o_R(1)\| Z_{\rm ess} \|_{\tilde S_1}$ term to the left and applying the $\dot H^\sigma$ Strichartz \eqref{eqStrichartz}.

\end{proof}

\subsection{Proof of nonlinear stability in $H^1$} \label{sec73}

We begin by proving the counterpart of Theorem \ref{thmfincodimstabHsigmaB}. Here we denote 
\be
  \tilde X_+:= X_+ \cap (\dot H^1)^2 \subset (\dot H^\sigma \cap \dot H^1)^2.
\ee
and will still use the natural isometry $\Pi: \tilde X_+ \to \dot H^\sigma \cap \dot H^1$ from \eqref{eqdefPi}. 

\begin{theorem}[$\dot H^\sigma \cap \dot H^1$ finite-codimensional asymptotic stability] \label{thmfincodimstabH1B}
    Let $d \ge 1$, $s_c \in (0, 1)$, and $b, Q_b$ satisfying  \eqref{eqselfsimilar}, \eqref{eqnonvanishing}, \eqref{selfsimilardecay}, \eqref{selfsimilardecayhigh}. 
  Then for every $0 < \sigma - s_c \ll 1$, the following statements hold true.
  \begin{enumerate}
      \item Finite codimensional stability: There exists $\epsilon_0 \ll 1$ and a Lipschitz map
      \[ \tilde \Phi: B_{\epsilon_0}^{\tilde X_+} \cap \tilde P_{cs} \tilde X_+ \to \tilde P_u \tilde X_+  \]
      such that for all $\tilde \e_{cs} \in B_{\epsilon_0}^{\dot H^\sigma \cap \dot H^1} \cap \Pi \tilde P_{cs} \tilde X_+$, the initial data
      \[ u_0 = Q_b +\tilde \e_{cs} + \Pi \tilde \Phi(\Pi^{-1}\tilde \e_{cs}) \]
      generates a solution of \eqref{eqNLS} blowing up at $T = \frac{1}{2b}$ satisfying  
      \bea 
      u(t, x) &=& \frac{1}{\l(t)^{\frac{2}{p-1}}}(Q_b + \e) \left(t, \frac{x}{\l(t)}\right)e^{i\tau(t)}, \label{eqss11}\\
        \| \e(t) \|_{\dot H^\sigma_x \cap \dot H^1_x} &\lesssim& (1-2bt)^{\frac{\sigma - s_c}{2}},\label{eqss21}
           \eea
           where $(\lambda(t),\gamma(t),\tau(t))$ are given by \eqref{eqlaw},
           and there exists $u_* \in \dot H^\sigma \cap \dot H^1$ such that 
        \be u(t) - \frac{1}{\l(t)^{\frac{2}{p-1}}}Q_b \left( \frac{x}{\l(t)}\right)e^{i\tau(t)}  \to u_*\quad \mathrm{in}\,\,\dot{H}^\sigma \cap \dot H^1 \quad as\quad t\to T.
		\label{eqss31}\ee
   
      Here $\tilde P_u$ and $\tilde P_{cs} := \tilde P_c + \tilde P_s$ are from Lemma \ref{lemJinv3}. 
    \item Contraction of the unstable-data map:  $\tilde \Phi$ satisfies $\tilde \Phi(0) = 0$ and for any $N \ge 1$,
      \be \| \tilde\Phi(\e_1) -\tilde \Phi(\e_2) \|_{H^N_{\la x \ra^N}} \lesssim_N \left( \| \e_1\|_{\tilde X_+} + \| \e_2 \|_{\tilde X_+} \right)^{\min \{ p-1, 1\}} \| \e_1 - \e_2 \|_{\tilde X_+}.\label{eqLipestPhi1}  \ee
  \end{enumerate}
\end{theorem}
\begin{proof}
    The proof is almost identical with that of Theorem \ref{thmfincodimstabHsigmaB}. We only sketch the proof and focus on the new elements.

    We use the same $\delta_0, p_0$ from \eqref{eqp0sec6}, the same requirement of $\delta_1, \sigma$ and the same exponents $\mu, r_1, r_2, q_1, \beta, \tilde p$ as in \eqref{eqdeltachoice}-\eqref{eqbeta}. Now we define the linear map $\tilde{\bm{\Psi}}(\tilde{\bm{Z}}) = (\tilde{\bm{\Psi}}_c(\tilde{\bm{Z}}), \tilde{\bm{\Psi}}_s(\tilde{\bm{Z}}), \tilde{\bm{\Psi}}_u(\tilde{\bm{Z}}))$ for the modified flow \eqref{eqZdecouple2} as
\be
  \left( \begin{array}{c} \tilde{\bm{\Psi}}_c(\tilde{\bm{Z}}) \\ \tilde{\bm{\Psi}}_s(\tilde{\bm{Z}}) \\ \tilde{\bm{\Psi}}_u(\tilde{\bm{Z}}) \end{array} \right) = \left( \begin{array}{c}
 e^{it\calH} \tilde Z_{c0} - i \int_0^t e^{i(t-\tau)\calH} \left(\tilde F_c - \tilde P_c \calH (\tilde Z_s + \tilde Z_u)  \right) d\tau \\ 
 e^{-t\tilde A_s} \tilde z_{s0} + \int_0^t e^{-(t-\tau)\tilde A_s}\tilde f_s d\tau \\
- \int_t^\infty e^{(\tau-t)\tilde A_u}\tilde f_u d\tau
  \end{array} \right)
\ee
The Banach space we use here is
\bee
 \tilde \XX(A_0) = \tilde S \times A_0 \XX_s \times A_0 \XX_u,\quad 
 \| \tilde{\bm{Z}} \|_{\tilde \XX(A_0)} = \max \{ \| \tilde Z_c \|_{\tilde S}, A_0 \| \tilde z_s \|_{\XX_s}, A_0 \| \tilde z_u \|_{\XX_u} \},
\eee
where $A_0 \gg 1$ is a constant chosen later, $\XX_\a = C^0_{e^{\frac{\tilde p + 1}{2}\mu t}}\left([0, \infty), \RR^{d_\a} \right)$ for $\a \in \{ s, u \}$ as before and the Strichartz norm $\tilde S = S \cap D^{1-\sigma} S$ with $S$ from \eqref{eqdefStrinorm}.  We remark that this is controlled by the Strichartz norm from Lemma \ref{lemStrichartz1}. The core proof of Theorem \ref{thmfincodimstabH1B} will again be reduced to proving contraction mapping of $\tilde{\bm{\Phi}}$ in $B^{\tilde \XX(A_0)}_{C_0\epsilon}$ with $\epsilon_0 \ll 1$, $C_0 \gg 1$ and $A_0 \gg 1$, namely the estimates similar to \eqref{eqPsicontract00}-\eqref{eqPsicontract04}. 

For the nonlinear estimate, the $L^{q_1'}_{e^{\mu t}} \dot W^{\sigma, r_1'}$ control namely \eqref{eqPsicontract00} still holds, and the $L^{q_1'}_{e^{\mu t}} \dot W^{1, r_1'}$ control follows from Lemma \ref{lemnonestderiv} that 
\bee
  \| F\|_{\dot W^{1, r_1'}} \lesssim \left| \begin{array}{ll}
      \| Z \|_{L^{pr_2^-}} + \| Z\|_{L^{pr_2}} \| Z \|_{\dot W^{1, r_1}} \lesssim \| Z \|^p_{\dot W^{1, r_1} \cap \dot W^{1, r_1^-}}  & 1 < p \le 2, \\
      \| Z \|^p_{\dot W^{1, r_1} \cap \dot W^{1, r_1^-}} + \| Z \|^2_{\dot W^{1, r_1} \cap \dot W^{1, r_1^-}}   & p > 2.
  \end{array}\right.
\eee
Then similar computations in the time-variable with \eqref{eqzProj1} implies for any $A_0 \ge 1$
\[ \| e^{\tilde p \mu \tau} F(Z)\|_{L^{q_1'} (\dot W^{\sigma, r_1'} \cap \dot W^{1, r_1'})} \lesssim_{A_0}  \| \bm{Z}\|_{\tilde \XX(A_0)}^{\min \{ p, 2 \}}. \]

The onto estimates of $\tilde z_s$, $\tilde z_u$ hold in the same way as  \eqref{eqPsicontract02}-\eqref{eqPsicontract03}, thanks to the same linear estimate of  $e^{-t\tilde A_s}$, $e^{t\tilde A_u}$ from Lemma \ref{lemJinv3} (6). For the onto estimate of $\tilde Z_c$, we first treat the additional source term $\tilde P_c \calH (\tilde Z_s + \tilde Z_u)$ by exploiting its high regularity \eqref{eqzProj1} to see 
\[ \| \tilde P_c \calH (\tilde Z_s + \tilde Z_u) \|_{\dot W^{\sigma, r_1'} \cap \dot W^{1, r_1'}} \lesssim |\tilde z_s| + |\tilde z_u|, \] 
which implies 
\be \|e^{\mu \tau}   \tilde P_c \calH (\tilde Z_s + \tilde Z_u) \|_{L^{q_1}(\dot W^{\sigma, r_1'} \cap \dot W^{1, r_1'})} \lesssim \|\tilde z_s\|_{\XX_s} + \|\tilde z_u\|_{\XX_u} \le 2A_0^{-1} \| \tilde{\bm{Z}} \|_{\tilde \XX(A_0)} \label{eqH1ZsutoZc} \ee
 thanks to higher exponential decay in $\XX_s$ and $\XX_u$. Therefore, thanks to $\tilde Z_c = P_c \tilde Z_c$ by Lemma \ref{lemJinv3} (5), Lemma \ref{lemStrichartz1} provides the counterpart of the linear Strichartz estimate of $e^{it\calH}$ in $\dot H^\sigma \cap \dot H^1$, implying that 
\[  \| \tilde{\bm{\Psi}}_c(\tilde{\bm{Z}}) \|_{\tilde S} \lesssim \| Z_{c0} \|_{\dot{H}^\sigma \cap \dot H^1} + C_{A_0} \| \bm{Z}\|_{\tilde \XX(A_0)}^{\min\{p, 2\}} + A_0^{-1}  \| \bm{Z}\|_{\tilde \XX(A_0)}. \]
This requires us to choose $A_0$ large enough to ensure smallness of $ \tilde z_s$, $\tilde z_u$ relatively to $\tilde Z_c$ (and then letting $\epsilon_0 \ll 1$ to absorb the $C_{A_0}$ constant by smallness of nonlinearity), so that this concludes the onto estimate for contraction mapping. 

Finally, for the contraction estimate, with the inverse mapping theorem argument, we can apply the same nonlinear estimate and \eqref{eqH1ZsutoZc} for additional term to conclude 
\bea
    \| \tilde{\bm{\Psi}}_\a(\tilde{\bm{Z}}) - \tilde{\bm{\Psi}}_\a(\tilde{\bm{Z}}') \|_{\XX_\a} &\lesssim_{A_0}& \epsilon^{\min \{p-1, 1 \}} \| \tilde{\bm{Z}} - \tilde{\bm{Z}}' \|_{\tilde \XX(A_0)}\quad  {\rm for}\,\, \a \in \{ u, s\};  \label{eqlipschitzsu1}\\
    \| \tilde{\bm{\Psi}}_c(\tilde{\bm{Z}}) - \tilde{\bm{\Psi}}_c(\tilde{\bm{Z}}') \|_{\tilde S} &\lesssim& \left( C(A_0)\epsilon^{\min \{p-1, 1 \}} + A_0^{-1} \right) \| \tilde{\bm{Z}} - \tilde{\bm{Z}}' \|_{\tilde \XX(A_0)}.\nonumber
\eea
Thus with $A_0 \ll 1$ and $\epsilon_0 \ll 1$, this is the contraction estimate and concludes that $\tilde{\bm{\Phi}}$ is contraction mapping. Following the argument in Step 1 as Theorem \ref{thmfincodimstabHsigmaB}, this implies all statements of Theorem \ref{thmfincodimstabH1B}. We only stress that the Lipschitz estimate \eqref{eqLipestPhi1} only depends on the difference estimate \eqref{eqlipschitzsu1} with $\a = u$, so the smallness is not affected by the new error in $\tilde Z_c$ evolution. The norm on LHS of \eqref{eqLipestPhi1} comes from \eqref{eqzProj1}. 
\end{proof}

Finally, we conclude the proof of Theorem \ref{thmasympstabH1}.

\begin{proof}[Proof of Theorem \ref{thmasympstabH1}] 
This theorem is mainly an application of Theorem \ref{thmasympstabH1}. So we fix a $0 < \sigma - s_c \ll 1$ throughout the proof. 

To begin with, we construct the first $H^1$ blowup. Denote $h_{R_1}$ for $R_1 \ge 1$ as 
\[ h_{R_1} = \tilde P_{cs} \Pi^{-1} \left((\chi_{R_1} - 1) Q_b \right) , \]
then $\| h_{R_1} \|_{\tilde X_+} = o_{R_1 \to \infty} (1)$. Then when $R_1 \gg 1$, we have $ h_{R_1} \in B^{\tilde X_+}_{\epsilon_0} \cap \tilde P_{cs} \tilde X_+$, so we define $\e^*_{R_1}$ as
\[ \e^*_{R_1} = Q_b (1 - \chi_{R_1}) +  \Pi h_{R_1} + \Pi \tilde \Phi (h_{R_1}) = \Pi \tilde P_u \Pi^{-1} \left((1 - \chi_{R_1}) Q_b \right) + \Pi \tilde \Phi (h_{R_1}). \]
The boundedness of $\tilde P_u$ (Lemma \ref{lemJinv3} (3)) and $\tilde \Phi$ \eqref{eqLipestPhi1} implies that 
\be \| \e^*_{R_1} \|_{H^3_{\la x \ra^3}} = o_{R_1 \to \infty} (1), \label{eqbdde*R1} \ee
and $\e^*_{R_1} \in C^\infty_c$ from Lemma \ref{lemJinv3} (2). Hence the initial data
\be Q_{b;R_1} := Q_b \chi_{R_1} + \e^*_{R_1} = Q_b + \Pi h_{R_1} + \Pi \tilde \Phi(h_{R_1}) \in C^\infty_c \label{eqdefQbR1} \ee
generates a blowup solution of \eqref{eqNLS} satisfying \eqref{eqss11}-\eqref{eqss31}.

We also define 
\be \tilde \Phi_{R_1} (h) = \tilde \Phi(h+h_{R_1}) - \tilde \Phi(h_{R_1}), \label{eqdefPhiR1} \ee
which is also Lipschitz with constant $\epsilon_1 + o_{R_1}(1)$ from $B^{\tilde X_+}_{\epsilon_1}$ to $H^1$ due to \eqref{eqLipestPhi1}.

\mbox{}

(i) For $R_1 \gg 1$ and $\epsilon_1 \ll 1$ such that $\| h_{R_1}\|_{\tilde X_+} + \epsilon_1 < \epsilon_0$, we define the manifold as 
\bea
  \calM &=& \left\{ h \in B^{\tilde X_+}_{\epsilon_1} \cap \tilde P_{cs} \tilde X_+ : Q_b + \Pi (h+h_{R_1}) + \Pi \tilde\Phi (h+h_{R_1}) \right\} \label{eqcalMdef1}\\
  &=&  \left\{ h \in B^{\tilde X_+}_{\epsilon_1} \cap \tilde P_{cs} \tilde X_+ : Q_{b;R_1} + \Pi ({\rm Id} + \tilde \Phi_{R_1}) h  \right\}  \label{eqcalMdef2}
\eea
 Therefore, \eqref{eqcalMdef2} indicates $\calM$ being a Lipschitz finite-codimensional manifold in $H^1$, while \eqref{eqcalMdef1} implies that it generates solutions satisfying \eqref{eqss11}-\eqref{eqss31}. To conclude Theorem \ref{thmasympstabH1} (i), it remains to check the limit profile $u^*$ \eqref{eqL2profH1} and compute the critical norm \eqref{eqcriticalnorm}. 

Recall that $p_c = \frac{d(p-1)}{2}$ such that $\dot{H}^{s_c} \hookrightarrow L^{p_c}$. Denote $\tilde{Q_b}$ and $\tilde{\e}$ to be the rescaled functions 
\[ u = e^{i\tau(t)}\l(t)^{-\frac{2}{p-1}} (Q_b + \e)\left( t, \frac{\cdot}{\l(t)}\right) =: \tilde{Q}_b + \tilde{\e}.\]
Then \eqref{eqcriticalnorm} is a consequence of the following four estimates
\bea
  \| \tilde \e (t) \chi_R \|_{\dot{H}^{s_c}} &\lesssim& 1,\label{eqpcext1}\\
     \| \tilde Q_b (t) \chi_R \|_{L^{p_c}} &\sim& |\log(T-t)|^{\frac{1}{p_c}},\label{eqpcext3}\\
  \| \tilde Q_b (t) \chi_R \|_{\dot{H}^{s_c}} &\sim& |\log(T-t)|^{\frac 12},\label{eqpcext2}\\
  \| u(t) (1-\chi_R) \|_{\dot{H}^{s_c}} &\lesssim& 1 \label{eqpcext}
\eea
where the constant is independent of $t$ as $t \to T$. We will prove them and particularly show \eqref{eqL2profH1} during the proof of \eqref{eqpcext}. 

\textit{Proof of  \eqref{eqpcext1}-\eqref{eqpcext2}.} For \eqref{eqpcext1}, we need the following estimate
\be \| f\chi_M \|_{\dot{H}^{s_c}} \lesssim M^{\sigma - s_c} \| f \|_{\dot{H}^\sigma}. \label{eqpoincare}\ee
for any $M > 0$, $0 \le s_c \le \sigma < \min \{\frac d2, 1 \}$. It can be easily proven through high-low decomposition of $f$. 
Applying this inequality and \eqref{eqbdd1}, we obtain
\bee \| \tilde{\e}(t) \chi_R \|_{\dot{H}^{s_c}} = \| \e(t) \chi_{R\l(t)^{-1}} \|_{\dot{H}^{s_c}} \lesssim  \| \e(t) \|_{\dot{H}^{\sigma}} (R\l(t)^{-1})^{\sigma - s_c} \lesssim 1 \eee

Next, recall from  \eqref{selfsimilardecay} that
\be |Q_b|\sim \la x \ra^{-\frac{2}{p-1}},\quad |Q_b'| \lesssim \la x \ra^{-\frac{p+1}{p-1}} \label{eqQbasympp}\ee
and for a large enough $R^* \gg 1$,
\be | Q_b(x) - Q_b(x')| \sim \la x \ra^{-\frac{2}{p-1}}\qquad \forall |x| \ge R^*,\,\, |x'| \ge 2 |x|.  \label{eqQbasympp2}\ee
Hence the $L^{p_c}$ estimate \eqref{eqpcext3} follows immediately
\[ \| \tilde{Q}_b \chi_R \|_{L^{p_c}} \sim \left(\int_{B_{2R\l(t)^{-1}}} \la x \ra^{-\frac{2}{p-1}p_c} dx\right )^{\frac{1}{p_c}} \sim |\log(T-t)|^{\frac{1}{p_c}}.\]
Using
\bee\| \tilde Q_b (t) \chi_R \|_{\dot{H}^{s_c}}^2 =  \| Q_b \chi_{\tilde R} \|_{\dot{H}^{s_c}}^2 &\sim& \iint_{\RR^{2d}} \frac{\left| (Q_b \chi_{\tilde R})(x-y) - (Q_b \chi_{\tilde R})(x)\right|^2}{|y|^{d+2s_c}}dxdy \\
&=:& \iint F(x,y)dxdy
\eee
where $\tilde R := R\l(t)^{-1}$.
The $\dot{H}^{s_c}$ estimate comes from a careful partition of $\RR^{2d}$. Indeed, let $\tilde R =: 2^{K}$, with $K \gg 1$, we define $K^*$ the smallest integer greater than $K$. Define the disjoint partition $\RR^{2d} = \cup_{i = 1}^5 I_i$ with
\bee
 I_1 &=& \{ |x| \ge 2^{K^* + 1} \}, \\
 I_2 &=& \{ |x| < 2^{K^* + 1},  |y| \ge 2^{K^* + 1} \},\\
 I_3 &=& \{ |x| < 2^{K^* + 1}, |y| \in [2^{K^* - 2}, 2^{K^* + 1}) \}\\
 I_4 &=& \{ |x| < 2^{K^* + 1}, |y| \in [2, 2^{K^* -2}) \}\\
 I_5 &=& \{ |x| < 2^{K^* + 1},  |y| < 4 \},
\eee
and further decompose $I_4$ with $k \in [2, K^*-3]$ by $I_4 = \cup_{k = 2}^{K^* -3} \cup_{l= 1}^4 I_{4,k,l}$
\bee
   I_{4, k, 1} &=& I_4 \cap  \{|x| < 2^{k-2}  \}, \\
   I_{4, k, 2} &=& I_4 \cap \{|x| \ge 2^{k-2}, |x-y| < 2^{k-2}  \}, \\
   I_{4, k, 3} &=& I_4 \cap \{ |x| \in [2^{k-2}, 2^{k+2}), |x-y| \ge 2^{k-2} \}, \\
   I_{4, k, 4} &=& I_4 \cap  \{|x| \ge 2^{k+2} \}.
\eee
Brute force computations with the support of $Q_b \chi_{\tilde R}$ and \eqref{eqQbasympp} yield
\bee \iint_{I_j} F(x, y) dxdy \lesssim 1,&& j = 1, 2, 3, 5;\\
\iint_{I_{4, k, l}} F(x, y) dxdy \lesssim 1,&& l = 1, 2, 3, 4; \quad k \in [2, K^* - 3]. \eee
And via \eqref{eqQbasympp2}, one can check the lower bound
\[ \iint_{I_{4, k, l}} F(x, y) dxdy \sim 1,\qquad l = 1, 2; \quad k \in [\log R^* + 5, K^* - 3]. \]
Thus by summing them up and using $K^* \sim \log \tilde{R} \sim |\log(T-t)|$, we obtain \eqref{eqpcext2}. 

\textit{Proof of \eqref{eqpcext} and \eqref{eqL2profH1}.} As a preparation, we first consider the subcritical $\dot{H}^{\tilde{\sigma}}$ convergence with $0 \le \tilde{\sigma} < s_c$. Consider the Strichartz pair 
\[ r_{\tilde{\sigma}} = \frac{d(p+1)}{d + \tilde{\sigma}(p-1)},\quad \gamma_{\tilde{\sigma}} = \frac{4(p+1)}{(p-1)(d-2\tilde{\sigma})},\quad \frac{2}{\gamma_{\tilde{\sigma}}} + \frac{d}{r_{\tilde{\sigma}}} = \frac{d}{2} \]
And let $\frac{1}{r_{\sigma}} - \frac{1}{r_{\tilde{\sigma}}} = \frac{\sigma -\tilde{\sigma}}{d}$
such that $\dot{W}^{\sigma, r_\sigma} \hookrightarrow \dot{W}^{\tilde{\sigma}, r_{\tilde{\sigma}}}$. Then by fractional Leibniz rule,
\[ \| D^{\tilde{\sigma}}(u |u|^{p-1}) \|_{L^{r_{\tilde{\sigma}}'}} \lesssim \| D^{\tilde{\sigma}} u \|_{L^{r_{\tilde{\sigma}}}}^p \lesssim \| D^{\sigma} u \|_{L^{r_\sigma}}^p. \]
From Duhamel formula and Strichartz estimate of $e^{it\Delta}$, for $0 \le  t_1 \le t_2 < T$
\bee
 \| u(t_1) - u(t_2) \|_{\dot{H}^{\tilde{\sigma}}} \lesssim \| e^{it_1\Delta} u_0 -e^{it_2\Delta} u_0 \|_{\dot{H}^{\tilde{\sigma}}} + \| D^{\sigma} u\|_{L^{p\gamma_{\tilde{\sigma}}'}_t L^{r_\sigma}_x [t_1, T)}^p 
\eee
Note that 
\[ \| e^{it_1\Delta} u_0 -e^{i(t_2 - t_0)\Delta} u_0 \|_{\dot{H}^{\tilde{\sigma}}}  = \| (e^{i(t_2 - t_1)\Delta} - I)u(t_0) \|_{\dot{H}^{\tilde{\sigma}}} = o_{|t_2 - t_1| \to 0}(1)  \]
 and 
\bee
 \| D^{\sigma} u\|_{L^{p\gamma_{\tilde{\sigma}}'}_t L^{r_\sigma}_x [t_1, T)}^p  \lesssim \left ( \int_{t_1}^T \l(t)^{-2+\gamma_{\tilde{\sigma}}' (s_c - \tilde{\sigma})}  (\| Q_b \|_{\dot{W}^{\sigma, r_\sigma}}^{p \gamma_{\tilde{\sigma}}'} + \| \e(t)\|_{\dot{W}^{\sigma, r_\sigma}}^{p\gamma_{\tilde{\sigma}}'}) dt \right )^{\frac{1}{\gamma_{\tilde{\sigma}}'}}  \\
 \lesssim (T-t_1)^{s_c - \tilde{\sigma}} (\| Q_b \|_{\dot{W}^{\sigma, r_\sigma}}^{p} + \| \e\|_{L^{p \gamma_{\tilde{\sigma}}' }_\tau \dot{W}^{\sigma, r_\sigma}_y}^p)  = o_{t_1 \to T}(1)
 \eee
using $Q_b \in \dot{W}^{\sigma, r_\sigma}$ and the boundedness of  $\dot H^\sigma$-Strichartz norm for $\e$ with $\frac{2}{p\gamma_{\tilde{\sigma}}'} + \frac{d}{r_\sigma} > \frac{d}{2}$ an admissible pair. We know $\{u(t)\}$ is a Cauchy sequence in $\dot{H}^{\tilde{\sigma}}$ as $t \to T$, yielding the existence of limit profile $u^* \in H^{\tilde \sigma}$ and the convergence \eqref{eqL2profH1}. In particular, $\| u(t) \|_{L^\infty_t \dot{H}^{\tilde{\sigma}} [0, T]} < \infty$.

By interpolation between $\dot{H}^{\tilde{\sigma}}$ and $\dot{H}^{\sigma}$, \eqref{eqpcext} is reduced to bound $\| u(t)(1-\chi_R) \|_{\dot{H}^\sigma}$. Decompose by
\[ D^{\sigma} (u(t) (1-\chi_R)) = (D^{\sigma} \tilde{Q}_b) (1-\chi_R) - [D^\sigma, \chi_R] \tilde{Q}_b + D^{\sigma} (\tilde{\e} (1-\chi_R)).  \]
The last term involving $\e$ is bounded by $O(1)$ from \eqref{eqbdd1}. The $O(1)$ bound of the first term follows a pointwise decay of $D^\sigma Q_b$
\[ |D^\sigma Q_b (x)| = C_{\sigma, d}  \left | \lim_{\epsilon \to 0} \int_{\{ |y| \ge \epsilon\}}\frac{Q_b (x-y) - Q_b(x)}{|y|^{d+\sigma}}dy \right| \lesssim_\sigma \la x \ra^{-\frac{2}{p-1}-\sigma}  \]
which is proven by splitting the integral in two and applying pointwise bounds of $Q_b$ and $\nabla Q_b$ \eqref{selfsimilardecay} respectively. Finally, by the self-similar scaling and decay, we can uniformly bound 
$$\sup_{t \in [0, T)}|\tilde{Q}_b (t, x)| \lesssim |x|^{-\frac{2}{p-1}}.$$
Thus the commutator has a pointwise bound
\bee  |[D^\sigma, \chi_R] \tilde{Q}_b| &\lesssim& \left |  \int \frac{\chi_R (x-y) - \chi_R(x)}{|y|^{d+\sigma}} \tilde{Q}_b(x-y)  dy \right | \\
&\lesssim &\mathbbm{1}_{B_{3R}} \left (|\cdot|^{-d-\sigma+1}\mathbbm{1}_{B_{\frac{R}{2}}} + |\cdot|^{-d-\sigma}\mathbbm{1}_{B_{\frac{R}{2}}^c} \right ) * |\cdot|^{-\frac{2}{p-1}} \\
&& + \left (|\cdot|^{-d-\sigma}\mathbbm{1}_{B_{\frac{R}{2}}^c}   \right ) * (|\cdot|^{-\frac{2}{p-1}} \mathbbm{1}_{B_{2R}} ) \\
&\lesssim_R & |x|^{-\frac{2}{p-1}} \mathbbm{1}_{B_{3R}} + \left (|\cdot|^{-d-\sigma}\mathbbm{1}_{B_{\frac{R}{2}}^c}   \right ) * (|\cdot|^{-\frac{2}{p-1}} \mathbbm{1}_{B_{2R}} ) 
 \eee
by again splitting the integral. This implies $\|  [D^\sigma, \chi_R] \tilde{Q}_b \|_{L^2} \lesssim 1$ from H\"older and Young's inequality.
This yields \eqref{eqpcext} and hence concludes \eqref{eqcriticalnorm}.

\mbox{}

(ii) Similar to the proof of Theorem \ref{thmasympstab}, we reduce the proof to matching initial data in order to apply Theorem \ref{thmfincodimstabH1B}, and solve the finite dimensional equation through Brouwer argument. 

For fixed $\e_0 \in B^{H^1}_{\epsilon_1}$, we define the $w_0^{\vec \a}$ from 
\[ (Q_{b;R_1}  + \e_0)(x) = \left(Q_{b;R_1} +  w^{(\vec \a)}_0 \right)\left( \frac{x-x_0}{\l_0} \right) \l_0^{-\frac d2 + s_c} e^{-i\theta_0}, \]
and the statement in (ii) boils down to finding $\vec \a \in B^{\RR^{d+2}}_{C_* \| \e_0 \|_{H^1}}$ such that 
\be   \tilde\Phi_{R_1} \left( (1-\tilde P_u) \Pi^{-1} w^{(\vec \a)}_0 \right) = \tilde P_u \Pi^{-1} w^{(\vec \a)}_0  \label{eqwalpha5} \ee
where $\tilde \Phi_{R_1}$ is defined as \eqref{eqdefPhiR1}. 

To solve \eqref{eqwalpha5}, we again take $r_{\vec \a; R_1} =Q_{b;R_1}(\l_0 \cdot + x_0)\l_0^{\frac d2 - s_c} e^{i\theta_0} - Q_{b;R_1}$ and $v_{\vec \a} = w_0^{\a} - r_{\vec \a; R_1}$. 
From Assumption \ref{assmodestab} and \eqref{eqexpliciteigenfunc}, we choose the basis of ${\rm Ran} P_u$ as $\{\varphi_{u;j}\}_{j=1}^{d+2} = \{ \xi_0, 2\xi_1, \zeta_1,...,\xi_d \}$, and define the modified eigenbasis $\{\tilde \varphi_{u;j}\}_{j=1}^{d+2}$ for $\tilde P_u$ correspondingly according to \eqref{eqdefmodspecproj}. 
Notice that 
\[  \tilde P_u \Pi^{-1} \left(\nabla_{\vec \a}  r_{\vec \a;R_1}\right)\Big|_{\vec \a = \vec 0} - \left( \xi_0, 2\xi_2, \zeta_1,...,\zeta_d \right) = o_{(\tilde X_+)^{d+2}}(1),\quad {\rm as}\,\, R_1 \to \infty,    \]
due to the definition of $Q_{b;R_1}$ \eqref{eqdefQbR1} and vanishing of residual \eqref{eqbdde*R1}. Therefore, defining the parametrization using the basis $\{\tilde \varphi_{u;j}\}_{j=1}^{d+2}$, from the definition of $\tilde P_u$ \eqref{eqdeftildePa}, we still have linear non-degeneracy that
$$\tilde \calP \tilde P_u \Pi^{-1} \left(\vec \a \cdot (\nabla_{\a} r_{\a;R_1})\big|_{\vec \a = \vec 0}\right) = ({\rm Id} + o_{\RR^{(d+2)\times (d+2)}}(1))\vec \a,\quad {\rm as \,\,} R_1 \to \infty.$$
The smallness of $v_{\vec\a}$, nonlinearity and Lipschitz constant of $\tilde \Phi_{R_1}$ \eqref{eqdefPhiR1} enables us to conclude the Brouwer argument and hence the whole proof. 
\end{proof}

\appendix
\section{Fractional Non-stationary Phase}\label{appA}

In Lemma \ref{lemanalyticity} and Lemma \ref{lemimpreg}, we want to make use of the $H^\delta$-regularity to integrate by parts and get decay from the quadratic oscillatory phase. For $\delta < 1$, we achieve this via mollification, with the mollifier's size adapted to the decay $r^{-1}$ we gain. The idea of this part is credited to Xiaodong Li. 

\begin{lemma}\label{lemfrac}
	Let smooth cutoff function $\psi \in C^\infty_{c, rad}(\RR^d)$, $\supp \psi \subset B_2$, non-negative and $\int_{\RR^d} \psi(x) dx = 1$. Define the rescaling of $\psi$ (or general $C^\infty_c$ functions) by $\psi_\l (r) := \l^{-d}\psi(\l^{-1} r)$. For $f \in \mathscr{S}(\RR^d)$ and $\delta \in (0, 1)$, we have
	\bea 
	\left\| f * \psi_{|\cdot|^{-1}} \right\|_{L^2\{|x| \ge 1/2 \}} &\lesssim_{\psi, d}& \|f\|_{L^2} \label{eqfrac1} \\
    \left\|\partial_r ( f * \psi_{|\cdot|^{-1}}) |\cdot|^{-1+\delta} \right\|_{L^2\{|x| \ge 1/2 \}} &\lesssim_{\psi, d, \delta}& \|f\|_{H^\delta} \label{eqfrac4} \\
    \left\| \left(f - f * \psi_{|\cdot|^{-1}}\right)|\cdot|^{\delta} \right\|_{L^2\{|x| \ge 1/2 \}} &\lesssim_{\psi, d, \delta}& \|f\|_{H^\delta} \label{eqfrac3}
	\eea
	where \eqref{eqfrac1} also holds for general $\psi \in C^\infty_c(B_2)$ with constant depending on $\psi$.
\end{lemma}

\begin{proof}
	(1) For \eqref{eqfrac1}, we have
	\bee
	&&\left\| f * \psi_{|\cdot|^{-1}} \right\|_{L^2\{|x| \ge 1/2 \}}^2 =\int_{B_{1/2}^c} \left| \int_{\RR^d} f(y) \psi_{|x|^{-1}}(x - y ) dy\right|^2 dx \\
	&\le& \int_{B_{1/2}^c} \left( \int_{\RR^d} |f(y)|^2  |\psi_{|x|^{-1}}(x - y )| dy \right) \left( \int_{\RR^d}  |\psi_{|x|^{-1}}(x - y )| dy \right)  dx\\
	& = & \| \psi\|_{L^1} \int_{\RR^d} |f(y)|^2 \int_{B_{1/2}^c} |\psi(|x|(x - y ))| |x|^d dx dy \\
	& \le & \| f \|_{L^2}^2 \| \psi\|_{L^1} \| \psi\|_{L^\infty} \sup_{y \in \RR^d}\int_{\left\{ x: |x| \ge \frac 12, |x - y| \le 2|x|^{-1} \right\}} |x|^d dx .
	\eee
    We claim that
    \be \sup_{y \in \RR^d}\int_{\left\{ x: |x| \ge \frac 12, |x - y| \le 2|x|^{-1} \right\}} |x|^d dx \lesssim_d 1. \label{eqfrac12} 
    \ee
    and \eqref{eqfrac1} follows with general $\psi \in C^\infty_c(B^2)$. Indeed, when $|y| \ge 8 $, noticing that 
    \[ \left\{ x: |x| \ge \frac 12, |x - y| \le 2|x|^{-1} \right\} \subset \{x: |x - y| \le 4 \}, \]
    we have $\frac 12 |y| \le |x| \le 2 |y|$. So 
    \bee \int_{\left\{ x: |x| \ge \frac 12, |x - y| \le 2|x|^{-1} \right\}} |x|^d dx \le \int_{|x - y| \le 4|y|^{-1}} 2^d |y|^d dx \lesssim_d 1.
    \eee
    Also, when $|y| \le 8$, 
    \bee \int_{\left\{ x: |x| \ge \frac 12, |x - y| \le 2|x|^{-1} \right\}} |x|^d dx \le \int_{|x - y| \le 4} |x|^d dx \le \int_{B_{12}} |x|^d dx \lesssim_d 1,
    \eee
    which verifies \eqref{eqfrac12}	and hence \eqref{eqfrac1}. \\
    
    (2) For \eqref{eqfrac4}, we first compute
    \bee \partial_r (f * \psi_{|x|^{-1}})(x) &=& (f * \partial_\l \psi_{|x|^{-1}})(x) (r^{-1})' + \frac{x}{|x|} \cdot (f * \nabla(\psi_{|x|^{-1}}))(x) \\
    &=&|x|^{-1} \left(f * [(d+(\cdot)\cdot\nabla)\psi]_{|x|^{-1}} \right)(x)  + x \cdot (f * (\nabla \psi)_{|x|^{-1}})(x)
    \eee
    Let $\tilde{\psi} := (d+x\cdot \nabla)\psi \in C^\infty_c(B^2)$, we have
    \bee &&\left\|\partial_r ( f * \psi_{|\cdot|^{-1}}) |\cdot|^{-1+\delta} \right\|_{L^2\{|x| \ge 1/2 \}} 
    \\
    &\le& \left\|( f * \tilde{\psi}_{|\cdot|^{-1}}) |\cdot|^{-2+\delta} \right\|_{L^2\{|x| \ge 1/2 \}} + \left\| ( f * (\nabla \psi)_{|x|^{-1}}) |\cdot|^{\delta} \right\|_{L^2\{|x| \ge 1/2 \}}.
    \eee
    Using \eqref{eqfrac1}, the first term is bounded by $C_{\psi, d, \delta} \| f\|_{L^2}$. So it suffices to prove
    \bea
    	\left\| \left(f * (\nabla \psi)_{|\cdot|^{-1}}\right) |\cdot|^{\delta} \right\|_{L^2\{|x| \ge 1/2 \}} &\lesssim_{\psi, d, \delta}& \|f\|_{\dot{H}^\delta} \label{eqfrac2} 
    \eea

    For \eqref{eqfrac2}, 
    \bee
        &&\left\| \left(f * (\nabla \psi)_{|\cdot|^{-1}}\right) |\cdot|^{\delta} \right\|_{L^2\{|x| \ge 1/2 \}}^2 = \int_{B_{1/2}^{c}} \left| \int_{\RR^d} f(x - y) (\nabla \psi)_{|x|^{-1}}(y ) dy \right|^2 |x|^{2\delta} dx \\
        &=&  \int_{B_{1/2}^{c}} \left| \int_{\RR^d} \left(f(x - y) - f(x)\right) (\nabla \psi)_{|x|^{-1}}( y ) dy \right|^2 |x|^{2\delta} dx \\
        & \le & \int_{B_{1/2}^{c}} \left( \int_{\RR^d} \left|f(x - y) - f(x)\right| \left(\frac{2|x|^{-1}}{|y|}\right)^{\frac{d+2\delta}{2}} |(\nabla \psi)_{|x|^{-1}}( y )| dy \right)^2 |x|^{2\delta} dx\\
        & \le & \int_{B_{1/2}^{c}} 2^{d+2\delta}|x|^{-(d+2\delta)}  \left( \int_{\RR^d} \frac{\left|f(x - y) - f(x)\right|^2}{|y|^{d+2\delta}} dy  \right) 
        \left( \int_{\RR^d}   |(\nabla \psi)_{|x|^{-1}}( y )|^2 dy \right) |x|^{2\delta} dx\\
        &\le& \| \nabla\psi\|_{L^2}^2 2^{d+2\delta} \int_{B^c_{1/2}}  \int_{\RR^d} \frac{\left|f(x - y) - f(x)\right|^2}{|y|^{d+2\delta}} dy dx
    \eee
    where we used the vanishing $\int_{\RR^d} (\nabla \psi)_\l (x) dx = \vec{0}$ for any $\l >0$ and its support and scaling. 
    With the equivalent expression of $\dot{H}^\delta$ norm \eqref{eqHdelta}, the above computation implies \eqref{eqfrac2} and hence \eqref{eqfrac4}.\\
    
    (3) Finally, for \eqref{eqfrac3}, note that  $\int \psi(x)dx = 1$ indicates
    \bee
    && \left\| \left(f - f * \psi_{|\cdot|^{-1}}\right)|\cdot|^{\delta} \right\|_{L^2\{|x| \ge 1/2 \}}^2 \\
    &=& \int_{B_{1/2}^{c}} \left| \int_{\RR^d} \left(f(x - y) - f(x)\right)  \psi_{|x|^{-1}}( y ) dy \right|^2 |x|^{2\delta} dx \\
    & \le & \| \psi \|_{L^2}^2 2^{d+2\delta} \int_{B^c_{1/2}}  \int_{\RR^d} \frac{\left|f(x - y) - f(x)\right|^2}{|y|^{d+2\delta}} dy dx.
    \eee
    The last inequality follows almost the same computation as the above one for \eqref{eqfrac2}.
\end{proof}

\section{Commutator estimates of multiplication in $\dot{H}^\sigma$}
\label{appB}

We consider several boundedness and commutator estimates for scalar potential $V$ in $\dot{H}^\sigma$. Recall our polynomial decay assumption on potential 
\be |V(x)|\lesssim \la x \ra^{-\a},\quad |\nabla V(x)| \lesssim \la x \ra^{-\a-1} \tag{\ref{eqpotbd0}} \ee
for some $\a > 0$.

\begin{lemma}\label{lempotbdd} For  $d \ge 1$, $0 < \sigma < \min\left\{1, \frac{d}{2} \right\} $ and $V \in C^1(\RR^d)$, for $2\le p < \frac{d}{\sigma}$, 
	\be	\| D^\sigma V D^{-\sigma} \|_{L^2 \to L^{p'}} + \| D^\sigma V D^{-\sigma} \|_{L^p \to L^2} \lesssim_{d, p, \sigma} \| V \|_{ \dot{W}^{\sigma, \left(\frac 12 + \frac \sigma d - \frac{1}{p}\right)^{-1}}}\label{eqleb1} \ee
	In particular, for $V$ satisfying \eqref{eqpotbd0} with $\a > 0$, $D^\sigma VD^{-\sigma} \in \L(L^2 \to L^{p'}) \cap \L(L^p\to L^2)$ for $\frac 12 \ge \frac 1p >\max\{\frac 12 - \frac{\a}{d}, \frac{\sigma}{d} \}$.
\end{lemma}

\begin{proof}
It is an easy application of Sobolev embedding and the fractional Lebnitz rule (see for example \cite[Theorem 1]{grafakos2014kato})
\be  \label{eqfracLeb} \| D^s (f_1 f_2)\|_{L^r} \lesssim_{d, s, r, p_1, p_2, q_1, q_2} \| D^s f_1 \|_{L^{p_1}} \| f_2 \|_{L^{q_1}} + \|f_1\|_{L^{p_2}} \| D^s f_2 \|_{L^{q_2}}  \ee
	with $1 > s > 0$, $r \in (1, \infty)$, $p_1, p_2, q_1, q_2 \in (1, \infty]$ satisfying $\frac{1}{r} = \frac{1}{p_1} + \frac{1}{q_1} = \frac{1}{p_2} + \frac{1}{q_2}$ and $f, g\in \mathscr{S}(\RR^d)$.
\end{proof}

Moreover, the shape of $V$ indicates a stronger bound: $\|\la x \ra^\a Vf \|_{L^2} \lesssim \| f \|_{L^2}$ in $L^2$. This remains almost true after conjugating $V$ with fractional derivatives.


\begin{lemma}\label{lempotbdd2}
	Let $d \ge 1$, $0 < \sigma < \min\left\{1, \frac{d}{2} \right\} $ and $V \in C^1(\RR^d)$ satisfying \eqref{eqpotbd0} with parameter $\a > 0$. We have the following statements:
	\begin{enumerate}
		\item[\rm (i)]For $0 \le \beta < \min\left\{ \alpha, 1\right\}$ and any $f \in \mathscr{S}(\RR^d)$, we have
		\be \| \la x \ra^{\beta} D^\sigma V D^{-\sigma} f \|_{L^2} \lesssim \| f \|_{L^2}. \label{eqpotbdd2} \ee
		\item[\rm (ii)]For $0 \le \b < \min\{\a, \frac d2 - \sigma\}$
		and $f \in \mathscr{S}(\RR^d)$, we have
		\be \| \la x \ra^\b D^{-\sigma} V D^{\sigma} f \|_{L^2} \lesssim \| f \|_{L^2} \label{eqpotbdd3} \ee
		\item[\rm (iii)]Let $0 < \a + \b < \gamma < 1$, we have
		\be \| \la x \ra^{\a} D^\sigma \la \cdot \ra^{\b} D^{-\sigma}\la \cdot \ra^{-\gamma} f \|_{L^2} \lesssim \| f \|_{L^2}.\label{eqpotbdd4} \ee
		Consequently, $\la x \ra^{\b} \in \L(\dot{H}^\sigma_{\la x \ra^\gamma} \to\dot{H}^\sigma_{\la x \ra^\alpha}) $.
        \item[\rm (iv)] Suppose additionally $V \in C^2(\RR^d)$ and $|\nabla^2 V(x)| \lesssim \la x \ra^{-\a - 2}$, then for $r < \frac d \sigma$, $\frac 1q + \frac 1r \in (1-\frac \sigma d, 1)$, we have 
        \be \|D^\sigma [V, D^{1-\sigma}] D^{-\sigma} f \|_{L^{r'}} \lesssim \| f \|_{L^q}. \label{eqpotbdd5} \ee
        Consequently, $[V, D^{1-\sigma}] \in \calL(\dot W^{\sigma, q} \to \dot W^{\sigma, r'} )$.
	\end{enumerate}
\end{lemma}

\begin{proof}
	(i) Since $\la x\ra^{\beta} V$ satisfies the polynomial decay condition \eqref{eqpotbd0} with parameter $\a-\beta > 0$, Lemma \ref{lempotbdd} suggests that it suffices to estimate the commutator
	\be \left\| [D^\sigma, \la \cdot \ra^\beta] VD^{-\sigma} f \right\|_{L^2} \lesssim \| f \|_{L^2}.\label{eqcom00} \ee
	Without less of generality, we assume $\a \in (\beta, 1)$.
	
	With the physical representation of $D^\sigma$ and the cancellation from the commutator, we have the decomposition
	\begin{align}
		&\left| [D^\sigma, \la \cdot \ra^\beta] VD^{-\sigma} f  \right|(x) = C \left| \int_{\RR^d} \frac{\la x \ra^{\beta} - \la x-y \ra^{\beta}}{|y|^{d+\sigma}}(VD^{-\sigma}f )(x-y) dy\right| \nonumber \\
		\lesssim& \int_{|y| \ge 2\la x \ra}|y|^{-(d+\sigma+\a-\b)} \left(|\cdot|^{-(d-\sigma)}* |f| \right)(x-y) dy \label{eqcom01} \\
		+& \int_{|y| \le 2\la x \ra} |y|^{-(d+\sigma - 1)} \la x\ra^{\b - 1} \la x -y\ra^{-\a} \left[\left(|\cdot|^{-(d-\sigma)} \mathbbm{1}_{B_{4\la x \ra}} \right) * |f| \right](x-y) dy  \label{eqcom02}  \\
		+& \int_{|y| \le 2 \la x \ra} |y|^{-(d+\sigma - 1)} \la x\ra^{\b - 1} \la x -y\ra^{-\a} \left[\left(|\cdot|^{-(d-\sigma)} \mathbbm{1}_{B_{4\la x \ra}^c} \right) * |f| \right](x-y) dy \label{eqcom03}
	\end{align}
	where $|V|$ is substituted by $\la x \ra^{-\a}$ and the fractional weight is bounded by
	\[\left| \la x \ra^{\beta} - \la x-y \ra^{\beta}  \right|\la x-y\ra^{-\a} \lesssim 
	\left\{ \begin{array}{ll}
		|y|^{\b -\a} & |y| \ge 2 \la x \ra, \\
		|y| \la x \ra^{\b-1} \la x-y \ra^{-\a} & |y| \le 2\la x \ra.
	\end{array} \right. \]
	Now we will bound the $L^2$-norm of \eqref{eqcom01}-\eqref{eqcom03} with $\| f \|_{L^2}$. 
	
	For \eqref{eqcom01}, using H\"older and Hardy-Littlewood-Sobolev inequality, we get
	\bee \eqref{eqcom01}\lesssim \| |\cdot|^{-(d+\sigma +\a-\b)} \|_{L^\frac{2d}{d+2\sigma}(B_{2\la x \ra}^c)} \| |\cdot|^{-(d-\sigma)} * f\|_{L^\frac{2d}{d-2\sigma}} \lesssim \la x \ra^{-\frac d2 -\a+\b}\| f\|_{L^2}.
	\eee
	Since $\a > \b$, we get $\| \eqref{eqcom01}\|_{L^2} \lesssim \| f \|_{L^2}$. 
	
	For \eqref{eqcom02}, noticing that $|z| \ge 4 \la x \ra$ and $|y| \le 2\la x \ra$ imply $|x-y-z| \sim |z|$, so
	\bee \eqref{eqcom02} \lesssim \la x \ra^{\b -1} \int_{B_{2\la x \ra}} |y|^{-(d+\sigma -1)} \la x-y\ra^{-\a} dy \int_{B_{4\la x \ra}^c} \frac{|f(z)|}{|z|^{d-\sigma}}dz \lesssim \la x \ra^{-\frac d2 -\a+\b}\| f \|_{L^2} \eee
	where the first integral is bounded by $\la x \ra^{1-\sigma -\a}$ when $\a < d$ via elementary computation. Therefore this term has the desired bound. 
	
	Finally, we write \eqref{eqcom03} as
	\bee \eqref{eqcom03} = \la x \ra^{\b - 1} \left(  |\cdot|^{-(d+\sigma -1)} \mathbbm{1}_{B_{2\la x \ra}} \right) * \left[ \la \cdot \ra^{-\a} \mathbbm{1}_{B_{3\la x \ra}} \left( |\cdot|^{-(d-\sigma)}\mathbbm{1}_{B_{4\la x \ra}} * |f| \right) \right](x). \eee 
	Then we can estimate their $L^2$ norm on dyadic annulus with $k \ge 0$
	\bee \| \eqref{eqcom03} \|_{L^2(B_{2^{k+1}} - B_{2^k})} 
	&\lesssim& 2^{-k(1-\b)} \| |\cdot|^{-(d+\sigma -1)} \|_{L^{p_1}(B_{2^{k+3}})} \| \la \cdot \ra^{-\a} \|_{L^{p_2}(B_{2^{k+3}})} \\
	&&\cdot \| |\cdot|^{-(d-\sigma)} \|_{L^{p_3}(B_{2^{k+3}})} \| f \|_{L^2} \\
	&\lesssim& 2^{-k(\a - \b)} \| f \|_{L^2}.
	\eee
	where we choose $p_1, p_2, p_3 \in (1, \infty)$ such that $\frac{1}{p_1} + \frac{1}{p_2} + \frac{1}{p_3} = 2$, $p_1 < \frac{d}{d+\sigma-1}$, $p_2 < \frac{d}{\a}$ and $p_3 < \frac{d}{d-\sigma}$ which is possible when $\a < 1$. The contribution on $B_1$ can be bounded similarly, and we conclude the estimate of \eqref{eqcom03} by summing all these estimates up using $\a > \b$. 
	
	Combining these three parts, we obtain \eqref{eqcom00} and hence \eqref{eqpotbdd2}.\\
	
	(ii) Without loss of generality, we assume $\a < \frac d2 - \sigma$. Because $\la x \ra^{\b} V \in L^\infty$, it suffices to show
	\be \left\|  \la x\ra^\b D^{-\sigma} [D^\sigma, V] f\right\|_{L^2} \lesssim \| f \|_{L^2}. \ee
	
	Notice that 
	\bee
	&&\la x\ra^\b D^{-\sigma} [D^\sigma, V] f(x) \\
	&=& C \la x \ra^\b \int_{\RR^d} \frac{1}{|y|^{d-\sigma}} \int_{\RR^d} \frac{V(x-y) - V(x-y-z)}{|z|^{d+\sigma}} f(x-y-z) dzdy. \eee
	Once again, we decompose the integral region and estimate the weights correspondingly. Here the estimate for interaction is 
	\be |V(x_1) - V(x_2)| \lesssim \left\{ \begin{array}{ll}
		\la x_1 \ra^{-\a-1} |x_1-x_2| & |x_1-x_2| \le \frac 12 \la x_1 \ra \\
		\la x_2 \ra^{-\a} & |x_1 - x_2| \in [\frac 18 \la x_1 \ra, 8 \la x_1 \ra] \\
		\la x_1 \ra^{-\a} & |x_1-x_2| \ge 2  \la x_1 \ra.
	\end{array} \right. \label{eqcomV}\ee
	So we will partition $y$ and $z$ into $3 \times 3$ regions:
	\begin{enumerate}
		\item 
		$\left\{ |y| \le \frac 12 \la x \ra, |z| \le \frac 14 \la x\ra \right\},
		\left\{ |y| \le \frac 12 \la x \ra, |z| \in \left[ \frac 14, 4 \right] \la x\ra\right\},
		\left\{ |y| \le \frac 12 \la x \ra, |z| \ge 4 \la x\ra\right\};$
		\item $\left\{ |y| \in \left[ \frac 12, 2\right] \la x \ra, |z| \le \frac 14 \la x-y \ra \right\},
		\left\{ |y| \in \left[ \frac 12, 2\right] \la x \ra, |z| \in \left[ \frac 14, 4 \right] \la x-y \ra\right\},\\
		\left\{ |y| \in \left[ \frac 12, 2\right] \la x \ra, |z| \ge 4 \la x-y \ra\right\};$
		\item $\left\{ |y| \ge 2 \la x \ra, |z| \le \frac 14 \la y \ra \right\},
		\left\{ |y| \ge 2 \la x \ra, |z| \in \left[ \frac 14, 4 \right] \la y \ra\right\}, 
		\left\{ |y| \ge 2 \la x \ra, |z| \ge 4 \la y \ra\right\}$.
	\end{enumerate}
	Therefore we have
	\bea 
	&&\left| \la x\ra^\b D^{-\sigma} [D^\sigma, V] f(x) \right| \nonumber \\
	&\lesssim& \la x \ra^{-1-\a +\b} \left(|\cdot|^{-(d-\sigma)} \mathbbm{1}_{B_{\frac 12 \la x \ra}} \right)* \left(|\cdot|^{-(d+\sigma-1)} \mathbbm{1}_{B_{\frac 14 \la x \ra}} \right) * |f| (x) \label{eqcom41} \\ 
	& + & \la x \ra^{-d-\sigma +\b} \int_{B_{\frac 12 \la x \ra}} |y|^{-(d-\sigma)} \int_{B_{4\la x \ra}} \frac{|f(x-y-z)|}{\la x -y-z\ra^{\a}}dz dy \label{eqcom42} \\
	&+& \la x \ra^{-\a +\b} \left(|\cdot|^{-(d-\sigma)} \mathbbm{1}_{B_{\frac 12 \la x \ra}} \right)* \left(|\cdot|^{-(d+\sigma)} \mathbbm{1}_{B^c_{4 \la x \ra}} \right) * |f| (x)  \label{eqcom43}\\
	&+& \la x \ra^{-d+\sigma+\b} \int_{B_{2\la x \ra} - B_{\frac 12 \la x \ra}} \la x -y\ra^{-\a - 1} \int_{B_{\frac 14 \la x-y\ra}} \frac{|f(x-y-z)|}{|z|^{d+\sigma -1}} dz dy \label{eqcom44}\\
	&+&\la x \ra^{-d+\sigma+\b} \int_{B_{2\la x \ra} - B_{\frac 12 \la x \ra}} \la x -y\ra^{-d- \sigma} \int_{B_{4 \la x -y \ra} - B_{\frac 14 \la x-y\ra}} \frac{|f(x-y-z)|}{\la x-y-z \ra^{\a}} dz dy \label{eqcom45} \\
	&+& \la x \ra^{-d+\sigma+\b} \int_{B_{2\la x \ra} - B_{\frac 12 \la x \ra}} \la x -y\ra^{-\a } \int_{B^c_{4 \la x-y\ra}} \frac{|f(x-y-z)|}{|z|^{d+\sigma}} dz dy \label{eqcom46}\\
	&+& \la x \ra^{\b} \int_{B_{2\la x \ra}^c} |y|^{-(d-\sigma+\a+1)} \int_{B_{\frac 14 \la y \ra}} \frac{|f(x-y-z)|}{|z|^{d+\sigma - 1}}  dz dy \label{eqcom47}\\
	&+& \la x \ra^{\b} \int_{B_{2\la x \ra}^c} |y|^{-2d} \int_{B_{4\la y \ra} - B_{\frac 14 \la y \ra}} \frac{|f(x-y-z)|}{\la x-y-z\ra^{\a}}  dz dy \label{eqcom48}\\
	&+& \la x \ra^{\b} \int_{B_{2\la x \ra}^c} |y|^{-(d-\sigma+\a)} \int_{B_{4 \la y \ra}^c} \frac{|f(x-y-z)|}{|z|^{d+\sigma}}  dz dy. \label{eqcom49}
	\eea
	We partition these 9 terms into four groups to estimate.
	\begin{enumerate}
		\item For \eqref{eqcom41} and \eqref{eqcom43}, we exploit the convolution structure to estimate the $L^2$ norm on non-homogeneous dyadic annulus, resulting in $\| \eqref{eqcom41}\|_{L^2(B_{2^{k+1}}-B_{2^k})} \lesssim 2^{-k(\a - \b)}\| f \|_{L^2}$, $\| \eqref{eqcom41}\|_{L^2(B_1)} \lesssim \| f \|_{L^2}$, and similar for \eqref{eqcom43}.  
		\item For \eqref{eqcom42}, \eqref{eqcom45} and \eqref{eqcom48}, a straightforward application of H\"older's inequality implies pointwise bound by $\la x \ra^{-\frac d2 -\a + \b} \| f \|_{L^2}$. 
		\item We address \eqref{eqcom47} and \eqref{eqcom49} similarly to \eqref{eqcom41} and \eqref{eqcom43}, but with an additional dyadic decomposition for the domain of $y$ to fix the integration domain of $z$. For example, 
		\bee
		&&\| \eqref{eqcom47}\|_{L^2(B_{2^{k+1}} - B_{2^k})}\\
		&\lesssim& 2^{k\b}\sum_{j \ge k} \left\| \left(|\cdot|^{-(d-\sigma+\a+1)}\mathbbm{1}_{B_{2^{j+1}} - B_{2^j}  }\right) * \left(|\cdot|^{-(d+\sigma-1)}\mathbbm{1}_{B_{2^{j}}}\right) * |f|\right\|_{L^2}\\
		&\lesssim& 2^{k\b} \sum_{j\ge k} 2^{-k\a} \| f\|_{L^2} \lesssim 2^{-k(\a-\b)} \| f\|_{L^2}.
		\eee
		\item For \eqref{eqcom44}-\eqref{eqcom46}, we first apply a change of variable $y' := x-y \in B_{4\la x \ra}$ so that $\la x \ra$ only affects the size of the integration domain of $y'$.
		By considering $x$ and $y$ dyadically, we can finally arrive at the same pointwise bound $\la x \ra^{-\frac d2 + \b -\a } \| f \|_{L^2}$. For instance, when $x \in B_{2^{k+1}} - B_{2^k}$, 
		\bee &&\eqref{eqcom46} \lesssim \la x \ra^{-d+\sigma +\b} \int_{B_{2^{k+3}}} \la y'\ra^{-\a} \int_{B_{4\la y'\ra}^c} \frac{|f(y'-z)|}{|z|^{d+\sigma}} dz dy \\
		&\lesssim& 2^{(-d+\sigma +\b)k} \sum_{0 \le j \le k}\int_{B_{2^{j+3}} - B_{2^{j+2}}} 2^{-j\a} \left[\left(|\cdot|^{-(d+\sigma)}\mathbbm{1}_{B_{2^{j+1}}^c}\right) * |f| \right] (y')dy' \\
		&&+ 2^{(-d+\sigma +\b)k} \int_{B_4} \left[\left(|\cdot|^{-(d+\sigma)}\mathbbm{1}_{B_{8}^c}\right) * |f| \right] (y')dy' \\
		&\lesssim & 2^{(-d+\sigma +\b)k} \sum_{0 \le j \le k} 2^{j(\frac{d}{2} - \sigma - \a)} \| f \|_{L^2} 
		\lesssim 2^{\left(-\frac d2 + \b -\a\right)k} \| f \|_{L^2}.
		\eee
			\end{enumerate}
	
These bounds conclude the proof. \\
	
	(iii) It suffices to estimate the commutator 
	\be \| \la x \ra^\a [D^\sigma, \la \cdot \ra^\b] D^{-\sigma} \la \cdot \ra^{-\gamma} f\|_{L^2} \lesssim \| f \|_{L^2}. \label{eqcom20}\ee
	
	We decompose the domain of integration for $y$ into 
	$\{|y| \ge 2\la x \ra\}$ and $\{|y| \le 2 \la x \ra\}$, and further decompose the domain of $z$ by comparing $|z|$ with $\frac 14 |y|$ or $4\la x \ra$ respectively. Estimating $|\la x \ra^\b - \la x-y\ra^\b|$, $|x-y-z|$ and $\la z \ra$ correspondingly, we obtain
	\begin{align} &\left| \la x \ra^\a [D^\sigma, \la \cdot \ra^\b] D^{-\sigma} \la \cdot \ra^{-\gamma} f(x)\right|  \nonumber\\
		=& C \left| \la x \ra^\a \int_{\RR^d} \frac{\la x \ra^\b - \la x -y\ra^{\b}}{|y|^{d+\sigma}} \int_{\RR^d} \frac{\la z \ra^{-\gamma} f(z)}{|x-y-z|^{d-\sigma}}dzdy \right| \nonumber \\
		\lesssim & \la x \ra^\a \int_{B_{2\la x \ra}^c} |y|^{-(d+\sigma -\b)} \int_{B_{\frac 14 |y|}} |y|^{-(d-\sigma)} \la z \ra^{-\gamma} |f(z)|  dz dy \label{eqcom21} \\
		+ & \la x \ra^\a \int_{B_{2\la x \ra}^c} |y|^{-(d+\sigma -\b + \gamma)} \left(|\cdot|^{-(d-\sigma)} *  |f|\right)(x-y)  dy\label{eqcom22} \\
		+ &  \la x \ra^{\a+\b-1} \left(|\cdot|^{-(d+\sigma -1)} \mathbbm{1}_{B_{2\la x \ra}} \right) *  \left(|\cdot|^{-(d-\sigma )} \mathbbm{1}_{B_{8\la x \ra}} \right) *  \left(\la \cdot\ra^{-\gamma} \mathbbm{1}_{B_{4\la x \ra}} |f| \right)(x) \label{eqcom23}\\
		+ & \la x \ra^{\a+\b-1} \int_{B_{2\la x \ra}}|y|^{-(d+\sigma -1)} dy \int_{B_{4\la x \ra}^c} \frac{|f(z)|}{|z|^{d-\sigma+\gamma}} dz \label{eqcom24}
	\end{align}
	For \eqref{eqcom21}, \eqref{eqcom22} and \eqref{eqcom24}, H\"older's and Hardy-Littlewood-Sobolev inequalities imply a pointwise bound by $\la x \ra^{-(d/2 + \gamma - \a -\b)} \| f \|_{L^2}$, so their $L^2$-norm is bounded by $\| f \|_{L^2}$. For \eqref{eqcom23}, similar to \eqref{eqcom03}, we can control their $L^2$-norm on dyadic annulus $B_{2^{k+1}} - B_{2^k}$ by $2^{-k(\gamma - \a -\b)} \| f \|_{L^2}$ and on $B_1$ by $\| f \|_{L^2}$. Combining these bounds, we get \eqref{eqcom20} and hence \eqref{eqpotbdd4}.

    \mbox{}

    (iv) We decompose 
    \[ D^\sigma [V, D^{1-\sigma}] D^{-\sigma}f = [V, D^{1-\sigma}] f + \left[ D^\sigma, [V, D^{1-\sigma}]  \right] D^{-\sigma} f\]
    and estimate both terms. The estimates are similar to the previous cases, with extra simplicity in that we can always set $\a = 0$. 

    \textit{Term 1: $[V, D^{1-\sigma}] f$.} By partitioning $|y| \ge 2\la x \ra$ or $|y| \le 2\la x \ra$, we have 
    \bee
     \left| [V, D^{1-\sigma}] f  \right| &\lesssim& \la x \ra^{-1} \left( |\cdot|^{-(d-\sigma)} \mathbbm{1}_{B_{2\la x \ra}} )* |f|  \right) + \left( |\cdot|^{-(d-\sigma+1)} \mathbbm{1}_{B_{2\la x \ra}}^c )* |f|  \right) \\
     &+& \la x \ra^{-(d+1-\sigma)}  \left(  \mathbbm{1}_{B_{4\la x \ra}}, |f|  \right)_{L^2}
    \eee
    For the first two terms, we decompose $\la x \ra$ in dyadic regions and apply Young's inequality to bound the $L^{r'}$ norm, and the last term can be estimated directly through H\"older's inequality. The range of $\frac 1{r'} - \frac 1q \in (0, \frac\sigma d)$ ensures the integrability and summability. 

    \textit{Term 2: $\left[ D^\sigma, [V, D^{1-\sigma}]  \right] D^{-\sigma} f$.} Denote $F := D^{-\sigma} f$ and $\frac 1{q_\sigma} := \frac 1q + \frac \sigma d$. Then from Sobolev embedding, it suffices to show 
    $$\| \left[ D^\sigma, [V, D^{1-\sigma}]  \right] F\|_{L^{r'}} \lesssim \| F \|_{L^{q_\sigma}}.$$ 
    By expanding both fractional derivative and exploit cancellations, we obtain 
    \[ \left[ D^\sigma, [V, D^{1-\sigma}]  \right] F = C \iint \frac{V(x) - V(x-y) - V(x-z) + V(x-y-z)}{|y|^{d+1-\sigma} |z|^{d+\sigma}} F(x-y-z) dydz. \]
    By partitioning $|y|$ and $|z|$ both according to $\frac 14 \la x \ra$ and use the pointwise bound of $\nabla V$, $\nabla^2 V$, we obtain 
    \bea
     && \left| \left[ D^\sigma, [V, D^{1-\sigma}]  \right] F(x) \right| \nonumber\\
     &\lesssim& \la x \ra^{-2} \cdot \left[ \left( |\cdot|^{-(d-\sigma)} \mathbbm{1}_{B_{\frac 14\la x \ra}}  \right) *  \left( |\cdot|^{-(d-1+\sigma)} \mathbbm{1}_{B_{\frac 14\la x \ra}}  \right) * |F| \right]  \label{eqdoucom1} \\
     &+& \int_{B_{\frac 14 \la x \ra}^c} \int_{B_{\frac 14 \la x \ra}} \frac{\la x \ra^{-1} + \la x-y \ra^{-1}}{|y|^{d+1-\sigma}|z|^{d-1+\sigma}} |F(x-y-z)| dydz \label{eqdoucom2}   \\
     &+& \int_{B_{\frac 14 \la x \ra}} \int_{B_{\frac 14 \la x \ra}^c} \frac{\la x \ra^{-1} + \la x-z \ra^{-1}}{|y|^{d-\sigma}|z|^{d+\sigma}} |F(x-y-z)| dydz  \label{eqdoucom3}  \\
     &+& \left( |\cdot|^{-(d+1-\sigma)} \mathbbm{1}_{B_{\frac 14\la x \ra}^c}  \right) *  \left( |\cdot|^{-(d+\sigma)} \mathbbm{1}_{B_{\frac 14\la x \ra}^c}  \right) * |F|  \label{eqdoucom4} 
    \eea
    Noticing that $s:=\frac{1}{q_\sigma} - \frac 1{r'} \in (0, \frac \sigma d)$, for \eqref{eqdoucom1} and \eqref{eqdoucom4}, we can partition the range of $\la x \ra$ dyadically and apply Young's inequality with suitable $L^{1+}$ norm on the homogeneous parts so as to exploit $\| F \|_{L^{q_\sigma}}$, which generates additional growth $\la x \ra^{\frac \sigma d}$ in \eqref{eqdoucom1} absorbed by $\la x \ra^{-2}$, and additional decay $\la x \ra^{-\frac \sigma d}$ in \eqref{eqdoucom4} for summability. The terms \eqref{eqdoucom2} and \eqref{eqdoucom3} are symmetric, and the part with $\la x \ra^{-1}$ can be estimated using exactly the same strategy as \eqref{eqdoucom1} or \eqref{eqdoucom4}. Lastly, we estimate the part in \eqref{eqdoucom2} with $\la x -y \ra^{-1}$ dyadically as 
    \bee
      && \left\| \int_{B_{\frac 14 \la x \ra}^c} \int_{B_{\frac 14 \la x \ra}} \frac{\la x-y \ra^{-1}}{|y|^{d+1-\sigma}|z|^{d-1+\sigma}} |F(x-y-z)| dydz  \right\|_{L^{r'}(B_{2^{n+1}} - B_{2^n})} \\
      &\lesssim& \left\| (|\cdot|^{-(d-1+\sigma)} \mathbbm{1}_{B_{2^{n+2}}}) * \left[ \la \cdot \ra^{-1} \cdot \left( (|\cdot|^{-(d+1-\sigma)} \mathbbm{1}_{B_{2^{n-3}}^c}) * |F| \right)  \right]  \right\|_{L^{r'}(B_{2^{n+1}})} \\
      &\lesssim& \left\||\cdot|^{-(d-1+\sigma)} \mathbbm{1}_{B_{2^{n+2}}}\right\|_{L^1}\cdot  \left\| |\cdot|^{-(d+1-\sigma)} \mathbbm{1}_{B_{2^{n-3}}^c} \right\|_{L^{(1-s)^{-1}}} \cdot \| F \|_{L^{(\frac 1{r'} +s)^{-1}}} \\
      &\lesssim& 2^{-sn} \| F \|_{L^{q_\sigma}}.
    \eee
    Summing over $n \ge 0$ using $s > 0$ implies the desired control.
\end{proof}

\section{Nonlinear estimates in fractional Sobolev space}\label{appC}

The following estimates in fractional Sobolev space are for the nonlinearity in \eqref{eqNL}. We first state the main result and give some remarks.

\begin{lemma}\label{lemfracdif}
	Given $d \ge 1$ and $W \in C^1(\RR^d \to \CC - \{ 0\})$ satisfying 
	\be |\nabla W(x) | \le C_W \la x \ra^{-a} |W(x)|,\quad \forall\, x \in \RR^d \label{eqslowvar} \ee
	for some $C_W > 0$, $a \in [0, 1]$, then for $p > 1$, $s \in (0, 1)$, $r \in (1, \infty)$ and $\e \in \mathscr{S}(\RR^d \to \CC)$, we have the following estimates for  
	\be R(\e)
	:=\bigg[ |W + \e|^{p-1} (W + \e) - |W|^{p-1} W - \frac{p+1}{2} |W|^{p-1} \e - \frac{p-1}{2} |W|^{p-3}W^2 \bar{\e} \bigg] \label{eqdefI}\ee
	with $\tilde{r}, r_1, r_2 \in (1, \infty)$ such that 
	$\frac{1}{r} - \frac{1}{\tilde{r}} \in \left [0, \frac{as}{d}\right )$, $\frac{1}{r} = \frac{1}{r_1} + \frac{1}{r_2}$ and $(p-1)r_1 > 1$:
	\begin{itemize}
		\item If $1 < p \le 2$, 
		\be \| R(\e)  \|_{\dot{W}^{s, r}} \lesssim_{d, p, s, r, \tilde{r}, r_1, r_2, a, C_W} \| |\e|^p\|_{L^{\tilde{r}}} +  \| \e \|_{L^{(p-1)r_1}}^{p-1} \| \e \|_{\dot{W}^{s, r_2}} \label{eqfracdif1} \ee
		\item If $p > 2$,
		\be 
		\begin{split}
			\| R(\e)  \|_{\dot{W}^{s, r}} \lesssim_{d, p, s, r,\tilde{r}, r_1, r_2, a, C_W} &\| |\e|^p \|_{L^{\tilde{r}}}  + \| |W|^{p-2} |\e|^2  \|_{L^{\tilde{r}}} \\
			+& \left( \| \e \|_{L^{(p-1)r_1}}^{p-1} + \||W|^{p-2}|\e|\|_{L^{r_1}} \right) \| \e \|_{\dot{W}^{s, r_2}}.
		\end{split} \label{eqfracdif2} \ee
		\end{itemize}
	\end{lemma}

\begin{remark}
	When $p$ is large, these estimates easily follow the fractional chain rule for $F''(W+t\e)$ with $F(z) := |z|^{p-1}z$. However, when $p -1 \ll 1$ (as $d \gg 1$),  $F''(u)$ may not even be continuous for generic $u \in C^\infty_c$, let alone estimates on $D^s$ derivatives.

	Fortunately, we exploit the specialty of $F$ and $W$ here: the nonlinearity $F$ only has a singularity at $0$ while $W$ does not vanish anywhere. Thus, with an appropriate estimate for $\nabla W$, the Taylor expansion of $F$ around $W$ will be good enough.
\end{remark}

\begin{proof}
Denote $W =: \Sigma + i \Theta$, $\e =: \e_1 + i\e_2$ and $R_1(\e) := \Re R(\e)$, $R_2 (\e) := \Im R(\e)$.
Also let $F_i: \RR^2 \to \RR$ with $i =1 , 2$ to be 
\be F_1(a, b) = (a^2 + b^2)^{\frac{p-1}{2}} a,\quad F_2(a, b) = (a^2 + b^2)^{\frac{p-1}{2}}b. \label{eqFdecomp}\ee
Then for $i = 1, 2$, 
\bee R_i(\e) = F_i\left(\Sigma+\e_1, \Theta + \e_2\right) - F_i ( \Sigma, \Theta) - \pa_1 F_i (\Sigma, \Theta)\e_1 - \pa_2 F_i(\Sigma, \Theta)\e_2 \eee 
From the expression of $F_i$, it is easy to see $F_i \in C^\infty(\RR^2-\{0\}, \RR)$ and
\be \left| \nabla^n F_i(a, b)  \right| \lesssim_n (a^2 + b^2)^{\frac{p-n}{2}},\quad \forall \, n \ge 0,\, \forall \,(a, b) \neq (0, 0). \label{eqFderivative} \ee
So if $(W+t\e)(x) \neq 0$ for any $t \in [0, 1]$, the following representation also holds
\be R_i(\e)(x) = \sum_{ j,k= 1}^2 \e_j(x) \e_k(x) \int_0^1 \pa_{j}\pa_k F_i (\Sigma(x) + t\e_1(x), \Theta(x) + t\e_2(x)) (1-t) dt  \label{eqI1derivative}\ee
We will prove \eqref{eqfracdif1}-\eqref{eqfracdif2} in two steps, following the strategy of \cite{MR1124294} or \cite{MR1766415}. For simplicity, we only prove the case $1 < p \le 2$ while the $p>2$ case follows in a similar way. \\

\underline{Step 1. Pointwise estimates.}
We first derive pointwise estimate for difference of $R(\e)(x+h) - R(\e)(x)$ with $x, h \in \RR^d$. For notational simplicity, we define the difference operator
\bee \Delta_h f (x) := f(x+h) -f(x). \eee
Then we claim the following estimates:
\be
	\begin{split} |\Delta_h R(\e)(x)| \lesssim_{p, C_W, a}& \left(|\e(x+h)|^{p} + |\e(x)|^p\right)\min\{1, \la x\ra^{-a} |h|\} \\
	+ & \left(|\e(x+h)|^{p-1} + |\e(x)|^{p-1}\right) \left|\Delta_h \e(x)\right|;  \end{split}\label{eqdif1}
    \ee
    We only consider estimates for $R_1(\e)$ as an example.

From \eqref{eqslowvar}, there exists a $c_0= c_0(C_W, a) > 0$ such that 
\be |W(x) - W(x+h)| \le \frac 13 |W(x)|, \quad \forall \,x \in \RR^d,\, \forall\, |h| \le \la x \ra^{a} c_0. \label{eqWsim} \ee
Then we distinguish two cases by comparing $|h|$ with $ \la x \ra^a c_0$. 
\\

\emph{Case 1. $|h| \ge  \la x\ra^{a} c_0$.} It suffices to bound $R_1(\e)(x)$ and $R_1(\e)(x+h)$ separately by $ |R(\e)(\cdot)| \lesssim_p |\e(\cdot)|^p$.

	 

If $|\e(x)| \ge \frac{1}{2} |W(x)|$, the trivial bound $|R_1(\e)(x)| \lesssim_p |W(x)|^{p} + |\e(x)|^p$ suffices. If $|\e(x)| \le \frac 12 |W(x)|$, from \eqref{eqFderivative}
we know 
\[ \sup_{t \in [0, 1]} \left| \nabla^2 F_1 (\Sigma(x) + t\e_1(x), \Theta(x) + t\e_2(x))\right| \lesssim_p |W(x)|^{p-2}. \]
Thus \eqref{eqI1derivative} yields
\[ |R_1(\e)(x)| \lesssim |W(x)|^{p-2} |\e(x)|^2 \lesssim_p | \e(x)|^{p}. \]
 
 \mbox{}
 
\emph{Case 2. $|h| \le  \la x\ra^{a}c_0$.} We consider the following three subcases.

\mbox{}

\emph{Case 2.1. $|\e(x)| \ge \frac{1}{6} |W(x)|$.} Since $|\e(x)| \gtrsim |W(x)| \sim |W(x+h)|$, we can estimate most interaction terms in a direct fashion
\bee 
   && |\Delta_h(|W|^{p-1}\Sigma)|+ |\Delta_h(|W|^{p-1}\e_1)| + |\Delta_h (|W|^{p-3}(\Sigma^2 - \Theta^2)\e_1)| \\
   &&\qquad + |\Delta_h (|W|^{p-3}\Sigma \Theta \e_2)| \\
   &\lesssim_{p, C_W}& |h| \la x\ra^{-a} |W(x)|^{p-1} \left(|\e(x)| + |\e(x+h)| +|W(x)|\right) + |W(x)|^{p-1}|\Delta_h \e(x)| \\
   & \lesssim_p& |h| \la x\ra^{-a}\left( |\e(x)|^p + |\e(x+h)|^p \right) + |\e(x)|^{p-1} |\Delta_h \e(x)|  \eee
 where we used $|\nabla W(x+th)| \le C_W  \la x+th\ra^{-a} |W(x+th)| \sim  \la x\ra^{-a}|W(x)|$ for $t \in [0, 1]$ to bound the difference of $W$. For the first term, we exploit the elementary estimate
 \be  \left| \frac{A^\a - B^\a}{A-B} \right| \sim_\a \max\{ A, B \}^{\a - 1}, \quad \a > 0  \label{eqalphadif} \ee
 for $A, B > 0$. Thus 
 \bee 
  &&\left|\Delta_h \left(|W+\e|^{p-1}(\Sigma + \e_1)\right)\right| \\
  & \le& \left| \frac{\Delta_h |W+\e|^{p-1}}{\Delta_h |W+\e|} \left(\Delta_h |W+\e| \right) (\Sigma + \e_1)(x+h)  \right| + |W+\e|^{p-1}(x) \left| \Delta_h (\Sigma + \e_1) \right| \\
  &\lesssim_{p, C_W}& \left( |W+\e|^{p-1}(x+h) +  |W+\e|^{p-1}(x)\right) \left( |h|  \la x\ra^{-a} |W(x)| + |\Delta_h \e| \right)  \\
  &\lesssim_{p}& |h| \la x\ra^{-a} \left( |\e(x)|^p + |\e(x+h)|^p \right) + \left( |\e(x)|^{p-1} + |\e(x+h)|^{p-1} \right)|\Delta_h \e(x)|.
 \eee
 These estimates implies both \eqref{eqdif1}.\\
   
\emph{Case 2.2. $|\e(x)| \le \frac{1}{6}|W(x)| \le |\e(x+h)|$.} Noticing the symmetry between $x$ and $x+h$ term and that $|\e(x+h)| \gtrsim |W(x)| \sim |W(x+h)|$, the bound in this case follows exactly the same argument as Case 2.1 above.\\

\emph{Case 2.3. $|\e(x)| \le \frac{1}{6}|W(x)|$, $|\e(x+h)| \le \frac{1}{6}|W(x)|$.} 
In the following computation, we slightly abuse notation to denote $W := W(x)$, $W' := W(x+h)$ and similarly for $\e$ and $\e'$. Applying \eqref{eqI1derivative}, we get 
\bee
  && R_1(\e)(x+h) - R_1(\e)(x) \\
  &=& \sum_{ j,k= 1}^2\Bigg[ \e_j' \e_k' \int_0^1 \pa_{j}\pa_k F_1 (W'+t\e') (1-t) dt -  \e_j \e_k \int_0^1 \pa_{j}\pa_k F_1 (W + t\e) (1-t) dt \Bigg] \\
   &=& \sum_{ j,k= 1}^2 \left( \Delta_h(\e_j \e_k) \right)  \int_0^1 \pa_{j}\pa_k F_1 (W'+t\e') (1-t) dt \\
    &+& \sum_{ j,k= 1}^2 \e_j \e_k \int_0^1 (1-t) \int_0^1 [(W'-W)+t(\e' - \e)] \cdot \nabla \pa_j \pa_k F_1 \big((1-s)(W+t\e) + s (W' + t\e')\big)  ds dt.
\eee
Thanks to \eqref{eqWsim} and the smallness of $\e(x)$, $\e(x+h)$, we have
\[ \left|(1-s)(W+t\e) + s (W' + t\e') - W  \right| \le \frac 12 |W(x)|,\quad \forall \, t, s \in [0, 1], \]
and hence all the derivatives acting on $F_1$ are well-defined. This and \eqref{eqFderivative} for $k = 3$ imply the boundedness 
\bee |\Delta_h R_1(\e)(x)| \lesssim_{p, C_W} |\Delta_h \e|(|\e'| + |\e|) |W|^{p-2}  
+ |\e|^2 \left( |h| \la x\ra^{-a} |W|
+ |\Delta_h \e| \right) |W|^{p-3}.  
\eee
which leads to the bounds \eqref{eqdif1}.\\

\underline{Step 2. Littlewood-Paley and maximal function estimates.} Next, we apply the argument in \cite{MR1766415} to derive estimates in fractional Sobolev spaces. Denote the homogeneous Littlewood-Paley decomposition as 
\[ P_k f = \mathscr{F}^{-1} (\psi_k\cdot) \hat{f})= \check{\psi}_k  * f,\quad k \in \ZZ, \]
and the Hardy-Littlewood maximal operator as $M$. Since $\psi_k = \psi ( 2^{-k} \cdot)$ with $\psi \in \C^\infty_c(B_2 -B_{\frac12})$ for $k \in \ZZ$, we have $\check{\psi}_k = 2^{kd} \check{\psi}(2^{k}\cdot) \in \mathscr{S}(\RR^d)$ and 
\bea
  |\check{\psi}_k(x)| &\lesssim_{d, m}& 2^{dk}\la 2^{k}x\ra^{-m},\quad \forall \, m \ge 0, \label{eqLP1}\\
  \int_{\RR^d} \psi_k (x) dx &=& 0.\label{eqLP2}
  \eea
The Littlewood-Paley theory indicates 
\be  \left\| f \right \|_{\dot{W}^{s, r}} \sim_{d, s, r} \left \| \sum_{j = -\infty}^{+\infty}2^{js}P_j f \right \|_{L^r} \sim_{d, s, r} \left \| \left( \sum_j 2^{2js} |P_j f(\cdot)|^2\right)^{\frac 12} \right \|_{L^r}  \label{eqLP}\ee
for $s \in (0, 1)$ and $r \in (1, \infty)$. We recall the following estimates involving maximal functions from 
\cite[Chap. 2 Lemma 3.1-3.5]{MR1766415}
\bea 
 &\begin{split}\int |P_k u(x) - P_k u(y)| \cdot |\check{\psi}_j (x-y)| dy \lesssim_d \min \{ 2^{k-j}, 1\}M(P_k u)(x), \end{split}\label{eqLPmax1}\\
 &\begin{split}
 &\int |P_k u(x) - P_k u(y)| \cdot |\check{\psi}_j (x-y)| H(y) dy \\
 &\qquad \lesssim_d \min \{ 2^{k-j}, 1\}\left[ M(P_k u)(x) MH(x) + M(|P_k u|H)(x) \right]
 \end{split}\label{eqLPmax2}
\eea 
where $H$ is a positive function and $k, j \in \ZZ$.

Now we derive a pointwise estimate for the function in \eqref{eqLP} and finish the proof. \\

Using the vanishing condition \eqref{eqLP2} and the pointwise estimates \eqref{eqdif1}, we have
\bea 
 &&\sum_j 2^{2js} |P_j R(\e)(x)|^2 
 \le \sum_{j} 2^{2js} \left[\int | R(\e)(y) - R(\e)(x)| |\check{\psi}_j(x-y)| dy \right]^2 \nonumber \\
 &\lesssim &\sum_j 2^{2js} \left[ \int \left(|\e(y)|^p + |\e(x)|^p\right) \min \{ 1, \la x \ra^{-a} |x-y|\} |\check{\psi}_j(x-y)| dy\right]^2 \label{eqdifterm1}\\ 
 &+& \sum_j 2^{2js} \left[\int \left(|\e(y)|^{p-1} + |\e(x)|^{p-1}\right) |\e(x) - \e(y)| |\check{\psi}_j(x-y)| dy\right]^2 \label{eqdifterm2}
\eea

For \eqref{eqdifterm1}, we claim
\be \eqref{eqdifterm1} \lesssim (M(|\e|^{p})(x))^2 \la x\ra^{-2as}. \ee
Indeed, the decay \eqref{eqLP1} implies
	\[ \int|\check{\psi}_j(x-y)| \min\{ 1, \la x \ra^{-a} |x-y|\}dy \lesssim \min \{2^{-j} \la x\ra^{-a},1\}. \]
	Then the bound on $|\e(x)|^p$ term follows
	\[ \sum_{j\in \ZZ} 2^{2js} \min \{2^{-2j} \la x \ra^{-2a},1\} \sim \la x \ra^{-2as} \]
thanks to $s \in (0, 1)$. Likewise, for the other term involving $|\e(y)|^p$, it suffices to show the maximal function bound
\[ A_j (x) := \int |\e(y)|^p|\check{\psi}_j(x-y)| \min\{ 1,  \la x \ra^{-a} |x-y|\}dy \lesssim \min \{2^{-j}\la x \ra^{-a},1\} M(|\e|^p)(x). \]
Let $j_0 \in \ZZ$ such that $2^{j_0} \le \la x \ra^{-a} < 2^{j_0 + 1}$. Take $m = d+2$, then for $j \le j_0$, 
\bee A_j(x) &\lesssim& \int_{|z| \le 2^{-j}}|\e(x-z)|^p 2^{dj} dz +  \sum_{k \ge -j} \int_{2^k \le |z| \le 2^{k+1}} |\e(x-z)|^p 2^{dj} 2^{-(j+k)m} dz \\ &\lesssim& M(|\e|^p)(x) \left( 1 + \sum_{k \ge -j}2^{(d-m)(j+k)} \right) \lesssim M(|\e|^p)(x); 
\eee
and for $j > j_0$, 
\bee A_j(x)  &\lesssim&  \int_{|z| \le 2^{-j}} |\e(x-z)|^p 2^{dj}2^{j_0-j} dz 
 + \sum_{k = -j}^{-j_0} \int_{2^k \le |z| \le 2^{k+1}} |\e(x-z)|^p 2^{dj}2^{j_0+k}2^{-m(j+k)} dz \\
&+& \sum_{k \ge -j_0+1} \int_{2^k \le |z| \le2^{k+1}} |\e(x-z)|^p 2^{dj}2^{-m(j+k)} dz \\
&\lesssim&  2^{j_0-j}M(|\e|^p)(x).  \eee

For \eqref{eqdifterm2}, we apply Littlewood-Paley decomposition to $\e(x) - \e(y)$ and then use \eqref{eqLPmax1}-\eqref{eqLPmax2} to get
\bee \eqref{eqdifterm2} &\lesssim& \sum_{j} 2^{2js} \left| \sum_k \int \left(|\e(y)|^{p-1} + |\e(x)|^{p-1}\right) |P_k \e(x) - P_k \e(y)| |\check{\psi}_j(x-y)| dy \right|^2\\ 
&\lesssim &\sum_{j} 2^{2js} \left[\sum_k \min \{ 2^{k-j}, 1\}  M(P_k \e)(x) M(|\e|^{p-1})(x) \right]^2 \\
&&+\sum_{j} 2^{2js} \left[\sum_k \min \{ 2^{k-j}, 1\} M(|P_k \e|\cdot |\e|^{p-1} )(x) \right]^2  \\
&\lesssim &  \left[M(|\e|^{p-1})(x)\right]^2 \sum_k 2^{2ks} \left[M(P_k \e)(x)\right]^2 +  \sum_k 2^{2ks} \left[ M(|P_k \e|\cdot |\e|^{p-1} )(x)\right]^2 
\eee
where the last inequality follows Young's inequality. 

From \eqref{eqLP} and these estimates of \eqref{eqLP1}-\eqref{eqLP2}, we bound the $\dot{W}^{s, r}$ norm by
\bee 
 && \| R(\e) \|_{\dot{W}^{s, r}}\\
   &\lesssim &\left\| M(|\e|^p) \la x\ra^{-as} \right\|_{L^r} +  \left \| M(|\e|^{p-1})(x)  \left\| 2^{ks} M(P_k \e)(x)\right\|_{l^2_k} \right\|_{L^{r}}\\
   &&+ \left \|  \left\| 2^{ks} M(|P_k \e|\cdot |\e|^{p-1} )(x)\right \|_{l^2_k} \right\|_{L^{r}}  \\
 & \lesssim &\| |\e|^p \|_{L^{\tilde{r}}} + \| \e \|_{L^{(p-1)r_1}}^{p-1} \left  \| \left\| 2^{ks} |P_k \e(x)|\right\|_{l^2_k}  \right \|_{L^{r_2}} +  \left \|  |\e|^{p-1}(x) \left\| 2^{ks} |P_k \e|(x) \right\|_{l^2_k} \right\|_{L^{r}} \\
  &\lesssim & \| |\e|^p \|_{L^{\tilde{r}}} +  \| \e \|_{L^{(p-1)r_1}}^{p-1} \| \e \|_{\dot{W}^{s, r_2}}
\eee
with $\frac{1}{r} - \frac{1}{\tilde{r}} \in \left [0, \frac{as}{d}\right )$, $\frac{1}{r} = \frac{1}{r_1} + \frac{1}{r_2}$ and $(p-1)r_1 > 1$. In the second-to-last inequality, we invoked the Fefferman-Stein maximal theorem \cite{MR284802}
\[ \left\| \|M(h_k)(x)\|_{l^2_k}  \right\|_{L^q_x} \lesssim \left\| \|h_k(x)\|_{l^2_k}  \right\|_{L^q_x} \]
for a sequence of functions $\{ h_k\}$ and $q \in (1, \infty)$.
\end{proof}

Next, we present a difference estimate in $L^p$ space. 
	\begin{lemma}\label{lemnondif}
	Given $d \ge 1$ and $W \in C^1(\RR^d \to \CC)$, $p > 1$ and $u, v \in \mathscr{S}(\RR^d \to \CC)$, for $R(u)$ defined as in \eqref{eqdefI},
	we have difference estimate for $r, r_1, r_2 \in [1, \infty]$ with $\frac{1}{r} = \frac{1}{r_1} + \frac{p-1}{r_2}$:
	\begin{itemize}
		\item If $1 < p \le 2$, 
		\be \| R(u) - R(v)  \|_{L^r} \lesssim_{d, p, r, r_1, r_2}
		\| u -v \|_{L^{r_1}}\left( \| u \|_{L^{r_2}}^{p-1} + \| v \|_{L^{r_2}}^{p-1} \right);
		\label{eqfracdif3} \ee
		\item If $p > 2$,
		\be 
		\begin{split} &\| R(u) - R(v)  \|_{L^r} \\
			\lesssim_{d, p, s, r, r_1, r_2}&
			\| u -v \|_{L^{r_1}}\left( \| u \|_{L^{r_2}}^{p-1} + \| v \|_{L^{r_2}}^{p-1} + \left( \| u \|_{L^{r_2}} + \| v \|_{L^{r_2}} \right) \| W \|_{L^{r_2}}^{p-2}\right). 
		\end{split} \label{eqfracdif4} 
		\ee
	\end{itemize}
\end{lemma}
\begin{proof}
	It suffices to prove the corresponding pointwise estimate for any $x$
	\bee |R(u) - R(v)| \lesssim_p \left\{ \begin{array}{ll}
		|u-v| (|u|^{p-1} + |v|^{p-1}) & 1 < p \le 2, \\
		|u-v| (|u|^{p-1} + |v|^{p-1} + (|u|+|v|)|W|^{p-2}) &  p \ge 2.
	\end{array} \right. \eee
 The proof is similar to Step 1 in Lemma \ref{lemfracdif} and thus is omited.
	\end{proof}

Finally, we give a pointwise derivative estimate for nonlinearity.
\begin{lemma}\label{lemnonestderiv}
    Let $d \ge 1$ and $W \in C^1(\RR^d \to \CC - \{ 0 \})$ satisfying \eqref{eqslowvar} with $a \in [0, 1]$, $p > 1$ and $u, v \in \mathscr{S}(\RR^d \to \CC)$, for $R(u)$ defined as in \eqref{eqdefI}. Then we have 
    \be
      |\nabla R(\e)| \lesssim \left| \begin{array}{ll}
          |\e|^{p-1} \cdot(|\nabla \e| + |\e| \la x \ra^{-a}) &  1 < p \le 2\\
          (|\e|^{p-2} + |W|^{p-2})\cdot|\e|\cdot(|\nabla \e| + |\e| \la x \ra^{-a})  &  p \ge 2
      \end{array}\right.
    \ee
\end{lemma}
\begin{proof}
    It is straightforward computation of Lebnitz rule after decomposing into two cases $|\e(x)| \le \frac 12|W(x)|$ or $|\e(x)| \ge \frac 12|W(x)|$.
\end{proof}

\section{Local Well-posedness in $\dot{H}^\sigma$}\label{appD}

In this section, we prove local well-posedness of the intercritical Schr\"odinger equation \eqref{eqNLS}
\bee \left\{\begin{array}{l} i\pa_t u+\Delta u+u|u|^{p-1}=0\\ u_{|t=0}=u_0 \in \dot{H}^\sigma \end{array}\right. \eee 
  and the renormalized system \eqref{eqZ} (we denote the time variable as $t$)
 \bee
   \left \{ \begin{array}{l}
    i \partial_t Z + \mathcal{H}Z = F, \\
    Z_{|t=0} = Z_0 \in \dot{H}^\sigma \times \dot{H}^\sigma. 
    \end{array}\right.
 \eee
 with $\calH$, $ F = \left( \begin{array}{c} -R(\e, Q_b) \\ \overline{R(\e, Q_b)} \end{array} \right)$ defined by \eqref{Hb}, \eqref{potential} and \eqref{eqNL}.
 Here we assume $s_c < \sigma \le \a_c = \frac{d}{d+2-2s_c}$ above the critical scaling. We will also use the exponents $r_1, r_2, q_1, \beta$ as in \eqref{eqr1r2q1}-\eqref{eqbeta}. We begin with the original system \eqref{eqNLS}.

\begin{proposition}\label{proplwp}
	Suppose $d \ge 1$ and $0 < s_c  < \min \left \{ 1, \frac d2\right \}$ and consider the Cauchy problem \eqref{eqNLS}. For any $s_c < \sigma \le \a_c$, we have the following local well-posedness result: when $d \neq 2$, for any initial data $u_0 \in \dot{H}^\sigma$, there exists $T = T(\| u \|_{\dot{H^\sigma}})$ such that \eqref{eqNLS} has a unique solution
	\[  u \in 
		C^0_t \dot{H}^\sigma_x  \cap L^2_t \dot{W}_x^{\sigma, 2^*}([-T, T] \times \RR^d) =: Y_d; \]
		when $d = 2$, for any $u_0 \in \dot{H}^\sigma$ and $r \in [2, \infty)$, there exists $T = T(\| u \|_{\dot{H^\sigma}}, r)$ such that 
		\eqref{eqNLS} has a unique solution
		\[  u \in 
		C^0_t \dot{H}^\sigma_x \cap L^{r'}_t \dot{W}_x^{\sigma, r}([-T, T] \times \RR^2) =: Y_2. \] 
		Moreover, $\| u \|_{Y_d} \lesssim \| u_0\|_{\dot{H}^\sigma}$ and the solution map $u_0 \mapsto u$ is Lipschitz.
	\end{proposition}
\begin{proof}
	The proof follows standard Banach fixed point argument and the $\dot{H}^\sigma$-subcritical nature of \eqref{eqNLS}. For clarity, we only prove the case $d \neq 2$. The $d =2$ case is no more than a substitution by non-endpoint Strichartz norm. 
	
	Define
	\[  \Psi_{u_0}(u) := e^{it\Delta} u_0 + i \int_0^t e^{i(t-s)\Delta} (|u|^{p-1}u)(s)ds. \]
	We will show $\Psi_{u_0}$ is contractive in some suitable metric space. 

    Let $M: = \| u_0\|_{\dot{H}^\sigma}$. Define the closed ball
	\be  X_{T, M}:= \left \{ u \in C^0_t \dot{H}^\sigma_x \cap L^2_t \dot{W}_x^{\sigma, 2^*}([-T, T] \times \RR^d): \| u \|_{L^\infty_t \dot{H}^\sigma_x \cap L^2_t \dot{W}_x^{\sigma, 2^*}} \le 2C_0 M \right  \}\label{eqball}\ee
	and metric
	\be
	\mathfrak{d}_T (u, v):= \| u-v\|_{L^{q_0}_t L^{r_0}_x([-T, T] \times \RR^d)}\label{eqmetric}\ee
	where $C_0$ is the constant in the Strichartz estimate for $e^{it\Delta}$. 
	Here $r_0 \in (2, 2^*)$ satisfies 
	\be \frac d2 - \frac{d}{r_0} =\frac{\frac d2 - \sigma}{\frac d2 - s_c}\label{eqrsigma} \ee
	and $q_0 = 2\frac{\frac d2 - s_c}{\frac d2 - \sigma}$ such that $(q_0, r_0)$ is Strichartz admissible. 
	From $s_c < \sigma \le \a_c = \frac{d}{d+2-2s_c}$, we know $\sigma \le \mathrm{RHS \,\, of\,\,}\eqref{eqrsigma} < \min \left \{ 1, \frac d2\right \}$ and thus $ (\dot{H}^\sigma \cap \dot{W}^{\sigma, 2^*})(\RR^d) \hookrightarrow L^{r_0}(\RR^d)$. Consequently, $(X_{T,M}, \mathfrak{d}_T)$ is a complete metric space for $T < \infty$. 
	
	We first check $\Psi_{u_0} : X_{T,M} \to X_{T,M}$. 
	By Strichartz estimate, 
	\bee 
	  \| \Psi_{u_0}(u) \|_{L^\infty_t \dot{H}^\sigma_x \cap L^2_t \dot{W}_x^{\sigma, 2^*}} 
	  \le C_0 \| u_0 \|_{\dot{H}^\sigma} + C_1 \left \| D^{\sigma} (|u|^{p-1}u) \right \|_{L_t^{q_1'} L^{r_1'}_x}.
	\eee
	Then by fractional chain rule for $C^1$ nonlinearity (see \cite{MR1124294}), we get
	\bee
	\left \| D^{\sigma} (|u|^{p-1}u) \right \|_{L_t^{q_1'} L^{r_1'}_x} \lesssim \left\| \| u\|_{\dot{W}_x^{\sigma, r_1}} \| u\|_{L^{pr_2}_x}^{p-1} \right\|_{L^{q_1'}_t} 
	\lesssim  \|u \|_{L^{q_1}_t \dot{W}^{\sigma, r_1}_x}^p  T^{\beta} \lesssim M^p T^{\beta} 
	\eee
    Thus the onto property of $\Psi_{u_0}$ is guaranteed by taking $T \sim \| u_0 \|_{\dot{H}^\sigma}^{-\frac{p-1}{\beta}} = M^{-\frac{p-1}{\beta}} $.

	
	Next, we check that $\Psi_{u_0}$ is a contraction concerning the metric $\mathfrak{d}_T$. Taking $\tilde{r}_0$ as $\frac{1}{\tilde{r}_0} = \frac{1}{r_0'} - \frac{1}{r_0}$, we then have $ \dot{H}^{\sigma} \hookrightarrow L^{\tilde{r}_0 (p-1)}$. Thus
	\bee && \| \Psi_{u_0} (u) - \Psi_{u_0}(v) \|_{L^{q_0}_t L^{r_0}_x} \lesssim \| |u|^{p-1} u - |v|^{p-1} v \|_{L^{q_0'}_t L^{r_0'}_x} \\
	&\lesssim  &\left\| \| u-v\|_{L_x^{r_0}} \left( \| u \|_{L_x^{\tilde{r}_0(p-1)}}^{p-1} + \| v \|_{L_x^{\tilde{r}_0(p-1)}}^{p-1} \right)\right\|_{L^{q_0'}_t}\\
	&\lesssim &\| u-v \|_{L^{q_0}_t L^{r_0}_x}\left ( \| u \|_{L^\infty_t \dot{H}^\sigma_x}^{p-1} + \| v \|_{L^\infty_t \dot{H}^\sigma_x}^{p-1}  \right) T^{\beta} \lesssim \| u-v \|_{L^{q_0}_t L^{r_0}_x} \| u_0 \|_{\dot{H}^\sigma}^{p-1} T^{\beta}.
		\eee
	Here the same $\beta$ comes from $\frac{1}{q_0'} - \frac{1}{q_0} = \frac{\sigma-s_c}{\frac d2 - s_c}= \beta$. By further shrinking $T$ by a constant if necessary, we can conclude that $\Psi_{u_0}$ is a contraction on $(X_{T,M}, \mathfrak{d}_T)$ for $d \neq 2$. We omit the standard arguments for the rest statements in this proposition.
		\end{proof}

Next, we derive a similar local-wellposedness result of the renormalized system \eqref{eqZ}.

\begin{proposition} \label{proplwp2}
		Suppose $d \ge 1$ and $0 < s_c  < \min \left \{ 1, \frac d2\right \}$ and consider the Cauchy problem \eqref{eqZ}. For any $s_c < \sigma \le \a_c$, we have the following local well-posedness result: when $d \neq 2$, for any initial data $Z_0 \in (\dot{H}^\sigma)^2$, there exists $T = T(\| Z_0 \|_{(\dot{H^\sigma})^2})$ such that \eqref{eqZ} has a unique solution
	\[  Z \in 
	C^0_t ([-T, T], (\dot{H}^\sigma_x(\RR^d))^2) \cap L^2_t ([-T, T], (\dot{W}_x^{\sigma, 2^*}(\RR^d))^2) =: \tilde{Y}_d; \]
	when $d = 2$, for any $Z_0 \in (\dot{H}^\sigma)^2$ and $r \in [2, \infty)$, there exists $T = T(\| Z_0 \|_{(\dot{H^\sigma})^2}, r)$ such that 
	\eqref{eqNLS} has a unique solution
	\[  Z \in 
		C^0_t ([-T, T], (\dot{H}^\sigma_x(\RR^d))^2) \cap L^{r'}_t ([-T, T], (\dot{W}_x^{\sigma, r}(\RR^d))^2) =: \tilde{Y}_2. \] 
	Moreover, $\| Z \|_{\tilde{Y}_d} \lesssim \|Z_0 \|_{(\dot{H}^\sigma)^2} $ and the solution map $Z_0 \mapsto Z$ is Lipschitz.
	\end{proposition}

\begin{proof}
	This is still a contraction argument as above with estimates similar to the proof of Proposition \ref{proplwp} and Theorem \ref{thmfincodimstabHsigmaB}. Again, for simplicity, we only discuss $d \neq 2$ and $1 < p \le 2$ case. 
	
	We view the potential as a source term and write 
	\bee \Psi_{Z_0}(Z) := e^{i\tau(\calH_b - ibs_c)}Z_0 + i \int_0^\tau e^{i(\tau - \tau') (\calH_b - ibs_c)} (F - VZ)(\tau')d\tau'.  \eee
    As in Proposition \ref{proplwp}, we take a distance given by the weaker norm $L^{q_0}_\tau L^{r_0}_y$. 
	
	Since we only consider well-posedness local in time, we may assume $T_0 \le 1$. Thereafter, all exponential timewise weight can be neglected and $e^{it(\calH_b - ibs_c)}$ has the Strichartz estimates in both $\dot{H}^\sigma$ and $L^2$.

	Using the nonlinear estimate from Lemma \ref{lemfracdif} and linear estimate Lemma \ref{lempotbdd} for potential, we get the estimate similar to \eqref{eqFnonlinear1} (with exponents borrowed there)
	\bee \| F \|_{\dot{W}^{\sigma, r_1'}} &\lesssim&  \| Z \|_{\dot{W}^{\sigma, r_1} \cap \dot{W}^{\sigma, r_1^-}}^p,\quad \| VZ \|_{\dot{W}^{\sigma, r_1'}} \lesssim \| Z \|_{ \dot{W}^{\sigma, r_1}}. \eee
	Like \eqref{eqnonlinearF}, the space-time estimate follows
	\bee
	\|F -VZ\|_{L^{q_1'}_\tau \dot{W}^{\sigma, r_1'}_y} 
	 \lesssim   (\|Z \|_{\tilde Y_d}^p + \|Z \|_{\tilde Y_d}) T^{\beta}
	\eee
	Now the Strichartz estimate for $e^{it(\calH_b-ibs_c)}$ in $\dot{H}^\sigma$ yields the onto estimate straightforwardly.
	
	As for the contraction estimate, notice that  
	\bee |R(Z_1, Q_b) - R(Z_1', Q_b)|
	 \lesssim  |Z_1 - Z_1'|(|Z_1|^{p-1} + |Z_1'|^{p-1} + |Q_b|^{p-1}). \eee
	Then mimicking Proposition \ref{proplwp}, we get
	\bee 
	\| (F-VZ) - (F'-VZ')\|_{L^{q_0'}_\tau L^{r_0'}_y} 
	\lesssim \| Z - Z'\|_{L^{q_0}_\tau L^{r_0}_y} \left ( \| Z \|_{L^\infty_\tau \dot{H}^\sigma_y}^{p-1} + \| Z' \|_{L^\infty_\tau \dot{H}^\sigma_y}^{p-1} + \| Q_b \|_{\dot{H}^\sigma}^{p-1}  \right) T^{\beta}
	\eee
	which suffices to conclude the contraction estimate. 
	
\end{proof}

\bibliographystyle{plain}
\bibliography{Bib-1}

\end{document}